%% file: IntersectionArxiv04.tex
\documentclass[11pt]{memo-l}
\usepackage[all]{xy}
\usepackage{amsmath}
\usepackage[psamsfonts]{amssymb}
\usepackage[psamsfonts]{eucal}
\usepackage{amsthm}
\usepackage[scaled]{helvet}
\usepackage{color}

\usepackage{geometry}
\usepackage{enumerate}
\usepackage{graphicx}
\usepackage{tikz}
\usetikzlibrary{arrows}
\theoremstyle{plain}
\newtheorem{theorem}{Theorem}

\newtheorem{proposition}{Proposition}[chapter]
\newtheorem{theoremb}[proposition]{Theorem}
\newtheorem{lemma}[proposition]{Lemma}
\newtheorem{corollary}[proposition]{Corollary}

\theoremstyle{definition}
\newtheorem{definition}[proposition]{Definition}
\newtheorem{example}[proposition]{Example}

\theoremstyle{remark}
\newtheorem{remark}[proposition]{Remark}

\newcommand{\chapref}[1]{Chapter~\ref{#1}}

\newcommand{\secref}[1]{Section~\ref{#1}}

\newcommand{\thmref}[1]{Theorem~\ref{#1}}
\newcommand{\propref}[1]{Proposition~\ref{#1}}
\newcommand{\lemref}[1]{Lemma~\ref{#1}}
\newcommand{\corref}[1]{Corollary~\ref{#1}}

\newcommand{\remref}[1]{Remark~\ref{#1}}
\newcommand{\exemref}[1]{Example~\ref{#1}}

\newcommand{\defref}[1]{Definition~\ref{#1}}

\def\R{{\mathbb R}}

\def\ov{\overline}
\def\ob{\underline}

\def\cA{{\mathcal A}}
\def\cB{{\mathcal B}}

\def\cD{{\mathcal D}}
\def\cE{{\mathcal E}}
\def\cF{{\mathcal F}}

\def\cI{{\mathcal I}}
\def\cJ{{\mathcal J}}

\def\cL{{\mathcal L}}
\def\cM{{\mathcal M}}
\def\cN{{\mathcal N}}

\def\cP{{\mathcal P}}

\def\cR{{\mathcal R}}
\def\cS{{\mathcal S}}
\def\cU{{\mathcal U}}
\def\cV{{\mathcal V}}

\def\cZ{{\mathcal Z}}
\def\Bl{{\mathcal Bl}}

\def\gN{{\mathfrak{N}}}
\def\tS{{\tt S}}
\def\tn{{\tt n}}
\def\tm{{\tt m}}
\def\B{\mathbb{B}}
\def\C{\mathbb{C}}

\def\L{\mathbb{L}}
\def\N{\mathbb{N}}
\def\Q{\mathbb{Q}}
\def\R{\mathbb{R}}
\def\Z{\mathbb{Z}}
\def\CP{{\mathbb{C}{\rm P}}}

\def\ker{{\rm Ker\,}}
\def\coker{{\rm Coker\,}}
\def\im{{\rm Im\,}}

\def\ext{{\rm Ext}}
\def\Tor{{\rm Tor}}

\def\id{{\rm id}}

\def\ev{{\rm ev}}

\def\dga{{{\rm{DGA}}}}
\def\cdga{{{\rm{CDGA}}}}
\def\cga{{{\rm{CGA}}}}
\def\cdgaf{{{\rm CDGA}_\cF}}
\def\ch{{\rm Ch}}
\def\pr{{\rm pr}}

\def\tC{{\widetilde{C}}}
\def\ttC{\ov{C}}

\def\ttH{\ov{H}}

\def\tF{{\widetilde{F}}}
\def\tDelta{{\widetilde{\Delta}}}
\def\tA{{\widetilde{A}_{PL}}}

\def\tG{{\widetilde{G}}}
\def\ttau{{\widetilde{\tau}}}
\def\fil{{{\pmb\Delta}^{[n]}_\cF}{ -Sets}}

\def\dset{{{\pmb\Delta}}{ -Sets}}
\def\cLe{{\cL}^{\tt exp}}
\def\GM{{\mathrm{GM}}}
\def\TW{{\mathrm{TW}}}
\def\reg{{\mathrm{reg}}}
\def\ffs{{filtered face set}}
\def\ffss{{filtered face sets}}
\def\Ffss{{Filtered face sets}}
\def\pd{{\dotplus}}
\def\Th{{\mathrm{Th}}}
\def\colim{\qopname\relax m{colim}}

  \makeindex
\begin{document}

\frontmatter

\title[Intersection Cohomology. Simplicial Blow-up and Rational Homotopy]{Intersection Cohomology. \\ Simplicial Blow-up and Rational Homotopy}

\author{David Chataur}
\address{D\'epartement de Mathematiques\\
         UMR 8524 et F\'ed\'eration CNRS Nord-Pas-de-Calais FR 2956\\
         Universit\'e de Lille~1\\
         59655 Villeneuve d'Ascq Cedex\\
         France}
\email{David.Chataur@math.univ-lille1.fr}

\author{Martintxo Saralegi-Aranguren}
\address{Laboratoire de Math{\'e}matiques de Lens\\  
      EA 2462 et F\'ed\'eration CNRS Nord-Pas-de-Calais FR 2956\\
      Universit\'e d'Artois\\
         SP18, rue Jean Souvraz\\
          62307 Lens Cedex\\
         France}
\email{saralegi@euler.univ-artois.fr}

\author{Daniel Tanr\'e}
\address{D\'epartement de Mathematiques\\
         UMR 8524 et F\'ed\'eration CNRS Nord-Pas-de-Calais FR 2956\\
         Universit\'e de Lille~1\\
         59655 Villeneuve d'Ascq Cedex\\
         France}
\email{Daniel.Tanre@univ-lille1.fr}
\thanks{The third author is partially supported by the MICINN grant MTM2010-18089,  ANR-11-BS01-002-01 ``HOGT" and ANR-11-LABX-0007-01  ``CEMPI''}

\date{\today}

\begin{abstract}  
Let $X$ be a pseudomanifold. In this text, we use a simplicial blow-up to define a cochain complex whose cohomology with coefficients in a field, is isomorphic to the intersection cohomology of $X$, introduced by M.~Goresky and R.~MacPherson.

We do it simplicially in the setting of a filtered version of face sets, also called simplicial sets without degeneracies, in the sense of C.P.~Rourke and B.J.~Sanderson.  We define perverse local systems over filtered face sets and intersection cohomology with coefficients in a perverse local system. In particular, as announced above when $X$ is a pseudomanifold, we get a perverse local system of cochains quasi-isomorphic to the intersection cochains of Goresky and MacPherson, over a field. We show also that these two complexes of cochains are quasi-isomorphic to a filtered version of Sullivan's differential forms over the field $\Q$. In a second step, we use these forms to extend Sullivan's presentation of  rational homotopy type to intersection cohomology.

For that, we construct a functor from  the category of filtered face sets to a category of perverse commutative differential graded ${\mathbb{Q}}$-algebras (\cdga's) due to Hovey. 
We establish also the existence and unicity of a positively  graded, minimal model of some perverse \cdga's, including the perverse forms over a \ffs~and their intersection cohomology. 
Finally, we prove the topological invariance of the minimal model  of a  PL-pseudomanifold whose regular part is connected, and this theory creates new topological invariants.
This point of view brings a definition of formality in the intersection setting and examples are given. In particular, we show that
any nodal hypersurface in $\CP(4)$, is intersection-formal.
 \end{abstract} 
 
\maketitle

\clearpage
\thispagestyle{empty}
\vspace*{13.5pc}

\tableofcontents

\include{intro07Arxiv}

\mainmatter
\chapter{Simplicial blow-up}\label{chap:blowup}
\section{Filtered face sets}\label{sec:filteredfacesets}

\begin{quote} 
Singular homology and cohomology of a topological space, $X$, can be obtained simplicially, thanks to the simplicial set of continuous simplices. In the case of a filtered space, there is a notion of filtered simplices (\defref{def:filteredsimplex}) whose set, denoted by $\ob{\rm ISing}_*^{\cF}(X)$, generates a chain complex (\defref{def:chaineintersection}) giving the intersection homology of a pseudomanifold (\propref{prop:intersectionetintersection}). In this section, we abstract the properties of $\ob{\rm ISing}_*^{\cF}(X)$ in the concept of \ffs,  giving a simplicial setting to intersection homology and cohomology. We present also the simplicial notions of link and expanded link which are crucial in the proofs of Sections \ref{sec:filteredandlocalsystem}, \ref{sec:proofofC} and \ref{sec:chaincochain}.  
\end{quote}

Let us recall the basics on face sets, introduced by Rourke and Sanderson (\cite{MR0300281}).
For any integer $n$, we set $[n]=\{0,1,\ldots,n\}$.
Denote by $\pmb\Delta$ the category whose morphisms are the \emph{injective} order-preserving maps $f\colon [n]\to [m]$.
A morphism of $\pmb{\Delta}$ is a composition of face operators $\delta_i\colon [n]\to[n+1]$ defined by
$$\delta_i(j)=\left\{
\begin{array}{cl}
j&\text{ if } j<i,\\
j+1&\text{ if } j\geq i,
\end{array}\right.$$
which satisfy $\delta_j\delta_i=\delta_i\delta_{j-1}$, if $i<j$.

\begin{definition}\label{def:ensembleface}
A \emph{face set} 
\index{Face!set}
is a contravariant functor, $\ob{K}$, from $\pmb{\Delta}$ to the category of sets, $[n]\mapsto \ob{K}_n$.
An element $\sigma\in\ob{K}_n$ is called \emph{a simplex of dimension $n$,} that we denote $|\sigma|=n$. The \emph{$n$-skeleton} of $\ob{K}$ is $K^{(n)}=\cup_{p\leq n}\ob{K}_p$.
The images by  $\ob{K}$ of face  operators are maps $\partial_i\colon \ob{K}_{n+1}\to \ob{K}_n$   such that $\partial_i\partial_j=\partial_{j-1}\partial_{i}$, if $i<j$.

If $\ob{K}$ and $\ob{K}'$ are face sets, a \emph{face map,}
\index{Face!map}
 $\ob{f}\colon\ob{K}\to\ob{K}'$, is a natural transformation between the two functors $\ob{K}$ and $\ob{K}'$. We denote by 
$\pmb{\Delta}-{\rm Sets}$ the category of face maps between face sets.
\end{definition}

We emphasize that morphisms between face sets are called face maps and that we keep the expression ``face operators''
\index{Face!operator}
 for the maps $\partial_{i}$ or $\delta_{i}$.
 For basic properties of face sets and their relation with the more classical simplicial setting, we refer to the foundational paper \cite{MR0300281},
 observing that the homotopy theory of face sets is the same as the homotopy theory of simplicial sets.
 
\begin{example}\label{exam:simplexestandard}
The face set $\ob{\Delta}^n$ is defined as follows: $\ob{\Delta}^n_p$ is the set of injective order-preserving maps from $[p]$ to $[n]$. It has only one  simplex of  dimension $n$, denoted by $[\Delta^n]$.
If $\ob{K}$ is a face set, there is a bijection between the set of elements  $\sigma\in \ob{K}_n$ and the set of face maps, $\underline{\sigma}\colon \ob{\Delta}^n\to \ob{K}$, defined by $\underline{\sigma}(f)=\ob{K}(f)(\sigma)$. In the sequel, we do not make any distinction between $\underline{\sigma}\colon\ob{\Delta}^n\to \ob{K}$ and $\sigma\in\ob{K}_n$.
\end{example}
 
Recall from \cite{MR0300281}, that the realization of a face set,
\index{Face!set!realization of a} $\ob{K}$, is the CW-complex, $\|\ob{K}\|$, defined by
$$\|\ob{K}\|=\bigcup_{n=0}^{\infty} (\ob{K}_n\times \Delta^n)/(\partial_i x,t)\sim (x,\delta_{i} t),$$
where
\begin{itemize}
\item $\Delta^n$ is the standard $n$-simplex of $\R^{n+1}$, whose vertices $v_0,\ldots,v_n$ verify  $v_i=(t_0,\ldots,t_n)$, $t_j=0$ if $i\neq j$ and $t_i=1$,
\item 
the map $\delta_{i}\colon \Delta^{n-1}\to \Delta^n$ is the linear application defined by 
$\delta_{i}(v_j)=v_{\delta_i(j)}$.
\end{itemize}
For instance, the realization of the face set $\ob{\Delta}$ described in Example~\ref{exam:simplexestandard} is the standard simplex $\Delta^n$ and, in the sequel, we identify $[n]$, $\ob{\Delta}^n$ and $\Delta^n$.
The CW-complex $\|\ob{K}\|$ has one $n$-cell for each $\sigma\in \ob{K}_n$, of characteristic map
$\tilde{\sigma}\colon \Delta^n\to \|\ob{K}\|$ defined by
$\tilde{\sigma}(t_0,\ldots,t_n)=[\sigma,(t_0,\ldots,t_n)]$, where $[-]$ denotes the equivalence classes in the realization $\|\ob{K}\|$.

\medskip
We adapt now these objects to the context of intersection theory. First, \emph{we fix an integer $n$ which corresponds to the formal dimension of filtered spaces; that means, we fix the number of elements in the filtration to~$n$,} see  \remref{rem:toutendimensionn} for the topological meaning of this process.

\medskip
Let ${\pmb\Delta}^{[n]}_\cF$ be the category whose
\begin{itemize}
\item objects are the join $\Delta=\Delta^{j_0}\ast\Delta^{j_1}\ast\cdots\ast\Delta^{j_n}$, where $\Delta^{j_i}$  is the simplex of dimension $j_i$, possibly empty, with the conventions $\Delta^{-1}=\emptyset$ and  $\emptyset\ast X=X$,
\item maps are the $\sigma\colon \Delta=\Delta^{j_0}\ast\Delta^{j_1}\ast\cdots\ast\Delta^{j_n}\to\Delta'=\Delta^{k_0}\ast\Delta^{k_1}\ast\cdots\ast\Delta^{k_n}$, 
of the shape $\sigma=\ast_{i=0}^n\sigma_i$, with $\sigma_i\colon \Delta^{j_i}\to \Delta^{k_i}$ an injective order-preserving map for each~$i$ or the map $\emptyset\to \Delta^{k_{i}}$.
\end{itemize}

For all $0\leq i\leq n$, any face operator, $\delta_{\ell}\colon \Delta^{j_{i}}\to \Delta^{j_{i}+1}\in \pmb\Delta$, gives rise to a face operator in 
${\pmb\Delta}^{[n]}_\cF$, obtained from the join with the identity map, and still denoted
$\delta_{\ell}\colon
\Delta^{j_{0}}\ast\cdots\ast
\Delta^{j_{i}}\ast\cdots\ast \Delta^{j_{n}}
\to 
\Delta^{j_{0}}\ast\cdots\ast
\Delta^{j_{i}+1}\ast\cdots\ast \Delta^{j_{n}}$.

\medskip
The category
${\pmb\Delta}^{[n],+}_\cF$ is the full subcategory of ${\pmb\Delta}^{[n]}_\cF$ whose objects are the $\Delta^{j_0}\ast\Delta^{j_1}\ast\cdots\ast\Delta^{j_n}$ with
$\Delta^{j_n}\neq\emptyset$, i.e., $j_n\geq 0$.

\begin{definition}\label{def:filteredfaceset}
A \emph{filtered face set}
\index{Filtered!face set}
 is a contravariant functor, $\ob{K}$, from the category ${\pmb\Delta}^{[n]}_\cF$ to the category of sets, i.e., $(j_0,\ldots,j_n)\mapsto \ob{K}_{j_0,\ldots,j_n}$. Any face operator, $\delta_{\ell}\colon \Delta^{j_{i}}\to \Delta^{j_{i}+1}$, induces a face operator, defined by
$\partial_{\ell}\colon \ob{K}_{j_{0},\ldots, j_{i}+1,\ldots,j_{n}}\to \ob{K}_{j_{0},\ldots,j_{i},\ldots,j_{n}}$, $\partial_{\ell}(\sigma)=\sigma\circ\delta_{\ell}$.

The restriction of a filtered face set, $\ob{K}$, to ${\pmb\Delta}^{[n],+}_\cF$ is denoted $\ob{K}_+$. 
Morphisms between filtered face sets are natural transformations; we call them \emph{filtered face maps.}
\index{Filtered!face map}
\end{definition}

\emph{Face sets} are  \ffss~whose simplices,
$\sigma\colon \Delta^{j_0}\ast\cdots\ast\Delta^{j_n}\to \ob{K}$, 
are such that $j_i=-1$ for all $i<n$.

\begin{definition}\label{def:perversedegreesimplexffs}
The \emph{perverse degree} of a simplex,
\index{Perverse!degree!of a simplex}
$\sigma\colon \Delta^{j_0}\ast\Delta^{j_1}\ast\cdots\ast\Delta^{j_n}\to \ob{K}$, of a filtered face set, $\ob{K}$, is the $(n+1)$-uple,
$$\|\sigma\|=(\|\sigma\|_0,\ldots,\|\sigma\|_n),$$  
where $\|\sigma\|_{\ell}=\dim (\Delta^{j_0}\ast\cdots\ast\Delta^{j_{n-\ell}})$ 
if $\Delta^{j_0} \ast\cdots\ast\Delta^{j_{n-\ell}}\neq \emptyset$
and $\|\sigma\|_\ell=-\infty$ otherwise.
\end{definition}

Observe that \emph{any filtered face map preserves the perverse degree.}
The next example is the main motivation for the introduction of  filtered face sets.

\begin{example}\label{exam:simplexesfiltrés} 
In \secref{sec:filteredspaces}, for any filtered space, $X_{0}\subset X_{1}\subset \cdots\subset X_{n}=X$, we define a filtered singular simplex as a continuous map,
$\sigma\colon \Delta^m\to X$, such that each $\sigma^{-1}X_{i}$ is a face of $\Delta^m$ or is the empty set. Such a map induces a decomposition $\Delta^m=\Delta^{j_{0}}\ast\cdots\ast\Delta^{j_{n}}\to X$,
with
$\sigma^{-1}X_{i} = \Delta^{j_0}\ast\Delta^{j_1}\ast\cdots\ast\Delta^{j_i}$. (\exemref{exam:filteredsimplex} provides an illustration of this situation.) All together, these filtered singular simplices define a \ffs, $\ob{\rm ISing}_*^{\cF}(X)$, by
$$\ob{\rm ISing}_*^{\cF}(X)_{j_0,\ldots,j_n}=\{\sigma\colon \Delta^{j_0}\ast\cdots\ast \Delta^{j_n}\to X, \,\sigma \text{ continuous and filtered}\}.$$
Moreover, any stratum preserving stratified map,
$X\to Y$,
 induces a filtered face map, 
 $\ob{\rm ISing}_*^{\cF}(X)\to \ob{\rm ISing}_*^{\cF}(Y)$,
 see \thmref{thm:stratifiedmapcomplex} and \corref{cor:amalgamation2}.
\end{example}

\begin{example}\label{exam:siplicialcomplex}
The second barycentric decomposition of a triangulated pseudomanifold gives rise also to a \ffs, see \cite[Proposition 1.1.4]{MR1346255}.
\end{example}

\begin{remark}\label{rem:toutendimensionn}
Imposing a common formal dimension, $n$, to filtered face sets, is not a restriction for the study of intersection homology.
Observe, from \defref{def:chaineadmis}, that the admissibility of a simplex of a filtered space, $X$, depends only on the codimension in   $X$.
Therefore, as noticed by G.~Friedman in \cite{MR2009092}, one can add a fixed integer to each index of the filtration of a space without modifying its intersection homology. A practical consequence is that we can always suppose that the spaces we are considering have the same formal dimension without loosing generality. 

The restriction to the subcategory 
${\pmb\Delta}^{[n],+}_\cF$
is due to the fact that
we do not consider simplices entirely included in the singular part, see \remref{rem:petitemaisbien}. One may observe also 
(see \defref{def:blowup})
that the blow-ups are not defined if $\Delta^{j_n}=\emptyset$.
\end{remark}

In \cite[Theorem~7.1]{MR0646078}, D. Sullivan gives a short proof of de Rham's theorem for PL-forms, based on the fact that the \cdga~of forms on the relative skeleton $(X^{(p)},X^{(p-1)})$ ``breaks into a product over $p$ cells''. For the \ffss, we need the analogue of the previous relative skeleton, satisfying a breaking decomposition, as PL-forms do in the case of a simplicial set. \propref{prop:relativehomeo} reaches this objective and we introduce first the necessary tools, as  the filtered notions of 
skeleta and links. They coincide with usual notions when $n=0$. 
At the end of this section, we detail \exemref{exam:linkandothers} which illustrates these definitions and may  help to their understanding.

For face sets, a skeleton is characterized by an integer, the dimension of simplices. Here, for the \ffss, we need two integers. One is still the dimension of simplices and the other one, called the depth, refers to the decomposition of the domain.

\begin{definition}\label{def:depth}
The \emph{depth of the simplex $\sigma\colon\Delta^{j_0}\ast\cdots\ast\Delta^{j_n}\to\ob{K}$}
\index{Depth!of a simplex}
 is the integer $v(\sigma)\in\{0,\ldots,n\}$ defined by
$$v(\sigma)=n-\min\,\{a\mid \Delta^{j_{a}}\neq \emptyset\}
,$$
i.e., $v(\sigma)=n-k$ if, and only if, $j_{k}\geq 0$ and $j_{i}=-1$ if $i< k$.
The \emph{depth $v(\ob{K})$ of a filtered face set,}
\index{Depth!of a filtered face set}
 $\ob{K}$, is the maximum of the depth of its simplices.
\end{definition}

\begin{definition}\label{def:filteredskeleton}
Let $r\geq 0$ and $k\geq 0$. The \emph{filtered skeleta} of a \ffs,
\index{Filtered!skeleta}
 ${\ob{K}}$, are the \ffss, ${\ob{K}}^{[r],k}$, defined by 
$${\ob{K}}^{[r],k}=
\{{\sigma} \in\ob{K}\mid
v(\sigma)<r\text{ or } (v(\sigma)=r \text{ and } j_{n-v(\sigma)}\leq k)\},$$
i.e., $\sigma\colon \Delta^{j_{0}}\ast\cdots\ast\Delta^{j_{n}}\to \ob{K}^{[r],k}$ if, and only if,
$j_{n-r}\leq k$ and $j_{i}=-1$ if $i<n-r$.
We set ${\ob{K}}^{[r]}=\cup_{k\geq 0} {\ob{K}}^{[r],k}$ and, by convention,  ${\ob{K}}^{[r],-1}={\ob{K}}^{[r-1]}$, if $r>0$, and
${\ob{K}}^{[0],-1}=\emptyset$.
 The \emph{regular part} of $\ob{K}$ is the face set $\ob{K}^{[0]}$.
 \index{Regular!part of a filtered face set}
\end{definition}

\begin{definition}\label{def:ffsconected}
A \ffs, $\ob{K}$, is \emph{connected} if its regular part, $\ob{K}^{[0]}$, is a connected face set.
\index{Filtered!face set!connected}
\end{definition}

Recall that $K^{[0]}$ is said connected if, for any couple $(v,v')$ of vertices, there exists a finite chain of 1-simplices,
$(\sigma_{0},\ldots,\sigma_{k})$, such that $\partial_{1}\sigma_{i}=\partial_{0}\sigma_{i+1}$, for all $0\leq i<k$, 
$\partial_{0}\sigma_{0}=v$ and $\partial_{1}\sigma_{k}=v'$.

\smallskip
With \defref{def:filteredspaceconnected}, we observe that the \ffs, $\ob{\rm ISing}_*^{\cF}(X)$, associated to  a connected filtered space, $X$, is connected.

\smallskip
Let $\ob{K}$ be a filtered face set. For any $\sigma\colon \Delta=\Delta^{j_0}\ast\cdots\ast\Delta^{j_n}\to \ob{K}^{[r]}\backslash K^{[r-1]}$, we denote by 
\begin{itemize}
\item $R_1(\sigma)$ the restriction of $\sigma$ to $\Delta^{j_{n-r}}$,
\item $R_2(\sigma)$ the restriction of $\sigma$ to
$\Delta^{j_{n-r+1}}\ast\cdots\ast\Delta^{j_n}$.
\end{itemize}
Observe  that $R_2$ can be written as a composite of face operators and that the domain of $R_2(\sigma)$ can be empty.
We denote by $\cJ(\ob{K},{[r],k})$ the set of $R_1(\sigma)$, when 
$\sigma\in\ob{K}^{[r],k}\backslash \ob{K}^{[r],k-1}$.
If $\sigma\in \ob{K}^{[r],k}\backslash \ob{K}^{[r],k-1}$, then $R_{1}(\sigma)\in \ob{K}^{[r],k}\backslash \ob{K}^{[r],k-1}$ and $R_{1}(R_{1}(\sigma))=R_{1}(\sigma)$. Therefore, we may write,
$$\cJ(\ob{K},{[r],k})=\{\sigma\in\ob{K}^{[r],k}\backslash \ob{K}^{[r],k-1} \mid R_{1}(\sigma)=\sigma\}.$$
 We set
$\xi \vartriangleleft \sigma$ if $\xi$ is a face (of any codimension) of a simplex $\sigma$.

\begin{definition}\label{def:linkstar}
Let $\ob{K}$ be a filtered face set and $\tau\in \cJ(\ob{K},{[r],k})$.
\begin{enumerate}[(a)]
\item \emph{The star of $\tau$ in $\ob{K}$} is the filtered face subset,
\index{Star of a simplex}
 $\ob{K}(\tau)$, of $\ob{K}$ defined  by
$$\ob{K}(\tau)=\{\xi\in\ob{K}\mid \exists \sigma\in\ob{K}\text{ with } R_1(\sigma)=\tau \text{ and } \xi \vartriangleleft \sigma\}.$$
\item \emph{The link of $\tau$ in $\ob{K}$}  is the filtered face subset, 
\index{Link!of a simplex}
$\cL(\ob{K},\tau)$, of $\ob{K}$ defined by 
$$\cL(\ob{K},\tau)=\{\xi\in\ob{K}\mid
\exists \sigma \in \ob{K} \text{ with } R_{1}(\sigma)=\tau \text{ and } \xi \vartriangleleft R_2(\sigma)\}.$$
\end{enumerate}
\end{definition}

The \ffs~$\ob{K}(\tau)$ is never empty and  the link $\cL(\ob{K},\tau)$ can be empty.
We observe also that 
\begin{equation}\label{equa:skeleta}
\ob{K}^{[r],k}=\ob{K}^{[r],k-1} \cup_{\tau\in \cJ(\ob{K},{[r],k})} \ob{K}( \tau).
\end{equation}

Let $\tau \in \cJ(\ob{K},{[r],k})$, $r\geq 0$ and $k\geq 1$. As $\ob{K}(\tau)$ is a filtered face set, the filtered face set $\ob{K}(\tau)(\partial_i\tau)$ is already defined for any face operator 
$\partial_i\colon \Delta^{k}\to\Delta^{k-1}$, i.e.,
$$\ob{K}(\tau)(\partial_i\tau)=\{\xi\in\ob{K}(\tau)\mid
\exists \sigma\in \ob{K}(\tau)\text{ with } R_1(\sigma)=\partial_i\tau\text{ and }\xi \vartriangleleft \sigma \}.$$
The fact that $\xi\in\ob{K}(\tau)$ can be detailed as: $\xi\in\ob{K}$ with the existence of $\gamma\in\ob{K}$ such that $R_{1}(\gamma)=\tau$ and $\xi \vartriangleleft \gamma$. If, moreover, $\xi\in \ob{K}(\tau)(\partial_i\tau)$, there exists also $\sigma\in\ob{K}(\tau)$ with $R_{1}(\sigma)=\partial_{i}\tau$ and $\xi\vartriangleleft\sigma$, from which we deduce $\xi \vartriangleleft\partial_{i}\gamma$. Thus, we can write,
$$\ob{K}(\tau)(\partial_i\tau)=
\{\xi\in\ob{K}\mid
\exists \gamma\in\ob{K} \text{ with }  
R_1(\gamma)=\tau \text{ and }\xi \vartriangleleft\partial_i\gamma \},
$$
since $R_{1}(\gamma)=\tau$ implies $R_{1}(\partial_{i}\gamma)=\partial_{i}\tau$.
\begin{definition}\label{def:ktaudeltau}
Let $\ob{K}$ be a filtered face set, $r\geq 0$, $k\geq 1$ and $\tau \in \cJ(\ob{K},{[r],k})$.
We define a filtered face set $\ob{K}(\tau,\partial\tau)$ by
$$\ob{K}(\tau,\partial\tau)=\cup_{i=0}^k\ob{K}(\tau)(\partial_i\tau).$$
\end{definition}
  
 Observe that 
 $\ob{K}(\tau)=\ob{K}(\tau)^{[r],k} \text{ and } \ob{K}(\tau,\partial \tau)=\ob{K}(\tau,\partial \tau)^{[r],k-1}$.
The next result is the determination of the intersection of the two filtered face sets in the right-hand side of equation (\ref{equa:skeleta}). 

\begin{proposition}\label{prop:intersectionKtau}
If $\ob{K}$ is a filtered face set, $r\geq 1$, $k\geq 0$ and $\tau \in \cJ(\ob{K}^{[r],k})$, we have
$$\ob{K}^{[r],k-1}\cap  \ob{K}( \tau)=\left\{
\begin{array}{l}
\ob{K}(\tau,\partial \tau), \text{ if } k\geq 1,\\[.2cm]
\cL(\ob{K},\tau), \text{ if } k=0.
\end{array}\right.$$
\end{proposition}
In the sequel, we abuse the notation by using  $\cL(\ob{K},\tau)=\ob{K}(\tau,\partial\tau)$ when $k=|\tau|=0$.

\begin{proof}
Suppose first $k\geq 1$. 
An element $\xi\in \ob{K}^{[r],k-1}\cap \ob{K}(\tau)$ is a face of a $\sigma\in\ob{K}(\tau)$ with $R_1(\sigma)=\tau$ and we have two possibilities:
\begin{itemize}
\item $v(\xi)=r$ and $j_{n-r}\leq k-1$. 
In this case, $R_1(\xi)$ is a proper face of~$\tau$. Let $\tau'=\partial_i\tau$ such that $R_1(\xi)\subset \partial_i\tau$ and set $\sigma'=\partial_i\sigma$. 
Then we have $R_1(\sigma')=\tau'$, $\sigma'\in\ob{K}(\tau)$ and $\xi$ is a face of $\sigma'$. 
This implies $\xi\in\ob{K}(\tau,\partial \tau)$.
\item or $v(\xi)\leq r-1$. In this case, $\xi$ is a face of $R_2(\sigma)$. Set $\tau'=\partial_i\tau$ and $\sigma'=\partial_i\sigma$ for a face operator $\partial_i$. Then, we have $\sigma'\in\ob{K}(\tau)$, $R_1(\sigma')=\tau'$ and $R_2(\sigma)$ is a face of $\sigma'$. Thus $\xi$ is also a face of $\sigma'$ and we conclude $\xi\in\ob{K}(\tau,\partial\tau)$.
\end{itemize}

The reverse inclusion comes directly from
$\ob{K}(\tau,\partial\tau)=\ob{K}(\tau,\partial\tau)^{[r],k-1}$.
The case $k=0$ is proved in a similar way, by using 
$\cL(\ob{K},\tau)=\cL(\ob{K},\tau)^{[r-1]}$. 
\end{proof}

This result can also be stated as follows.

\begin{corollary}\label{cor:intersectionKtau}
Let $\ob{K}$ be a filtered face set, $r\geq 1$, $k\geq 0$. Then we have a push-out,
$$
\xymatrix{ 
\sqcup_{\tau\in \cJ(\ob{K}^{[r],k})}\ob{K}(\tau,\partial \tau)\ar[r]\ar[d]&\sqcup_{\tau\in\cJ(\ob{K}^{[r],k})}\ob{K}(\tau)\ar[d]\\
\ob{K}^{[r],k-1}\ar[r]&
\ob{K}^{[r],k}.
}$$
\end{corollary}

The expanded link, defined below, allows ``the break'' of $(\ob{K}^{[r],k},\ob{K}^{[r],k-1})$ in \propref{prop:relativehomeo}.

\begin{definition}\label{def:expandedlink}
Let $\ob{K}$ be a filtered face set, $r\geq 1$, $k\geq 0$ and $\tau \in \cJ(\ob{K},{[r],k})$.
The \emph{expanded link} of $\tau$ in $\ob{K}$
\index{Link!of a simplex!expanded}
 is the filtered face set $\cLe(\ob{K},\tau)$ defined by
$$\cLe(\ob{K},\tau)=\left\{
(R_2\sigma,\sigma)\mid
\sigma\in\ob{K}, R_1(\sigma)=\tau \text{ and } \sigma\neq \tau\right\}
,$$
with face operators,  
$$\partial_i(R_2(\sigma),\sigma)=(\partial_iR_2(\sigma),\partial_{i+1+k}\sigma)=
(R_2(\partial_{i+1+k}\sigma),\partial_{i+1+k}\sigma),$$
for all $i\in\{0,\ldots,\dim (R_2(\sigma),\sigma)=\dim R_2(\sigma)\}$, and perverse degrees,
$$\|(R_2(\sigma),\sigma)\|_\ell=\|R_2(\sigma)\|_{\ell},$$
for all $\ell\in\{0,\ldots,n\}$.
\end{definition}

First, we note that the previous definition of $\partial_{i}$ implies that the expanded link is stable with face operators and thus is a well defined \ffs. By definition, a simplex $(\alpha,\sigma)\in \cLe(\ob{K},\tau)$
if, and only if,
$R_1(\sigma)=\tau$ and $R_2(\sigma)=\alpha$.
The difference between the link and the expanded link lies in the fact that two distinct simplices, $\sigma$ and $\sigma'$, of $\ob{K}(\tau)$, verifying $R_2(\sigma)=R_2(\sigma')$, give the same element in $\cL(\ob{K},\tau)$ and two distinct elements in
$\cLe(\ob{K},\tau)$. In the expanded link, we keep track of $\sigma$ and $\sigma'$ as parameters. Also, the condition $\sigma\neq \tau$ means
$R_2(\sigma)\neq \emptyset$, for any $(\alpha,\sigma)\in \cLe(\ob{K},\tau)$.
\exemref{exam:linkandothers} illustrates the difference between the two links.
Observe also that, by definition, we have
$\cLe(\ob{K},\tau)=\cLe(\ob{K},\tau)^{[r-1]}$.

 \begin{definition}\label{def:relativeisomorphism}
 A \emph{morphism of pairs of filtered face sets,} $f\colon (\ob{K}_1,\ob{L}_1)\to (\ob{K}_2,\ob{L}_2)$,  is a filtered face map, $f\colon \ob{K}_1\to \ob{K}_2$, such that $f(\ob{L}_1)\subset \ob{L}_2$.
 A \emph{relative isomorphism,} $f\colon (\ob{K}_1,\ob{L}_1)\to (\ob{K}_2,\ob{L}_2)$, 
 \index{Relative isomorphism of filtered face sets}
 is a morphism of pairs of filtered face sets, such that 
 the restriction of $f$ to the complementary subsets is a bijection respecting the perverse degree, $f\colon \ob{K}_1\backslash \ob{L}_1\cong \ob{K}_2\backslash \ob{L}_2$.
 \end{definition}
 
 \begin{example}\label{exam:relativeiso}
 From \propref{prop:intersectionKtau}, we deduce that the map
 $$(\sqcup_{\tau\in \cJ(\ob{K},{[r],k})} \ob{K}(\tau), \sqcup_{\tau\in \cJ(\ob{K},{[r],k})} \ob{K}(\tau,\partial\tau))\to
 (\ob{K}^{[r],k},\ob{K}^{[r],k-1})$$
 is a relative isomorphism.
 \end{example}
 
 \begin{example}\label{exam:simplejoin}
Let $\ob{K}$ be a filtered face set of depth $v(\ob{K})\leq r-1$ with $r\in\{1,\ldots,n\}$. Let  $k\in\N$. 
We denote by $\Delta^k\ast\ob{K}$ the filtered face set, of depth $r$, generated by the simplices,
$$
\xymatrix@1{
\Delta^k\ast\Delta^{j_{n-r+1}}\ast\cdots\ast\Delta^{j_{n}}\ar[rr]^-{\id\ast\sigma}&&\Delta^k\ast \ob{K},
}$$
with $\sigma\in\ob{K}$. Set  $j_{n-r}=k$.
The simplices of this \ffs~$\Delta^k\ast\ob{K}$ are of three types: $\partial_I\Delta^k$, $(\partial_I\Delta^k)\ast \sigma$ with $I\subsetneqq \{0,\ldots,k\}$, and those of $\ob{K}$.
The perverse degree of the simplices of $\ob{K}$ is keeping unchanged. For the other ones, we set:
\begin{itemize}
\item $\|\partial_I\Delta^k\|_{r'}=\left\{
\begin{array}{cl}
-\infty&\text{ if } r< r',\\
k-|I|&\text{ if } r\geq r',
\end{array}\right.
$
\item $\|\partial_I\Delta^k\ast\sigma\|_{r'}=\left\{
\begin{array}{cl}
-\infty&\text{ if } r<r',\\
k-|I|&\text{ if } r'=r, \\
&\text{ or } (r>r' \text{ and } \|\sigma\|_{r'}=-\infty),\\
k-|I|+\|\sigma\|_{r'}+1&\text{ if }r>r' \text{ and } \|\sigma\|_{r'}\neq-\infty.
\end{array}\right.$
\end{itemize}

If $\ob{L}$ is a subcomplex of $\Delta^k$, we define similarly a filtered face set $\ob{L}\ast\ob{K}$, called \emph{the join of $\ob{L}$ and $\ob{K}$,} as, for instance, $\partial \Delta^k\ast\ob{K}$.
\index{Join}
There exists certainly a notion of join of two filtered face sets (see \cite{MR2061452}) but we do not need it.
\end{example}

 \begin{proposition}\label{prop:relativehomeo}
 Let $\ob{K}$ be a filtered face set, $r\geq 1$, $k\geq 0$ and $\tau\in\cJ(\ob{K},{[r],k})$. There exists a relative isomorphism,
 $$f\colon (\Delta^k\ast\cLe(\ob{K},\tau),\partial\Delta^k\ast\cLe(\ob{K},\tau))\to (\ob{K}(\tau), \ob{K}(\tau,\partial\tau)),$$
 preserving the perverse degree.
 \end{proposition}
 
 \begin{proof}
 Let $(\alpha,\sigma)\in \cLe(\ob{K},\tau)$
 and 
 $I\subsetneqq \{0,\ldots,k\}$.
 The application $f$ is defined by
 $$f(\partial_I\Delta^k)=\partial_I\tau,\;
 f(\alpha,\sigma)=\alpha \text{ and }
 f(\partial_I\Delta^k\ast(\alpha,\sigma))=\partial_I\sigma.$$
 By construction, we have
$f(\partial\Delta^k\ast\cLe(\ob{K},\tau))=\ob{K}(\tau,\partial \tau)$.
Decompose $\partial_I\Delta^k\ast(\alpha,\sigma)$ in $\Delta^{j_{0}}\ast\cdots\ast\Delta^{j_{n}}$. If $r<r'$, then the product $\Delta^{j_{0}}\ast\cdots\ast\Delta^{j_{n-r'}}$ is empty and we have $\|\partial_I\Delta^k\ast(\alpha,\sigma)\|_{r'}=-\infty$. As
$\tau\in \cJ(\ob{K},{[r],k})$ and $R_{1}(\sigma)=\tau$, we know that
$\|\partial_I\Delta^k\ast(\alpha,\sigma)\|_{r'}=k-|I|$, if
$r'=r$ or  ($r> r'$ and $\|\alpha\|_{r'}=-\infty$). Finally, in perverse degree $r'$, $r'<r$ and $\|\alpha\|_{r'}\neq -\infty$, we have to take in account 
$\alpha=R_{2}(\sigma)$, and we get
$\|\partial_I\Delta^k\ast(\alpha,\sigma)\|_{r}=k-|I|+\|\alpha\|_{r'}+1$,
the element +1 coming from the fact that a join $\beta_{1}\ast\beta_{2}$ has for dimension $\dim\beta_{1}+\dim\beta_{2}+1$.
From these observations, the definition of $f$ and  \exemref{exam:simplejoin}, we may check first that $f$ keeps the perverse degree. 
\begin{eqnarray*}
\|f(\partial_I\Delta^k\ast(\alpha,\sigma))\|_{r'}&=&\|\partial_I\sigma\|_{r'}=-\infty=
\|\partial_I\Delta^k\ast(\alpha,\sigma)\|_{r'}, \text{ if } r<r',\\
\|f(\partial_I\Delta^k\ast(\alpha,\sigma))\|_{r'}&=&\|\partial_I\sigma\|_{r'}=
k-|I|=\|\partial_I\Delta^k\ast(\alpha,\sigma)\|_{r'},\\
&&\;\text{ if } r'=r \text{ or } (r> r' \text{ and } \|\alpha\|_{r'}=-\infty),\\
\|f(\partial_I\Delta^k\ast(\alpha,\sigma))\|_{r'}&=&\|\partial_I\sigma\|_{r'}=
k-|I|+\|\alpha\|_{r'}+1=\|\partial_I\Delta^k\ast(\alpha,\sigma)\|_{r'},\\
&&\; \text{ if } r>r' \text{ and } \|\alpha\|_{r'}\neq -\infty.
\end{eqnarray*}
A similar computation is done for the other cases. We leave also to the reader the verifications of the compatibility with face operators.

We study now the restriction to the complementary subsets. 
Let $\gamma\in\ob{K}(\tau)\backslash \ob{K}(\tau,\partial\tau)$. 
By definition, we have $R_1(\gamma)=\tau$. 
If $R_2(\gamma)=\emptyset$, then $\gamma=\tau=f(\Delta^k)$. 
If $R_2(\gamma)\neq\emptyset$, then $\gamma=f(\Delta^k\ast(R_2(\gamma),\gamma))$. 
This gives the surjectivity and the injectivity is obvious from inspection of the definition of $f$.
 \end{proof}

\begin{example}\label{exam:linkandothers}
We determine $\ob{K}^{[r],k}$ and the links for the  face set, $\ob{K}$, defined by the following picture
\setlength{\unitlength}{.3mm}
$$
\begin{picture}(200,130)(00,-120)

\put(0,-50){\makebox(0,0){$\bullet$}}
\put(25,0){\makebox(0,0){$\bullet$}}
\put(50,-50){\makebox(0,0){$\bullet$}}
\put(25,-100){\makebox(0,0){$\bullet$}}
\put(120,-16){\makebox(0,0){$\bullet$}}

\put(-10,-58){\makebox(0,0){$b$}}
\put(33,8){\makebox(0,0){$a$}}
\put(55,-58){\makebox(0,0){$b'$}}
\put(33,-108){\makebox(0,0){$a$}}
\put(128,-25){\makebox(0,0){$b''$}}

\put(25,-44){\makebox(0,0){$\tau$}}
\put(25,-25){\makebox(0,0){$F$}}
\put(25,-75){\makebox(0,0){$F'$}}
\put(75,-25){\makebox(0,0){$F''$}}

\put(3,-25){\makebox(0,0){$\alpha$}}
\put(48,-25){\makebox(0,0){$\alpha'$}}
\put(3,-75){\makebox(0,0){$\alpha''$}}
\put(48,-75){\makebox(0,0){$\alpha'''$}}
\put(85,-40){\makebox(0,0){$\beta$}}
\put(70,0){\makebox(0,0){$\beta'$}}

\linethickness{,7mm}
\put(0,-50){\line(1,0){50}} 
\thinlines
\put(0,-50){\line(1,2){25}} 
\put(25,0){\line(1,-2){25}} 
\put(0,-50){\line(1,-2){25}} 
\put(50,-50){\line(-1,-2){25}}

\put(50,-50){\line(-1,-2){25}} 

\put(50,-50){\line(2,1){70}} 
\put(25,0){\line(6,-1){95}}

\end{picture} 
$$
$$\begin{array}{llll}
\text{ and }&\ob{K}_{(-1,0)}=\{a,b''\},&
\ob{K}_{(-1,1)}=\{\beta'\},&
\\
&\ob{K}_{(0,-1)}=\{b,b'\},&
\ob{K}_{(0,0)}=\{\alpha,\alpha',\alpha'',\alpha''',\beta\},&
\ob{K}_{(0,1)}=\{F''\},\\
&&\ob{K}_{(1,-1)}=\{\tau\},&
\ob{K}_{(1,0)}=\{F,F'\}.\\
\end{array}$$
Denoting by $\langle -\rangle$ the face set generated by $-$, we have:\\
$\ob{K}^{[0],0}=\ob{K}_{(-1,0)}=\langle a,b''\rangle$,
$\ob{K}^{[0],1}=\ob{K}^{[0],0}\cup \ob{K}_{{(-1,1)}}=\langle \beta'\rangle$,\\
$\ob{K}^{[1],0}=\ob{K}^{[0],1}\cup \ob{K}_{(0,-1)}\cup \ob{K}_{(0,0)}\cup \ob{K}_{(0,1)}=\langle \alpha,\alpha'',\alpha''',F''\rangle$,\\
$\ob{K}^{[1],1}=\ob{K}^{[1],0}\cup \ob{K}_{(1,-1)}\cup \ob{K}_{(1,0)}=\ob{K}$,\\
$\ob{K}(\tau)=\langle F,F'\rangle$, $\ob{K}(\tau,\partial \tau)=\langle \alpha,\alpha', \alpha'', \alpha'''\rangle$, $\ob{K}(b)=\langle \alpha,\alpha''\rangle$, 
\\
$\cJ(\ob{K},{[0],0})=\{a,b''\}$, $\cJ(\ob{K},{[0],1})=\{\beta'\}$, $\cJ(\ob{K},{[1],0})=\{b,b'\}$, $\cJ(\ob{K},{[1],1})=\{\tau\}$,
$\cL(\ob{K},\tau)=\langle a\rangle$,
${{\cL}^{\tt exp}}(\ob{K},\tau)=\langle (a,F),(a,F')\rangle$.
\end{example}

\section{Perverse local systems on filtered face sets}\label{sec:filteredandlocalsystem}

\begin{quote} 
We define perverse local systems of coefficients, $\cM$, over a filtered face set $\ob{K}$ and the intersection cohomology of $\ob{K}$ with coefficients in $\cM$, for a loose perversity $\ov{q}$.
The blow-up $\tF$ of certain local systems over face sets, $F$, are examples of this situation. We detail the particular case of universal systems, as cochains  and Sullivan forms (over $\Q$). The main result (\thmref{thm:extendable}) gives sufficient conditions on the local systems, $F$ and $G$, for having an isomorphism between the intersection cohomologies of a filtered face set, $\ob{K}$, with coefficients in the blow-ups $\tF$ and~$\tG$.

In this section, all cochain complexes are over a commutative ring $R$. The particular case $R=\Q$ is explicitly quoted if necessary.
\end{quote} 

 We introduce  the notion of perversity which is the fundamental tool of intersection theory. We follow the convention of \cite{MR800845}.
 
\begin{definition}\label{def:perversité} 
A \emph{loose perversity} is a map $\ov{q}\colon \N\to\Z$, $i\mapsto \ov{q}(i)$, such that $\ov{q}(0)=0$. 
\index{Perversity!loose}
\index{Perversity}
A loose perversity is \emph{positive} if $\ov{q}(i)\geq 0$, for any $i\in\N$.

A \emph{perversity} is a loose perversity such that
$\ov{q}(i)\leq\ov{q}(i+1)\leq\ov{q}(i)+1$, for all $i\geq 1$.
A \emph{Goresky-MacPherson perversity} (or \emph{GM-perversity})
\index{Perversity!of Goresky-MacPherson}\index{GM!perversity}
 is a perversity such that $\ov{q}(1)=\ov{q}(2)=0$.
\end{definition}

If $\ov{q}_1$ and $\ov{q}_2$ are two loose perversities, we set $\ov{q}_1\leq \ov{q}_2$ if we have $\ov{q}_1(i)\leq\ov{q}_2(i)$, for all $i\in\N$. The lattice of GM-perversities, denoted $\cP^n$, admits a maximal element, $\ov{t}$, called the \emph{top perversity} and defined by $\ov{t}(i)=i-2$, if $i\geq 2$.
\index{Perversity!top}

To these perversities, we  add an element, $\ov{\infty}$, which is the constant map on $\infty$. We call it the \emph{infinite perversity} despite the fact that it is not a perversity in the sense of the previous definition.
\index{Perversity!infinite}

Homology and cohomology theories on a face set, $\ob{K}$, can be expressed with local systems, see \cite{MR736299} for instance in the simplicial case. We adapt these local systems to the filtered situation as follows. 

\begin{definition}\label{def:perversecochains}
A \emph{strict perverse cochain complex} is a cochain complex, $(C,d)$, 
\index{Strict perverse!cochain complexes}
in which each element, $\omega$, has a perverse degree,
$$\|\omega\|=(\|\omega\|_1,\ldots,\|\omega\|_n),$$
with $\|\omega\|_i\in\N\cup\{-\infty,\infty\}$,  such that
 $\|\omega+\omega'\|_{i}\leq \max(\|\omega\|_{i},\|\omega'\|_{i})$, for any $i\in\{1,\ldots,n\}$.

If $\ov{q}$ is a loose perversity, a cochain $\omega\in C$ is \emph{$\ov{q}$-admissible} if 
$\|\omega\|_{i}\leq \ov{q}(i)$, for any $i\in\{1,\ldots,n\}$. A  cochain $\omega\in C$ is of \emph{intersection for $\ov{q}$} (or of $\ov{q}$-intersection) if $\omega$ and $d\omega$ are $\ov{q}$-admissible. 
\index{Cochain!admissible for a perversity}
\index{Cochain!of intersection for a perversity}
We denote by
$$C_{\ov{q}}=\{\omega\in C\mid
\|\omega\|_{i}\leq \ov{q}(i)
\text{ and }
\|d\omega\|_{i}\leq \ov{q}(i) \text{ for any } i\in\{1,\ldots,n\}\},$$ the complex of $\ov{q}$-intersection cochains and by $H_{\ov{q}}(C)$ its homology.

\emph{A morphism of strict perverse cochain complexes,} $f\colon (C,d)\to (C',d)$, is a morphism of cochain complexes which decreases the perverse degree, i.e., $\|f(\omega)\|\leq \|\omega\|$, for any $\omega$.
\index{Strict perverse!cochain complexes!morphism of}

 A \emph{strict perverse differential graded algebra} (henceforth strict perverse \dga)
 \index{Strict perverse!DGA}
  is a strict perverse cochain complex and a DGA such that
$\|\omega\cdot \omega'\|_{i}\leq \|\omega\|_{i}+\|\omega'\|_{i}$, for any $i\in\{1,\ldots,n\}$.
A morphism of strict perverse cochain complexes, compatible with the products, is called a \emph{morphism of strict perverse} DGA's.
\end{definition}

\begin{definition}\label{def:localsystemfiltered}
A \emph{strict perverse local system of cochains over a filtered face set}, $\ob{K}$,
\index{Strict perverse!cochain complexes!local system of}
 is a family of strict perverse cochain complexes, $\cM_\sigma$, indexed by the simplices $\sigma$ of $\ob{K}_{+}$, and a family of cochain maps, $\tilde{\partial}_i\colon \cM_\sigma\to \cM_{\partial_i\sigma}$, decreasing the perverse degree
(i.e., $\|\tilde{\partial}_{i} \omega\|_{j}\leq \|\omega\|_{j}$, for any $j\in \{1,\ldots, n\}$)
and such that $\tilde{\partial}_i\tilde{\partial}_j=\tilde{\partial}_{j-1}\tilde{\partial}_i$, if 
$i<j$ and $\partial_i\partial_j$ corresponds to $\delta_j\delta_i\in {\pmb\Delta}^{[n],+}_\cF$.

The \emph{space of global sections} of $\cM$ over $\ob{K}$ is the strict perverse cochain complex $\cM(\ob{K})$ defined as follows: 
\index{Space!of global sections}
an element $\omega\in \cM^j(\ob{K})$ is a function which assigns to each simplex $\sigma\in\ob{K}_{+}$ an element $\omega_\sigma\in \cM^j_\sigma$ such that
$\omega_{\partial_i\sigma}=\tilde{\partial}_i(\omega_\sigma)$ for all $\sigma$  and all face operators $\delta_i\in{\pmb\Delta}^{[n],+}_\cF$. The \emph{perverse degree of $\omega\in \cM(\ob{K})$} is the supremum 
\index{Perverse!degree!of a global section}
(possibly infinite)
of the perverse degrees of the $\omega_\sigma$, for all $\sigma\in\ob{K}_{+}$, i.e.,
$\|\omega\|=(
\sup_{\sigma}\|\omega_{\sigma}\|_{1},\ldots,
\sup_{\sigma}\|\omega_{\sigma}\|_{n})$.
The laws and differential on $\cM(\ob{K})$ are defined by
$
(\lambda\omega+\mu\omega')_{\sigma}=\lambda \omega_{\sigma}+\mu\omega'_{\sigma}$,
$(d\omega)_{\sigma}=d\omega_{\sigma}$.

If  $\cM$ is a local system  of strict perverse \dga's, the space of  global sections is a  
strict perverse
\dga, with the law
$(\omega\cdot\omega')_{\sigma}=\omega_{\sigma}\cdot\omega'_{\sigma}$.
\end{definition}

\begin{definition}\label{def:localsystemadmissible}
Let $\ov{q}$ be a  loose perversity and $\cM$ be a 
strict
perverse local system of cochains over a filtered face set $\ob{K}$.   
We denote by $\cM_{\ov{q}}({\ob{K}})$ the complex 
 of  global sections which are of intersection for $\ov{q}$ and by $H_{\ov{q}}^*({\ob{K}};\cM)$ its homology, called the \emph{intersection cohomology of ${\ob{K}}$ with coefficients in $\cM$, for the loose perversity $\ov{q}$.}
 \index{Intersection cohomology!with coefficients in a local system}
\end{definition} 

In the case of a \ffs~of formal dimension~0 (i.e., $n=0$) the 
strict
perverse local systems are the usual local systems of \cite{MR736299}.  

\medskip
Let $\cM$ be a 
strict
perverse local system on a filtered face set $\ob{K}$. Observe that any filtered face map, $f\colon \ob{L}\to \ob{K}$, induces a 
strict
perverse local system, $f^*{\cM}$, on $\ob{L}$, defined by
$$(f^*{\cM})_\sigma=\cM_{f\circ \sigma}.$$
Therefore, a 
strict
perverse local system, $\cM$, on $\ob{K}$ induces a 
strict
perverse local system on any filtered face subset of $\ob{K}$; we still denote by $\cM$ these induced systems.

\begin{definition}\label{def:universalsystem}
\index{Universal system!of cochains}
A \emph{universal system (of co\-chains)} is a 
contravariant
functor $F$ from the category
$\pmb\Delta$
to the category of cochain complexes, \emph{free as $R$-modules.} If $F$ takes its values in the category of \dga's, free as $R$-modules, we say that $F$ is a \emph{universal system of \dga's.}
\index{Universal system!of \dga's}
\end{definition}

A universal system of cochains  defines a local system of cochains on any \emph{face set,} $\ob{K}$, by setting
$F_\sigma=F(\Delta^{|\sigma|})$
and
$\tilde{\partial}_i=F(\delta_i)\colon F_\sigma\to F_{\partial_i\sigma}$.
The cochain complex of global sections is denoted $F(\ob{K})$. If $f\colon \ob{L}\to \ob{K}$ is a filtered face map, then 
$f^*F=F$ and we obtain a morphism of cochain complexes,
$F(f)\colon F(\ob{K})\to F(\ob{L})$.

\smallskip
The PL-forms of Sullivan and the cochains are examples of universal systems of \dga's. We recall their construction.

\begin{example}[Cochain algebra]\label{exam:simplicialcochain}
Let $R$ be a commutative ring.
We define a universal system of DGA's from the simplicial cochain complex, by setting 
$C^r([n])=C^r(\ob{\Delta}^n;R)$.
 If $\ob{K}$ is a face set, the  associated DGA of global sections, $C^*(\ob{K};R)$, is defined by:
\begin{itemize}
\item $C^j(\ob{K};R)$, is the free $R$-module of all set  maps  $f\colon \ob{K}_j\to R$,
\item the differential $\delta$ is given by 
$$(\delta f)(\sigma)=\sum_{i=0}^{j+1}(-1)^if(\partial_i\sigma),\,
f\in C^j(\ob{K};R),\,\sigma\in \ob{K}_{j+1},$$
\item the product of $f\in C^j(\ob{K};R)$ and $g\in C^k(\ob{K};R)$ is defined by
$$(f\cdot g)(\sigma)=f(\partial_j^F\sigma)\, g(\partial_k^B\sigma),\,
\sigma\in \ob{K}_{j+k},$$
where $\partial_j^F\colon \ob{K}_{j+k}\to \ob{K}_j$ and
$\partial_k^B\colon \ob{K}_{j+k}\to \ob{K}_k$ are given by
$\partial^F_j\sigma=(\partial_{j+1}\circ\cdots\circ\partial_{j+k})\sigma$
and
$\partial_k^B\sigma=(\partial_0\circ\cdots\circ \partial_0)\sigma$.
\end{itemize}
\end{example}

\begin{example}[Sullivan's polynomial forms]\label{exam:sullivanforms}
\index{Sullivan!polynomial forms of}
Let $R=\Q$. We define a universal system of CDGA's by setting
$$A_{PL}([n])=\land (t_0,\ldots,t_n,dt_0,\ldots,dt_n)/(t_0+\cdots +t_n=1,\,dt_0+\cdots +dt_n=0),$$ 
with $|t_i|=0$ and $|dt_i|=1$.  
Geometrically, the elements of $A_{PL}([n])$ are the polynomial differential forms, with rational coefficients, on the simplex $\Delta^n$. 
The face operator, $\delta_{i}\colon [n]\to [n+1]$, induces $A_{PL}(\delta_{i})\colon A_{PL}([n+1])\to A_{PL}([n])$, defined by: $A_{PL}(\delta_{i})(t_{k})$ is equal to $t_{k}$ if $k<i$, 0 if $k=i$ and $t_{k-1}$ if $k>i$.
The associated CDGA of global sections on a face set, $\ob{K}$, is denoted $A_{PL}(\ob{K})$.
\end{example}

We mention the existence of local systems which are not universal. For instance, if $\ob{f}\colon\ob{E}\to\ob{K}$ is a Kan fibration between face sets,  for any $\sigma\colon \Delta\to \ob{K}$, we denote by $\ob{E}_\sigma$ the pullback of $\ob{f}$ along $\sigma$. The association $\sigma\mapsto A_{PL}(\ob{E}_\sigma)$ is a local system over $\ob{K}$ which is the key tool in \cite{MR736299} and \cite{MR1653355}. These more general local systems are not considered here.

\medskip
We introduce now the fundamental notion of \emph{blow-up of a universal system}, based on the blow-up of a simplex which first appears in \cite{MR1143404}, see also \cite{MR2210257}.

\begin{definition}\label{def:blowupsimplex}
The  \emph{blow-up of a filtered simplex,}
\index{Blow-up!of a filtered simplex}
$\Delta=\Delta^{j_0}\ast\cdots\ast\Delta^{j_n}$, with $j_{n}\geq 0$,
is the map,
$$\mu\colon \widetilde{\Delta}=c\Delta^{j_0}\times\cdots\times c\Delta^{j_{n-1}}\times \Delta^{j_n}\to
\Delta=\Delta^{j_0}\ast\cdots\ast\Delta^{j_n},$$
defined by
\begin{eqnarray*}
\mu([y_0,s_0],\ldots,[y_{n-1},s_{n-1}],y_n)&=&
s_0y_0+(1-s_0)s_1y_1+\cdots\\
&&
+(1-s_0)\cdots (1-s_{n-2})s_{n-1}y_{n-1}\\&&
+(1-s_0)\cdots (1-s_{n-2})(1-s_{n-1})y_n,
\end{eqnarray*}
where  $c\Delta^k=\Delta^k\times [0,1]/\Delta^k\times \{0\}$ is the cone on $\Delta^k$, $[y_i,s_i]\in c\Delta^{j_i}$ and $y_n\in\Delta^{j_n}$. 
\end{definition}

The prism $\widetilde{\Delta}$ is sometimes also called the blow-up of  $\Delta=\Delta^{j_0}\ast\cdots\ast\Delta^{j_n}$.
The map $\mu$ \emph{induces a diffeomorphism} between the interior of the blow-up $\widetilde{\Delta}$ and the interior of the simplex $\Delta$.
In the previous representation of the elements of $\widetilde{\Delta}$, the face $\Delta^{j_i}\times\{1\}$ of the simplex $c\Delta^{j_i}$ corresponds to $s_i=1$.

\begin{remark}\label{rem:blowup}
The blow-up map, $\mu$, can also be defined as follows. We express a point of $c\Delta^{j_{i}}$ by
$(x_{i},t_{i})$, with $x_i=(x_{i,0},\ldots,x_{i,j_{i}})\in\R^{j_i+1}$, $t_i\in\R$, $t_{i}+\sum_{k=0}^{j_{i}}x_{i,k}=1$, and set
$$\mu((x_0,t_0),\ldots,(x_{n-1},t_{n-1}),x_n)
=
x_0+t_0x_1+t_0t_1x_2+\cdots +t_0\cdots t_{n-1}x_n.$$
In this context, the face $\Delta^{j_i}\times\{1\}$ of the simplex $c\Delta^{j_i}$ corresponds to the set of elements satisfying the equation $t_i=0$ and the cone point corresponds to $(0,1)$. In the case $\Delta^{j_{i}}=\emptyset$ with $i<n$, we have $c\Delta^{j_{i}}=\{(0,1)\}$. 
\end{remark}

\begin{example} We draw two explicit examples of the blow-up of a simplex. 

\definecolor{zzttqq}{rgb}{0.6,0.2,0.9}

\makebox{}

\medskip

\begin{tikzpicture}
\draw [color=black] (0,0)-- (1,0.5);
\draw [color=black] (1,0.5)--  (0,2);
\draw [color=black]  (0,2)-- (0,0);
\draw [color=black]  (1,0.5)-- (-2,0.5);
\draw [color=black]  (0,0)-- (-2,0.5);
\draw [color=black]  (0,2)-- (-2,0.5);
\fill [color=zzttqq] (-2,0.5) circle (3pt);
\draw[color=black] (-2.5,0.5) node {$\Delta^{j_0}$};
\draw[color=black] (-0.5,-1.2) node {$\Delta = \Delta^{j_0} *\Delta^{j_1}$  };
\draw[color=black] (2.5,-1.3) node {has for blow-up };
\draw[color=black] (5.5,-1.2) node {$\widetilde\Delta = c \Delta^{j_0} \times  \Delta^{j_1}$ };
\draw[color=black] (1.2,1.5) node {$\Delta^{j_1}$};
\draw[color=zzttqq] (2.5,2.5) node {$\left(\Delta^{j_0} \times \{1\} \right)\times \Delta^{j_1} $};
\draw[->] (0.8,1.3) -- (0.3,0.8);
\draw[color=zzttqq, ->] (2.5,2.2) -- (4.3,1.3);
\draw [color=black] (6,0)-- (7,0.5);
\draw [color=black] (7,0.5)--  (6,2);
\draw [color=black]  (6,2)-- (6,0);
\fill[color=red,fill=zzttqq,fill opacity=0.3] (4,0) -- (5,0.5) -- (4,2) -- (4,0) -- cycle;
\draw [color=zzttqq] (4,0)-- (5,0.5);
\draw [color=black] (5,0.5)--  (7,0.5);
\draw [color=black] (6,2)--  (4,2);
\draw [color=zzttqq]  (4,2)-- (4,0);
\draw [color=black]  (6,0)-- (4,0);
\draw [color=zzttqq]  (5,0.5)-- (4,2);
\draw[color=black] (7.2,1.5) node {$\Delta^{j_1}$};
\draw[->] (6.8,1.3) -- (6.3,0.8);
\draw[color=black] (5,-0.3) node {$c\Delta^{j_0}$};
\end{tikzpicture}

\vspace{1cm}

\begin{tikzpicture}
\draw [color=black] (0,0)-- (1,0.5);
\draw [color=black] (1,0.5)--  (0,2);
\draw [color=black]  (0,2)-- (0,0);
\draw [color=black]  (1,0.5)-- (-2,0.5);
\draw [color=red,very thick,dashed]  (0,0)-- (-2,0.5);
\draw [color=black]  (0,2)-- (-2,0.5);
\fill [color=zzttqq] (-2,0.5) circle (3pt);
\draw[color=black] (-2.5,0.5) node {$\Delta^{j_0}$};
\draw[color=black] (0,-0.2) node {$\Delta^{j_1}$};
\draw[color=black] (-0.5,-2) node {$\Delta = \Delta^{j_0} *\Delta^{j_1} *\Delta^{j_2}$  };
\draw[color=black] (3,-2.1) node {has for blow-up };
\draw[color=black] (6.5,-2) node {$\widetilde\Delta = c \Delta^{j_0} \times c\Delta^{j_1} \times \Delta^{j_2}$ };
\draw[color=black] (1,1.5) node {$\Delta^{j_2}$};
\draw[color=zzttqq] (2.5,2.5) node {$\left(\Delta^{j_0} \times \{1\} \right)\times c\Delta^{j_1} \times \Delta^{j_2}$};
\draw[color=zzttqq, ->] (2.5,2.2) -- (3.8,1.6);
\draw[color=red] (2.5,-1) node {$ c\Delta^{j_0}  \times \left(\Delta^{j_1} \times \{1\} \right)\times \Delta^{j_2}$};
\draw[color=red, ->] (3,-0.5) -- (5,0.3);
\draw [color=black] (6,0)-- (7,0.5);
\draw [color=black] (7,0.5)--  (6,2);
\draw [color=black]  (5,1.5)-- (6,2);
\fill[color=black,fill=zzttqq,fill opacity=0.2,dashed] (4,0) -- (5,0.5) -- (4,2) -- (3,1.5) -- (4,0) -- cycle;
\fill[color=black,fill=red,fill opacity=0.3] (4,0) -- (6,0)-- (5,1.5)  -- (3,1.5) -- (4,0) -- cycle;
\draw [color=zzttqq] (4,0)-- (5,0.5);
\draw [color=black] (5,0.5)--  (7,0.5);
\draw [color=black] (6,2)--  (4,2);
\draw [color=red,very thick]  (4,0)-- (3,1.5);
\draw [color=red,very thick,dashed]  (6,0)-- (4,0);
\draw [color=red,very thick,dashed]  (3,1.5)-- (5,1.5);
\draw [color=zzttqq]  (3,1.5)-- (4,2);
\draw [color=red,very thick,dashed] (6,0 )--  (5,1.5);
\draw [color=zzttqq]  (5,0.5)-- (4,2);
\draw[color=black] (7.2,1.5) node {$\Delta^{j_2}$};
\draw[color=black] (5,-0.3) node {$c\Delta^{j_0}$};
\draw[color=black] (7,0) node {$c\Delta^{j_1}$};
\end{tikzpicture}


%

The faces containing $\Delta^{j_{i}}\times\{1\}$ as a factor, which play a fundamental role in the next definition, have been shadowed in the previous drawings.
\end{example}


As already observed in \cite{MR1143404} in the case of differential forms,
any universal system of cochains, $F$, 
gives rise to a 
\emph{
strict
perverse local system of cochains,} $\tF$, on any filtered face set, $\ob{K}$, defined as follows.
For any simplex, $\sigma\colon\Delta^{j_0}\ast\cdots\ast \Delta^{j_n}\to\ob{K}_{+}$, we set
$$\tF_\sigma=F(c\Delta^{j_0})\otimes\cdots\otimes F(c\Delta^{j_{n-1}})\otimes F(\Delta^{j_n}).$$
With a method similar to a construction of Brylinski (\cite{MR1197827}), any  $\omega_{\sigma}\in \tF_{\sigma}$
can be provided with an extra degree, called the \emph{perverse degree}, 
\index{Perverse!degree}
that we describe now.
Let 
$\ell\in \{1,\ldots,n\}$ such that  $\Delta^{j_{n-\ell}}\neq\emptyset$, the restriction of $\omega_{\sigma}$,
$$\omega_{\sigma,n-\ell}\in 
F(c\Delta^{j_0})\otimes\cdots\otimes F(\Delta^{j_{n-\ell}}\times\{1\})\otimes\cdots\otimes F(c\Delta^{j_{n-1}})\otimes F(\Delta^{j_n}),$$
can be written
$\omega_{\sigma,n-\ell}=\sum_k \omega'_{\sigma,n-\ell}(k)\otimes \omega''_{\sigma,n-\ell}(k)
 $, with 
 \begin{itemize}
\item $\omega'_{\sigma,n-\ell}(k)\in F(c\Delta^{j_0})\otimes\cdots\otimes F(c\Delta^{j_{n-\ell-1}})\otimes
F(\Delta^{j_{n-\ell}}\times\{1\})$ and
\item $\omega''_{\sigma,n-\ell}(k)\in  F(c\Delta^{j_{n-\ell+1}})\otimes\cdots\otimes F(\Delta^{j_n})$.
\end{itemize}

\begin{definition}\label{def:transversedegreeblowup}
If $\omega_{\sigma,n-\ell}\neq 0$, the \emph{$\ell$-perverse degree,} $\|\omega_{\sigma}\|_{\ell}$, of  
$\omega_{\sigma}$ is equal to 
$$\|\omega_\sigma\|_{\ell}=\sup_k\left\{ |\omega''_{\sigma,n-\ell}(k)|\text{ such that } \omega'_{\sigma,n-\ell}(k)\neq 0\right\},$$
where $|\omega''_{\sigma,n-\ell}(k)|$ is the degree of 
$\omega''_{\sigma,n-\ell}(k)$,
as an element of the graded module\linebreak
$F(c\Delta^{j_{n-\ell+1}})\otimes\cdots\otimes F(\Delta^{j_n})$.
If $\omega_{\sigma,n-\ell}= 0$ or $\Delta^{j_{n-\ell}}=\emptyset$, we set
$\|\omega_{\sigma}\|_{\ell}=-\infty$.
\end{definition}
Observe now that any face operator,
$\delta_{i}\colon \Delta^{j_{0}}\ast\cdots\ast\Delta^{j_{n}}\to
 \Delta^{k_{0}}\ast\cdots\ast\Delta^{k_{n}}\in {\pmb\Delta}^{[n],+}_\cF$,
 induces a chain map
 $${\delta}_{i}^*\colon F(c\Delta^{k_{0}})\otimes \cdots\otimes F(c\Delta^{k_{n-1}})\otimes F(\Delta^{k_{n}})
 \to F(c\Delta^{j_{0}})\otimes \cdots\otimes F(c\Delta^{j_{n-1}})\otimes F(\Delta^{j_{n}}),
 $$
which decreases the perverse degree. 
These chain maps are the operators,
$\tilde{\partial}_{i}\colon F_{\sigma}\to F_{\partial_{i}\sigma}$,
required in \defref{def:localsystemfiltered} and we have proved that \emph{$\tF$ is a 
strict
perverse local system of cochains over $\ob{K}$.}
%
%

Observe that we have two notions of perverse degree, one for the simplices of a \ffs~(\defref{def:perversedegreesimplexffs}) and the previous one for global sections associated to  a universal system. As they qualify different objects, they are different in essence but we will specify the one in use when there is some possible ambiguity. They are also linked in some sense, as we show now in the following remark.

\begin{remark}\label{rem:degreealongfibre}
Let $\sigma\colon \Delta^{j_{0}}\ast\cdots\ast\Delta^{j_{n}}\to \ob{K}_{+}$.
In the case of differential forms, the $\ell$-perverse degree of $\omega_{\sigma}$ is the degree
along the fibers of $\omega_{\sigma,n-\ell}$ relatively to the projection (see \cite[Page 122]{MR0336650})
$$ 
c\Delta^{j_0}\times \cdots\times\left(\Delta^{j_{n-\ell}}\times\{1\}\right)\times\cdots\times c\Delta^{j_{n-1}}\times \Delta^{j_n} \to
c\Delta^{j_0}\times\cdots\times\left(\Delta^{j_{n-\ell}}\times\{1\}\right).$$
Because of the dimension of the fiber, the perverse degree, $\|\omega_{\sigma}\|$, of the form $\omega_{\sigma}$ and the perverse degree, $\|\sigma\|$, of the simplex $\sigma$ satisfy the inequality,
$$\|\omega_{\sigma}\|\leq (\dim\sigma)-1-\|{\sigma}\|.$$
\end{remark}

\begin{definition}\label{def:blowup}
Let $F$ be a universal system  
and  $\ob{K}$ be a filtered face set. The 
strict
perverse local system, $\tF$, described above, is called the \emph{blow-up} of $F$ over $\ob{K}$. 
\index{Blow-up!of a universal system}
\end{definition}

\begin{example}
The algebra of Sullivan's polynomial forms and the  cochain algebra of Examples \ref{exam:sullivanforms} and \ref{exam:simplicialcochain} give the main  examples of blow-ups used in this work, denoted respectively by $ \widetilde{A}_{PL}$ and $\tC^*$. The cochain complexes of $\ov{q}$-intersection are denoted  by $ \widetilde{A}_{PL,\ov{q}}$ and $\tC^*_{\ov{q}}$. 
The differential forms corresponding to $\ov{q}=\ov{0}$ are the forms introduced by A.~Verona in \cite{MR771120}. 
\begin{definition}\label{def:thomwhitney}
Let $R$ be a commutative ring and $\ob{K}$ be a \ffs.
The complex $\tC^*(\ob{K})$ over $R$
is called the \emph{Thom-Whitney complex} with coefficients in $R$. If $\ov{q}$ is a loose perversity, the homology of  $\tC^*_{\ov{q}}(\ob{K})$ is denoted
$H^*_{\TW,\ov{q}}(\ob{K};R)$, and called the \emph{Thom-Whitney cohomology} (henceforth TW-cohomology) of $\ob{K}$ with coefficients in $R$.
\index{TW!cohomology}
\index{Thom-Whitney|see{TW}}
\index{TW!complex}
\end{definition}
We compare it with the Goresky-MacPherson cohomology in \secref{sec:chaincochain}.
\end{example}

\begin{proposition}\label{prop:inducedbyfacemaps}
Let $F$ be a universal system.
Any filtered face map, ${f}\colon \ob{L}\to\ob{K}$, induces
 a  cochain map, ${f}^*\colon \tF(\ob{K})\to \tF(\ob{L})$,
defined by
$({f}^*\omega)_{\sigma}=\omega_{{f}(\sigma)}$.
Moreover, the map $f^*$ decreases the perverse degree and induces a cochain map
$f^*\colon \tF_{\ov{q}}(\ob{K})\to \tF_{\ov{q}}(\ob{L})$,
for any loose perversity $\ov{q}$. In the case of a universal system of DGA's, the map ${f}^*\colon \tF(\ob{K})\to \tF(\ob{L})$ is a morphism of strict perverse DGA's.
\end{proposition}

\begin{proof}
Let $\omega\in\tF(\ob{K})$ and $\sigma\in\ob{L}$.  Any map ${f}$ which is compatible with face operators satisfies
$$(f^*\omega)_{\partial_i\sigma}=\omega_{f(\partial_i\sigma)}=\omega_{\partial_if(\sigma)}=
\tilde{\partial}_i\omega_{f(\sigma)}=
\tilde{\partial}_i(f^*(\omega_{\sigma})).$$
Thus $f^*\omega$ is an element of $\tF(\ob{L})$. The compatibility with the perverse degree (or the laws of algebras) is direct.
\end{proof}

The next result is  the key point in the existence of quasi-isomorphisms. It is based on the following definitions introduced in \cite{MR736299}.

\begin{definition}\label{def:extendable}%
A universal  system, $F$, is \emph{extendable} if the restriction map,
\index{Universal system!extendable}
${F}({\Delta}^n)\to{F}({\partial\Delta}^n)$,
is surjective, for any $n\geq 1$.
\end{definition}

\begin{definition} \label{def:ofdifferentialcoeff}%
A universal  system, $F$, is of \emph{differential coefficients} if, for any $n$ and any $i$,  the face operator $\delta_i\colon \Delta^{n-1}\to \Delta^n$ induces an isomorphism
$ H(F(\Delta^n))\cong H(F(\Delta^{n-1}))$.
\index{Universal system!of differential coefficients}
\end{definition}

\begin{theorem}\label{thm:extendable}
Let $R$ be a commutative ring and $\ov{q}$ be a loose perversity. 
Let $(F,G)$ be a pair of  extendable universal  systems of differential coefficients, with a natural transformation,
$\psi\colon F\to G$, given by
$\psi_{\Delta^{i}}\colon F(\Delta^{i})\to G(\Delta^{i})$.
Then, there exists  a morphism
$\Psi\colon \tF_{\ov{q}}(\ob{K})\to \tG_{\ov{q}}(\ob{K})$, 
defined as follows: if $\omega \in \tF(\ob{K})$,
 $\sigma\colon \Delta^{j_0}\ast\cdots\ast \Delta^{j_n}\to \ob{K}_{+}$ and $\omega_{\sigma}
 =\omega_0\otimes\cdots\otimes\omega_n\in F(c\Delta^{j_0})\otimes\cdots\otimes F(c\Delta^{j_{n-1}})\otimes F(\Delta^{j_n})$, we set
 $(\Psi(\omega))_{\sigma}=\psi_{c\Delta^{j_{0}}}(\omega_{0})\otimes\cdots\otimes\psi_{\Delta^{j_{n}}}(\omega_{n})$.
 If $\psi$ is also a natural transformation of DGA's, then $\Psi$ induces a morphism, $\tF(\ob{K})\to \tG(\ob{K})$, of strict perverse DGA's.
 
 Moreover,  if $\psi_{\Delta^0}$ induces an isomorphism, 
$H^0(\psi_{\Delta^0})\colon H^0(F(\Delta^0))\cong H^0(G(\Delta^0))\cong R$, then the  map, induced by $\Psi$ in homology, is an isomorphism,
$$\Psi^*\colon H^*_{\ov{q}}(\ob{K};\tF)
\xrightarrow[]{\cong}
H^*_{\ov{q}}(\ob{K};\tG).$$
\end{theorem}

If $R=\Q$, S. Halperin proves (\cite[Chapters 13 and 14]{MR736299}) that the two local systems of \exemref{exam:simplicialcochain} and  \exemref{exam:sullivanforms} are extendable  universal systems of differential coefficients.
Furthermore, the classical integration map,
$\int\colon A_{PL}(\Delta^n)\to C^*(\Delta^n)$,
 defines a morphism  between the blow-ups 
 of Sullivan's forms and the Thom-Whitney cochains,
$$\int\colon \widetilde{A}_{PL}(\ob{K})\to \widetilde{C}^*(\ob{K}).$$
In the case of face sets, Sullivan proves that this integration is a quasi-isomorphism, see \cite[Theorem 7.1]{MR0646078}. This  theorem of Sullivan and \thmref{thm:extendable} imply the next result.

\begin{corollary}\label{cor:Aplandcochains}
Let $R=\Q$, $\ob{K}$ be a filtered face set and $\ov{q}$ be a loose perversity. 
The integration map
$\int\colon \widetilde{A}_{PL,\ov{q}}(\ob{K})\to \tC^{*}_{\ov{q}} (\ob{K})$
induces an isomorphism in homology.
\end{corollary}

The compatibility with products, of the map induced by $\int$ in cohomology, is detailed in \propref{prop:producto}.

\section{Proof of Theorem~A
}\label{sec:proofofC} 
\begin{quote}
This section is devoted to the proof  of \thmref{thm:extendable}. For that, we introduce the notion of filtered theory of cochains that will be used also for the comparison between Goresky-MacPherson and Thom-Whitney cochains in \secref{sec:chaincochain}. All cochain complexes are over a commutative ring $R$.
\end{quote}

\begin{definition}\label{def:perversetheory}
A functor $C^*$ from the category of filtered face sets to the category of  cochain complexes, is a \emph{filtered theory of cochains} if the following properties are satisfied.
\index{Filtered!theory of cochains}
\begin{enumerate}[(a)]
\item {\sc Extension axiom.} For any \ffs, $\ob{K}$, any $r\geq 0$ and any $k\geq 0$, the restriction map,
$C^*(\ob{K}^{[r],k})\to C^*(\ob{K}^{[r],k-1})$,
is surjective. The kernel is denoted by $C^*(\ob{K}^{[r],k},\ob{K}^{[r],k-1})$.
\item {\sc Relative isomorphism axiom.} Each relative isomorphism (see \defref{def:relativeisomorphism}), of the shape 
$f\colon (\ob{K}^{[r],k},\ob{K}^{[r],k-1})\to (\ob{L}^{[r],k},\ob{L}^{[r],k-1})$, induces an isomorphism
$$C^*(\ob{K}^{[r],k},\ob{K}^{[r],k-1})\cong C^*(\ob{L}^{[r],k},\ob{L}^{[r],k-1}).$$
\item {\sc Wedge axiom.} 
Let $(\ob{K}_{i})_{i}$ be a family of \ffss~such that
the complementary subsets, $ \ob{K}_i^{[r],k}\backslash\ob{K}_i^{[r],k-1}$, are  disjoint. Then, there is
an isomorphism
$$C^*(\cup_i \ob{K}_i^{[r],k},\cup_i \ob{K}_i^{[r],k-1}))\cong\prod_iC^* (\ob{K}_i^{[r],k},\ob{K}_i^{[r],k-1}).$$
\item {\sc Filtered Dimension axiom.} For any filtered face set, $\ob{K}$, and  any $r$, we have
$H^m(C^*(\ob{K}^{[r]},\ob{K}^{[r],k}))=0$ if $m< k$.
\end{enumerate}
\end{definition}

Recall the \ffs~$\Delta^k\ast\ob{K}$ defined in \exemref{exam:simplejoin}.

\begin{definition}\label{def:mapconelike}
Let $R$ be a commutative ring and $\Psi\colon C^*\to D^*$ be a natural transformation between two filtered theories of cochains, such that 
$\Psi(\ob{K})\colon C^*(\ob{K})\to D^*(\ob{K})$
is a quasi-isomorphism for a \ffs~$\ob{K}$, of depth $v(\ob{K})<n$. 
The natural transformation $\Psi$ is
 \emph{cone-compatible for $\ob{K}$} if the map
$$\Psi(\Delta^k\ast\ob{K})\colon C^*(\Delta^k\ast\ob{K})\to D^*(\Delta^k\ast\ob{K})$$
is a quasi-isomorphism, for any $k\geq 0$.
\index{Cone-compatible}
\end{definition}

\begin{proposition}\label{prop:qipervestheory}
Let $R$ be a commutative ring. 
Let  $\Psi\colon C^*\to D^*$ be a  natural transformation between two filtered theories of cochains, cone-compatible for any \ffs~of depth strictly less than $n$,
and inducing an isomorphism in homology, \linebreak
$H^*(\Psi)\colon C^*(\{\vartheta\})\to D^*(\{\vartheta\})$,
for any singleton $\{\vartheta\}$, 
of any depth $\leq n$.
Then the cochain map $\Psi(\ob{K})\colon C^*(\ob{K})\to D^*(\ob{K})$
is a quasi-isomorphism for any filtered face set $\ob{K}=\ob{K}^{[n]}$.
\end{proposition}
 
\begin{proof}
Suppose the result is true for any filtered face set $\ob{L}$ such that $\ob{L}={\ob{L}}^{[r],k-1}$.
If $r=0$, $k=1$, the hypothesis of induction is satisfied by hypothesis and the wedge axiom. The extension axiom  gives a morphism of exact sequences 
$$
\xymatrix{ 
0\ar[r]&C^*({\ob{K}}^{[r],k},{\ob{K}}^{[r],k-1})\ar[r]\ar[d]_{f_1}
&C^*({\ob{K}}^{[r],k})\ar[r]\ar[d]^{f_2}&
C^*({\ob{K}}^{[r],k-1})\ar[r]\ar[d]^{f_3}&0\\
0\ar[r]&D^*({\ob{K}}^{[r],k},{\ob{K}}^{[r],k-1})\ar[r]
&D^*({\ob{K}}^{[r],k})\ar[r]&
D^*({\ob{K}}^{[r],k-1})\ar[r]&0,
}$$
where the morphisms $f_{i}$ are induced by $\Psi$, $i=1,\,2,\,3$.
The map $f_3$ is a quasi-isomorphism by induction. To prove that $f_2$ is one also, we are reduced to study the map $f_1$ induced between the kernels. 
Recall 
that $\cJ(\ob{K},{[r],k})$ is the set of $\tau=R_1(\sigma)$, when $\sigma\in\ob{K}^{[r],k}\backslash \ob{K}^{[r],k-1}$.  From \exemref{exam:relativeiso},
the wedge and the relative isomorphism axioms,  we get 
\begin{equation}\label{equa:relativesplit}
C^*( {\ob{K}}^{[r],k},{\ob{K}}^{[r],k-1})=
\prod_{\tau\in 
\cJ(\ob{K},{[r],k})}
C^*(\ob{K}(\tau), \ob{K}(\tau,\partial\tau))
\end{equation}
and a similar result for $D^*$. 

$\bullet$ Suppose first $k\neq 0$. 
From \propref{prop:relativehomeo} and the relative isomorphism axiom, there exist isomorphisms,
$$C^*(\ob{K}(\tau), \ob{K}(\tau,\partial\tau))\cong C^*(\Delta^k\ast\cLe(\ob{K},\tau),\partial \Delta^k\ast\cLe(\ob{K},\tau))$$
and
$$D^*(\ob{K}(\tau), \ob{K}(\tau,\partial\tau))\cong D^*(\Delta^k\ast\cLe(\ob{K},\tau),\partial \Delta^k\ast\cLe(\ob{K},\tau)).$$
As we have already observed, we have
$$\cLe(\ob{K},\tau)=(\cLe(\ob{K},\tau))^{[r-1]}
\text{ and }
\partial \Delta^k\ast\cLe(\ob{K},\tau)=(\partial \Delta^k\ast\cLe(\ob{K},\tau))^{[r],k-1}.$$
Therefore, from the induction hypothesis and the cone compatibility condition, we obtain quasi-isomorphisms, induced by 
$\Psi$,
$C^*(\Delta^k\ast\cLe(\ob{K},\tau))\to D^*(\Delta^k\ast\cLe(\ob{K},\tau))$
and
$C^*(\partial \Delta^k\ast\cLe(\ob{K},\tau))\to D^*(\partial \Delta^k\ast\cLe(\ob{K},\tau)),
$
which bring out a quasi-isomorphism, also induced by $\Psi$,
$$C^*(\Delta^k\ast\cLe(\ob{K},\tau),\partial \Delta^k\ast\cLe(\ob{K},\tau))\to
D^*(\Delta^k\ast\cLe(\ob{K},\tau),\partial \Delta^k\ast\cLe(\ob{K},\tau)).
$$

$\bullet$  If $k=0$, we have $\ob{K}(\tau,\partial\tau)=\cL(\ob{K},\tau)$ and
$\ob{K}^{[r],-1}=\ob{K}^{[r-1]}$. If $\cLe(\ob{K},\tau)\neq \emptyset$, the argument above is still valid. 
If $\cLe(\ob{K},\tau)= \emptyset$, then $\ob{K}(\tau)=\tau=\{\vartheta\}$ and we apply the hypothesis on the singletons.

Finally, we  have proved that $f_{1}$ is a quasi-isomorphism and the result is established in the case
${\ob{K}^{[r]}}=\ob{K}^{[r],k}$ for some~$k$.

The axiom of filtered dimension implies the nullity of  the relative cohomologies,\linebreak 
$H^m (C( {\ob{K}^{[r]}},{\ob{K}}^{[r],k}))$ and 
$H^m(D( {\ob{K}^{[r]}},{\ob{K}}^{[r],k}))$,
in any degree $m$ for $k$ great enough. Therefore,  we get the isomorphism 
$H^m(C({\ob{K}^{[r]}}))\cong H^m(D({\ob{K}^{[r]}}))$, for any $m$.
As $\ob{K}=\ob{K}^{[n]}$, the proof is done.
\end{proof}

 \propref{prop:qipervestheory} admits a variation on the connectivity of the involved \ffss~(used in \propref{prop:casonulo}) whose proof follows from an inspection of the previous arguments.
 
 \begin{proposition}\label{prop:qipervestheoryconnected}
 Let $R$ be a commutative ring. 
Let  $\Psi\colon C^*\to D^*$ be a  natural transformation between two filtered theories of cochains, cone-compatible for any \emph{connected} 
\ffs~of depth strictly less than $n$,
and inducing an isomorphism in homology, 
$H^*(\Psi)\colon C^*(\{\vartheta\})\to D^*(\{\vartheta\})$,
for any singleton $\{\vartheta\}$,
of any depth $\leq n$.
\index{Cone-compatible}
Then $\Psi(\ob{K})\colon C^*(\ob{K})\to D^*(\ob{K})$
is a quasi-isomorphism for any filtered face set $\ob{K}=\ob{K}^{[n]}$ whose expanded links,
$\cLe(\ob{K},\tau)$,
are connected.
 \end{proposition}
 
 We determine a class of universal local systems generating filtered theory of cochains by blow-ups.
 
\begin{proposition}\label{prop:systemlocalperversetheory} 
 Let $R$ be a commutative ring. 
Let $F$ be  an extendable universal  system of differential coefficients, such that $H^0(F(\{\vartheta\}))=R$, and $\ov{q}$ be a loose perversity. Then, the cochain complex $\tF_{\ov{q}}$ is a filtered theory of cochains.
\index{Filtered!theory of cochains}
\end{proposition}
 
\begin{proof}
We follow the properties of \defref{def:perversetheory}.

\medskip
$\bullet$ {\sc Extension axiom.} 
Let $\ob{K}$ be a \ffs~and $\tau\in \cJ(\ob{K},{[r],k})$.
We first prove that
$\tF_{\ov{q}}(\ob{K}(\tau))\to \tF_{\ov{q}}( \ob{K}(\tau,\partial\tau))$ is surjective. 

--- Begin with  $k>0$. 
Let $\sigma\in \ob{K}(\tau)\backslash \ob{K}(\tau,\partial\tau)$ and $\omega\in \tF(\ob{K}(  \tau,\partial\tau))$. 
If $\sigma\colon \Delta^k \ast \Delta^{j_{n-r+1}}\ast\cdots\ast \Delta^{j_n}\to \ob{K}_{+}$, we denote by
${\omega}_{\partial'\sigma}$ the restriction of $\omega$ to
$\partial'\sigma=\partial\Delta^k\ast  \Delta^{j_{n-r+1}}\ast\cdots\ast \Delta^{j_n}$. This restriction can be decomposed as
${\omega_{\partial'\sigma}}=\sum_m\omega'_m\otimes \omega''_m$
with $\omega'_m\in F(c\partial \Delta^k)$ and
$\omega''_m\in F(c\Delta^{j_{n-r+1}})\otimes \cdots\otimes F(\Delta^{j_n})$.
By hypothesis and \cite[Proposition 12.21]{MR736299}, the map $F(c\Delta^k)\to F(c\partial \Delta^k)$ admits a section,  $\rho$, and we may define 
$\mu(\omega)_{\sigma}=\sum_m\rho(\omega'_m)\otimes \omega''_m$. By construction, the correspondence $\sigma\mapsto \mu(\omega)_{\sigma}$ is compatible with face operators and we get a global section $\mu(\omega)\in \tF(\ob{K}(\tau))$.

As in the expression of $\sigma\colon \Delta^k \ast \Delta^{j_{n-r+1}}\ast\cdots\ast \Delta^{j_n}\to \ob{K}_{+}$, the euclidean simplices, $\Delta^{j_{i}}$, are empty for $i<n-r$, we have $\|\mu(\omega)_{\sigma}\|_{\ell}=-\infty$ if $r<\ell$. If $\ell=r$, by definition of the perverse degree of a global section, we consider the restriction,
$\beta=\sum_{m}\beta'_{m}\otimes \omega''_{m}$ of $\mu(\omega)_{\sigma}$ to
$F(\Delta^k\times\{1\})\otimes F(c\Delta^{j_{n-r+1}})\otimes\cdots\otimes F(\Delta^{j_{n}})$. The $r$-perverse degree of 
$\mu(\omega)_{\sigma}$ is the maximum of the (cohomological) degrees of the $ \omega''_{m}$'s, with $\beta'_{m}\neq 0$. As $k>0$, we have $\partial\Delta^k\neq\emptyset$ and 
 $\|\mu(\omega)_{\sigma}\|_{r} \leq\|\omega_{\partial'\sigma}\|_{r}$. A similar argument works for $\ell<r$ and we have proved
$\|\mu(\omega)_{\sigma}\|\leq \|\omega_{\partial'\sigma}\|$. 
As $\omega$ is of $\ov{q}$-intersection, we have
\begin{equation}\label{equa:extensionaxiom}
\|\mu(\omega)\|_{r}\leq \ov{q}(r) \text{ and } \|\mu(d\omega)\|_{r}\leq \ov{q}(r).
\end{equation}
This proves the $\ov{q}$-admissibility of $\mu(\omega)$ and gives also
 \begin{eqnarray*}
 \|d\mu(\omega)_{\sigma}\|_{r}&\leq_{(1)}&\max(\|d\mu(\omega)_{\sigma}-\mu(d\omega)_{\sigma}\|_{r},\|\mu(d\omega)_{\sigma}\|_{r})\\
 &\leq_{(2)}&\max(\|\sum_{m}(d\rho(\omega'_{m})-\rho(d\omega'_{m}))\otimes \omega''_{m}\|_{r},\ov{q}(r))\\
 &\leq_{(3)}&\max(\ov{q}(r),\ov{q}(r))=\ov{q}(r),
 \end{eqnarray*}
where 
\begin{itemize}
\item $\leq_{(1)}$ comes from the compatibility condition of the perverse degree with sums, see \defref{def:perversecochains},
\item $\leq_{(2)}$ is the replacement of $\mu(\omega)_{\sigma}$ by its value,
\item $\leq_{(3)}$ is a computation similar at the previous one used in the proof of the displayed formula (\ref{equa:extensionaxiom}).
\end{itemize}
Therefore,  we get an induced surjective map,
$\tF_{\ov{q}}(\ob{K}(\tau))\to \tF_{\ov{q}}( \ob{K}(\tau,\partial\tau))$.

--- We consider now the case $k=0$. 
Here, in perverse degree $r$, we have $\|\omega_{\partial'\sigma}\|_{r}=-\infty$. Therefore, for having the inequality
$\|\mu(\omega)_{\sigma}\|\leq \|\omega_{\partial'\sigma}\|$,
we need a section,
$\mu(\omega)_{\sigma}$,
with a vanishing restriction to
$F(\{\vartheta\}\times \{1\})\otimes F(c\Delta^{j_{n-r+1}})\otimes\cdots\otimes F(\Delta^{j_{n}})$.
We proceed now to the construction of this section.

As $H^0(F(\{\vartheta\})=R$, there are two cocycles in $F^0(\{\vartheta\})$ which correspond to $0\in R$ and $1\in R$, respectively. We denote them by~0 and~1.
The surjectivity of $F(\Delta^1)=F(c\{\vartheta\})\to
F(\partial \Delta^1)=F((\{\vartheta\}\times\{1\})\sqcup( \{\vartheta\}\times \{0\}))$
 implies the existence of
${\alpha}\in F^0(c\{\vartheta\})$ such that
$${\alpha}_{|\{\vartheta\}\times\{1\}}=0 \text{ and }
{\alpha}_{|\{\vartheta\}\times\{0\}}=1. 
$$
Let $\sigma\colon \{\vartheta\}\ast \Delta^{j_{n-r+1}}\ast\cdots\ast \Delta^{j_n}\to \ob{K}_{+}$ be a simplex of $\ob{K}(\{\vartheta\})$.
We identify the product $c\Delta^{j_{n-r+1}}\times\cdots\times \Delta^{j_n}$
to the subset
$\left(\{\vartheta\}\times\{0\}\right)\times c\Delta^{j_{n-r+1}}\times\cdots\times \Delta^{j_n}$
of
$c\{\vartheta\}\times c\Delta^{j_{n-r+1}}\times\cdots\times \Delta^{j_n}$ and denote by $\sigma'$ the restriction of $\sigma$ to
$c\Delta^{j_{n-r+1}}\times\cdots\times \Delta^{j_n}$.
Let $\omega\in\widetilde{F}(\cL(\ob{K},\{\vartheta\}))$. We set $\mu(\omega)_{\sigma}={\alpha}\otimes \omega_{\sigma'}$. This construction  is compatible with the face operators and define an element $\mu(\omega)\in \tF(\ob{K}(\{\vartheta\}))$ whose restriction to $\widetilde{F}(\cL(\ob{K},\{\vartheta\}))$ coincides with $\omega$. As  
the restriction of ${\alpha}$ and $d{\alpha}$ to $\{\vartheta\}\times\{1\}$
are equal to 0, we have, for any $\omega'\in\widetilde{F}(\cL(\ob{K},\{\vartheta\}))$, 
$$
\|d{\alpha}\otimes \omega'\|_{\ell}\leq \|{\alpha}\otimes \omega'\|_{\ell}=
\left\{
\begin{array}{lcl}
-\infty&\text{ if } & \ell\geq r,\\
\|\omega'\|_{\ell}&\text{ if }& \ell <r.
\end{array}\right.
$$
By applying these inequalities alternatively to $\omega'=\omega_{\sigma'}$ and $\omega'=d\omega_{\sigma'}$, we obtain
\begin{eqnarray*}
\|\mu(\omega)_{\sigma}\|
&=&
\|{\alpha}\otimes \omega_{\sigma'}\|\leq \|\omega_{\sigma'}\|,\\
\|d\mu(\omega)_{\sigma}\|
&\leq&
\max\left(\|d{\alpha}\otimes\omega_{\sigma'}\|,\|{\alpha}\otimes d\omega_{\sigma'}\|\right)\\
&\leq&
\max\left(\|\omega_{\sigma'}\|,\|d\omega_{\sigma'}\|\right).
\end{eqnarray*}
The surjectivity of
$\tF_{\ov{q}}(\ob{K}(\{\vartheta\}))\to \tF_{\ov{q}}(\cL(\ob{K},\{\vartheta\}))$
follows.

\medskip
--- We prove now the surjectivity of
$\tF_{\ov{q}}(\ob{K}^{[r],k})\to \tF_{\ov{q}}(\ob{K}^{[r],k-1})$, by detailing only the case $k>0$, the case $k=0$ being similar.
Observe,  from \corref{cor:intersectionKtau}, that $\ob{K}^{[r],k}$ can be obtained as a push-out, built on disjoint unions,
\begin{equation}\label{equa:pushout}
\xymatrix{ 
\sqcup_{\tau\in \cJ(\ob{K},{[r],k})}\ob{K}(\tau,\partial \tau)\ar[r]\ar[d]&\sqcup_{\tau\in\cJ(\ob{K},{[r],k})}\ob{K}(\tau)\ar[d]\\
\ob{K}^{[r],k-1}\ar[r]&
\ob{K}^{[r],k}.
}\end{equation}
By definition, the  blow-up $\tF_{\ov{q}}$  is compatible with colimits and we get a pullback
$$\xymatrix{ 
\tF_{\ov{q}}\left(\sqcup_{\tau\in \cJ(\ob{K},{[r],k})}\ob{K}(\tau,\partial \tau)\right)&\ar[l]
\tF_{\ov{q}}\left(\sqcup_{\tau\in\cJ(\ob{K},{[r],k})}\ob{K}(\tau)\right)\\
\tF_{\ov{q}}\left(\ob{K}^{[r],k-1}\right)\ar[u]&
\tF_{\ov{q}}\left(\ob{K}^{[r],k}\right).\ar[l]\ar[u]
}$$
By definition also, we have 
$\tF_{\ov{q}}\left(\sqcup_{\tau\in \cJ(\ob{K},{[r],k})}\ob{K}(\tau,\partial \tau)\right)=
\prod_{\tau\in \cJ(\ob{K},{[r],k})} \tF_{\ov{q}}(\ob{K}(\tau,\partial \tau)$
and \linebreak
$\tF_{\ov{q}}\left(\sqcup_{\tau\in\cJ(\ob{K},{[r],k})}\ob{K}(\tau)\right)=
\prod_{\tau\in \cJ(\ob{K},{[r],k})} \tF_{\ov{q}}(\ob{K}(\tau))$.
The surjectivity of the map
$\tF_{\ov{q}}(\ob{K}(\tau))\to \tF_{\ov{q}}( \ob{K}(\tau,\partial\tau))$
implies the surjectivity of
$\tF_{\ov{q}}(\ob{K}^{[r],k})\to \tF_{\ov{q}}(\ob{K}^{[r],k-1})$.

\smallskip
Moreover, observe that, if $\ob{L}$ and $\ob{L}'$ are \ffss, connected by a push-out of filtered face maps,
\begin{equation}\label{equa:addcell}\xymatrix{ 
\sqcup_{\tau\in I}\ob{K}(\tau,\partial \tau)\ar[r]\ar[d]&\sqcup_{\tau\in I}\ob{K}(\tau)\ar[d]\\
\ob{L'}\ar[r]&
\ob{L},
}\end{equation}
as in (\ref{equa:pushout}), the previous argument implies the surjectivity of $\tF_{\ov{q}}(\ob{L})\to\tF_{\ov{q}}(\ob{L'})$. 

\medskip
$\bullet$ {\sc Relative isomorphism axiom.} \\ A relative isomorphism 
$f\colon (\ob{K}^{[r],k},\ob{K}^{[r],k-1})\to (\ob{L}^{[r],k},\ob{L}^{[r],k-1})$
induces a morphism of complexes
$$f^*\colon \tF_{\ov{q}}(\ob{L}^{[r],k},\ob{L}^{[r],k-1})\to \tF_{\ov{q}}(\ob{K}^{[r],k},\ob{K}^{[r],k-1}).
$$
If $\eta\in \tF_{\ov{q}}(\ob{K}^{[r],k},\ob{K}^{[r],k-1})$ and $\sigma\in\ob{L}^{[r],k}$, we set
$$\omega_{\sigma}=\left\{
\begin{array}{ccl}
\eta_{f^{-1}(\sigma)}&\text{ if }& \sigma\in \ob{L}^{[r],k}\backslash \ob{L}^{[r],k-1},\\
0&\text{ if }& \sigma\in \ob{L}^{[r],k-1}.
\end{array}\right.$$
As $f^{-1}(\sigma)$ exists and is uniquely defined, the element $\omega$  is well defined.
We have to check the compatibility with face operators. Let $\partial_j$ be a face operator and $\sigma\in\ob{L}^{[r],k}$.
\begin{itemize}
\item if $\partial_j\sigma\in \ob{L}^{[r],k}\backslash \ob{L}^{[r],k-1}$, then we have
$$\omega_{\partial_j\sigma}=\eta_{f^{-1}(\partial_j\sigma)}
=\eta_{\partial_jf^{-1}(\sigma)}=\tilde{\partial}_j\eta_{f^{-1}(\sigma)}=\tilde{\partial}_j\omega_{\sigma}.$$
(We used ${\partial_jf^{-1}(\sigma)}={f^{-1}(\partial_j\sigma)}$ which can be verified by composing with $f$.)
\item if $\partial_j\sigma\in \ob{L}^{[r],k-1}$, then we have
$\tilde{\partial}_j\omega_{\sigma}=\tilde{\partial}_j\eta_{f^{-1}(\sigma)}=\eta_{\partial_jf^{-1}(\sigma)}=0$, because ${\partial_jf^{-1}(\sigma)}\in \ob{K}^{[r],k-1}$. (If ${\partial_jf^{-1}(\sigma)}\in  \ob{K}^{[r],k}\backslash \ob{K}^{[r],k-1}$ then we should have $\partial_j\sigma\in \ob{L}^{[r],k}\backslash \ob{L}^{[r],k-1}$.)
By construction, we have also $\omega_{\partial_j\sigma}=0$.
\end{itemize}

\medskip
$\bullet$ {\sc Wedge axiom.} The restriction maps give a morphism of complexes
$$\tF_{\ov{q}}(\cup_i\ob{K}_i^{[r],k},\cup_i\ob{K}_i^{[r],k-1})\to
\prod_i\tF_{\ov{q}}(\ob{K}_i^{[r],k},\ob{K}_i^{[r],k-1}).$$
This morphism  admits an inverse: for any family $(\omega_i)_i\in \prod_i\tF_{\ov{q}}(\ob{K}_i^{[r],k},\ob{K}_i^{[r],k-1})$, we define a global section $\omega\in \tF_{\ov{q}}(\cup_i\ob{K}_i^{[r],k},\cup_i\ob{K}_i^{[r],k-1})$ by
$$\omega_{\sigma}=\left\{
\begin{array}{cl}
(\omega_i)_{\sigma}&\text{ if } \sigma\in \ob{K}_i^{[r],k}\backslash\ob{K}_i^{[r],k-1},\\
0&\text{ otherwise.}
\end{array}\right.$$
The index $i$ such that $ \sigma\in \ob{K}_i^{[r],k}\backslash\ob{K}_i^{[r],k-1}$ being unique, the element $\omega_{\sigma}$  is well defined. We have to check its compatibility with face operators. Let $ \sigma\in \ob{K}_i^{[r],k}$.
\begin{itemize}
\item If $\partial_j\sigma\in \ob{K}_i^{[r],k}\backslash\ob{K}_i^{[r],k-1}$, we have
$\omega_{\partial_j\sigma}=(\omega_i)_{\partial_j\sigma}=\tilde{\partial}_j(\omega_i)_{\sigma}=\tilde{\partial}_j\omega_{\sigma}$.
\item If $\partial_j\sigma\in \ob{K}_i^{[r],k-1}$, we have $\omega_{\partial_j\sigma}=0$ and
$\tilde{\partial}_j\omega_{\sigma}=\tilde{\partial}_j(\omega_i)_{\sigma}=(\omega_i)_{\partial_j\sigma}=0$, because
$\omega_i\in \tF_{\ov{q}}(\ob{K}_i^{[r],k},\ob{K}_i^{[r],k-1})$.
\end{itemize}

\medskip
$\bullet$ {\sc Filtered dimension axiom.}
We have to prove the nullity of the relative cohomology with coefficients in $\tF_{\ov{q}}$, i.e.,
$H^m_{\ov{q}}( {\ob{K}^{[r]}},{\ob{K}}^{[r],k};\tF)=0$ for $m<k$.
We establish it in three steps.
 
\smallskip
(i) If $\tau\in \cJ(\ob{K},{[r],k})$ and  $m<k$, we prove$$H^m_{\ov{q}}(\ob{K}(\tau), \ob{K}(\tau,\partial\tau);\tF)=0.$$ 
We denote by $\varLambda^{m,k}$ the subcomplex of $\Delta^k$ generated by the simplices containing the vertex $v_0$ and of dimension less than, or equal, to $m$ and by $\varLambda^{m,\tau}$ the restriction of $\tau$ to $\varLambda^{m,k}$. 
(Observe that $\varLambda^{k,k}=\Delta^k$ and $\varLambda^{k,\tau}=\tau$.)
 Let $f_i\colon \Delta^m\to \varLambda^{m,k}$, $i\in I_{m}$, the $m$-dimensional simplices of $\varLambda^{m,k}$. We  set $\tau_i=\varLambda^{m,\tau}\circ f_i$ and
 define $\ob{K}(\varLambda^{m,\tau})=\cup_{i\in I} \ob{K}(\tau)(\tau_i)$,
 where (see \defref{def:linkstar})
 $$\ob{K}(\tau)(\tau_{i})=\{\xi\in\ob{K}(\tau)\mid \exists \sigma\in\ob{K}(\tau) \text{ with } \xi \vartriangleleft \sigma \text{ and } R_{1}(\sigma)=\tau_{i}\}.$$
 We first prove that
\begin{equation}\label{equa:relative}
H^*_{\ov{q}}(\ob{K}(\tau),\ob{K}(\varLambda^{m,\tau});\tF)=0, \text{ for any } m,\, 0\leq m\leq  k.
\end{equation}
If $\tau\in \cJ(\ob{K},{[r],0})$, then $\varLambda^{0,\tau}=\tau$ and we have $H^*_{\ov{q}}(\ob{K}(\tau),\ob{K}(\varLambda^{0,\tau});\tF)=0$ as required. We use a first induction on the dimension $|\tau|$ of $\tau$, by supposing that
(\ref{equa:relative})
is true for any $\tau'$ of dimension strictly less than $|\tau|$. 

Secondly, we do a decreasing induction assuming
$H^*_{\ov{q}}(\ob{K}(\tau),\ob{K}(\varLambda^{j,\tau});\tF)=0$ for some $j$, $j\leq k$.
As $\varLambda^{k,\tau}=\tau$, we have 
$H^*_{\ov{q}}(\ob{K}(\tau),\ob{K}(\varLambda^{k,\tau});\tF)=0$ and the inductive property is fulfilled for $j=k$. The proof of (\ref{equa:relative}) is reduced to
\begin{equation}\label{equa:induction2}
H^*_{\ov{q}}(\ob{K}(\tau),\ob{K}(\varLambda^{j-1,\tau});\tF)=0.
\end{equation}
Using the (already established) wedge axiom, we have an isomorphism
$$\tF_{\ov{q}}(\ob{K}(\varLambda^{j,\tau}), \ob{K}(\varLambda^{j-1,\tau}))\cong
\prod_{i\in I_{j}}
\tF_{\ov{q}}(\ob{K}(\tau_i),\ob{K}(\varLambda^{j-1,\tau_i})).$$
The first induction hypothesis (on the dimension of $\tau$) gives
$$H^*_{\ov{q}}(\ob{K}(\tau_i),\ob{K}(\varLambda^{j-1,\tau_i});\tF)=0,$$
from which we deduce
\begin{equation}\label{equa:relative2}
H^*_{\ov{q}}(\ob{K}(\varLambda^{j,\tau}),\ob{K}(\varLambda^{j-1,\tau});\tF)=0.
\end{equation}
By combining the short exact sequence of pairs with the snake lemma, the extension axiom generates a short exact sequence
$$
\xymatrix@1{
0\ar[r]&
\tF_{\ov{q}}(\ob{K}(\tau),\ob{K}(\varLambda^{j,\tau}))
\ar[r]&
\tF_{\ov{q}}(\ob{K}(\tau),\ob{K}(\varLambda^{j-1,\tau}))
\ar[r]&
\tF_{\ov{q}}(\ob{K}(\varLambda^{j,\tau},\ob{K}(\varLambda^{j-1,\tau}))
\ar[r]&
0.
}$$
This sequence and the equality (\ref{equa:relative2}) imply
$$
H^*_{\ov{q}}(\ob{K}(\tau),\ob{K}(\varLambda^{j-1,\tau});\tF)\cong
H^*_{\ov{q}}(\ob{K}(\tau),\ob{K}(\varLambda^{j,\tau});\tF),
$$
which is trivial with the second (decreasing) induction. Finally, we have proved the equality
(\ref{equa:relative}).

\smallskip
Consider now $\iota\colon \Delta^{k-1}\to\Delta^k$, the face opposite to $v_0$ and set $\tau'=\tau\circ\iota$. With the extension axiom, we have an exact sequence,
$$
{\xymatrix@=16pt{
0\ar[r]&
\tF_{\ov{q}}(\ob{K}(\tau),\ob{K}(\tau,\partial\tau))
\ar[r]&
\tF_{\ov{q}}(\ob{K}(\tau),\ob{K}(\varLambda^{k-1,\tau}))
\ar[r]&
\tF_{\ov{q}}(\ob{K}(\tau'),\ob{K}(\tau',\partial\tau'))
\ar[r]&
0.}
}$$
The equality (\ref{equa:relative}) implies the existence of an isomorphism of degree +1 between the relative cohomologies involving $\tau$ and $\tau'$.
Starting from $\tau'=\{\vartheta\}$ and
$\tF_{\ov{p}}(\ob{K}(\tau'),\ob{K}(\tau',\partial\tau'))=
\tF_p(\ob{K}(\{\vartheta\}),\cL(\ob{K},\{\vartheta\}))$,
whose cohomology is zero in negative degree, we deduce
$$
H^m_{\ov{q}}(\ob{K}(\tau),\ob{K}(\tau,\partial\tau);\tF)=0, \text{ if } m< |\tau|=k,
$$ 
and the result follows.

\smallskip
(ii) As 
the map
 $(\cup_{\tau\in \cJ(\ob{K},{[r],k})} \ob{K}(\tau), \cup_{\tau\in \cJ(\ob{K},{[r],k})} \ob{K}(\tau,\partial\tau))\to
 (\ob{K}^{[r],k},\ob{K}^{[r],k-1})$
 is a relative isomorphism, the  previous step, the wedge and the relative isomorphism axioms imply
 $$H^m_{\ov{q}}( {\ob{K}}^{[r],k},{\ob{K}}^{[r],k-1};\tF))=
 \prod_{\tau\in \cJ(\ob{K},{[r],k})} H^m_{\ov{q}}(\ob{K}(\tau),\ob{K}(\tau,\partial\tau),\tF)=0,$$
for any $m$, $m<k$. 

\smallskip
(iii) Finally, we prove $H^m_{\ov{q}}( {\ob{K}^{[r]}},{\ob{K}}^{[r],k};\tF)=0$,  if $m< k$.\\
Let $\omega\in \tF_{\ov{q}}( {\ob{K}^{[r]}},{\ob{K}}^{[r],k})$, of degree $m$, with $d\omega=0$.  The restriction $\omega_{|{\ob{K}}^{[r],k+1}}$ is a coboundary (with step (ii)) and there exists $\psi_0\in \tF_{\ov{q}}( {\ob{K}}^{[r],k+1},{\ob{K}}^{[r],k})$ such that
$\omega_{|{\ob{K}}^{[r],k+1}}=d\psi_0$. 
By iterating the extension axiom, the global section $\psi_0$ can be extended in a global section defined on ${\ob{K}^{[r]}}$ that we still denote $\psi_0$, i.e., we have $\psi_{0}\in \tF_{\ov{q}}(\ob{K}^{[r]},\ob{K}^{[r],k})$ and 
$\omega-d\psi_{0}\in \tF_{\ov{q}}(\ob{K}^{[r]},\ob{K}^{[r],k+1})$.

Suppose that, for any $i$, $0\leq i\leq j$, we have built global sections
$\psi_{i}\in \tF_{\ov{q}}(\ob{K}^{[r]},\ob{K}^{[r],k+i})$
such that
$\omega-d(\sum_{i=0}^j\psi_{i})\in \tF_{\ov{q}}(\ob{K}^{[r]},\ob{K}^{[r],k+j+1})$.
The restriction of
$\omega-d(\sum_{i=0}^j\psi_{i})$
to $\ob{K}^{[r],k+j+2}$ being a coboundary (by step (ii)), there exists
$\psi_{j+1}\in
\tF_{\ov{q}}(\ob{K}^{[r],k+j+2},\ob{K}^{[r],k+j+1})$
such that
$d\psi_{j+1}=(\omega-d(\sum_{i=0}^j\psi_{i}))_{|\ob{K}^{[r],k+j+2}}$. We extend 
$\psi_{j+1}$ in a global section defined on $\ob{K}^{[r]}$, still denoted $\psi_{j+1}$. This section verifies
$$\psi_{j+1}\in
\tF_{\ov{q}}(\ob{K}^{[r]},\ob{K}^{[r],k+j+1})
\text{ and }
\omega-\sum_{i=0}^{j+1}d\psi_{i}\in
\tF_{\ov{q}}(\ob{K}^{[r]},\ob{K}^{[r],k+j+2}).
$$
 We set $\psi=\sum_{i=0}^{\infty} \psi_i$. When we apply $\psi$ to a simplex $\sigma\in\ob{K}^{[r]}_{+}$, the sum $\psi_{\sigma}$ is finite and we have $\psi\in \tF_{\ov{q}}( {\ob{K}^{[r]}},{\ob{K}}^{[r],k})$ with $\omega=d\psi$.
\end{proof}

\begin{definition}\label{def:troncature}
For any positive integer $s$ and any cochain complex, $C^*$, over the commutative ring $R$, the \emph{$s$-truncation} of $C^*$ is the cochain complex, $\tau_{\leq s} C^*$, defined  by
\index{Truncation@$s$-Truncation}
$$(\tau_{\leq s} C)^r=\left\{
\begin{array}{cl}
C^r&\text{ if } r<s,\\
\cZ C^s&\text{ if } r=s,\\
0&\text{ if } r>s,
\end{array}\right.$$
where $\cZ C^s$ denotes the $R$-module of cocycles of $C$ in degree $s$.
\end{definition}

\begin{proposition}\label{prop:blowupcohomologycone} 
Let $R$ be a commutative ring and $\ell$ be an integer such that $1\leq \ell\leq n$.
Let $F$ be  an extendable universal  system of differential coefficients and $\ov{q}$ be a loose perversity.
Let $\ob{K}$ be a filtered face set of depth $v(\ob{K})\leq \ell-1$ and $\Delta^k\ast\ob{K}$ be a join of depth $\ell$ with $k\geq 0$.
Then we have
$$H^i_{\ov{q}}(\Delta^k\ast\ob{K};\tF)=
H^i(F(\{\vartheta\})\otimes \tau_{\leq\ov{q}(\ell)}\tF_{\ov{q}}(\ob{K})).$$
\end{proposition}

\begin{proof}
An element of $F(c\Delta^k)$ is  determined by its value on the maximal simplex $[c\Delta^k]$. 
As the universal system, $F$, and $\Delta^k$ are given, $F(c\Delta^k)_{[c\Delta^k]}$ is already determined and
an element $\eta\in F(c\Delta^k)\otimes \tF(\ob{K})$ is given by the values $\eta_{\alpha}\in\tF(\ob{K})$ for $\alpha\in\ob{K}_{+}$.
 Also, for the same reason, any element of $\tF(\Delta^k\ast\ob{K})$ is  determined by the values $\omega_{\alpha}\in\tF(\ob{K})$ for 
$\alpha\in\ob{K}_{+}$. This remark implies the existence of an isomorphism,
$$\tF(\Delta^k\ast\ob{K})\cong F(c\Delta^k)\otimes \tF(\ob{K}).$$

Consider now the short exact sequence
$$\xymatrix@1{
0\ar[r]&
F(c\Delta^k,\Delta^k\times\{1\})\ar[r]&
F(c\Delta^k)\ar[r]&
F(\Delta^k\times\{1\})\ar[r]
&
0.
}$$
As $F$ is a universal system of differential coefficients and $F(-)$ is free as $R$-module, this sequence splits and we have an isomorphism of cochain complexes,
$$F(c\Delta^k)\cong F(c\Delta^k,\Delta^k\times\{1\})\oplus F(\Delta^k\times\{1\}).$$
Let $\alpha\colon \Delta^{j_{n-\ell+1}}\ast\cdots\ast\Delta^{j_n}\to \ob{K}_{+}$,
$\eta \in F(c\Delta^k)\otimes \tF(\ob{K})$ and $\ov{q}$ be a loose perversity. We decompose
 $$\eta_{\Delta^k\ast\alpha}=\sum_i(\eta'_{i,\alpha}+\eta''_{i,\alpha})\otimes \eta'''_{i,\alpha}\in 
( F(c\Delta^k,\Delta^k\times\{1\})\oplus F(\Delta^k\times\{1\}))\otimes \tF(\ob{K}).$$
 By definition, the perverse degree of $\eta_{\Delta^k\ast\alpha}$ is determined as follows:
 \begin{itemize}
 \item if $\ell<\ell'$, then 
 $\|\eta_{\Delta^k\ast\alpha}\|_{\ell'}=-\infty$,
 \item if $\ell'=\ell$, then, the $\ell$-perverse degree of $\eta'_{i,\alpha}\otimes \eta'''_{i,\alpha}$ is equal to $-\infty$. Therefore,
 \begin{itemize}
\item if $\eta''_{i,\alpha\mid \Delta^k\times\{1\}}\otimes \eta'''_{i,\alpha}=0$ for all $i$, we have 
$ \|\eta_{\Delta^k\ast\alpha}\|_{\ell}=-\infty$, 
\item if not, we have
 $$\|\eta_{\Delta^k\ast\alpha}\|_{\ell}=
\displaystyle{\max_i\{|\eta'''_{i,\alpha}|\text{ such that } {\eta''_{i,\alpha}}_{|\Delta^k\times\{1\}}\neq 0\},}
$$
\end{itemize}
  \item if $\ell>\ell'$, then, 
  \begin{itemize}
  \item if $\eta''_{i,\alpha}=0$ for all $i$, we have ${\|\eta_{\Delta^k\ast\alpha}\|}_{\ell'}=-\infty$, 
  \item if not, we have
 $${\|\eta_{\Delta^k\ast\alpha}\|}_{\ell'}=\max_i\left\{\|\eta'''_{i,\alpha}\|_{\ell'}\text{ such that } \eta''_{i,\alpha}\neq 0\right\},$$
 \end{itemize}
 and the perverse degree of $\eta$ coincide with the perverse degree of its component in $\tF(\ob{K})$.
 \end{itemize}
 As a consequence, if we consider the $\ell'$-perverse degree for $\ell'<\ell$, the fact of being of $\ov{q}$-intersection involves only the $\ov{q}$-intersection condition of the component in $\tF(\ob{K})$. 
 For the $\ell$-perverse degree, the situation is different: there is no restriction for the component which involves
 $F(c\Delta^k,\Delta^k\times\{1\})$
 and there is a restriction on the (cohomological) degree in $\tF(\ob{K})$ for the component having
 $F(\Delta^k\times\{1\})$ as first factor.
 In summary, we have a decomposition of $\tF_{\ov{q}}(\Delta^k\ast \ob{K})$ as
 $$(F(c\Delta^k,\Delta^k\times\{1\})\otimes \tF_{\ov{q}}(\ob{K}))\oplus (F(\Delta^k\times\{1\})\otimes \tau_{\leq\ov{q}(\ell)}\tF_{\ov{q}}(\ob{K})).
 $$
 
 As $F$ is a system of differential coefficients, there is a quasi-isomorphism between
 $F(\Delta^k\times\{1\})$ and $F(\{\vartheta\})$
 and 
 we get the quasi-isomorphisms:
 \begin{eqnarray*}
 \tF_{\ov{q}}(\Delta^k\ast \ob{K})&\simeq&
F(\Delta^k\times\{1\})\otimes \tau_{\leq\ov{q}(\ell)}\tF_{\ov{q}}(\ob{K})\\
&\simeq&
F(\{\vartheta\})\otimes \tau_{\leq\ov{q}(\ell)}\tF_{\ov{q}}(\ob{K}).
\end{eqnarray*}
\index{TW!cohomology!of a cone}
\end{proof}

\begin{corollary}\label{cor:blowupcohomologycone}
If we add $H^*(F(\{\vartheta\}))=R$ to the hypotheses of \propref{prop:blowupcohomologycone}, we obtain:
$$H^i_{\ov{q}}(\Delta^k\ast\ob{K};\tF)=\left\{
\begin{array}{cl}
H^i_{\ov{q}}(\ob{K};\tF)&\text{ if } i\leq \ov{q}(\ell),\\
0&\text{ if } i> \ov{q}(\ell).
\end{array}\right.$$
\end{corollary}

\begin{proof}
By \defref{def:universalsystem}, the module $F(\{\vartheta\})$ is $R$-free. Thus,  there is a K\"unneth spectral sequence, converging to 
$H^*(F(\{\vartheta\})\otimes \tau_{\leq\ov{q}(\ell)}\tF_{\ov{q}}(\ob{K}))$,
and whose second page is
$$E_{2}^{i,j}=\oplus_{u+v=j} \Tor^i(H^u(F(\{\vartheta\})),H^v(\tau_{\leq\ov{q}(\ell)}\tF_{\ov{q}}(\ob{K}))).
$$
As $H^*(F(\{\vartheta\})=R$, all the terms $\Tor^i(-,-)$ are zero if $i>0$, and this spectral sequence collapses in
$$E_{2}^{0,j}=H^j(\tau_{\leq\ov{q}(\ell)}\tF_{\ov{q}}(\ob{K}))\cong E_{\infty}^{0,j}\cong H^j(F(\{\vartheta\})\otimes \tau_{\leq\ov{q}(\ell)}\tF_{\ov{q}}(\ob{K}))=H^j_{\ov{q}}(\Delta^k\ast\ob{K};\tF),$$
where the last equality comes from \propref{prop:blowupcohomologycone}. By definition of the truncation, we have
$$H^j(\tau_{\leq\ov{q}(\ell)}\tF_{\ov{q}}(\ob{K}))=\left\{
\begin{array}{cl}
H^j_{\ov{q}}(\ob{K};\tF)&\text{ if } j\leq \ov{q}(\ell),\\
0&\text{ if } j> \ov{q}(\ell).
\end{array}\right.$$
\end{proof}

\begin{proof}[Proof of \thmref{thm:extendable}]
Let $\sigma\colon \Delta^{j_0}\ast\cdots\ast \Delta^{j_n}\to \ob{K}_{+}$. We define
$$\psi_{\sigma}\colon
\widetilde{F}_{\sigma}=F(c\Delta^{j_0})\otimes\cdots\otimes F(\Delta^{j_n})\to
\widetilde{G}_{\sigma}=G(c\Delta^{j_0})\otimes\cdots\otimes G(\Delta^{j_n})$$
by $\psi_{\sigma}=\psi_{c\Delta^{j_0}}\otimes\cdots\otimes \psi_{\Delta^{j_n}}$. By construction, the map $\psi_{\sigma}$ decreases the perverse degree and induces a cochain map $\Psi({\ob{K}})\colon \widetilde{F}_{\ov{q}}(\ob{K})\to \widetilde{G}_{\ov{q}}(\ob{K})$. Also, in the case of a natural transformation between universal systems of DGA's, the induced map
$\tF(\ob{K})\to \tG(\ob{K})$
is compatible with the laws of algebras, as follows directly from their construction in \defref{def:localsystemfiltered}.

The hypothesis of \propref{prop:qipervestheory} are satisfied: the blow-ups $\tF_{\ov{q}}$ and $\tG_{\ov{q}}$ are filtered theories of cochains thanks to \propref{prop:systemlocalperversetheory}. Moreover, the natural transformation $\tF_{\ov{q}}(-)\to \tG_{\ov{q}}(-)$
is cone-compatible for any \ffs, thanks to \corref{cor:blowupcohomologycone}.
\end{proof}

\section{Cochains on  filtered face sets}\label{sec:chaincochain} 

\begin{quote}
We compare the Thom-Whitney  complex (cf. \defref{def:thomwhitney}) to a cochain complex which is the linear dual of the analogous for \ffss~of the classical Goresky and MacPherson intersection chain complex. 
The main result (\thmref{thm:thetwocochains}) is the existence of a quasi-isomorphism between these two complexes, over any field of coefficients, for complementary perversities.
\end{quote}

 Let $\ob{K}$ be a filtered face set and $R$ be a commutative ring. The chain complex $C_*^{\GM}(\ob{K};R)$ is defined as follows,
 \begin{itemize}
 \item as module, $C_*^{\GM}(\ob{K};R)$ is the free $R$-module generated by the simplices $\sigma\in\ob{K}$,
 \item any simplex $\sigma\colon \Delta^{j_0}\ast\cdots\ast \Delta^{j_n}\to\ob{K}$ has a perverse degree defined by
 $\|\sigma\|_{\ell}=\dim(\Delta^{j_0}\ast\cdots\ast\Delta^{j_{n-\ell}})$,
 if $\Delta^{j_0}\ast\cdots\ast\Delta^{j_{n-\ell}}\neq\emptyset$ and 
 $\|\sigma\|_{\ell}=-\infty$ otherwise,
 \item the differential of a simplex
 $\sigma\colon \Delta^k= \Delta^{j_0}\ast\cdots\ast \Delta^{j_n}\to\ob{K}$ is given as usual by
 $\partial\sigma=\sum_{i=0}^k(-1)^i\sigma\circ\delta_{i}$, with $\delta_{i}\in {\pmb\Delta}^{[n]}_\cF$.
 \end{itemize}
 Let $\ov{p}$ be a loose perversity.
A simplex $\sigma\colon \Delta= \Delta^{j_0}\ast\cdots\ast \Delta^{j_n}\to\ob{K}$
is $\ov{p}$-admissible if
$$\dim (\Delta^{j_0}\ast\cdots\ast\Delta^{j_{n-\ell}})=\|\sigma\|_\ell\leq \dim \Delta -\ell+\ov{p}(\ell),$$
for any $\ell$, such that $\ell\in\{1,\ldots,n\}$. (See \remref{rem:perversity0} for the case $\ell=0$.)
\index{Chain!admissible for a perversity}\index{Chain!of intersection for a perversity}
A chain $c$ is \emph{$\ov{p}$-admissible} if $c$ can be written as a linear combination of $\ov{p}$-admissible simplices. A chain $c$ is of \emph{$\ov{p}$-intersection} if $c$ and $\partial c$ are $\ov{p}$-admissible.
We denote by $C_{*}^{\GM,\ov{p}}(\ob{K};R)$ the complex of $\ov{p}$-intersection chains. 
The \emph{GM-cochain complexes} are the duals of these complexes, i.e.,
\index{GM!cochain complex}\index{Goresky-MacPherson|see{GM}}\index{GM!cohomology}
$C^{*}_{\GM}(\ob{K};R)=\hom(C_*^{\GM}(\ob{K};R),R)$
{ and }
$C^{*}_{\GM,\ov{p}}(\ob{K};R)=\hom(C^{\GM,\ov{p}}_{*}(\ob{K};R),R)$.

We denote by
$H_*^{\GM,\ov{p}}(\ob{K};R)$
the homology of the complex $C_{*}^{\GM,\ov{p}}(\ob{K};R)$,
by
$H^*_{\GM,\ov{p}}(\ob{K};R)$
 the homology of the complex $C^{*}_{\GM,\ov{p}}(\ob{K};R)$ and called them, respectively, Goresky and MacPherson homology and cohomology of $\ob{K}$.
 
 If $X$ is a filtered space and $\ob{L}=\ob{\rm ISing}_*^{\cF}(X)$, as in \exemref{exam:simplexesfiltrés}, we have
 $C^*_{\GM,\ov{p}}(\ob{L};R)=\hom(C_*^{\ov{p}}(X;R),R)$, where $C_*^{\ov{p}}(X;R)$ is introduced in \defref{def:chaineintersection}.

\begin{proposition}\label{prop:GMcochainsperversetheory}
Let $R$ be a commutative ring and $\ov{p}$ be a loose perversity. The \emph{GM-cochain complex} $C^{*}_{\GM,\ov{p}}(-;R)$ is a filtered theory of cochains.
\index{Filtered!theory of cochains}
\end{proposition}

\begin{proof}
As all the chain and cochain complexes of this proof are over $R$, we do not mention explicitly the coefficients. We  establish the properties of \defref{def:perversetheory} for a \ffs, $\ob{K}$.

\medskip
$\bullet$ {\sc Extension axiom.}
We denote by $\cR$ the map induced by the canonical inclusion,
$$\cR\colon \hom(C^{\GM,\ov{p}}_{*}(\ob{K}^{[r],k}),R)\to
\hom(C^{\GM,\ov{p}}_{*}(\ob{K}^{[r],k-1}),R).
$$
A chain $c\in C_*(\ob{K}^{[r],k})$ can be decomposed as $c=c_1+c_2$, with
$c_1\in C_{*}(\ob{K}^{[r],k-1})$
and
$c_2\in C_{*}(\ob{K}^{[r],k}\backslash \ob{K}^{[r],k-1})$.
If $c$ is of $\ov{p}$-intersection, then $c_1$, $c_2$ are $\ov{p}$-admissible and the  boundary, $\partial c$,  can be written as a linear combination of $\ov{p}$-admissible simplices. We decompose $\partial c_2$ in
$\partial c_2=\partial'c_2+\partial''c_2$, with
$\partial'c_2\in C_{*}(\ob{K}^{[r],k-1})$
and
$\partial''c_2\in C_{*}(\ob{K}^{[r],k}\backslash \ob{K}^{[r],k-1})$.
The elements of the linear combination $\partial''c_2$ cannot be canceled with elements coming from $\partial' c_2$ or $\partial c_1$, therefore $\partial'' c_2$ is $\ov{p}$-admissible.

Let $\partial_i$ be any face operator and $\sigma$ be a $\ov{p}$-admissible simplex. If we have strict inequalities, 
$\|\partial_i\sigma\|_{\ell}<\|\sigma\|_{\ell}$, for any $\ell$, then, directly from the definition, we observe that $\partial_i\sigma$ is $\ov{p}$-admissible. As these strict inequalities occur when $\partial_i$ is a face operator acting on the first factor $\Delta^{j_{n-r}}$ of
$\Delta^{j_{n-r}}\ast\cdots\ast\Delta^{j_n}$, we get the $\ov{p}$-admissibility of $\partial'c_2$.
We have proved that $c_2$ is of $\ov{p}$-intersection; therefore $c_1=c-c_2$ is of $\ov{p}$-intersection also.

This property allows the construction of a section, $s$, to $\cR$,  as follows:\\
Let $\Phi\in\hom(C^{\GM,\ov{p}}_{*}(\ob{K}^{[r],k-1}),R)$, we define
$s(\Phi)\in \hom(C^{\GM,\ov{p}}_{*}(\ob{K}^{[r],k}),R)$ by
$s(\Phi)(c_1+c_2)=\Phi(c_1)$.

\medskip
$\bullet$ {\sc Relative isomorphism axiom.} Let
$$f\colon (\ob{K}^{[r],k},\ob{K}^{[r],k-1})\to
(\ob{L}^{[r],k},\ob{L}^{[r],k-1})$$
be a relative isomorphism preserving the perverse degree.
The relative isomorphism axiom is equivalent to the fact that $f$ induces an isomorphism between the quotients,
$$\ov{f}_*\colon\frac{C^{\GM,\ov{p}}_{*}(\ob{K}^{[r],k})}{C^{\GM,\ov{p}}_{*}(\ob{K}^{[r],k-1})}
\stackrel{\cong}{\longrightarrow}
\frac{C^{\GM,\ov{p}}_{*}(\ob{L}^{[r],k})}{C^{\GM,\ov{p}}_{*}(\ob{L}^{[r],k-1})}.
$$
Let $c\in C_{*}(\ob{K}^{[r],k})$. As in the previous item, we have a unique decomposition of $c$ as
$c=c_{1}+c_{2}$, with
$c_{1}\in C_{*}(\ob{K}^{[r],k-1})$
and
$c_{2}$ a linear combination of simplices of $\ob{K}^{[r],k}\backslash \ob{K}^{[r],k-1}$. 
Moreover,
if $c\in C^{\GM,\ov{p}}_{*}(\ob{K}^{[r],k})$ then $c_{1}\in C^{\GM,\ov{p}}_{*}(\ob{K}^{[r],k-1})$ and $c_{2}$ is 
a linear combination of simplices of $\ob{K}^{[r],k}\backslash \ob{K}^{[r],k-1}$ that is of $\ov{p}$-intersection. A similar decomposition exists for the elements of $C^{\GM,\ov{p}}_{*}(\ob{L}^{[r],k})$.

We denote by $[c]$ the class of $c\in C^{\GM,\ov{p}}_{*}(\ob{K}^{[r],k})$ in the quotient above. As $f$ is a bijection between
$\ob{K}^{[r],k}\backslash \ob{K}^{[r],k-1}$
and
$\ob{L}^{[r],k}\backslash \ob{L}^{[r],k-1}$,
we have $\ov{f}_*([c])=\ov{f}_*([c_{2}])$,
which is a linear combination of simplices in
$\ob{L}^{[r],k}\backslash \ob{L}^{[r],k-1}$, that is of $\ov{p}$-intersection, by the hypothesis on the conservation of the perverse degree. This implies that $\ov{f}_*$ is an isomorphism.

\medskip
$\bullet$ {\sc Wedge axiom.} The dual of a direct sum being a product, we have to prove that the canonical injections  induce an isomorphism,
$$\psi\colon 
\bigoplus_i\frac{C^{\GM,\ov{p}}_{*}\left(\ob{K}_i^{[r],k}\right)}{C^{\GM,\ov{p}}_{*}\left(\ob{K}_i^{[r],k-1}\right)}
\stackrel{\cong}{\longrightarrow}
\frac{C^{\GM,\ov{p}}_{*}\left(\cup_i\ob{K}_i^{[r],k}\right)}{C^{\GM,\ov{p}}_{*}\left(\cup_i\ob{K}_i^{[r],k-1}\right)}.$$
Let $([c_i])_i$ be an element of the left-hand term.
The map $\psi$ is defined by $\psi(([c_i])_i)=\sum_i[c_i]$, where $[-]$ denotes the equivalence classes of the quotients.

As we already did,  we decompose a chain $c\in C^{\GM,\ov{p}}_{*}\left(\cup_i\ob{K}_i^{[r],k}\right)$ in
$c=c'+c''$,
with
$c'\in C^{\GM,\ov{p}}_{*}\left(\cup_i\ob{K}_i^{[r],k-1}\right)
$
and
$c''\in C^{\GM,\ov{p}}_{*}\left(\cup_i\ob{K}_i^{[r],k}\backslash \cup_i\ob{K}_i^{[r],k-1}\right)$. As the intersection of the sets, $\ob{K}^{[r],k}_i\backslash \ob{K}^{[r],k-1}_i$, is empty, the element $c''$ can be written $c''=\sum_ic''_i$, in a unique way, with 
$c''_i\in C^{\GM,\ov{p}}_{*}\left(\ob{K}_i^{[r],k}\backslash \ob{K}_i^{[r],k-1}\right)$.
An inverse to $\psi$ is defined by
$\psi^{-1}([c])=([c''_i])_i$.

\medskip
$\bullet$ {\sc Filtered dimension axiom.}
Let $c\in C^{\GM,\ov{p}}_{\ast}(\ob{K}^{[r]})$ of homological degree $|c|$ such that $|c|< k$. Denote by $[c]$ the class of $c$ in the quotient
$$
\frac{C^{\GM,\ov{p}}_{*}(\ob{K}^{[r]})}{C^{\GM,\ov{p}}_{*}(\ob{K}^{[r],k})}.
$$
For degree reason, we have $[c]=0$ and the filtered dimension axiom is proved.
\end{proof}

The next result generalizes \propref{prop:propertiesintersectionhomology} (v), which corresponds to the case $k=0$ and
$\ob{K}=\ob{\rm ISing}_*^{\cF}(X)$.

\begin{proposition}\label{prop:coneenGM} 
 Let $R$ be a commutative ring and $\ell$ be an integer such that $1\leq \ell\leq n$.
Let $\ob{K}$ be a filtered face set of depth $v(\ob{K})\leq \ell-1$ and $\Delta^k\ast\ob{K}$ be a join of depth $\ell$ with $k\geq 0$.
\index{Intersection homology!of a cone}
For any  perversity $\ov{p}$ and any $k\in\N$, we have
$$H_{i}^{\GM,\ov{p}}(\Delta^k\ast\ob{K};R)=\left\{
\begin{array}{cl}
H_{i}^{\GM,\ov{p}}(\ob{K};R)
&\text{ if } i\leq \ell -2 -\ov{p}(\ell),\\
R&\text{ if } i>\ell-2-\ov{p}(\ell) \text{ and } i=0,\\
0&\text{ if } i>\ell -2 -\ov{p}(\ell) \text{ and } i\neq 0.
\end{array}\right.
$$
\end{proposition}

\begin{proof}
As the elements of degree zero are playing an important role, we suppose that each chain complex $C_{*}(-)$ is augmented. We denote   by $\varepsilon$ the various augmentations and by  $\ttC_{*}(-)=C_{*}(-)\oplus R$  the associated augmented complexes, with the component $R$ in degree -1.
 We define a linear map,
 $$\psi\colon \ttC_*(\Delta^k)\otimes \ttC_*(\ob{K})\to \ttC_{*+1}(\Delta^k\ast\ob{K}),$$
 by  $\psi(\gamma\otimes \sigma)=\gamma\ast\sigma$, for any $\gamma\in \ttC_{*}(\Delta^k)$ and any $\sigma\in \ttC_{*}(\ob{K})$, with the convention $\gamma\ast 1=\gamma$, $1\ast\sigma=\sigma$ and $1\ast 1=1$. We check easily that this map, of degree +1, is  compatible with the differentials. 
 Moreover, $\psi$ is injective on the basis of simplices, thus injective. The surjectivity comes from the definition of $\Delta^k\ast\ob{K}$ in \exemref{exam:simplejoin}.

Recall, from  \exemref{exam:simplejoin}, the determination of the perverse degrees of the simplices of $\Delta^k\ast\ob{K}$.
Let $\ell'\geq 1$, $\sigma\colon \Delta^N=\Delta^{j_{n-\ell+1}}\ast\cdots\ast\Delta^{j_n}\to\ob{K}$
  and 
  $\gamma\colon \Delta^{k'}\to \Delta^k$. 
  The perverse degree of $\sigma$ is the same, as simplex of $\ob{K}$ or as simplex of $\Delta^k\ast\ob{K}$. 
For $\gamma\in \Delta^k\subset \Delta^k\ast\ob{K}$, we have
$$\|\gamma\|_{\ell'}=\left\{
\begin{array}{cl}
-\infty&
\text{ if } \ell<\ell',\\
k'&
\text{ if } \ell'\leq \ell.
\end{array}\right.$$
The perverse degree of $\gamma\ast\sigma\in \Delta^k\ast\ob{K}$ is determined by 
 $$\|\gamma\ast\sigma\|_{\ell'}=\left\{\begin{array}{cl}
 -\infty&\text{ if } \ell<\ell',\\
 k'&\text{ if } \ell'=\ell \text{ or } (\ell>\ell' \text{ and } \|\sigma\|_{\ell'}=-\infty),\\
 k'+\|\sigma\|_{\ell'}+1&\text{ if }\ell>\ell' \text{ and } \|\sigma\|_{\ell'}\neq-\infty.
 \end{array}\right.$$
 From these determinations, we deduce the $\ov{p}$-admissibility conditions of the various types of simplices of $\Delta^k\ast\ob{K}$.
 
 $\bullet$ The simplex $\sigma$ is $\ov{p}$-admissible as simplex of $\ob{K}$ if, and only if, it is admissible as simplex of $\Delta^k\ast\ob{K}$.
 
 $\bullet$ The $\ov{p}$-admissibility condition of $\gamma\colon \Delta^{k'}\to \Delta^k$ as simplex of $\Delta^k\ast\ob{K}$ is, by definition,
\begin{equation}\label{equa:admissiblegamma}
\|\gamma\|_{\ell'}\leq k'-\ell'+\ov{p}(\ell'), \text{ for any } \ell'\geq 1.
\end{equation}
 \begin{itemize}
 \item[---] If $\ell<\ell'$, this condition is always satisfied.
 \item[---] If $\ell'\leq \ell$, this condition is satisfied if, and only if, $0\leq -\ell'+\ov{p}(\ell')$. As $\ov{p}$ is a perversity, all these conditions are equivalent to $\ell\leq \ov{p}(\ell)$.
 \end{itemize}

  $\bullet$
The $\ov{p}$-admissibility condition of
 $\gamma\ast\sigma\in \Delta^k\ast\ob{K}$
 is, by definition,
\begin{equation}\label{equa:padmissible}
\|\gamma\ast\sigma\|_{\ell'}\leq k'+N+1-\ell'+\ov{p}(\ell'), \text{ for any } \ell'\geq 1.
\end{equation}
  \begin{itemize}
 \item[---] If $\ell<\ell'$, this condition is always satisfied.
  \item[---] If $\ell'<\ell$ and $\|\sigma\|_{\ell'}\neq -\infty$, this condition is satisfied if, and only if, $\|\sigma\|_{\ell'}\leq N-\ell'+\ov{p}(\ell')$. This last inequality is exactly the $\ov{p}$-admissibility of $\sigma$ at $\ell'$.
 \item[---] If $\ell'=\ell$ or ($\ell>\ell'$ and $\|\sigma\|_{\ell'}=-\infty$), the condition (\ref{equa:padmissible}) is equivalent  to $N\geq \ell' - 1 -\ov{p}(\ell')$. As $\ov{p}$ is a perversity,  all these inequalities are equivalent to 
 $N\geq \ell-1-\ov{p}(\ell)$.
 \end{itemize}
 
 We determine now $H_{*}^{\GM,\ov{p}}(\Delta^k\ast\ob{K};R)$ by considering two cases. 
 
 \medskip
 \emph{First case: $\ov{p}(\ell)\geq \ell$.}
The previous determination of perverse degrees implies the bijectivity of
$$\psi\colon \ttC_{*}(\Delta^k)\otimes \ttC_{*}^{\GM,\ov{p}}(\ob{K})\to \ttC_{*+1}^{\GM,\ov{p}}(\Delta^k\ast\ob{K}),$$
from which we deduce $\ttH^{\GM,\ov{p}}_{*}(\Delta^k\ast\ob{K};R)=0$.

\medskip
 \emph{Second case: $\ov{p}(\ell)\leq \ell -1$.}
By comparing to the previous case, we observe that the chains of $\ov{p}$-intersection are sums of chains of the shape
$\nu_{1}\ast \nu_{2}$ with $\nu_{2}\in \ttC_{*}^{\GM,\ov{p}}(\ob{K})$, such that $|\nu_{2}|>\ell-1-\ov{p}(\ell)$ or, with $|\nu_{2}| =\ell-1-\ov{p}(\ell)$ and $\partial \nu_{2}=0$.
By setting,
$$ \tau^{\ov{p}}_{\ell}\ttC_*(\ob{K})=
 \left\{\begin{array}{cl}
 0&\text{ if } *<\ell-1-\ov{p}(\ell),\\
 \cZ \ttC_{*}^{\GM,\ov{p}}(\ob{K})&\text{ if } *=\ell-1-\ov{p}(\ell),\\
 \ttC_{*}^{\GM,\ov{p}}(\ob{K})&\text{ if } *>\ell-1-\ov{p}(\ell),
 \end{array}\right.$$
 the previous calculations of perverse degree give an isomorphism
\begin{equation}\label{equa:shortcone}
\psi\colon  D_{*}=\left(C_*(\Delta^k)\otimes  \tau^{\ov{p}}_{\ell}\ttC_*(\ob{K})\right)\oplus \left(R\otimes \ttC_{*}^{\GM,\ov{p}}(\ob{K})\right)
 \to \ttC_{*+1}^{\GM,\ov{p}}(\Delta^k\ast\ob{K}).
 \end{equation}
   For computing the homology of $D_{*}$, we decompose it in the next short exact sequence,
$$\xymatrix{
 0\ar[r]&
 R\otimes \ttC_{*}^{\GM,\ov{p}}(\ob{K})\ar[r]&
 D_{*}\ar[r]&
 C_*(\Delta^k)\otimes  \tau^{\ov{p}}_{\ell}\ttC_*(\ob{K})\ar[r]&
 0.
 }$$
 In the first term, on the left, the factor $R$ is concentrated in degree -1 and the exact sequence can be rewritten as
 $$\xymatrix{
 0\ar[r]&
\ttC_{*+1}^{\GM,\ov{p}}(\ob{K})\ar[r]&
 D_{*}\ar[r]&
 C_*(\Delta^k)\otimes  \tau^{\ov{p}}_{\ell}\ttC_*(\ob{K})\ar[r]&
 0.}
 $$
 The associated long exact sequence with its connecting map, $\delta$, appears as
 $$\ldots\to H_{i+1}(\ttC^{\GM,\ov{p}}(\ob{K}))
 \to H_{i}(D)\to
 H_{i}(\tau_{\ell}^{\ov{p}}\ttC(\ob{K}))\xrightarrow[]{\delta}
 H_{i}(\ttC^{\GM,\ov{p}}(\ob{K}))\to
 H_{i-1}(D)\to\ldots
 $$
 The connecting map, $\delta$, is induced by the canonical inclusion
 $\tau_{\ell}^{\ov{p}}\ttC(\ob{K}))
 \hookrightarrow
 \ttC^{\GM,\ov{p}}(\ob{K})$. By definition of
 $\tau_{\ell}^{\ov{p}}\ttC(\ob{K})$
 this inclusion induces an injection in homology. Thus, with the isomorphism (\ref{equa:shortcone}), we have
 \begin{eqnarray*}
\ttH_{i}^{\GM,\ov{p}}(\Delta^k\ast\ob{K};R)
&=&
\coker\left(H_{i}\tau_{\ell}^{\ov{p}}\ttC_{*}(\ob{K})\to
\ttH_{i}^{\GM,\ov{p}}(\ob{K};R)\right)\\
&=&
\left\{\begin{array}{cl}
\ttH_{i}^{\GM,\ov{p}}(\ob{K};R)&\text{ if }  i\leq \ell -2 -\ov{p}(\ell),\\
0&\text{ if } i>\ell-2-\ov{p}(\ell).\end{array}\right.
\end{eqnarray*}
This argument determines the reduced cohomology of $\Delta^k\ast\ob{K}$. The non-reduced one has an extra term $R$ in degree~$i=0$. In the case $i\leq \ell-2-\ov{p}(\ell)$, this component appears in
$H_{i}^{\GM,\ov{p}}(\ob{K};R)$ but for $0\geq \ell-1-\ov{p}(\ell)$, we have to add
$H_{0}^{\GM,\ov{p}}(\Delta^k\ast\ob{K};R)=R$, as stated.
\end{proof}

We may note that the isomorphism $\ttH_{i}^{\GM,\ov{p}}(\Delta^k\ast\ob{K};R)\cong \ttH_{i}^{\GM,\ov{p}}(\ob{K};R)$, obtained for low degrees in \propref{prop:coneenGM}, is induced by the inclusion $\ob{K}\hookrightarrow \Delta^k\ast\ob{K}$.

\smallskip
In the particular case of the cone on a smooth manifold, the next result already appears in \cite{TheseSalem}.

\begin{corollary}\label{cor:coneenGM}
Let $R$ be a principal ideal domain.
\index{GM!cohomology!of a cone}
Let $\ob{K}$ be a filtered face set of depth $v(\ob{K})\leq \ell-1$,  with $\ell\in\{1,\ldots,n\}$, and $\Delta^k\ast\ob{K}$ be a join of depth $\ell$ with $k\geq 0$.
For any  perversity $\ov{p}$ and any $k\in\N$, we have
$$H^{i}_{\GM,\ov{p}}(\Delta^k\ast\ob{K};R)=\left\{
\begin{array}{cl}
H^{i}_{\GM,\ov{p}}(\ob{K};R)&\text{ if } i\leq  \ell -2 -\ov{p}(\ell),\\
R&\text{ if } i\geq\ell-1-\ov{p}(\ell) \text{ and } i=0,\\
\ext(H_{i-1}^{\GM,\ov{p}}(\ob{K};R),R)&\text{ if } i= \ell -1 -\ov{p}(\ell) \text{ and } i\neq 0,\\
0&\text{ if } i\geq \ell  -\ov{p}(\ell) \text{ and } i\neq 0.
\end{array}\right.
$$
\end{corollary}

\begin{proof}
As $R$ is a principal ideal domain, the complex $C_{*}^{\GM,\ov{p}}(\ob{K};R)$ is free as $R$-module and we may use  the universal coefficients formula. The announced results for
$ i= \ell -1 -\ov{p}(\ell)$
and
$i\geq \ell  -\ov{p}(\ell)$
are direct consequences of this formula and \propref{prop:coneenGM}. For the last case, $i\leq  \ell -2 -\ov{p}(\ell)$, we consider the morphism of {short} exact sequences induced by the \ffs~map $\ob{K}\to\Delta^k\ast\ob{K}$:
$${
\small
 \xymatrix@=14pt{
0\ar[r]&
\ext(H_{i-1}^{\GM,\ov{p}}(\Delta^k\ast\ob{K}),R)\ar[r]\ar[d]&
H^i_{\GM,\ov{p}}(\Delta^k\ast\ob{K})\ar[r]\ar[d]&
\hom(H_{i}^{\GM,\ov{p}}(\Delta^k\ast\ob{K}),R)\ar[d]
\ar[r]&0
\\
0\ar[r]&
\ext(H_{i-1}^{\GM,\ov{p}}(\ob{K}),R)\ar[r]&
H^i_{\GM,\ov{p}}(\ob{K})\ar[r]&
\hom(H_{i}^{\GM,\ov{p}}(\ob{K}),R)
\ar[r]&0
}}$$
As the left-hand side and right-hand side arrows are isomorphisms, the middle one is an isomorphism also.
\end{proof}

\begin{theorem}\label{thm:thetwocochains}
Let $R$ be a \emph{field.}
Let $\ob{K}$ be a filtered face set, $\ov{p}$ and $\ov{q}$  be two  perversities such that
$\ov{q}\geq 0$ and $\ov{p}(i)+\ov{q}(i)=i-2$ for any $i\in \{1,\ldots,n\}$. 
 Then, the {GM-cochain complex,} $C^{*}_{\GM,\ov{p}}(\ob{K})$, and the Thom-Whitney  complex,
$\tC^*_{\ov{q}}(\ob{K})$, are related by a quasi-isomorphism, i.e.,
$H^*_{\GM,\ov{p}}(\ob{K};R)\cong H^*_{\TW,\ov{q}}(\ob{K};R)$. 
\index{GM!cohomology}\index{TW!cohomology}
\end{theorem}

\begin{proof}
First, we define $\chi\colon \tC^*(\ob{K})\to\hom(C^{\GM}_{*}(\ob{K}),R)$. 
If $\sigma\colon \Delta=\Delta^{j_0}\ast\cdots\ast \Delta^{j_n}\to \ob{K}_{+}$,
recall that 
$\tC^*(\ob{K})_{\sigma}=C^*(c\Delta^{j_0})\otimes\cdots\otimes C^*(\Delta^{j_n})$.
If $\Phi\in \tC^*(\ob{K})$, 
 we may write $\Phi_{\sigma}$ as $\Phi_{\sigma}=\sum_{i}\Phi_{0,\sigma,i}\otimes\cdots\otimes \Phi_{n,\sigma,i}\in C^*(c\Delta^{j_0})\otimes\cdots\otimes C^*(\Delta^{j_n})$.
The cochains $\Phi_{k,\sigma,i}$ can be evaluated on the maximal simplex in each factor and we set
$$\chi(\Phi)(\sigma)=\sum_{i}\Phi_{0,\sigma,i}([c\Delta^{j_0}])\cdot\ldots\cdot\Phi_{n,\sigma,i}([\Delta^{j_n}]).$$ 
Denote by $\Phi_{i}=\Phi_{0,\sigma,i}\otimes\cdots\otimes \Phi_{n,\sigma,i}$.
The compositions of $\chi$ with the differentials verify
\begin{eqnarray*}
\chi(d\Phi_{i})(\sigma)&=&
\sum_{k=0}^n\pm \Phi_{0,\sigma,i}([c\Delta^{j_0}])\cdot\ldots\cdot
d\Phi_{k,\sigma,i}([c\Delta^{j_k}])
\cdot\ldots\cdot\Phi_{n,\sigma}^i([\Delta^{j_n}]),\\
d\chi(\Phi_{i})(\sigma)
&=&
\chi(\Phi_{i})(\partial \sigma)\\
&=&
\sum_{k=0}^n\pm \Phi_{0,\sigma,i}([c\Delta^{j_0}])\cdot\ldots\cdot
\Phi_{k,\sigma,i}(c\partial[\Delta^{j_k}])
\cdot\ldots\cdot\Phi_{n,\sigma,i}([\Delta^{j_n}]).
\end{eqnarray*}
From the  definition of the differential of a cochain, we have  
$d\Phi_{k,\sigma,i}([c\Delta^{j_k}])=\Phi_{k,\sigma,i}(\partial[ c\Delta^{j_k}])$. 
The sign convention of the boundary operator of a cone is given by
$\partial [c\Delta^{j_i}]=c\partial[\Delta^{j_i}]+ (-1)^{j_{i}+1}[\Delta^{j_i}\times\{1\}]$. 
Therefore, to get the equality 
$d\chi(\Phi)=\chi(d\Phi)$, we have to prove the nullity of  the products
$$\Pi_k^i=\Phi_{0,\sigma,i}([c\Delta^{j_0}])\cdot\ldots\cdot
\Phi_{k,\sigma,i}([\Delta^{j_k}\times\{1\}])
\cdot\ldots\cdot\Phi_{n,\sigma,i}([\Delta^{j_n}]),$$
for any $k\in\{0,\ldots,n-1\}$. 
Suppose $\Delta^{j_{k}}\neq\emptyset$ and the restriction of $\Phi_{k,\sigma,i}$  to $\Delta^{j_k}\times\{1\}$   not equal to~0. 
For having $\Pi_k^i\neq 0$,  the perverse degrees of $\Phi_{i}$ must verify
$$\|\Phi_{i}\|_{n-k}=\dim(c\Delta^{j_{k+1}}\times\cdots\times c\Delta^{j_{n-1}}\times \Delta^{j_n}),$$ which is equivalent to
\begin{equation}\label{equa:stokes}
\|\Phi_{i}\|_{n-k}+\|\sigma\|_{n-k}=\dim \Delta-1.
\end{equation}
If $\Phi_{i}\in \tC^*_{\ov{q}}(\ob{K})$ and $\sigma\in C_{*,\ov{p}}^{\GM}(\ob{K})$, their perverse degrees verify
$$\|\Phi_{i}\|_{n-k}\leq \ov{q}(n-k)\text{ and } \|\sigma\|_{n-k}\leq \dim\Delta-(n-k)+\ov{p}(n-k)$$
which imply
\begin{eqnarray*}
\|\Phi_{i}\|_{n-k}+\|\sigma\|_{n-k}&\leq& \dim \Delta-(n-k)+\ov{p}(n-k)+\ov{q}(n-k)\\
&\leq& \dim\Delta-2.
\end{eqnarray*}
Thus the condition (\ref{equa:stokes}) cannot be satisfied and the products $\Pi_k^i$ are equal to~0 for any $k\in\{0,\ldots,n-1\}$. We have proved the compatibility of the map
$$\chi\colon \tC^*_{\ov{q}}(\ob{K})\to\hom(C^{\GM,\ov{p}}_{*}(\ob{K}),R)$$ 
with the differentials. 
The conclusion is a recollection of previous results.
\begin{itemize}
\item From \propref{prop:systemlocalperversetheory} and the fact that $C^*$ is an extendable universal system of coefficients (\cite[Chapter 14]{MR736299}), $\tC^*_{\ov{q}}$ is a filtered theory of cochains.
\item \propref{prop:GMcochainsperversetheory} says that $C^*_{\GM,\ov{p}}$ is a filtered theory of cochains too.
\item As we are working over a field and with a perversity, $\ov{p}$, the conclusion of \corref{cor:coneenGM} simplifies in
$$H^{i}_{\GM,\ov{p}}(\Delta^k\ast\ob{K};R)=\left\{
\begin{array}{cl}
H^{i}_{\GM,\ov{p}}(\ob{K};R)&\text{ if } i\leq  \ell -2 -\ov{p}(\ell),\\
0&\text{ if } i\geq \ell -1 -\ov{p}(\ell).
\end{array}\right.
$$
(Observe that, with the hypotheses on $\ov{p}$ and $\ov{q}$, we cannot have $0=\ell-1-\ov{p}(\ell)$.)
This computation coincides with \corref{cor:blowupcohomologycone}.
\item As noted before, the equality in low degrees between the cohomology of $\ob{K}$ and the cohomology of $\Delta^k\ast\ob{K}$ is induced by the inclusion $\ob{K}\hookrightarrow \Delta^k\ast\ob{K}$. Therefore, a quasi-isomorphism,
$\tC^*(\ob{K})\to\hom(C^{\GM}_{*}(\ob{K}),R)$, induced by a natural map, gives also a quasi-isomorphism
$\tC^*(\Delta^k\ast \ob{K})\to\hom(C^{\GM}_{*}(\Delta^k\ast \ob{K}),R)$.
\item Finally, \propref{prop:qipervestheory} implies that the map $\chi$ is a quasi-isomorphism.
\end{itemize}
\end{proof}

\begin{remark}\label{rem:perversitytop}
The hypothesis \emph{$R$ is a field} is used for having an isomorphism between 
$H^*_{\GM,\ov{p}}(\Delta^k\ast\ob{K};R)$
and
$H^*_{\TW,\ov{q}}(\Delta^k\ast\ob{K};R)$,
by killing the {\rm Ext}-term. We may observe also that this {\rm Ext}-term may also be avoided in a particular case on a principal ideal domain.

More precisely, let $\ov{t}'$ be the  perversity defined by $\ov{t}'(i)=i-2$, for any $i\in\{1,\ldots,n\}$. This perversity coincides with the top perversity $\ov{t}$ (\defref{def:perversité}) for any $i\in\{2,\ldots,n\}$ and the difference between $\ov{t}$ and $\ov{t}'$ does not matter if we work with filtered spaces without strata of codimension~1, as it is the case for pseudomanifolds. 
Directly from the previous proof, for any \ffs, $\ob{K}$, and any \emph{principal ideal domain,} $R$, we have an isomorphism
$$H^*_{\TW,\ov{0}}(\ob{K};R)\cong H^*_{\GM,\ov{t}'}(\ob{K};R),$$
since, in \corref{cor:coneenGM}, the  Ext-term appears in  homological degree~0  and we have\linebreak
 $\ext(H_{0}^{\GM,\ov{t}'}(\ob{K};R),R)=0$.
\end{remark}

The first part of the next corollary is a direct consequence of \thmref{thm:thetwocochains} \and \corref{cor:Aplandcochains}.
The second part follows from  \propref{prop:intersectionetintersection} and 
 \cite{MR800845}.

\begin{corollary}\label{cor:quasiisossurQ}
Let $R=\Q$. 
Let  $\ob{K}$ be a filtered face  set, $\ov{p}$ and $\ov{q}$ be  perversities 
such that
$\ov{q}\geq 0$ and $\ov{p}(i)+\ov{q}(i)=i-2$ for any $i\in \{1,\ldots,n\}$.
Then the complexes 
$\widetilde{A}_{PL,\ov{q}}(\ob{K})$, $\tC^*_{\ov{q}}(\ob{K})$ and $C^{*}_{\GM,\ov{p}} (\ob{K})$  are related by quasi-isomorphisms.

In the particular case that $\ov{p}$ and $\ov{q}$ are GM-perversities and $\ob{K}$ is the filtered face set associated to a pseudomanifold, $X$, the homology of these complexes coincides with the original Goresky-MacPherson intersection cohomology of~$X$.
\end{corollary}

\begin{remark}\label{rem:petitemaisbien}
{The simplices in $\ob{K}\backslash \ob{K}_+$ (i.e., $j_{n}=-1$) are not considered in $\tC^*_{\ov{q}}(\ob{K})$, by definition, and they do not appear  in 
$C^*_{\GM,\ov{p}}(\ob{K})$}
 if $\ov{q}(1)\geq 0$, or equivalently if $\ov{p}(1)\leq -1$. Indeed, any $\ov{p}$-admissible simplex 
 $\sigma\colon\Delta=\Delta^{j_0}\ast\cdots\ast\Delta^{j_n}\to \ob{K}$
 verifies
 $$
  \|\sigma\|_1=\dim(\Delta^{j_0}\ast\cdots\ast\Delta^{j_{n-1}})\leq \dim \Delta - 1+\ov{p}(1)\leq \dim\Delta -2.
 $$
If $\Delta^{j_0}\ast\cdots\ast \Delta^{j_{n-1}}\neq \emptyset$, we get $2\leq \dim \Delta-\dim (\Delta^{j_0}\ast\cdots\ast \Delta^{j_{n-1}})=j_n+1$ and $j_n\geq 1$. This implies also that all the $\ov{p}$-admissible 0-simplices and 1-simplices belong to the regular part of $\ob{K}$.
\end{remark}

\section{Particular cases: $\ov{q}=\ov{0}$, $\ov{q}=\ov{\infty}$, cone and suspension}\label{sec:normalisation}
\begin{quote}
We  characterize the Thom-Whitney intersection cohomology in the cases $\ov{q}=\ov{\infty}$ and $\ov{q}=\ov{0}$.  For $\ov{q}=\ov{0}$, we need to introduce the normalization of a \ffs, as Goresky and MacPherson do (see \cite{MR572580}) in the framework of pseudomanifolds. Here, the connectivity of the link is not sufficient, we need to use the expanded link.
\end{quote}

Recall  from \defref{def:filteredskeleton} that the face set $\ob{K}^{[0]}$ is  the regular part of the \ffs, $\ob{K}$.

\begin{proposition}\label{prop:casoinfinito}
Let $R$ be a commutative ring. For any \ffs, $\ob{K}$, the restriction map, 
$\tC_{\ov{\infty}}^*(\ob{K})\to C^*(\ob{K}^{[0]})
$,
induces an isomorphism,
$$H^*_{\TW,\ov{\infty}}(\ob{K};R)\cong
H^*(\ob{K}^{[0]};R),$$
where the last term is the ordinary cohomology of the face set  $\ob{K}^{[0]}$.
\index{TW!$\ov{\infty}$-cohomology}
\end{proposition}

\begin{proof}
 If $\omega\in\tC_{\ov{\infty}}^*(\ob{K})$, we have an element
  $\omega_{\sigma}\in C^*(c\Delta^{j_0})\otimes\cdots\otimes C^*(c\Delta^{j_{n-1}})\otimes C^*(\Delta^{j_n})$, 
  for  each $\sigma\colon \Delta^{j_0}\ast\cdots\ast\Delta^{j_n}\to \ob{K}_+$. 
  The simplices $\sigma$ of $\ob{K}^{[0]}$ correspond to the particular case  $j_i=-1$ for all $i\leq n-1$. 
  A global section, still denoted  $\omega\in C^*(\ob{K}^{[0]})$, is thus defined by restriction from $\ob{K}_+$ to $\ob{K}^{[0]}$.
  
We check easily that $\ob{K}\mapsto \tC_{\ov{\infty}}^*(\ob{K})$ and $\ob{K}\mapsto C^*(\ob{K}^{[0]})$ are  filtered theory of cochains, see \propref{prop:systemlocalperversetheory}. Let  $\ob{K}$ be a \ffs~of depth  $v(\ob{K})=\ell-1$, $\ell\in\{1,\ldots,n\}$, and such that $\tC_{\ov{\infty}}^*(\ob{K})\to C^*(\ob{K}^{[0]})$ is a quasi-isomorphism. 
Let $\Delta^k\ast\ob{K}$ be the join of depth~$\ell$.
Using \propref{prop:qipervestheory}, we are reduced to prove that
  $\tC_{\ov{\infty}}^*(\Delta^k\ast\ob{K})\to C^*((\Delta^k\ast\ob{K})^{[0]})$ is a quasi-isomorphism. On the right-hand side, we have
  $(\Delta^k\ast\ob{K})^{[0]}=\ob{K}^{[0]}$. On the left-hand side, with \corref{cor:blowupcohomologycone}, we know that
  $$H^*_{\TW,\ov{\infty}}(\Delta^k\ast\ob{K};R)\cong H^*_{\TW,\ov{\infty}}(\ob{K};R).$$
This isomorphism being induced by the canonical inclusion $\ob{K}\hookrightarrow \Delta^k\ast\ob{K}$, the hypotheses of \propref{prop:qipervestheory} are fulfilled and the statement is proved.
\end{proof}

For the study of the $\ov{0}$-cohomology, we introduce a concept similar to the normalization of Goresky and MacPherson in \cite[Page 151]{MR572580}. We denote by $\sigma \vartriangleleft \phi$ the relation ``$\sigma$ is a face of $\phi$''.

\begin{definition}\label{def:ffsregulier}
A \ffs, $\ob{K}$, is called \emph{normal} if  the two following conditions are satisfied.
\index{Filtered!face set!normal}
\begin{enumerate}[(a)]
\item Any simplex $\sigma\in\ob{K}\backslash \ob{K}_+$ is a face of a simplex $\phi$ of $\ob{K}_+$, i.e., $\sigma \vartriangleleft \phi$.
\item Moreover, the simplex $\phi\in\ob{K}_+$ is unique in the following sense.\\ Let $\sigma\in\ob{K}\backslash \ob{K}_+$. For any pair $(\phi,\phi')$ of simplices of $\ob{K}_+$ such that
$\sigma \vartriangleleft \phi$
and 
$\sigma \vartriangleleft \phi'$, there exists a family
$\phi_1,\ldots,\phi_m$ 
of simplices of $\ob{K}_+$ such that 
$\sigma \vartriangleleft \phi_i$ for $i\in\{1,\,\ldots,\,m\}$ and
$\phi \vartriangleleft\phi_1 \vartriangleright \cdots \vartriangleleft \phi_m \vartriangleright \phi'$.
\end{enumerate}
\index{Link!of a simplex!expanded}
\end{definition}

Recall the expanded link, $\cLe(\ob{K},\tau)$, introduced in \defref{def:expandedlink}.

\begin{proposition}\label{prop:propertynormalffs}
Let $\ob{K}$ be a normal \ffs. For any $r>0$, $k\geq 0$ and  $\tau\in \cJ(\ob{K},{[r],k})$, the expanded link,
$\cLe(\ob{K},\tau)$, is a connected, non empty and normal \ffs.
\end{proposition} 

\begin{proof}
As $r>0$, we have $\tau\in \ob{K}\backslash \ob{K}_+$ and there exists $\sigma\in±\ob{K}_+$ such that $\tau \vartriangleleft \sigma$. We may suppose $R_1(\sigma)=\tau$. As $\sigma\neq\tau$, we have $(R_2(\sigma),\sigma)\in \cLe(\ob{K},\tau)$ and $\cLe(\ob{K},\tau)\neq\emptyset$.

We prove now the connectivity of $\cLe(\ob{K},\tau)$ by connecting any element $(R_2(\sigma'),\sigma')\in \cLe(\ob{K},\tau)$ to $(R_2(\sigma),\sigma)$. We consider two cases.
\begin{itemize}
\item If $\sigma'\in \ob{K}_+$, as $R_1(\sigma')=R_1(\sigma)=\tau$, we may apply the uniqueness axiom of normal \ffs. There exists a family 
$\phi_1,\ldots,\phi_m$ 
of simplices of $\ob{K}_+$ such that 
$\tau \vartriangleleft \phi_i$ for $i\in\{1,\,\ldots,\,m\}$ and
$\sigma \vartriangleleft\phi_1 \vartriangleright \cdots \vartriangleleft \phi_m \vartriangleright \sigma'$.
We may suppose $R_1(\phi_i)=\tau$  and, as $\phi_i\neq\tau$, we have $(R_2(\phi_i),\phi_i)\in \cLe(\ob{K},\tau)_{+}$
for $i\in\{1,\ldots,\,m\}$.  This implies the next relations in $\cLe(\ob{K},\tau)$,
$$(R_2(\sigma),\sigma) \vartriangleleft
(R_2(\phi_1),\phi_1) \vartriangleright
\cdots
\vartriangleleft
(R_2(\phi_m),\phi_m)
\vartriangleright
(R_2(\sigma'),\sigma').$$
\item If $\sigma'\notin \ob{K}_+$, there exists $\phi\in\ob{K}_+$ such that $\sigma' \vartriangleleft \phi$. As $R_1(\sigma')=\tau$, we may suppose $R_1(\phi)=\tau$ and we get
$(R_2(\sigma'),\sigma') \vartriangleleft (R_2(\phi),\phi)$. By applying  the first case to 
$(R_2(\phi),\phi)\in \cLe(\ob{K},\tau)_{+}$, 
we can relate $(R_2(\sigma'),\sigma')$ to $(R_2(\sigma),\sigma)$.
\end{itemize}
We are reduced to the proof of the normality of $\cLe(\ob{K},\tau)$. 

Let $(R_2(\sigma),\sigma)\in \cLe(\ob{K},\tau)\backslash \cLe(\ob{K},\tau)_+$. Then $\sigma\in\ob{K}\backslash \ob{K}_+$ and there exists $\phi\in\ob{K}_+$ with $\sigma \vartriangleleft \phi$. We may suppose $R_1(\phi)=\tau$ and we get
$(R_2(\sigma),\sigma) \vartriangleleft (R_2(\phi),\phi)$ with $(R_2(\phi),\phi)\in \cLe(\ob{K},\tau)_{+}$. This gives  the first property of a normal \ffs.  For the second one,  consider
$(R_2(\sigma),\sigma)\in \cLe(\ob{K},\tau)\backslash \cLe(\ob{K},\tau)_+$,
$(R_2(\phi),\phi)\in   \cLe(\ob{K},\tau)_+$ and
$(R_2(\phi'),\phi')\in   \cLe(\ob{K},\tau)_+$ such that
$$
(R_2(\phi),\phi)\vartriangleright
(R_2(\sigma),\sigma)\vartriangleleft
(R_2(\phi'),\phi').$$
We deduce
$\phi$ and $\phi'$ in $\ob{K}_+$, $\sigma\in \ob{K}\backslash \ob{K}_+$, $R_1(\phi)=R_1(\phi')=R_1(\sigma)=\tau$
and
$\phi
\vartriangleright
\sigma
\vartriangleleft
\phi'$. From the uniqueness condition in a normal \ffs, we get the existence of a family
$\phi_1,\ldots,\phi_m$ 
of simplices of $\ob{K}_+$ such that 
$\sigma \vartriangleleft \phi_i$ for $i\in\{1,\,\ldots,\,m\}$ and
$\phi \vartriangleleft\phi_1 \vartriangleright \cdots \vartriangleleft \phi_m \vartriangleright \phi'$.
We may suppose $R_1(\phi_i)=\tau$ which gives
$(R_2(\sigma),\sigma) \vartriangleleft (R_2(\phi_i),\phi_i)$ and
$(R_2(\phi_i),\phi_i)\in \cLe(\ob{K},\tau)_{+}$
for all $i\in\{1,\ldots,\,m\}$,  and
$$(R_2(\phi),\phi) \vartriangleleft
(R_2(\phi_1),\phi_1) \vartriangleright
\cdots
\vartriangleleft
(R_2(\phi_m),\phi_m)
\vartriangleright
(R_2(\phi'),\phi').$$
\end{proof}

\begin{proposition}\label{prop:casonulo}
Let $R$ be a principal ideal domain and $\ov{t}'$ be the perversity defined by $\ov{t}'(i)=i-2$, $i\neq 0$. For any normal \ffs, $\ob{K}$, we have isomorphisms,
$$H^*_{\TW,\ov{0}}(\ob{K};R)\cong H^*_{\GM,\ov{t}'}(\ob{K};R)\cong H^*(\ob{K};R),$$
where the last term is the ordinary cohomology of the face set underlying $\ob{K}$.
\index{TW!$\ov{0}$-cohomology}\index{GM!top-cohomology}
\end{proposition}

\begin{proof}
The first isomorphism comes directly from \thmref{thm:thetwocochains} and \remref{rem:perversitytop}.

For the second one, we use the canonical inclusion, $C_{*}^{\GM,\ov{t}'}(\ob{K})\to C_{*}(\ob{K})$,
to define, by duality, a cochain map,
$C^*(\ob{K})\to C^*_{\GM,\ov{t}'}(\ob{K})$.
Using the classical theory of ordinary cochains on a face set, we know that the association
$\ob{K}\mapsto C^*(\ob{K})$
is a filtered theory of cochains. (The proof of \propref{prop:GMcochainsperversetheory}, in which all perverse degrees are taking out,
gives an explicit confirmation of this fact.)
Therefore, $C^*(\ob{K})\to C^*_{\GM,\ov{t}'}(\ob{K})$ is a cochain map between two filtered theories of cochains.
As $\ob{K}$ is normal, its expanded links are connected (cf. \propref{prop:propertynormalffs}) and an application of \propref{prop:qipervestheoryconnected} gives the second isomorphism.
\end{proof}

\begin{remark}\label{rem:3.3or3.4}
The cochain map,
$C^*(-)\to C^*_{\GM,\ov{t}'}(-)$,
is not cone-compatible in general. Indeed, if $\ob{L}$ is a \ffs, non connected as face set, such that
$C^*(\ob{L})\to C^*_{\GM,\ov{t}'}(\ob{L})$
is a quasi-isomorphism, we have
$H^0(\Delta^k\ast\ob{L};R)=R$, 
because $\Delta^k\ast\ob{L}$ is connected as face set,
and
$H^0_{\GM,\ov{t}'}(\Delta^k\ast\ob{L};R)\cong H^0_{\GM,\ov{t}'}(\ob{L};R)\cong H^0(\ob{L};R)\neq R$.
Thus, 
we must restrict to connected expanded links on this point. 
That justifies the use of \propref{prop:qipervestheoryconnected} in the proof of \propref{prop:casonulo}. (\exemref{exam:linkconnexenevautpas} comes back on this point.)
\end{remark}

We prove now that we can associate to any \ffs~a unique normal \ffs, with the same intersection cohomology.

\begin{definition}\label{def:normalisation}
The \emph{normalization of a \ffs,} $\ob{K}$, is a map of \ffss,
\index{Filtered!face set!normalization of}\index{Normalization of a filtered face set}
$$\cN\colon N(\ob{K})\to \ob{K},$$ such that
$N(\ob{K})$ is normal and the restriction map, $\cN\colon N(\ob{K})_+\to \ob{K}_+$, is an isomorphism.
\end{definition}

\begin{proposition}\label{prop:existencenormal}
Any \ffs, $\ob{K}$, admits a unique normalization
defined by
$$N(\ob{K})=\ob{K}_+\cup \{\langle \sigma,\phi\rangle\mid \sigma\in\ob{K}\backslash \ob{K}_+,\;\phi\in\ob{K}_+ \text{ and } \sigma \vartriangleleft \phi\},$$
where
\begin{itemize}
\item $\langle -,-\rangle$ denotes the equivalence class for the equivalence relation generated by
$$(\sigma,\phi)\sim (\sigma,\phi')\iff \phi \vartriangleleft \phi',$$
\item the perverse degree of elements of $\ob{K}_+$ is kept and 
$\|\langle \sigma,\phi\rangle\|_i=\|\sigma\|_i$.
\end{itemize}
Moreover, if $F$ is an extendable universal  system of differential coefficients, over a commutative ring $R$, of blow-up $\tF$, the map $\cN$ induces
\begin{itemize}
\item an isomorphism
$H_{\ov{q}}^*(N(\ob{K});\tF)\cong H_{\ov{q}}^*(\ob{K};\tF)$, for any loose perversity $\ov{q}$ or $\ov{q}=\ov{\infty}$,
\item
an isomorphism
$H_*^{\GM,\ov{p}}(N(\ob{K});R)\cong H_*^{\GM,\ov{p}}(\ob{K};R)$, for the Goresky-MacPherson homology if $\ov{p}$ is a loose perversity such that $\ov{p}(1)<  0$.
\end{itemize}
\end{proposition} 

\begin{proof}
Let $\sigma\colon \Delta^{j_{0}}\ast\cdots\ast\Delta^{j_{k}}\to \ob{K}$ be a simplex with $k<n$ and
$\phi\in\ob{K}_{+}$ with $\sigma \vartriangleleft \phi$. The simplex $\langle\sigma,\phi\rangle$ of $N(\ob{K})$ is considered also as
$\langle\sigma,\phi\rangle \colon \Delta^{j_{0}}\ast\cdots\ast\Delta^{j_{k}}\to N(\ob{K})$. A simplex of  this type cannot be in $N(\ob{K})_{+}$ and we have $N(\ob{K})_+=\ob{K}_+$.
We specify now the definition of a face operator, $\partial_i$, by considering three cases:
\begin{itemize}
\item if $\alpha\in\ob{K}_+$ is a simplex whose $\partial_i$-face in $\ob{K}$ belongs to $\ob{K}_+$, we keep the same face operator than in $\ob{K}$,
\item if $\alpha\in\ob{K}_+$ is a simplex whose $\partial_i$-face in $\ob{K}$ belongs to $\ob{K}\backslash\ob{K}_+$, we set
$\partial_i \alpha=\langle \partial_i\alpha,\alpha\rangle$, (in this case, $\partial_i$ is the last face operator of  $\alpha$),
\item we set $\partial_i\langle\sigma,\phi\rangle=\langle\partial_i\sigma,\phi\rangle$ otherwise.
\end{itemize}
Let $\alpha\colon \Delta=\Delta^{j_0}\ast\cdots\ast\Delta^{j_n}\to \ob{K}$ with $j=\dim\Delta$. For the commutation rule of face operators, we can reduce the verification to the next two cases:
\begin{itemize}
\item if $j_n=0$ and $i<j$, we have
$$\partial_i\partial_j\alpha=\partial_i\langle \partial_j\alpha,\alpha\rangle=\langle \partial_i\partial_j\alpha,\alpha\rangle=\langle \partial_{j-1}\partial_i\alpha,\alpha\rangle=\langle \partial_{j-1}\partial_i\alpha,\partial_{i}\alpha\rangle=\partial_{j-1}\partial_i\alpha,$$
\item if $j_n=1$ and $i=j-1$, we have
$$\partial_i\partial_j\alpha=
\langle\partial_i\partial_j\alpha,\partial_j\alpha\rangle=
\langle\partial_i\partial_j\alpha,\alpha\rangle=
\langle\partial_{j-1}\partial_i\alpha,\alpha\rangle=\langle\partial_{j-1}\partial_i\alpha,\partial_i\alpha\rangle=\partial_{j-1}\partial_i\alpha.$$
\end{itemize}

We continue with the verification of the properties of a normal \ffs. 
Let $\langle\sigma,\phi\rangle\in N(\ob{K})\backslash N(\ob{K})_+$ with
$\phi\colon \Delta^{j_0}\ast\cdots\ast\Delta^{j_n}\to \ob{K}_{+}$. 
Let $J$ be the smallest subset of indices such that $\partial_J\phi\notin\ob{K}_+$ and choose $I$ such that $\partial_I\partial_J\phi=\sigma$. By definition of face operators, we have
$$\partial_I\partial_J\phi=\partial_I\langle\partial_J\phi,\phi\rangle=\langle\partial_I\partial_J\phi,\phi\rangle=\langle\sigma,\phi\rangle$$
and $\langle\sigma,\phi\rangle \vartriangleleft \phi$,
 which means that $\langle\sigma,\phi\rangle$
is a face of an element of $N(\ob{K})_+$. We study now the uniqueness property of the definition of normal \ffs.
Let $\langle\sigma,\phi\rangle\in N(\ob{K})\backslash N(\ob{K})_+$ and $\beta$, $\beta'$ in $N(\ob{K})_+$ such that
$\langle \sigma,\phi\rangle \vartriangleleft \beta$
and 
$\langle \sigma,\phi\rangle \vartriangleleft \beta'$.
Let $J$ be the smallest subset such that $\partial_J\beta\notin\ob{K}_+$ and $I$ such that
$\langle\sigma,\phi\rangle=\partial_I\partial_J\beta$. By definition of the face operator, we have
$$\langle\sigma,\phi\rangle=\partial_I\partial_J\beta=\partial_I\langle\partial_J\beta,\beta\rangle=\langle\partial_I\partial_J\beta,\beta\rangle,$$
which implies $\sigma=\partial_I\partial_J\beta$ (in $\ob{K}$!) and
$\langle\sigma,\phi\rangle=\langle\sigma,\beta\rangle$. Similarly, we prove $\langle\sigma,\beta'\rangle=\langle\sigma,\phi\rangle$ and deduce
$$\langle\sigma,\phi\rangle=\langle\sigma,\beta\rangle=\langle\sigma,\beta'\rangle.$$
By definition of the equivalence relation, there exists a family
of elements of $\ob{K}_+=N(\ob{K})_+$, $(\beta_1,\ldots,\beta_p,\ldots,\beta_m)$, with
\begin{equation}\label{equa:facesandfaces}
\beta \vartriangleleft \beta_1 \vartriangleright\cdots \vartriangleleft\beta_p \vartriangleright \phi \vartriangleleft \beta_{p+1} \vartriangleright \cdots \vartriangleleft \beta_m \vartriangleright \beta'
\text{ and } \sigma \vartriangleleft \beta_i,
\end{equation}
for all $i\in\{1,\ldots,m\}$. We are reduced  to prove that $\langle\sigma,\phi\rangle \vartriangleleft \beta_i$, for all $i\in\{1,\ldots,m\}$. Observe that the relations (\ref{equa:facesandfaces}) imply
$\langle\sigma,\phi\rangle=\langle \sigma,\beta_i\rangle$, for all $i\in\{1,\ldots,m\}$.
Let $J$ be the smallest set of indices such that
$\partial_J\beta_i\notin \ob{K}_+$ and $I$ such that $\partial_I\partial_J\beta_i=\sigma$. By definition of the face operators, we get
$$\partial_I\partial_J\beta_i=\partial_I\langle\partial_J\beta_i,\beta_i\rangle=\langle\partial_I\partial_J\beta_i,\beta_i\rangle=\langle\sigma,\beta_i\rangle,$$
which implies $\langle\sigma,\beta_i\rangle \vartriangleleft \beta_i$. We have obtained
$\langle\sigma,\phi\rangle=\langle\sigma,\beta_i\rangle \vartriangleleft \beta_i$.

A map of \ffss, $\cN\colon N(\ob{K})\to \ob{K}$, is defined by $\cN(\alpha)=\alpha$, if $\alpha\in\ob{K}_+$ and $\cN(\langle\sigma,\phi\rangle)=\sigma$. The verification of the compatibility with face operators is direct from the definitions and the restriction of $\cN$ is the identity map from $N(\ob{K})_+$ to $\ob{K}_+$. 

Suppose now that $\cN\colon N(\ob{K})\to \ob{K}$ and $\cN'\colon N'(\ob{K})\to \ob{K}$ are two normalizations of $\ob{K}$. We construct a map of \ffss, $\gN\colon N(\ob{K})\to N'(\ob{K})$, such that $\cN'\circ\gN=\cN$, by:
$$\begin{array}{ll}
\gN(\alpha)=\cN'^{-1}(\cN(\alpha)),&\text{ if } \alpha\in N(\ob{K})_+,\\
\gN(\alpha)=\partial_I(\cN'^{-1}(\cN(\phi))),&\text{ if } \alpha=\partial_I\phi\in N(\ob{K})\backslash N(\ob{K})_+\text{ and } \phi\in N(\ob{K})_+.
\end{array}$$
We first have to prove that $\gN$ is well defined in the case $\alpha\notin N(\ob{K})_+$. Let $\phi'\in N(\ob{K})_+$ such that $\alpha=\partial_{I'}\phi'$. We may suppose that $\alpha \vartriangleleft \phi \vartriangleleft \phi'$ and set $\phi=\partial_J\phi'$. We have $\alpha=\partial_I\partial_{J}\phi'$ and
$$\partial_I(\cN'^{-1}(\cN(\phi))=\partial_I(\cN'^{-1}(\cN(\partial_J\phi'))=\partial_I\partial_J(\cN'^{-1}(\cN(\phi'))),$$
which proves that $\gN$ is well defined.

We establish now the bijectivity of the map $\gN$. A map $\gN'\colon N'(\ob{K})\to N(\ob{K})$ is defined in a similar manner and, on $N(\ob{K})_+$, they are obviously inverse. If $\alpha=\partial_I\phi$, we have
$$\gN'(\gN(\alpha))=\gN'(\partial_I(\cN'^{-1}(\cN(\phi))))=\partial_I(\cN^{-1}(\cN'(\cN'^{-1}(\cN(\phi)))))=
\partial_I\phi=\alpha.$$

The map $\gN$ preserves the perverse degree. That is obvious on $N(\ob{K})_+$. Let $\alpha\in N(\ob{K})\backslash N(\ob{K})_+$. Then there exists $\phi\in N(\ob{K})_+$ and we may suppose $\alpha=\partial_m\phi$, with $m=|\phi|$. This implies
$$\|\gN(\alpha)\|_{\ell}=\|\partial_m\cN'^{-1}(\cN(\phi))\|_{\ell}=\|\cN'^{-1}(\cN(\phi))\|_{\ell}-1=
\|\phi\|_{\ell}-1=\|\alpha\|_{\ell}.$$

For proving the compatibility of $\gN$  with face operators, we consider three cases.
\begin{itemize}
\item It is direct if $\alpha$ and $\partial_i\alpha$ are in $N(\ob{K})_+$.
\item Let  $\alpha\in N(\ob{K})_+$ with $\partial_i\alpha\notin N(\ob{K})_+$. Then, we have
$$\partial_i\gN(\alpha)=\partial_i(\cN'^{-1}(\cN(\alpha)))=\gN(\partial_i \alpha).$$
\item Let $\alpha\notin N(\ob{K})_+$. There exists $\phi\in N(\ob{K})_+$ with $\alpha=\partial_I\phi$ and we have
$$\partial_i\gN(\alpha)=\partial_i\partial_I(\cN'^{-1}(\cN(\phi)))=\gN(\partial_i \alpha).$$
\end{itemize}

The equality $\cN'\circ \gN=\cN$ is obvious on $N(\ob{K})_+$. If $\alpha=\partial_I\phi$, $\alpha\notin N(\ob{K})_+$ and $\phi\in N(\ob{K})_+$, we have
$$\cN'(\gN(\alpha))=\cN'(\partial_I(\cN'^{-1}(\cN(\phi)))) =
\partial_I \cN(\phi)=\cN(\partial_I\phi)=\cN(\alpha).$$

We have established the existence and unicity of the normalization. We show now the existence of isomorphisms in cohomology and homology.

The restriction map of $\cN$ to  $N(\ob{K})_+$ being an isomorphism,  $N(\ob{K})_+\cong \ob{K}_+$, 
the map $\cN$ induces an isomorphism for the  cohomology with coefficients  in $\tF_{\ov{q}}$, 
for any loose perversity~$\ov{q}$ or $\ov{q}=\ov{\infty}$.

Let $\sigma\colon \Delta=\Delta^{j_0}\ast\cdots\ast\Delta^{j_n}\to N(\ob{K})$ be a $\ov{p}$-admissible simplex for a 
loose perversity $\ov{p}$ such that  $\ov{p}(1)< 0$.
From $\|\sigma\|_{1}=\dim(\Delta^{j_0}\ast\cdots\ast\Delta^{j_{n-1}})$ and $$\|\sigma\|_{1}\leq \dim \Delta -1 +\ov{p}(1)\leq \dim\Delta-2,$$
we deduce $j_n\geq 1$ and  $\sigma\in \ob{K}_+\cong N(\ob{K})_+$. A similar argument works for $\partial \sigma$ and we get  an isomorphism
of chain complexes,
$C_*^{\GM,\ov{p}}(N(\ob{K}))\cong C_*^{\GM,\ov{p}}(\ob{K})$.
\end{proof}

\begin{example}\label{exam:linkconnexenevautpas}
This example shows  that the connectivity of $\cL(\ob{K},\tau)$ is not sufficient for having a quasi-isomorphism between 
$C(\ob{K})$ and $\tC_{\ov{0}}(\ob{K})$. Consider a face set, $\ob{K}$, which looks like a pinch ribbon. It is formed of two triangles, $\Delta^2_{(1)}$ and $\Delta^2_{(2)}$, with a common vertex and the opposite edge in common also. We can represent it as
$${\footnotesize \xymatrix@=12pt{
a\ar@{-}[rrrd]&&&&&&b\ar@{-}[llld]\\
&&&c\ar@{-}[llld]\ar@{-}[rrrd]\\
d\ar@{-}[uu]&&&&&&\ar@{-}[uu]e}
}$$ 
with $[ad]=[be]$. 
We decompose the two triangles as $\{c\}\ast [ad]$ and $\{c\}\ast[be]$ creating a \ffs~still denoted $\ob{K}$. The blow-up of these triangles are
$\tilde{\Delta}_{(1)}^2=c\Delta^0\times \Delta^1=\Delta^1_{(1)}\times \Delta^1$
and
$\tilde{\Delta}_{(2)}^2=c\Delta^0\times \Delta^1=\Delta^1_{(2)}\times \Delta^1$,
with $\Delta^1_{(1)}\neq \Delta^1_{(2)}$
and the second factor in common. The cochains on these blow-ups (compatible with the restriction to the common faces) have the behavior of cochains on a rectangle. Thus, there is no cohomology of degree 1 in perverse degree $\ov{0}$, i.e.,
$H_{\TW,\ov{0}}^1(\ob{K};\Z)=0$. The cohomology of the face set $\ob{K}$ has a generator in degree~1 and, in this case, we have
$H_{\TW,\ov{0}}^1(\ob{K};\Z)\neq H^1(\ob{K};\Z)$. It is easy to check that the link of $\{c\}$ is connected but its expanded link is not connected.

The normalization $N(\ob{K})$ of $\ob{K}$ is formed of $\ob{K}_+$ and of two 0-simplices 
$c_{1}=\langle \{c\}, \Delta^2_{(1)}\rangle$ and
$c_{2}=\langle \{c\}, \Delta^2_{(2)}\rangle$. 
The determination of the other boundaries shows that in $N(\ob{K})$  the two triangles still have a common edge
but the two cone points are now different. Thus $N(\ob{K})$ can be represented as
$${\footnotesize \xymatrix@=12pt{
&&&a=b\ar@{-}[dd]\ar@{-}[drrr]&&&\\
c_1\ar@{-}[rrru]\ar@{-}[rrrd]&&&&&&c_2\ar@{-}[llld]\\
&&&c=d&&&}
}$$ 
with $N(\ob{K})_+\cong \ob{K}_+\cong N(\ob{K})\backslash \{c_1,c_2\}\cong \ob{K}\backslash\{c\}$.
\end{example}

\begin{example}\label{exam:cone}
Here, we work over $\Q$.
\emph{The cone on the face set  $\ob{S}$} is  the \ffs~$c\ob{S}=\{\vartheta\}\ast \emptyset\ast\cdots\ast\emptyset\ast\ob{S}\in {\pmb\Delta}^{[n]}_\cF$. 
\index{Cone on a face set}
The simplices of $c\ob{S}$ are of three kinds,\linebreak
$\{\vartheta\}\ast\sigma\colon \Delta^0\ast \emptyset\ast\cdots\ast\emptyset\ast\Delta^{j_{n}}\to c\ob{S}$,
$\{\vartheta\}\colon \Delta^0\ast \emptyset\ast\cdots\ast\emptyset\to \{\vartheta\}\subset c\ob{S}$
or $\sigma\colon \emptyset\ast\cdots\ast\emptyset\ast\Delta^{j_{n}}\to \ob{S}\subset c \ob{S}$,
with $\sigma\in\ob{S}$. 
Observe that the $\{\vartheta\}\ast\sigma$ and $\sigma$ are the only elements of $c\ob{S}_{+}$.
The blow-up of  Sullivan's forms on these simplices is defined by  
$\widetilde{A}_{PL}(c\ob{S})_{\{\vartheta\}\ast\sigma}=A_{PL}(c\Delta^0)\otimes A_{PL}(\Delta^{j_{n}})$
and
$\widetilde{A}_{PL}(c\ob{S})_{\sigma}=\Q\otimes A_{PL}(\Delta^{j_{n}})$.
As the compatibility to face operators involves only the face operators on $\{\vartheta\}\ast \emptyset\ast\cdots\ast\emptyset\ast\ob{S}$, we are reduced to the 
compatibility on the factors $A_{PL}(\Delta^{j_{n}})$ and we get
$$\widetilde{A}_{PL}(c\ob{S})=A_{PL}(\Delta^1)\otimes A_{PL}(\ob{S})=\land(t,dt)\otimes  A_{PL}(\ob{S}).$$
Recall from \remref{rem:blowup} that in this system of coordinates, the face $\{\vartheta\} \times\{1\}$ of $c\Delta^0$ corresponds to $t=0$. Thus the perverse degrees are determined by
$\|t\otimes\omega\|=\|dt\otimes\omega\|=-\infty$, $\|1\otimes \omega\|=|\omega|$,
if $\omega \in A_{PL}(\ob{S})$ with $\omega\neq 0$.

As there is only one singular stratum, a loose perversity, $\ov{q}$, is given by a non-negative number, $\ov{q}(n)$. We want to determine the forms of $\ov{q}$-intersection. First, if this form has  a component in $t$ or $dt$, there is no restriction on the second factor and all the elements of 
$\land^+(t,dt)\otimes A_{PL}(\ob{S})$ are $\ov{q}$-admissible. (Here, $\land^+(t,dt)$ is the sum of elements of the shape $t^{\alpha}$, $t^{\beta}dt$, with $\alpha\geq 1$ and $\beta\geq 0$.) On the other hand, the element $1\otimes \omega$ is $\ov{q}$-admissible if, and only if, $|\omega|\leq\ov{q}(n)$. From that, we deduce that the set of forms of $\ov{q}$-intersection is
$$\left(A_{PL}(\Delta^1)\otimes A_{PL}(\ob{S})\right)_{\ov{q}}
=
 \left(\land^+(t,dt)\otimes A_{PL}(\ob{S})\right)\oplus
 A_{PL}^{< \ov{q}(n)}(\ob{S})\oplus \cZ A_{PL}^{\ov{q}(n)}(\ob{S}).
 $$
 The \cdga~$\land^+(t,dt)$ being contractible, there is a  quasi-isomorphism,
 $$\left(\widetilde{A}_{PL}(c\ob{S})\right)_{\ov{q}}
 \simeq  A_{PL}^{< \ov{q}(n)}(\ob{S})\oplus \cZ A_{PL}^{\ov{q}(n)}(\ob{S})=\tau_{\leq \ov{q}(n)}A_{PL}(\ob{S}),$$
 where the truncation $\tau_{\leq \ov{q}(n)}$ is introduced in \defref{def:troncature}.
This last cochain complex involves only the \cdga~$A_{PL}(\ob{S})$ and inherits a structure of \cdga. We may define a structure of strict perverse algebra on
$A_{PL}(\ob{S})$ (see \defref{def:perversecochains}) by setting $\|\omega\|=\max(|\omega|,|d\omega|)$, if $\omega\neq 0$. For this structure, we observe that
$A_{PL}(\ob{S})_{\ov{q}}=A_{PL}^{< \ov{q}(n)}(\ob{S})\oplus \cZ A_{PL}^{\ov{q}(n)}(\ob{S})=\tau_{\leq \ov{q}(n)}A_{PL}(\ob{S})$. 
In conclusion, \emph{the blow-up of Sullivan's forms on a cone, $c\ob{S}$, is related by a quasi-isomorphism to a strict perverse \cdga, defined on the  PL-forms, $A_{PL}(\ob{S})$.}

We end this example with an explicit computation.
 In the case of the \ffs~associated to $c\C P(2)$, only three values of $\ov{q}(n)$ suffice for the description of all the possible intersection cohomologies; they are $\ov{q}(n)=0,\, 2$ and~4. 
 If we denote by $\ov{\ell}$ the perversity such that $\ov{\ell}(n)=\ell$, we have
  \begin{eqnarray*}
 H^i_{\ov{0}}(c\C P(2);\Q)=\Q, &\text{ if } i=0 \text{ and } 0 \text{ otherwise, }\\
 H^i_{\ov{2}}(c\C P(2);\Q)=\Q, &\text{ if } i=0,\,2 \text{ and } 0 \text{ otherwise, }\\
H^i_{\ov{4}}(c\C P(2);\Q)=\Q, &\text{ if } i=0,\,2,\,4 \text{ and } 0 \text{ otherwise.}
 \end{eqnarray*}
 If we choose $n=1$, we did a computation of perverse cohomologies with filtrations that do not correspond to any geometrical notion of dimension but, as expected, the results coincide with the classical determination of the perverse cohomology of a cone. 
 If we want to obtain these cohomologies from $C^*_{\GM}(\ob{K};\Q)$ with the same filtration, we 
 need to choose perversities $\ov{p}$, related to the previous $\ov{q}$'s by
 $\ov{p}+\ov{q}=-1$, which corresponds to the loose perversities, $\ov{p}(1)=-1$, $\ov{p}(1)=-3$ and $\ov{p}(1)=-5$. Evidently, if we choose a geometric filtration of length $n=5$, we recover usual GM-perversities, $\ov{p}$.
\end{example}

 \begin{example}\label{exam:suspensionCP2}
 We work over $\Q$.
\emph{The suspension of the face set $\ob{S}$} is the \ffs~$\Sigma \ob{S}=\{\vartheta_1,\vartheta_2\}\ast \emptyset\ast\cdots\ast\emptyset\ast\ob{S}\in {\pmb\Delta}^{[n]}_\cF$.  
\index{Suspension of a face set}
By definition, the blow-up of  Sullivan's forms is 
  $$\widetilde{A}_{PL}(\Sigma\ob{S})=\left(A_{PL}(\Delta^1)\oplus A_{PL}(\Delta^1)\right)_{t_1=t_2=0}\otimes A_{PL}(\ob{S}),$$
   where $\left(A_{PL}(\Delta^1)\oplus A_{PL}(\Delta^1)\right)_{t_1=t_2=0}$ is the subspace of the direct sum generated by polynomials, $P(t_1,dt_1)+Q(t_2,dt_2)$,
  such that $P(0,0)=Q(0,0)$. 
These polynomials coincide on one side of the intervals; thus, the situation corresponds to two intervals attached together by one vertex and this can be  assimilated to one interval, $I$. We may therefore replace the previous  complex by
$C=A_{PL}(\Delta^1)\otimes A_{PL}(\ob{S})=\land(t,dt)\otimes A_{PL}(\ob{S})$. In this replacement, the 
boundary of the first component, is given by the values $t=0$ and $t=1$. 
Thus any polynomial $f$ in $t$ such that $f(0)=f(1)=0$ has a restriction to theses faces equal to zero. 
This observation allows the determination of the perverse degrees as,
\begin{itemize}
\item if $f\in \land t$, then $\| f\otimes \omega\|=-\infty$ if ($f(0)=f(1)=0$ or $\omega=0$) and
$\| f\otimes \omega\|=|\omega|$ otherwise,
\item if $\alpha=f dt$ with $f\in \land t$, then $\|\alpha\otimes\omega\|=-\infty$.
\end{itemize}  
In conclusion, \emph{the blow-up  of Sullivan's forms on the suspension, $\Sigma \ob{S}$, 
is related by a quasi-isomorphism to a strict perverse \cdga, defined on $\land(t,dt)\otimes A_{PL}(\ob{S})$.}

We introduce now a second cochain complex, $E_{\ov{q}}$, related to $C_{\ov{q}}$ by a quasi-isomorphism.
Let $\ov{q}$ be a loose perversity, given by a non-negative number, $\ov{q}(n)$. We define a cochain complex
by
$$(\ttau^{\geq \ov{q}(n)}A_{PL}(\ob{S}))^i=
\left\{\begin{array}{cl}
0&\text{if }i<\ov{q}(n),\\
(\cZ_cA_{PL}(\ob{S}))^i&\text{if } i=\ov{q}(n),\\
A_{PL}(\ob{S})^i&\text{if } i>\ov{q}(n),
\end{array}\right.
$$
where $\cZ_cA_{PL}(\ob{S})$ denotes a supplementary subspace of the vector space of cocycles. This complex has the same cohomology than the truncation denoted by
$\tau^{\geq \ov{q}(n)+1}A_{PL}(\ob{S})$ 
in \cite[Page 52]{MR2401086}, i.e.,
$H^i(\ttau^{\geq \ov{q}(n)}A_{PL}(\ob{S}))=H^i(A_{PL}(\ob{S}))$, if $i\geq \ov{q}(n)+1$ and 0 otherwise. 
We note that our truncation, $\ttau$, depends on the choice of a supplementary subspace and therefore is not canonical, 
but it has the advantage of being a subspace of $A_{PL}(\ob{S})$ and allows the construction of a morphism,
  $$\varphi\colon  E_{\ov{q}}=\tau_{\leq \ov{q}(n)}A_{PL}(\ob{S})\oplus s(\ttau^{\geq \ov{q}(n)}A_{PL}(\ob{S}))\to C_{\ov{q}},$$
  defined by
$\varphi(\eta)=1\otimes \eta$ if $|\eta|\leq \ov{q}(n)$ and $\varphi(s\eta)=dt\otimes \eta$, 
if $\eta\in \ttau^{\geq \ov{q}(n)}A_{PL}(\ob{S})$.  
We have to prove that $\varphi$ is a quasi-isomorphism.
It is easy to see that
$C_{\ov{q}}^{<\ov{q}(n)}=(\land(t,dt)\otimes A_{PL}(\ob{S}))^{<\ov{q}(n)}$
and
$\cZ C_{\ov{q}}^{\ov{q}(n)}=\cZ (\land(t,dt)\otimes A_{PL}(\ob{S}))^{\ov{q}(n)}$, which imply that $H^{j}(\varphi)$ is an isomorphism for $j\leq \ov{q}(n)$.

\medskip
In degree $\ov{q}(n)+1$, if
$\omega=\sum_{i\geq 0}t^i \alpha_i +\sum_{i\geq 0} t^i\,dt\,\beta_i\in \land(t,dt)\otimes A_{PL}(\ob{S})$ is such that $d\omega=0$, we know that
$$\omega-\alpha_0=d\left(\sum_{i\geq 0}\frac{t^{i+1}}{i+1}\beta_i\right).$$
If $\omega$ is an element of $C_{\ov{q}}$, we have $\alpha_0=0$ and we observe that the forms
$\frac{t^{i+1}}{i+1}\beta_i$ are elements of $C_{\ov{q}}^{\ov{q}(n)}$. Thus any cocycle of degree $\ov{q}(n)+1$ is a coboundary and we obtain
$H^{\ov{q}(n)+1}(C_{\ov{q}}^*)=0$.

In degrees strictly greater than $\ov{q}(n)+1$,  from the previous determination of the perverse degree, there is no restriction for the polynomials in $t$, $f(t)$, in the expressions $f(t)dt\beta_{i}\in C_{\ov{q}}^{>\ov{q}(n)}$ 
 but the  polynomials $f(t)$ in the expressions $f(t)\alpha_{i}\in C_{\ov{q}}^{>\ov{q}(n)}$ have to verify $f(0)=f(1)=0$. 
 In short, an element $\omega\in C_{\ov{q}}^{>\ov{q}(n)}$ is of the shape
$$\omega=\sum_{i\geq 0}(t^{i+2}-t^{i+1})\alpha_{i}+\sum_{i\geq 0}t^idt\beta_{i}.$$
We construct a morphism
$\psi\colon  C_{\ov{q}}^{>\ov{q}(n)}\to s(A^{>\ov{q}(n)}_{PL}(\ob{S}))$ and a homotopy $K$ 
 by
$$\psi(\omega)=s\sum_{i\geq 0}\frac{\beta_{i}}{i+1} \text{ and } 
K(\omega)=\sum_{i\geq 0}\frac{t^{i+1}}{i+1}\beta_{i}-t\left(\sum_{i\geq 0}\frac{\beta_{i}}{i+1}\right).$$
We check $\psi\circ d =d\circ \psi$, $K\circ d+d\circ K=\id-\varphi\circ\psi$, which imply that $H^{j}(\varphi)$ is an isomorphism for $j> \ov{q}(n)$.

\smallskip
 In the case of $\Sigma\C P(2)$, only  three values of $\ov{q}(n)$ are sufficient for the description of all the possible intersection cohomologies. If we denote by $\ov{\ell}$ the perversity such that $\ov{\ell}(n)=\ell$, we have
   \begin{eqnarray*}
   H^i_{\ov{0}}(\Sigma\C P(2);\Q)=\Q
   &\text{ if } i=0,\,3,\,5 \text{ and } 0 \text{ otherwise, }\\
H^i_{\ov{2}}(\Sigma\C P(2);\Q)=\Q
&\text{ if } i=0,\,2,\,5 \text{ and } 0 \text{ otherwise, }\\
 H^i_{\ov{4}}(\Sigma\C P(2);\Q)=\Q
 &\text{ if } i=0,\,2,\,4 \text{ and } 0 \text{ otherwise.}
    \end{eqnarray*} 
\end{example}

\section{Homotopy of filtered  face sets}\label{sec:homotopyffs}

\begin{quote}
In this section, we define the product of a \ffs~with a face set and use it for a definition of homotopy between filtered face maps. Kan fibrations are also defined in the category of \ffss~and linked to their analogue in the category of face sets, see \propref{prop:kanetKan}.
\end{quote}

Let $\Delta^N=\Delta^{j_0}\ast\cdots\ast \Delta^{j_n}$ be a filtered simplex and $\Delta^k$ be a simplex. We describe a \ffs~ $(\Delta^{j_0}\ast\cdots\ast \Delta^{j_n})\otimes \Delta^k$ by its vertices, as follows.
Denote by $z_0<\cdots<z_k$ the vertices of $\Delta^k$ and by $a_{\ell}^{0}<\cdots < a_{\ell}^{j_{\ell}}$ the  vertices of   $\Delta^{j_{\ell}}$, for  any $\ell\in \{0,\ldots,n\}$.
 We extend the order between the vertices of the $\Delta^{j_{k}}$'s to an order on the vertices of $\Delta^N$ by setting
 $a^s_{\ell}\leq a^{s'}_{\ell'}$ if $\ell <\ell'$ or ($\ell=\ell'$ and $s\leq s'$).
 The vertices of $(\Delta^{j_0}\ast\cdots\ast \Delta^{j_n})\otimes \Delta^k$ are the couples $(a_{\ell}^{s},z_i)$ of vertices of $\Delta^N$ and $\Delta^k$. 
Consider a sequence of distinct vertices,
$\left((a^{s_{i}}_{\ell_{i}},z_{w_{i}})\right)_{0\leq i\leq \nu}\,$,
with $\ell_{i}\in \{0,\ldots,n\}$,
$s_{i}\in \{0,\ldots,j_{\ell_{i}}\}$
and
$w_{i}\in \{0,\ldots,k\}$. Such a sequence spans a simplex of
 $(\Delta^{j_0}\ast\cdots\ast \Delta^{j_n})\otimes \Delta^k$
if, and only if, we have
$a_{\ell_{i}}^{s_{i}}\leq a_{\ell_{i+1}}^{s_{i+1}}$
and
$z_{w_{i}}\leq  z_{w_{i+1}}$, for all $i\in\{0,\ldots,\nu-1\}$.
For instance, the maximal simplices are sequences of distinct $(k+N+1)$ couples. The filtration degree is given by the index $\ell$, i.e.,
if $\nabla$ is a simplex of
$(\Delta^{j_0}\ast\cdots\ast \Delta^{j_n})\otimes \Delta^k$,
the number $1+\|\nabla\|_{\ell}$ is equal to the cardinal of the set of vertices whose first component belongs to
$\Delta^{j_0}\ast\cdots\ast \Delta^{j_{n-\ell}}$.

This product is still of formal dimension $n$, as \ffs, and  this construction is compatible with face operators, see \cite[Section~3]{MR0300281}. This justifies the next definition.

\begin{definition}\label{def:productffsandfaceset}
Let $\ob{K}\in \fil$ be a \ffs~and $\ob{T}$ be a face set. \emph{The  
product of $\ob{K}$ and $\ob{T}$}  is the
\ffs~defined by$$\ob{K}\otimes \ob{T}=\colim_{\Delta^{j_0}\ast\cdots\ast \Delta^{j_n}\to \ob{K}}\left(\colim_{\Delta^k\to \ob{T}}\;(\Delta^{j_0}\ast\cdots\ast \Delta^{j_n})\otimes \Delta^k\right).$$
\end{definition}

This product 
is clearly associative, i.e.,
$$\ob{K}\otimes(\ob{T}\otimes \ob{S})=(\ob{K}\otimes\ob{T})\otimes \ob{S},$$
 extends naturally  in a bifunctor,
$$\otimes\colon \fil \times \dset\to \fil,$$
and commutes with colimits by definition.

\medskip
We define now some structures on the sets of morphisms. We keep the previous notation: $\ob{K}$, $\ob{L}$ are  \ffss~and $\ob{T}$ is a  face set. We define \emph{a \ffs, $\ob{K}^{\ob{T}}$,} by
$$(\ob{K}^{\ob{T}})_{j_0,\ldots,j_n}=\hom_{\fil}((\Delta^{j_0}\otimes\cdots\otimes\Delta^{j_n})\otimes\ob{T},\ob{K}),$$
and \emph{a face set, $\hom^{\pmb\Delta}(\ob{K},\ob{L})$,} by
$$\hom^{\pmb\Delta}(\ob{K},\ob{L})_k=\hom_{\fil}(\ob{K}\otimes \Delta^k,\ob{L}).$$
The following formulae of adjunctions derive directly from the definitions and the usual preservation of colimits by a $\hom$-functor .

\begin{proposition}\label{prop:adjonctionproduct}
If $\ob{K}$, $\ob{L}$ are two \ffss~and $\ob{T}$ is a face set,  we have natural bijections,
\begin{eqnarray*}\hom_{\fil}(\ob{K}\otimes\ob{T},\ob{L})&\cong& \hom_{\fil}(\ob{K},\ob{L}^{\ob{T}})\\
&\cong&
\hom_{\dset}(\ob{T},\hom^{\pmb\Delta}(\ob{K},\ob{L})).
\end{eqnarray*}
\end{proposition}

The two canonical maps, $\iota_0,\,\iota_1\colon \Delta^0\to \Delta^1$, give a notion of homotopy in the category $\fil$.

\begin{definition}\label{def:homotopyffs}
Let $f,\,g\colon \ob{K}\to \ob{L}$ be two filtered face maps. They are \emph{homotopic} if there exists a filtered face map, $F\colon \ob{K}\otimes \Delta^1\to \ob{L}$, such that $F\circ (\id\otimes \iota_0)=f$ and $F\circ (\id\otimes \iota_1)=g$. We denote this relation by $f\simeq g$.
\index{Homotopy!of filtered face maps}
\end{definition}

In particular, the two injections,
$\id\otimes\iota_{i}\colon \ob{K}\to \ob{K}\otimes \Delta^1$, for $i=0,\,1$,
are homotopic.
As in the case of face sets, for getting an equivalence relation, we have to impose some restrictions. Let $m\geq 1$. We denote by $\Delta^{m,k}$ the face subset of $\partial \Delta^m$ obtained by removing the $k$-th $(m-1)$-face.

\begin{definition}\label{def:Kanffs}
A filtered face map, $f\colon \ob{K}\to \ob{K}'$, is a \emph{Kan fibration} if, for each \ffs, $\ob{L}$, and each commutative diagram of filtered face maps,
$$\xymatrix{
\ob{L}\otimes \Delta^{m,k}\ar[r]^-{\varphi}\ar[d]&\ob{K}\ar[d]^f\\
\ob{L}\otimes\Delta^m\ar[r]^-{\psi}\ar@{-->}[ur]^-{g}&\ob{K}',
}$$
there exists a filtered face map, $g\colon \ob{L}\otimes \Delta^m\to \ob{K}$, extending $\varphi$ and lifting~$\psi$. In the case $\ob{K}'=\ob{K}'_+=\emptyset\ast\cdots\ast\emptyset\ast\Delta^0$, the \ffs~$\ob{K}$ is called a \emph{Kan \ffs}.
\index{Kan fibration}
\end{definition}

This definition is connected to the notion of Kan fibration between face sets, see \cite[Section~5]{MR0300281}.

\begin{proposition}\label{prop:kanetKan}
A filtered face map, $f\colon \ob{K}\to \ob{K}'$, is a Kan fibration if, and only if, the corresponding filtered face  map,
$\hat{f}\colon \hom^{\pmb\Delta}(\ob{L},\ob{K})\to \hom^{\pmb\Delta}(\ob{L},\ob{K}')$,
is a Kan fibration for any \ffs, $\ob{L}$.
\end{proposition}

\begin{proof}
With  \propref{prop:adjonctionproduct}, the existence of a filtered face map $g$, making commutative the following diagram,
$$\xymatrix{
\ob{L}\otimes \Delta^{m,k}\ar[r]^-{\varphi}\ar[d]&\ob{K}\ar[d]^{f}\\
\ob{L}\otimes \Delta^m\ar[r]^-{\psi}\ar@{-->}[ru]^-{g}&\ob{K}',
}$$
is equivalent to the existence of a filtered face map, $\hat{g}$, making commutative the following diagram,
$$\xymatrix{
\Delta^{m,k}\ar[r]^-{\hat{\varphi}}\ar[d]&\hom^{\pmb\Delta}(\ob{L},\ob{K})\ar[d]^{\hat{f}}\\
\Delta^m\ar[r]^-{\hat{\psi}}\ar@{-->}[ur]^-{\hat{g}}&\hom^{\pmb\Delta}(\ob{L},\ob{K}'),}$$
where $\hat{\varphi}$, $\hat{\psi}$ correspond to $\varphi$, $\psi$ in the adjunction correspondence, respectively.
This last property is the definition of a Kan fibration in $\dset$, see \cite[Section~5]{MR0300281}.
\end{proof}

\begin{corollary}\label{cor:Kanmapspace}
Let $\ob{K}$ be a Kan \ffs~and $(\ob{T},\ob{S})$ be a pair of face sets. Then the map $\ob{K}^{\ob{T}}\to \ob{K}^{\ob{S}}$ is a Kan fibration in $\fil$.
\end{corollary}

\begin{proof}
Let $\ob{L}$ be a \ffs. 
With \propref{prop:kanetKan}, we have to show the existence of
a dotted arrow making commutative the following diagram in $\dset$,
$$\xymatrix{
\Delta^{m,k}\ar[r]\ar[d]&\hom^{\Delta}(\ob{L},\ob{K}^{\ob{T}})\ar[d]\\
\Delta^m\ar[r]\ar@{-->}[ru]&\hom^{\Delta}(\ob{L},\ob{K}^{\ob{S}}).
}$$
This problem of existence can also be rewritten as the following lifting problem in $\dset$,
$$\xymatrix{
\Delta^m\otimes \ob{S}\cup_{\Delta^{m,k}\otimes\ob{S}}\Delta^{m,k}\otimes \ob{T}\ar[r]\ar[d]&
\hom^{\Delta}(\ob{L},\ob{K}).\\
\Delta^m\otimes \ob{T}\ar@{-->}[ru]&
}$$
As $\ob{K}\to\emptyset\ast\cdots\ast \emptyset\ast\Delta^0$ is a Kan fibration, by \propref{prop:kanetKan}, the induced map
$\hom^{\Delta}(\ob{L},\ob{K})\to \hom^{\Delta}(\ob{L},\emptyset\ast\cdots\ast \emptyset\ast\Delta^0)=
\Delta^0$ is a Kan fibration (i.e.,  $\hom^{\Delta}(\ob{L},\ob{K})$ is a Kan face set) and this last dotted arrow exists.
\end{proof}

Two maps, $f,\,g\colon \ob{K}\to \ob{L}$ in $\fil$, can be viewed as elements of $\hom^{\pmb\Delta}(\ob{K},\ob{L})_0$. 
From \defref{def:homotopyffs}, they are homotopic in $\fil$~if, and only if, there exists 
$F\in \hom^{\pmb\Delta}(\ob{K},\ob{L})_1$ 
such that $\partial_0 F=f$ and $\partial_1 F=g$, 
where $\partial_0$, $\partial_1$ are face operators of the face set 
$\hom^{\pmb\Delta}(\ob{K},\ob{L})$. This
observation implies the next result, by definition of homotopy groups of Kan face sets, see \cite[Section~6]{MR0300281}.

\begin{corollary}\label{cor:homotopyffs}
If $\ob{K}$ is a Kan \ffs, the relation of homotopy between filtered face maps, of codomain $\ob{K}$, is an equivalence relation and the set of equivalence classes verifies
$$[\ob{L},\ob{K}]_{\fil}=\pi_0\hom^{\pmb\Delta}(\ob{L},\ob{K}).$$
\end{corollary}

\begin{proof}
We have only to check that
$\hom^{\Delta}(\ob{L},\ob{K})$
is a Kan face set and this is already done at the end of the proof of \corref{cor:Kanmapspace}.
\end{proof}

We study now the behavior of the product of a \ffs~and a face set with intersection cohomology.

\begin{proposition}\label{prop:preKunneth}
Let $\ob{K}$ be a \ffs, $f\colon\ob{X}\to\ob{Y}$ be a face map between face sets and $\ov{p}$ be a loose perversity. Then, the map
$\id\otimes f\colon \ob{K}\otimes\ob{X}\to \ob{K}\otimes\ob{Y}$
induces a chain map
$$(\id\otimes f)_*\colon C_*^{\GM,\ov{p}}(\ob{K}\otimes\ob{X})\to
C_*^{\GM,\ov{p}}(\ob{K}\otimes\ob{Y}).$$
Moreover, if $f$ is homotopic to $g$, then we have 
$H_*^{\GM,\ov{p}}(\id\otimes f)=H_*^{\GM,\ov{p}}(\id\otimes g)$. In particular, if $f$ is a homotopy equivalence, the chain map
$(\id\otimes f)_*$ is a quasi-isomorphism. 
\end{proposition}

\begin{proof}
The map $\id\otimes f$ induces a chain map $C_*(\ob{K}\otimes\ob{X})\to C_*(\ob{K}\otimes\ob{Y})$ and we have only to
determine its behavior with the perverse degree. As it is a local computation, we suppose $f\colon \Delta^k\to\Delta^{k'}$ and consider the map $\id\otimes f\colon  (\Delta^{j_0}\ast\cdots\ast\Delta^{j_n})\otimes \Delta^k\to (\Delta^{j_0}\ast\cdots\ast\Delta^{j_n})\otimes \Delta^{k'}$.

Let $\sigma\colon \nabla\to (\Delta^{j_0}\ast\cdots\ast\Delta^{j_n})\otimes \Delta^k$ be a simplex of the product defined at the beginning of this section; it has a perverse degree induced by the perverse degree of a product. Thus, $\nabla$ has a decomposition such that
$\|\sigma\|_{\ell}=\dim\left(\nabla^{k_0}\ast\cdots\ast\nabla^{k_{n-\ell}}\right)$. Recall also that $\nabla$ can be described by its vertices and these last ones are couples
of a vertex of
$\Delta^{j_0}\ast\cdots\ast\Delta^{j_n}$
and a vertex of $\Delta^k$. In this context, the integer
$1+\|\sigma\|_{\ell}$ is the cardinal of the set of vertices whose first component belongs to
$\Delta^{j_0}\ast\cdots\ast\Delta^{j_{n-\ell}}$. The image 
$((\id\otimes f)\circ \sigma)(\nabla)$ is a simplex whose vertices are the images of the vertices of $\nabla$ by $\id\otimes f$. Thus the number of vertices whose first components are in
$\Delta^{j_0}\ast\cdots\ast\Delta^{j_{n-\ell}}$
cannot increase and we get an induced  map 
$C_*^{\GM,\ov{p}}((\Delta^{j_0}\ast\cdots\ast\Delta^{j_n})\otimes \Delta^k)\to
C_*^{\GM,\ov{p}}((\Delta^{j_0}\ast\cdots\ast\Delta^{j_n})\otimes \Delta^{k'})$.

For the second part of the statement, it is sufficient to prove that the two injections
$\iota_0,\,\iota_1\colon \ob{K}\to \ob{K}\otimes \Delta^1$
induce the same map in homology. The argument used in the proof of \lemref{lem:homotopie1} can be, word for word, adapted to this situation.
\end{proof}

\chapter{Rational algebraic models}\label{chap:rational}

In this chapter, we consider algebraic structures defined on the field of the rational numbers and any vector space is supposed to be a rational vector space.

\section{Perverse differential graded algebras}\label{sec:perversedifferentialalgebras}

\begin{quote}
We collect here some results of Hovey \cite{MR2544388} concerning the closed model structure on the category of perverse \cdga's and specify the notion of homotopy. A definition of minimal model and its unicity up to isomorphisms are provided. 
\end{quote}

The notion of strict perverse CDGA, introduced in \secref{sec:filteredandlocalsystem},  is extended in the one of perverse CDGA defined by M. Hovey, see \cite{MR2544388}. We enlarge the lattice ${\cP}^n$ of GM-perversities used in \cite{MR2544388} by considering the lattice $\hat{\cP}^n=\cP^n\cup\{\ov{\infty}\}$.
 If $\ov{p}$ and $\ov{q}$ are elements of $\hat{\cP}^n$, we define $\ov{p}\oplus \ov{q}$ as the smallest element, $\ov{r}$, of $\hat{\cP}^n$ such that $\ov{p}+\ov{q}\leq \ov{r}$, see \cite{MR572580}. (In particular, if $\ov{p}(i)+\ov{q}(i)> i-2$ for some $i>2$, then $\ov{p}\oplus\ov{q}=\ov{\infty}$.)
We may construct $\ov{p}\oplus \ov{q}$ by induction (see \cite{MR2529162}), starting from $(\ov{p}\oplus\ov{q})(n)=\ov{p}(n)+\ov{q}(n)$. The inductive step is given by:
\begin{eqnarray*}
\text{ if } \ov{p}(k)+\ov{q}(k)<(\ov{p}\oplus\ov{q})(k+1),&\text{ one sets}&
(\ov{p} \oplus\ov{q})(k)=(\ov{p} \oplus\ov{q})(k+1)-1,\\
\text{ if } \ov{p}(k)+\ov{q}(k)=(\ov{p}\oplus\ov{q})(k+1),&\text{ one sets}&
(\ov{p} \oplus\ov{q})(k)=(\ov{p} \oplus\ov{q})(k+1).
\end{eqnarray*}
This law is commutative, associative and has the null perversity as neutral element. 

\begin{definition}\label{def:perversevs}
A \emph{perverse vector space} is 
\index{Perverse!vector space}\index{Perverse!linear map}
a functor from $\hat{\cP}^n$ to the category of vector spaces. 
If this functor takes value in the category of graded vector spaces, we use the expression \emph{perverse graded vector space.}
\emph{Perverse linear maps} are natural transformations between perverse vector spaces.
\end{definition}

If $M_{\bullet}$ is a perverse vector space, we denote by $\varphi^{\ov{q}}_{\ov{p}}\colon M_{\ov{p}}\to M_{\ov{q}}$ the morphism associated to $(\ov{p}\leq\ov{q})$.
As in  \cite[Section~2]{MR2544388},we define a tensor product of perverse vector spaces by:
\begin{equation}\label{equa:perversedegreeproduct}
(M_{\bullet}\otimes N_{\bullet})_{\ov{r}}=\colim_{\ov{p}+\ov{q}\leq \ov{r}} M_{\ov{p}}\otimes N_{\ov{q}}.
\end{equation}

For any perversity $\ov{p}\in\hat{\cP}^n$, there is an evaluation functor, $\ev_{\ov{p}}$, which associates to a perverse vector space, $M_{\bullet}$, the vector space $M_{\ov{p}}$. This functor admits a left adjoint, $\cV_{\ov{p}}$, that generates  free objects.

\begin{definition}\label{def:freeperversevs}
\index{Perverse!vector space!$\ov{p}$-free}
Let $V$ be a $\Q$-vector space and $\ov{p}\in\hat{\cP}^n$. The \emph{$\ov{p}$-free perverse vector space generated by $V$}
is defined by,
\begin{itemize}
\item $\cV_{\ov{p}}(V)_{\ov{r}}=V$ if $\ov{r}\geq \ov{p}$ and 0 otherwise;
\item the map
$\varphi_{\ov{r}_{1}}^{\ov{r}_{2}}$
is the identity map $V\to V$ if $\ov{p}\leq \ov{r}_{1}\leq \ov{r}_{2}$, and the canonical injection $0\to V$ or the identity map $0\to 0$ otherwise.
\end{itemize}
\end{definition}

 The \emph{projective perverse vector spaces} are direct factors of 
$\oplus_{\ov{p}} \cV_{\ov{p}}(V(\ov{p}))_{\bullet}$, where the $V(\ov{p})$'s are vector spaces.
More precisely, (\cite[Proposition 2.1]{MR2544388}),
a  perverse vector space, $M_{\bullet}$, is \emph{projective} if the map
$\colim_{\ov{p}<\ov{q}}M_{\ov{p}}\to M_{\ov{q}}$
is a split monomorphism for all  $\ov{q}$.

\begin{definition}\label{def:perversecga}
A \emph{perverse algebra} is a monoid in the category of perverse vector spaces. In the case of a  commutative monoid, we have a \emph{perverse commutative graded algebra} (henceforth perverse \cga).
\index{Perverse!algebra}\index{Perverse!CGA}
We denote by $\Q_{\bullet}$ the commutative perverse algebra defined by $\varphi_{\ov{p}}^{\ov{q}}=\id\colon \Q_{\ov{p}}=\Q\to \Q_{\ov{q}}=\Q$, for any $\ov{p}\leq \ov{q}$.
\end{definition}

A \emph{perverse chain complex} is a functor from $\hat{\cP}^n$ in
the category of chain complexes.
\index{Perverse!chain complex}\index{Perverse!chain map}\index{Perverse!cochain complex}
A \emph{perverse chain map} is a perverse linear map compatible with the differentials. We denote by $\ch(\hat{\cP}^n)$ the associated category.

The projective objects of $\ch(\hat{\cP}^n)$ are the perverse chain complexes, $A_{\bullet}$, formed  of  perverse projective vector spaces and such that each perverse chain map into an acyclic complex, $A_{\bullet}\to B_{\bullet}$,  is chain homotopic to zero, (\cite[Section~3]{MR2544388}). 

\begin{theoremb}[cf. \cite{MR1912401}, \cite{MR2544388}]
The category $\ch(\hat{\cP}^n)$ is a closed model category for the following classes of objects.
\begin{itemize}
\item The \emph{weak equivalences} are the  perverse  chain maps $f_{\bullet}\colon M_{\bullet}\to N_{\bullet}$ such that
$H_*(f_{\ov{p}})\colon H_*(M_{\ov{p}})\to H_*(N_{\ov{p}})$ is an isomorphism for any perversity~$\ov{p}\in\hat{\cP}^n$.
\item The \emph{fibrations} are the surjections.
\item The \emph{cofibrations} are the injective maps with a projective cokernel.%
\end{itemize}
This structure is monoidal and each object of $\ch(\hat{\cP}^n)$ is fibrant.
\end{theoremb}

\begin{definition}\label{def:perversecdga}
\index{Perverse!DGA}\index{Perverse!CDGA}
A \emph{perverse differential graded algebra} (henceforth perverse  \dga)~is a monoid in the category $\ch(\hat{\cP}^n)$.
In the case of a  commutative monoid, we have a \emph{perverse commutative differential graded algebra} (henceforth perverse \cdga).
We denote by $\cdgaf$ the category of  perverse \cdga's.
\end{definition}

\begin{theoremb}[cf.  \cite{MR2544388}]
The category $\cdgaf$  is a closed model category for the following classes of objects.
\begin{itemize}
\item The \emph{weak equivalences} are the morphisms which are weak equivalen\-ces in $\ch(\hat{\cP}^n)$.
\item The \emph{fibrations} are the surjections.
\item The \emph{cofibrations} are defined by the lifting property relatively to trivial fibrations.
\end{itemize}
\end{theoremb}

To any strict perverse cochain complex, $A$, we associate a perverse cochain complex, $A_{\bullet}$, defined by 
\begin{equation}\label{equa:strictpasstric}A_{\ov{p}}=\{\omega \mid \|\omega\|_i\leq \ov{p}(i) \text{ and } \|d\omega\|_i\leq \ov{p}(i)\text{ for all } i\}.
\end{equation}
Moreover, if $A$ is a strict perverse DGA,  then $A_{\bullet}$ is a perverse DGA and $H(A_{\bullet})$ is a perverse algebra. 

For instance, if  $\ob{K}$ is a \ffs~and $F$ is a universal system, the blow-up, $\tF(\ob{K})$, generates a perverse cochain complex $\tF(\ob{K})_{\bullet}$. In particular, the blow-up
of Sullivan's forms gives a perverse \cdga, $\cA(\ob{K})_\bullet$, defined by
$\cA(\ob{K})_{\ov{q}}= \widetilde{A}_{PL,\ov{q}}(\ob{K})$.

In the perverse cohomology of the suspension of $\C P(2)$ (see  \exemref{exam:suspensionCP2}) an element of $H_{\ov{0}}^3(\Sigma\CP(2))$ disappears in $H_{\ov{2}}^3(\Sigma\CP(2))$ 
and we have no inclusion of $H_{\ov{0}}^3(\Sigma\CP(2))$ into $H_{\ov{2}}^3(\Sigma\CP(2))$.
If the perverse cochain complex is coming from a strict perverse cochain complex,  any 
morphism
$\varphi^{\ov{q}}_{\ov{p}}\colon A_{\ov{p}}\to A_{\ov{q}}$
is injective. Therefore, the cohomology of a strict perverse CDGA is a perverse algebra which, in general, is not coming from a structure of strict perverse algebra.
 For the study of formality, it is important that the involved cochain complexes and their cohomology are objects of the same category. 
 Thus, the introduction of  the perverse CDGA's defined by M. Hovey (\cite{MR2544388}) is essential in this context. 

\smallskip
We introduce now free objects in the category of perverse CGA's.
The  perverse tensor algebra on a perverse graded vector space, $M_{\bullet}$, is defined by $T(M_{\bullet})= \oplus_{k\geq 0} M_{\bullet}^{\otimes k}$, where the tensor products $(M_{\bullet}^{\otimes k})_{\bullet}$ are defined in (\ref{equa:perversedegreeproduct}). 
The \emph{perverse commutative tensor algebra}   is defined as 
$\land M_{\bullet}=\oplus_{k\geq 0} M_{\bullet}^{\otimes k}/\Sigma_{k}$, where the action of the symmetric group $\Sigma_{k}$ is characterized by the next action of the transposition $(i, i+1)$,
$$x_0\otimes \cdots \otimes x_k\mapsto 
(-1)^{|x_i|\,|x_{i+1}|}
x_0\otimes \cdots \otimes x_{i-1}\otimes x_{i+1}\otimes x_i\otimes x_{i+2}\otimes \cdots \otimes x_n.$$
If there is no ambiguity, the element $x\otimes x'$ of $\land M_{\bullet}$ is also denoted $xx'$.

\begin{definition}\label{def:freeperversecga}
\index{Perverse!CGA!free}
A perverse \cga, $B_{\bullet}$, is called \emph{free} if there exists a family of graded vector spaces, $(V_{[\ov{p}]})_{\ov{p}\in\hat{\cP}^n}$, such that $B_{\bullet}=\land \left(\oplus_{\ov{p}\in\hat{\cP}^n}\cV_{\ov{p}}(V_{[\ov{p}]})_{\bullet}\right)$.
\end{definition}

From the definition, we deduce the existence of a canonical injection $\Q_{\bullet}\to B_{\bullet}$, for any free perverse \cga, $B_{\bullet}$.
Observe also that a free perverse \cga, $B_{\bullet}$, comes from a strict perverse \cga:
the elements  $v\in V_{[\ov{p}]}^i$ have a  degree denoted by $|v|=i$, and a perverse degree denoted by $\|v\|=\ov{p}$. This perverse degree is extended to the free \cga~by $\|v_1v_2 \|=\|v_1\| + \|v_2\|$. 
We use this structure to simplify the notation in $$B=\land (\oplus_{\ov{p}\in\hat{\cP}^n} V_{[\ov{p}]}),$$
with $B_{\ov{p}}$  the vector space generated by the products, $v_{1} \ldots v_{k}$, with  $v_{i}\in V_{[\ov{p}_{i}]}$ and  $ \ov{p}_{1}+\cdots+\ov{p}_{k}\leq \ov{p}$, i.e.,
\begin{equation}\label{equa:freepbar}
B_{\ov{p}}=\langle v_{1} \ldots v_{k}\mid v_{i}\in V_{[\ov{p}_{i}]}\text{ and } \ov{p}_{1}+\cdots+\ov{p}_{k}\leq \ov{p}\rangle.
\end{equation}
For a perverse \cga, $B_{\bullet}$, we denote by $B^k_{\ov{p}}$ the set of elements of degree $k$ and perverse degree $\ov{p}$. 
We describe some particular cases of free perverse \cga's. 
\begin{itemize}
\item If $V_{[\ov{p}]}=\Q x$, with $x$ of odd degree, and $V_{[\ov{r}]}=0$ for all $[\ov{r}]\neq \ov{[p]}$, then
$B_{\ov{q}}=\Q x$ if $\ov{q}\geq \ov{p}$ and $B_{\ov{q}}=\Q$ otherwise.
\item If $V_{[\ov{p}]}=\Q x$, with $x$ of even degree $k$, and $V_{[\ov{r}]}=0$ for all $[\ov{r}]\neq \ov{[p]}$, then
$B_{\ov{q}}^{ik}=\Q x^i$ if $\ov{q}\geq i\ov{p}$ and $B_{\ov{q}}=\Q$ otherwise.
\end{itemize}
As $\land (V\oplus W)=\land V \otimes \land W$, the general situation can be expressed as  tensor products of these two cases.

\smallskip
We introduce now  the concept of Sullivan minimal model of perverse \cdga's. Recall that, if $V$ is a graded vector  space,
$(\land V)^+$ is the ideal of $\land V$ formed of the elements of strictly positive degree and $\land^{\geq 2}V$ the ideal formed of decomposable elements in $\land V$.

\begin{definition}\label{def:algebredeSullivan}
A \emph{Sullivan minimal perverse \cdga}  is a strict perverse \cdga, $(B,d)$, free as perverse \cga, i.e.,
\index{Perverse!CDGA!minimal}
$B=\land (\oplus_{\ov{p}\in\hat{\cP}^{n}}V_{[\ov{p}]})$, and such that, for any $\ov{p}\in \hat{\cP}^{n}$
and any $m\geq 1$,
we have
$V_{\ov{p}}=\oplus_{k\geq 1} V_{\ov{p}}^k$, and $V^m_{\ov{p}}$ can be written as,
$V^m_{\ov{p}}=\cup_{i=0}^{\infty}V_{[\ov{p}]}(i)^m
$,
with $V_{[\ov{p}]}(i)^m=V_{[\ov{p}]}(i-1)^m\oplus V_{[\ov{p}]}[i]^m$ and
\begin{equation}\label{equa:minimalinduction}
dV^m_{[\ov{p}]}[i]\subset \left((\land \oplus_{\ov{r}<\ov{p}}V_{[\ov{r}]})\otimes
(\land V_{[\ov{p}]}^{<m})\otimes (\land V^m_{[\ov{p}]}(i-1))\right)_{\ov{p}}.
\end{equation}
\end{definition}

Observe that the condition (\ref{equa:minimalinduction}) implies that $(\land V_{[\ov{0}]},d)$ is a classical minimal \cdga, and that the canonical inclusion,
$(\land (\oplus_{\ov{p}\in{\cP}^{n}}V_{[\ov{p}]}),d)\to
(\land (\oplus_{\ov{p}\in\hat{\cP}^{n}}V_{[\ov{p}]}),d)$,
is a relative minimal \cdga, see \cite[(6.5)]{MR736299}.
\begin{definition}\label{def:minimalmodel}
If $\rho_{\bullet}\colon B_{\bullet}\to A_{\bullet}$
\index{Sullivan!minimal model}\index{Perverse!model}
is a quasi-isomorphism of perverse \cdga's, we say that $B_{\bullet}$ is a \emph{perverse model} of $A_{\bullet}$. If $B_{\bullet}$
is a Sullivan minimal perverse \cdga, we call it a \emph{Sullivan minimal model} of $A_{\bullet}$, or a minimal model of 
$A_{\bullet}$.
\end{definition}

\begin{proposition}
Sullivan minimal perverse \cdga's are cofibrant in $\cdgaf$.
\end{proposition}

\begin{proof}
This is a classical argument. If $k\geq 1$ and $\ov{p}\in\cP^n$, we denote by $\Q(k,\ov{p})$ the free perverse cochain complex generated by $z$, $|z|=k$, $\|z\|=\ov{p}$, $dz=0$. If $k\geq 2$, we define $\Q(k-1,k,\ov{p})$ as the free perverse cochain complex generated by $x$ and $y$, $|x|=k-1$, $\|x\|=\ov{p}$, $|y|=k$, $\|y\|=\ov{p}$ and $dx=y$.
The natural perverse cochain maps,
$\Q\to\Q(k,\ov{p})$
and
$\Q(k,\ov{p})\to \Q(k-1,k,\ov{p})$, $z\mapsto y$,
are cofibrations in $\ch(\hat{\cP}^n)$, see \cite[Section~3]{MR2544388}. Thus the perverse \cdga's maps,
$\Q\to\land \Q(k,\ov{p})$
and
$\land \Q(k,\ov{p})\to \land \Q(k-1,k,\ov{p})$,
are cofibrations in $\cdgaf$.

As Sullivan minimal perverse \cdga's can be obtained as successive push-outs built from the previous cofibrations, we get the result, see \cite[Proposition 7.5]{MR0425956} for a similar treatment. 
\end{proof}

As in the classical case (\cite[Theorem 5.1]{MR0646078}), minimal models are unique up to isomorphisms.

\begin{proposition}\label{prop:unicityminimal}
 Any quasi-isomorphism between Sullivan minimal perverse \cdga's is an isomorphism.
\end{proposition}

\begin{proof}
Let $\varphi\colon B=(\land\oplus_{\ov{p}}V_{[\ov{p}]},d)\to
C= (\land\oplus_{\ov{p}}W_{[\ov{p}]},d)$ be a quasi-isomorphism between two Sullivan minimal perverse \cdga's.
 We adapt to the perverse setting a proof made by Antonio G\'omez-Tato in \cite{MR1207067}, by building a perverse \cdga's map, 
 $\psi\colon C\to B$, such that $\varphi\circ\psi=\id$.
 
 First, we observe that the restriction
 $\varphi\colon (\land V_{[\ov{0}]},d)\to (\land W_{[\ov{0}]},d)$
 is a weak equivalence between (classical) minimal models therefore it is an isomorphism and the map $\psi$ is defined on $\land W_{[\ov{0}]}$.
 
 Let $\ov{p}$ be a perversity. By definition, (see (\ref{equa:minimalinduction})) there exists a filtration of $W^k_{[\ov{p}]}$ as
 $W^k_{[\ov{p}]}(i)=W^k_{[\ov{p}]}(i-1)\oplus W^k_{[\ov{p}]}[i]$
 such that
 \begin{equation}
dW_{[\ov{p}]}^k[i]\subset (\land\oplus_{\ov{r}<\ov{p}}W_{[\ov{r}]})\otimes (\land W_{[\ov{p}]}^{< k})\otimes (\land W_{[\ov{p}]}^{k}(i-1)).
\end{equation}
 Set $M= (\land\oplus_{\ov{r}<\ov{p}}W_{[\ov{r}]})\otimes (\land W_{[\ov{p}]}^{< k})\otimes (\land W_{[\ov{p}]}^{k}(i-1))$. We suppose that $\psi$, such that $\varphi\circ\psi=\id$, is defined on $M$. By induction, the construction of $\psi$  on $C$ is therefore reduced to its extension to
 $W_{[\ov{p}]}^k[i]$. With the 5-lemma, we get a quasi-isomorphism
 $$\ov{\varphi}\colon \left(B/\psi(M)\right)_{\ov{p}}\to \left(C/M\right)_{\ov{p}}.$$
 Let $(w_i)$ be a basis of $W_{[\ov{p}]}^k[i]$. Any $w_i$ is a cocycle in $\left(C/M\right)_{\ov{p}}$ that cannot be a coboundary.  We deduce the existence of $v_i\in B_{\ov{p}}$, $m_i\in M_{\ov{p}}$,  $m'_i\in M_{\ov{p}}$ such that
 $$\varphi(v_i)=w_i+m_i \text{ and } dv_i=\psi(m'_i).$$
 We set $\psi(w_i)=v_i-\psi(m_i)$ and we check:
 \begin{eqnarray*}
 dw_i&=&d\varphi(v_i)-dm_i=\varphi(\psi(m'_i))-dm_i= m'_i-dm_i,\\
 \psi(dw_i)&=&\psi(m'_i)-\psi(dm_i)=dv_i-d\psi(m_i) =d\psi(w_i),\\
 \varphi(\psi(w_i))&=&\varphi(v_i)-\varphi(\psi(m_i))=\varphi(v_{i})-m_{i}=w_i.
 \end{eqnarray*}
 From the equality $\varphi\circ\psi=\id$, we deduce that $\psi$ is a quasi-isomorphism as well and the same construction applied to $\psi$ brings a morphism $\varphi'$ such that $\psi\circ\varphi'=\id$. Therefore $\psi$ is an isomorphism and $\varphi=\psi^{-1}$ is one also.
\end{proof}
  
\begin{definition}\label{def:cylinderadgc}
A \emph{path object} of an element $B_{\bullet}$ of  $\cdgaf$ is given by
\index{Path object}
$$\xymatrix{
(B \otimes \land(t,dt))_{\bullet}\ar@<.4ex>[r]^-{\pi_0}
\ar@<-.4ex>[r]_-{\pi_1}&B_{\bullet}
}$$
where $|t|=0$,  $|dt|=1$, 
$(B \otimes \land(t,dt))_{\ov{p}}=B_{\ov{p}}\otimes \land(t,dt)$
and the maps $\pi_i$ are defined by the identity on $B_{\bullet}$ and the evaluations
$\pi_i(t)=i$, for $i=0,\,1$.
\emph{Let $A_{\bullet}$ be cofibrant.}
Two  morphisms,
$f_{0},\,f_{1}\colon A_{\bullet}\to B_{\bullet}$ of $\cdgaf$,
 are \emph{homotopic} if there exists a  diagram in $\cdgaf$, 
$$\xymatrix{
A_{\bullet}\ar@<.6ex>[rr]^-{f_{0}}\ar@<-.4ex>[rr]_-{f_{1}}\ar[rrd]_F&&B_{\bullet}\\
&& (B \otimes \land(t,dt))_{\bullet}
\ar@<.4ex>[u]^{\pi_0}
\ar@<-.4ex>[u]_{\pi_1}
}$$
such that $\pi_0\circ F=f_{0}$ and $\pi_1\circ F=f_{1}$.
\end{definition}

\begin{remark}
The sequence
$B_{\bullet}\xrightarrow[]{\iota}(B \otimes \land(t,dt))_{\bullet}\xrightarrow[]{(\pi_{0},\pi_{1})}B_{\bullet}\oplus B_{\bullet}$
is a decomposition of the diagonal $B_{\bullet}\to B_{\bullet}\oplus B_{\bullet}$ in a weak-equivalence followed by a fibration. This corresponds to the classical definition of a path object in a closed model category.
\end{remark}

As $\cdgaf$ is a closed model category, the following properties  are standard, see \cite{MR0223432}.

\begin{proposition}\label{prop:propertiesofhomotopy}
Let $A'_{\bullet}$ and $A_{\bullet}$ be cofibrant objects of $\cdgaf$.
\begin{enumerate}
\item  If
$$\xymatrix@1{
A'_{\bullet}\ar[r]^g&A_{\bullet}\ar@<.6ex>[r]^-{f_0}\ar@<-.4ex>[r]_-{f_1}&B_{\bullet}\ar[r]^h&B_{\bullet}'
}$$
is a diagram of morphisms of $\cdgaf$,  the next properties are satisfied.
\begin{enumerate}
\item If $f_0$ is homotopic to $f_1$, then $h\circ f_0\circ g$ is homotopic to $h\circ f_1\circ g$.
\item Homotopy  is an equivalence relation on the set of morphisms from $A_{\bullet}$ to $B_{\bullet}$. We denote this relation by $\simeq$.
\end{enumerate}
\item  Any weak equivalence, $h\colon B_{\bullet}\to B'_{\bullet}$ in $\cdgaf$, induces an isomorphism between the homotopy classes,
$$h_{\sharp}\colon [A_{\bullet},B_{\bullet}]\to[A_{\bullet},B_{\bullet}'].$$
\item  Any weak equivalence, $g\colon A'_{\bullet}\to A_{\bullet}$ in $\cdgaf$, induces an isomorphism between the homotopy classes,
$$g^{\sharp}\colon [A_{\bullet},B_{\bullet}]\to [A'_{\bullet},B_{\bullet}].$$
\end{enumerate}
\end{proposition}

The notion of homotopy can be extended to any pair of morphisms, not necessarily with a cofibrant domain. With \propref{prop:propertiesofhomotopy} (3),  the next definition does not depend on the choice of a cofibrant model of $B_{\bullet}$.

\begin{definition}\label{def:homotopyallmaps}
Let $f_0,\,f_1\colon B_{\bullet}\to B'_{\bullet}$ be two morphisms in $\cdgaf$ and $\varphi\colon A_{\bullet}\to B_{\bullet}$ be a cofibrant model. The morphisms $f_0$ and $f_1$ are \emph{homotopic} if $f_0\circ \varphi\simeq f_1\circ\varphi$.
\index{Homotopy!of perverse filtered CDGA maps}
\end{definition}

 \section{Balanced perverse cochain complex}\label{sec:balanced}
 
 \begin{quote}
 In this section, we define and study a notion of balanced perverse \cdga's, which are the basis of the construction of a minimal model in the next section. (If $\ob{K}$ is a connected \ffs~and $F$ is a universal system of \cdga's, we prove in \secref{sec:modelCS} that the blow-up, $\tF(\ob{K})_{\bullet}$, and its cohomology, $H(\tF(\ob{K})_{\bullet})$, are balanced perverse \cdga's.)
 \end{quote}
 
We begin with a description  of the predecessors of a given GM-perversity $\ov{p}$. 

\begin{definition}\label{def:peak}
Let $\ov{p}\in {\cP}^n$ be a GM-perversity. 
\index{Peak of a GM-perversity}\index{Perversity!of Goresky-MacPherson!peak of a}
A \emph{peak} of $\ov{p}$ is an integer $j\in\{3,\ldots,n-1\}$ such that $\ov{p}(j+1)=\ov{p}(j)>\ov{p}(j-1)$. The integer $n$ is a peak if $\ov{p}(n)>\ov{p}(n-1)$. We  order  the peaks $(\tn_1,\ldots,\tn_s)$ of $\ov{p}$ by
$3\leq \tn_1<\tn_2<\ldots <\tn_s\leq n$.
For each peak $\tn_i$ of $\ov{p}$, we define a perversity $\ov{p}_i$ by
$$\left\{\begin{array}{l}
\ov{p}_i(k)=\ov{p}(k), \text{ if } k\ne \tn_i,\\
\ov{p}_i(\tn_i)=\ov{p}(\tn_i)-1.
\end{array}\right.$$
\end{definition}
We check easily the following properties.
\begin{enumerate}[(i)]
\item The previous perversities $(\ov{p}_1,\ldots,\ov{p}_s)$ are in ${\cP}^n$.
\item The perversities $(\ov{p}_1,\ldots,\ov{p}_s)$ are the predecessors of $\ov{p}$. \emph{We order them by the corresponding peaks.}
\item We have $\ov{p}(\tn_i)\geq i$ and $\ov{p}_i(\tn_i)\geq i-1$.
\item If $(\ov{p}_1,\ldots,\ov{p}_s)$ are the ordered  predecessors of $\ov{p}$, we denote by $\ov{p}_{i,\ell}$ the GM-perversity $\inf(\ov{p}_i,\ov{p}_\ell)$. The GM-perversities
$$(\ov{p}_{1,2}, \ov{p}_{1,3},\ldots,\ov{p}_{1,s})
$$
are  ordered predecessors of $\ov{p}_1$ and the corresponding ordered peaks are
$(\tn_2,\ldots,\tn_s)$.
\end{enumerate}

We define now the main objects of this section.

\begin{definition}\label{def:balanced}
\index{Perverse!cochain complex!balanced}
A perverse cochain complex, $A_{\bullet}$, defined by the  morphisms, \linebreak
  $\varphi^{\ov{p}}_{\ov{q}}\colon A_{\ov{q}}\to A_{\ov{p}}$, for any $\ov{q}\leq \ov{p}$, is called \emph{balanced} if it satisfies the next properties, for any GM-perversity, $\ov{p}$, and \emph{any} sequence of ordered predecessors of $\ov{p}$, $(\ov{p}_{\nu_{1}},\ldots,\ov{p}_{\nu_{j}})$.
\begin{enumerate}[(a)]
\item  Recall $\ov{p}_{\nu_{i},\nu_{\ell}}=\inf(\ov{p}_{\nu_{i}},\ov{p}_{\nu_{\ell}})$. For any $j\geq 1$ and any $k\leq \ov{p}(\tn_{\nu_{j}})-j+1$, the  sequence,
$$\xymatrix@1{
\oplus_{1\leq i< \ell\leq j}A^k_{\ov{p}_{\nu_{i},\nu_{\ell}}}\ar[r]^-{\varphi}
&
\oplus_{i=1}^jA^k_{\ov{p}_{\nu_{i}}}
\ar[r]^-{\psi}
&
A^k_{\ov{p}},
}$$
is exact, where
\begin{itemize}
\item the morphism $\varphi$ is defined on $\beta_{i,\ell}\in A_{\ov{p}_{\nu_{i},\nu_{\ell}}}$ by
$\varphi(\beta_{i,\ell})=(\alpha_{m})_{1\leq m\leq j}$,
with $\alpha_{i}=\varphi_{\ov{p}_{\nu_{i},\nu_{\ell}}}^{\ov{p}_{\nu_{i}}}(\beta_{i,\ell})$, $\alpha_{\ell}=-\varphi_{\ov{p}_{\nu_{i},\nu_{\ell}}}^{\ov{p}_{\nu_{\ell}}}(\beta_{i,\ell})$ and $\alpha_{m}=0$ otherwise,
\item  and the morphism $\psi$ by
$\psi(\alpha_{1},\ldots,\alpha_{j})=\varphi_{\ov{p}_{\nu_{1}}}^{\ov{p}}(\alpha_{1})+\cdots+\varphi_{\ov{p}_{\nu_{j}}}^{\ov{p}}(\alpha_{j})$.
\end{itemize}
\item The map $\varphi_{\ov{p}_{\nu_{j}}}^{\ov{p}}\colon
A^0_{\ov{p}_{\nu_{j}}}\to A^0_{\ov{p}}$
is surjective if $\ov{p}(\tn_{\nu_{j}})\geq 2$.
\item For any GM-perversity $\ov{q}$, with $\ov{q}\leq \ov{p}$, the induced morphisms,\\
\centerline{$H^1(\varphi^{\ov{p}}_{\ov{q}})\colon H^1(A_{\ov{q}})\to H^1(A_{\ov{p}})$ and  
$H^1(\varphi^{\ov{\infty}}_{\ov{q}})\colon H^1(A_{\ov{q}})\to H^1(A_{\ov{\infty}})$,}
are injective.
\end{enumerate}
\index{Perverse!CDGA!balanced}
\emph{A balanced perverse \cdga}~is a perverse \cdga~which is balanced as perverse cochain complex.
\end{definition}

\begin{remark}\label{rem:petitaj=1}
In the case of a sequence of predecessors of length $j=1$, Property (a) of \defref{def:balanced} can also be stated as follows:\\
for any predecessor $\ov{p}_{i}$ of $\ov{p}$, the map $\varphi^{\ov{p}}_{\ov{p}_{i}}\colon A^k_{\ov{p}_{i}}\to A_{\ov{p}}^k$ is an injection, for any $k\leq \ov{p}(\tn_{i})$.
\end{remark}

\begin{definition}\label{def:+point}
Let $A_{\bullet}$ be a perverse cochain complex and $({\ov{p}_{\nu_{i}}})_{1\leq i\leq j}$ be an ordered sequence of predecessors of a GM-perversity, $\ov{p}$. The \emph{dotted sum,} $\pd_{i=1}^{ j} A_{\ov{p}_{\nu_{i}}}$, is the cochain complex,   quotient of $\oplus_{i=1}^jA^k_{\ov{p}_{\nu_{i}}}$ by the image of $\varphi$, i.e., 
$$\pd_{i=1}^{ j} A_{\ov{p}_{\nu_{i}}}=\oplus_{i=1}^jA_{\ov{p}_{\nu_{i}}}/\varphi(\oplus_{1\leq i< \ell\leq j}A_{\ov{p}_{\nu_{i},\nu_{\ell}}}).$$
We denote by $\langle \alpha_{1},\ldots,\alpha_{j}\rangle\in \pd_{i=1}^{ j} A_{\ov{p}_{\nu_{i}}}$ the class of 
$(\alpha_{1},\ldots,\alpha_{j})\in \oplus_{i=1}^jA_{\ov{p}_{\nu_{i}}}$
and by \linebreak
$\xymatrix@1{
\oplus_{i=1}^jA_{\ov{p}_{\nu_{i}}}\ar[r]_-{\Psi}&\pd_{i=1}^jA_{\ov{p}_{\nu_{i}}}\ar[r]_-{\ov{\psi}}&A_{\ov{p}}
}$
the decomposition of $\psi$.
\end{definition}

The dotted sum, $\pd$, can also be defined inductively, on the number of predecessors, as shows the second property of the next statement.
 
\begin{proposition}\label{prop:+point2}
Let $A_{\bullet}$ be a perverse cochain complex and $({\ov{p}_{\nu_{i}}})_{1\leq i\leq j}$ be an ordered sequence of predecessors of a GM-perversity, $\ov{p}$. The following properties are satisfied.
\begin{enumerate}[(i)]
\item Property (a) of \defref{def:balanced} is equivalent to the injectivity, in degrees $k \leq \ov{p}(\tn_{\nu_{j}})-j+1$, of the map, 
$\ov{\psi}\colon 
\pd_{i=1}^j A^k_{\ov{p}_{\nu_{i}}}
\to A^k_{\ov{p}}$.
\item Let $(\ov{\psi}_{1},\ov{\varphi})\colon  \pd_{\ell=2}^{ j} A_{\ov{p}_{\nu_{1},\nu_{\ell}}}\to 
A_{\ov{p}_{\nu_{1}}}\oplus (\pd_{\ell=2}^{ j} A_{\ov{p}_{\nu_{\ell}}})$, where $\ov{\psi}_{1}$ is the  map  $\ov{\psi}$ associated to $\ov{p}_{1}$ and $\ov{\varphi}$ is induced by $\varphi$. 
(See the proof for an explicit description of these maps.)
There is an isomorphism,
$$\pd_{1\leq i\leq j} A_{\ov{p}_{\nu_{i}}}\cong (A_{\ov{p}_{\nu_{1}}}\oplus (\pd_{\ell=2}^{ j} A_{\ov{p}_{\nu_{\ell}}}))/(\ov{\psi}_{1},\ov{\varphi})(\pd_{\ell=2}^{ j} A_{\ov{p}_{\nu_{1},\nu_{\ell}}}).
$$
In the sequel, we identify these two expressions.
\end{enumerate}
\end{proposition}

\begin{proof}
The first property is obvious. As for the second one, 
let $(\alpha_{1},\ldots,\alpha_{j})\in \oplus_{i=1}^j A_{\ov{p}_{\nu_{i}}}$
and $(\beta_{i,\ell})_{1\leq i<\ell\leq j}\in \oplus_{1\leq i<\ell\leq j} A_{\ov{p}_{\nu_{1},\nu_{\ell}}}$. 
Using the notation of \defref{def:balanced}, we define:
\begin{itemize}
\item $f((\beta_{i,\ell})_{1\leq i<\ell\leq j})=\langle \beta_{1,2},\ldots,\beta_{1,\ell}\rangle\in \pd_{\ell=2}^j A_{\ov{p}_{\nu_{1},\nu_{\ell}}}$,
\item $f'(\alpha_{1},\ldots,\alpha_{j})=(\alpha_{1},\langle \alpha_{2},\ldots,\alpha_{j}\rangle)\in
A_{\ov{p}_{\nu_{1}}}\oplus (\pd_{\ell=2}^j A_{\ov{p}_{\nu_{\ell}}})$,
\item $\ov{\psi}_{1}(\langle \beta_{1,2},\ldots,\beta_{1,j}\rangle)=
\varphi^{\ov{p}_{\nu_{1}}}_{\ov{p}_{\nu_{1},\nu_{2}}}(\beta_{1,2})
+\cdots+
\varphi^{\ov{p}_{\nu_{1}}}_{\ov{p}_{\nu_{1},\nu_{j}}}(\beta_{1,j})
\in A_{\ov{p}_{\nu_{1}}}$,
\item $\ov{\varphi}(\langle \beta_{1,2},\ldots,\beta_{1,j}\rangle)=
\langle \varphi(\beta_{1,2}),\ldots,\varphi(\beta_{1,j})\rangle\in
\pd_{\ell=2}^j A_{\ov{p}_{\nu_{\ell}}}$.
\end{itemize}
The equality
$f'\circ \varphi=(\ov{\psi}_{1},\ov{\varphi})\circ f$
follows from the definitions of these maps. In the next commutative diagram, the maps ${\Psi}$ and $\ov{\Psi}$ are defined as  cokernels of ${\varphi}$ and $(\ov{\psi}_{1},\ov{\varphi})$, respectively.
$$\xymatrix{
\oplus_{1\leq i< \ell\leq j}A_{\ov{p}_{\nu_{i},\nu_{\ell}}}\ar[r]^-{\varphi}
\ar[d]_{f}&
\oplus_{i=1}^jA_{\ov{p}_{\nu_{i}}}
\ar[r]^-{\Psi}
\ar[d]_{f'}&
\pd_{i=1}^j A_{\ov{p}_{\nu_{i}}}
\ar[d]^{f''}\\
\pd_{\ell=2}^j A_{\ov{p}_{\nu_{1},\nu_{\ell}}}\ar[r]^-{(\ov{\psi}_{1},\ov{\varphi})}&
A_{\ov{p}_{\nu_{1}}}\oplus (\pd_{\ell=2}^j A_{\ov{p}_{\nu_{\ell}}})\ar[r]^-{\ov{\Psi}}&
C
}$$
As $\ov{\Psi}$ and $f'$ are surjective, the map $f''$ is surjective also. Let $\ov{x}\in \pd_{i=1}^j A_{\ov{p}_{\nu_{i}}}$ such that $f''(x)=0$. Then there exists $x\in \oplus_{i=1}^jA_{\ov{p}_{\nu_{i}}}$ such that $\Psi(x)=\ov{x}$. 
By exactitude of the bottom line, there is $z\in \pd_{\ell=2}^j A_{\ov{p}_{\nu_{1},\nu_{\ell}}}$ such that $(\ov{\psi}_{1}\oplus\ov{\varphi})(z)=f'(x)$ and $u\in \oplus_{1\leq i< \ell\leq j}A_{\ov{p}_{\nu_{i},\nu_{\ell}}}$ such that $f(u)=z$. The element $x-\varphi(u)$ is in the kernel of $f'$, which is the image by $\varphi$ of
$\oplus_{2\leq i< \ell\leq j}A_{\ov{p}_{\nu_{i},\nu_{\ell}}}$, by definition of $\pd$. This implies, the existence of $v\in
\oplus_{2\leq i< \ell\leq j}A_{\ov{p}_{\nu_{i},\nu_{\ell}}}$ such that
$x-\varphi(u)=\varphi(v)$. Finally, $\ov{x}=\Psi(x)=\Psi(\varphi(u+v))=0$ and the map $f''$ is injective. 
\end{proof}

The next statement is a major point in the construction of a minimal model of balanced perverse \cdga's.

\begin{corollary}\label{cor:+point2}
Let $({\ov{p}_{\nu_{i}}})_{1\leq i\leq j}$ be an ordered sequence of predecessors of a GM-perversity, $\ov{p}$. 
For any balanced perverse  cochain complex, $A_{\bullet}$, the  following properties are satisfied.
\begin{enumerate}[(i)]
\item The sequence $$\xymatrix{
0\ar[r]&
\pd_{\ell=2}^j A_{\ov{p}_{\nu_{1},\nu_{\ell}}}^k\ar[r]^-{(\ov{\psi}_{1},\ov{\varphi})}&
A_{\ov{p}_{\nu_{1}}}^k\oplus (\pd_{\ell=2}^j A_{\ov{p}_{\nu_{\ell}}})^k\ar[r]^-{\ov{\Psi}}&
\pd_{i=1}^j A_{\ov{p}_{\nu_{i}}}^k
\ar[r]&0,
}$$
is exact, for any $k\leq \ov{p}(\tn_{\nu_{j}})-j+2$.
\item The homomorphism
$H^1(\ov{\psi})\colon H^1(A_{\ov{p}_{\nu_{1}}}\pd\cdots\pd A_{\ov{p}_{\nu_{j}}})\to
H^1(A_{\ov{p}})$
is injective.
\end{enumerate}
\end{corollary}

\begin{proof}
(i) With property (ii) of \propref{prop:+point2}, the proof is reduced to the fact that the map
$(\ov{\psi}_{1},\ov{\varphi})$
is injective in the specified degrees, for $j\geq 2$. From the first property of \propref{prop:+point2}, we know that the map
$\ov{\psi}_{\nu_{1}}\colon \pd_{\ell=2}^j A_{\ov{p}_{\nu_{1},\nu_{\ell}}}\to A_{\ov{p}_{\nu_{1}}}$ is injective in degree $k\leq \ov{p}_{\nu_{1}}(\tn_{\nu_{j}})-(j-1)+1=\ov{p}(\tn_{\nu_{j}})-j+2$, and the result is established.

(ii) This property is true for $j=1$ by (a) of \defref{def:balanced}, see \remref{rem:petitaj=1}. Let $j\geq 2$ and consider a cocycle, $\langle \alpha_{1},\ldots,\alpha_{j}\rangle\in A_{\ov{p}_{\nu_{1}}}\pd\cdots\pd A_{\ov{p}_{\nu_{j}}}$, of degree 1 such that there exists $f\in A^0_{\ov{p}}$ with
$df=\varphi_{\ov{p}_{\nu_{1}}}^{\ov{p}}(\alpha_{1})+\cdots+ \varphi_{\ov{p}_{\nu_{j}}}^{\ov{p}}(\alpha_{j})$.
 By (b) of \defref{def:balanced}, there exists $g\in A^0_{\ov{p}_{\nu_{j}}}$ such that $f=\varphi_{\ov{p}_{\nu_{j}}}^{\ov{p}}(g)$ and $\langle \alpha_{1},\ldots,\alpha_{j}-dg\rangle$ is in the kernel of
$\ov{\psi}\colon A^1_{\ov{p}_{\nu_{1}}}\pd\cdots\pd A^1_{\ov{p}_{\nu_{j}}}\to A^1_{\ov{p}}$.
As $\ov{p}(\tn_{\nu_{j}})-j+1\geq 1$, we obtain 
$\langle \alpha_{1},\ldots,\alpha_{j}-dg\rangle=0$ (cf. \propref{prop:+point2}(i)) and
$\langle \alpha_{1},\ldots,\alpha_{j}\rangle=d\langle 0,\ldots,0,g\rangle$.
\end{proof}

If, for any $\ov{p}$, the cochain complex $A_{\ov{p}}$ is a subcomplex of $A_{\ov{\infty}}$, the ordinary sum of vector
 subspaces is defined and can be different from the dotted sum. (We  consider the case of  $\cA(\ob{K})_{\bullet}$
in \remref{rem:+ou+point}.) But these two sums, $+$ and $\pd$, coincide for free perverse \cga's.

\begin{proposition}\label{prop:capandgraded}
Let $B=\land \oplus_{\ov{p}} V_{[\ov{p}]}$ be a free perverse \cga. The following properties are satisfied.  
\begin{enumerate}[(i)]
\item For any perversities, $\ov{p}$, $\ov{q}_1,\ldots,\ov{q}_j$, we have
\begin{eqnarray*}
B_{\ov{q}_1}\cap B_{\ov{q}_2}&=&B_{\inf(\ov{q}_1,\ov{q}_2)},\\
B_{\ov{p}}\cap (B_{\ov{q}_1}+\cdots+B_{\ov{q}_j})&=&(B_{\ov{p}}\cap B_{\ov{q}_1})+\cdots+(B_{\ov{p}}\cap B_{\ov{q}_j}).
\end{eqnarray*}
\item For any ordered sequence, $({\ov{p}_{\nu_{i}}})_{1\leq i\leq j}$, of predecessors of a GM-perversity, $\ov{p}$, one has
$$B_{\ov{p}_{\nu_{1}}}\pd\cdots\pd B_{\ov{p}_{\nu_{j}}}=B_{\ov{p}_{\nu_{1}}}+\cdots+ B_{\ov{p}_{\nu_{j}}}.$$
\end{enumerate}
\end{proposition}

\begin{proof}
(i) 
Define $B_{[\ov{p}]}=\{v_{1}\ldots v_{k}\mid v_{i}\in V_{\ov{q}_{i}} \text{ and } \ov{q}_{1}+\cdots+\ov{q}_{k}=\ov{p}\}$. From \defref{def:freeperversecga}, the equality (\ref{equa:freepbar}) can also be written as
$B_{\ov{p}}=\oplus_{\ov{r}\leq \ov{p}}B_{[\ov{r}]}$.
The inclusion
$(B_{\ov{p}}\cap B_{\ov{q}_1})+\cdots+(B_{\ov{p}}\cap B_{\ov{q}_j})
\subset
B_{\ov{p}}\cap (B_{\ov{q}_1}+\cdots+B_{\ov{q}_j})$
being obvious, let $\omega\in B_{\ov{p}}\cap (B_{\ov{q}_1}+\cdots+B_{\ov{q}_j})$. 
Denote by $$\cJ_i=\{\ov{r}\mid \ov{r}\leq \ov{q}_i\}
\text{ and }
\cJ=\{\ov{r}\mid\ov{r}\leq \ov{p}\}.$$ 
We write $\omega$ as a linear combination of elements of $B_{[\ov{r}]}$, $\ov{r}\in \cup_{i=1}^j \cJ_i$,
$$\omega=\sum_{i=1}^j\sum_{\ov{r}\in\cJ_{i}}v_{i,[\ov{r}]},
\text{ with } v_{i,[\ov{r}]}\in B_{[\ov{r}]}.$$
Decompose now each $\cJ_i$ in two disjoint subsets
$$\cJ_{i,1}=\cJ_i\cap \cJ=\{\ov{r}\mid \ov{r}\leq \ov{q}_i\text{ and } \ov{r}\leq \ov{p}\},\; %
\cJ_{i,2}=\cJ_i\backslash \cJ_{i,1}=\{\ov{r}\mid \ov{r}\leq \ov{q}_i\text{ and } \ov{r}\not\leq \ov{p}\}.$$
Observe  that
$$\sum_{i=1}^j\sum_{\ov{r}\in\cJ_{i,2}}v_{i,[\ov{r}]}=\omega -
\sum_{i=1}^j\sum_{\ov{r}\in\cJ_{i,1}}v_{i,[\ov{r}]}\in B_{\ov{p}}.$$
Any element of $B_{\ov{p}}$ can be written as a sum of elements of $B_{[\ov{r}]}$ with $\ov{r}\in\cJ$.
The intersections $\cJ\cap \cJ_{i,2}$, $B_{[\ov{r}]}\cap B_{[\ov{r}']}$ being the empty set if $\ov{r}\neq\ov{r}'$, we get
$\sum_{i=1}^j\sum_{\ov{r}\in\cJ_{i,2}}v_{i,[\ov{r}]}=0$
and
$$\omega=\sum_{i=1}^j\sum_{\ov{r}\in\cJ_{i,1}}v_{i,[\ov{r}]}\in (B_{\ov{p}}\cap B_{\ov{q}_1})+\cdots+(B_{\ov{p}}\cap B_{\ov{q}_j}).$$

(ii) Let $j=2$. By definition of $\pd$ and Property (i), we have
$B_{\ov{p}_{\nu_{1}}}+B_{\ov{p}_{\nu_{2}}}=B_{\ov{p}_{\nu_{1}}}\pd B_{\ov{p}_{\nu_{2}}}$. 
From (i), we deduce that
$+_{i\geq 2}B_{\ov{p}_{\nu_{1},\nu_{i}}}=B_{\ov{p}_{\nu_{1}}}\cap \left( +_{i\geq 2}B_{\ov{p}_{\nu_{i}}}\right)$
is the kernel of
\linebreak
$B_{\ov{p}_{\nu_{1}}}\oplus \left( +_{i\geq 2}B_{\ov{p}_{\nu_{i}}}\right)\to +_{i\geq 1} B_{\ov{p}_{\nu_{i}}}$. Thus, by induction, we have a commutative diagram,
$$\xymatrix{
0\ar[r]&
+_{i\geq 2}B_{\ov{p}_{\nu_{1},\nu_{i}}}\ar[r]&
B_{\ov{p}_{\nu_{1}}}\oplus \left( +_{i\geq 2}B_{\ov{p}_{\nu_{i}}}\right)\ar[r]&
+_{i\geq 1} B_{\ov{p}_{\nu_{i}}}\ar[r]&
0\\
0\ar[r]&
\pd_{i\geq 2}B_{\ov{p}_{\nu_{1},\nu_{i}}}\ar[r]\ar[u]_{\cong}&
B_{\ov{p}_{\nu_{1}}}\oplus \left( \pd_{i\geq 2}B_{\ov{p}_{\nu_{i}}}\right)\ar[r]\ar[u]_{\cong}&
\pd_{i\geq 1} B_{\ov{p}_{\nu_{i}}}\ar[r]\ar@{-->}[u]_{\cong}&
0,
}$$
in which the two vertical isomorphisms induce a dotted isomorphism between
$+_{i\geq 1} B_{\ov{p}_{\nu_{i}}}$
and 
$\pd_{i\geq 1} B_{\ov{p}_{\nu_{i}}}$.
\end{proof}

We exhibit now examples of balanced perverse cochain complexes.

 \begin{proposition}\label{prop:formsandaxioms}
 Let $F$ be a universal system of coefficients and $\ob{K}$ be a \ffs. Then, the perverse cochain complex of global sections,
 $\tF(\ob{K})_{\bullet}$, and its cohomology, $H(\tF(\ob{K})_\bullet)$, are balanced.
\end{proposition}

The proof uses interesting properties of the perverse cochain complex of global sections, that we state independently.

\begin{lemma}\label{lem:formsandcohomology}
Let $F$ be a universal system of coefficients, $\ob{K}$ be a \ffs, $\ov{p}$ be a GM-perversity and $\ov{q}$ be a predecessor of $\ov{p}$.
We denote by $\tm$ the associated peak, i.e., $\ov{p}(i)=\ov{q}(i)$ if $i\neq \tm$ and $\ov{p}(\tm)=\ov{q}(\tm)+1$.
 Then the following properties are satisfied.
 \begin{enumerate}[(i)]
 \item For any $k\leq \ov{p}(\tm)-2$, we have $\tF(\ob{K})^k_{\ov{p}}= \tF(\ob{K})^k_{\ov{q}}$.
 \item The morphism $f^k\colon H^k(\tF(\ob{K})_{\ov{q}})\to H^k(\tF(\ob{K})_{\ov{p}})$,
 induced by the inclusion,
 $\tF(\ob{K})_{\ov{q}}\subset \tF(\ob{K})_{\ov{p}}$,
is surjective if $k\leq \ov{p}(\tm)-1$
and injective if $k\leq \ov{p}(\tm)$.
 \end{enumerate}
\end{lemma}

\begin{proof}
(i) Let $\omega\in \tF(\ob{K})_{\ov{p}}^k$. We consider two cases:
\begin{itemize}
\item if $\ell\neq \tm$, then $\max(\|\omega\|_{\ell}, \| d\omega\|_{\ell})\leq \ov{p}(\ell)= \ov{q}(\ell)$,
\item if $\ell=\tm$, then
$\max(\|\omega\|_{\tm},\|d\omega\|_{\tm})\leq k+1\leq \ov{p}(\tm)-1= \ov{q}(\tm)$.
\end{itemize}
Therefore $\omega\in \tF(\ob{K})_{\ov{q}}^k$. The reverse inclusion is obvious. 

(ii) (a) Let $\omega\in \tF(\ob{K})^{k}_{\ov{p}}$, $d\omega=0$. We consider two cases:
\begin{itemize}
\item if $\ell\neq \tm$, then $\|\omega\|_{\ell}\leq \ov{p}(\ell)=\ov{q}(\ell)$,
\item if $\ell= \tm$, then $\|\omega\|_{\tm}\leq k\leq \ov{p}(\tm)-1 = \ov{q}(\tm)$.
\end{itemize}
This implies $\omega\in \tF(\ob{K})_{\ov{q}}$ and the surjectivity of $f^k$.

(b) Let $\omega\in \tF(\ob{K})^{k}_{\ov{q}}$ and $\alpha\in \tF(\ob{K})^{k-1}_{\ov{p}}$ with $d\alpha=\omega$. We consider two cases:
\begin{itemize}
\item if $\ell\neq \tm$, then $\max(\|\alpha\|_{\ell},\|\omega\|_{\ell})\leq \ov{q}(\ell)$,
\item  if $\ell= \tm$, then $\max(\|\alpha\|_{\tm},\|\omega\|_{\tm})\leq \max(k-1, \ov{q}(\tm))=\ov{q}(\tm)$.
\end{itemize}
This implies $\alpha\in \tF(\ob{K})_{\ov{q}}$ and the injectivity of $f^k$.\\
(Observe that the key point of this proof is the inequality $\|\omega\|\leq |\omega|$.)
\end{proof}

\begin{proof}[Proof of \propref{prop:formsandaxioms}]
(i) First, we verify the properties of \defref{def:balanced} for  the perverse cochain complex $\tF(\ob{K})_\bullet$. 

We begin with Property (a). It is trivially verified for $j=1$. Suppose $j\geq 2$. Let $k\leq \ov{p}(\tn_{\nu_{j}})-j+1$ and $(\omega_{i})_{1\leq i\leq j}\in \oplus_{i=1}^j\tF(\ob{K})_{\ov{p}_{\nu_{i}}}$ such that $\omega_{1}+\cdots+\omega_{j}=0$. 
If $j=2$, we obtain $\omega_{1}=-\omega_{2}\in \tF(\ob{K})_{\ov{p}_{\nu_{1},\nu_{2}}}^k$ and
$(\omega_{1},\omega_{2})=\varphi(\omega_{1})$, as expected.
If $j\geq 3$, from (i) of \lemref{lem:formsandcohomology}, we have $\omega_{i}\in \tF(\ob{K})_{\ov{p}_{\nu_{i},\nu_{j}}}$, for any $i\neq j$.
As $\omega_{j}=-\omega_{1}-\cdots -\omega_{j-1}$, we can write
$(\omega_{1},\ldots,\omega_{j})=\varphi((\omega_{i,\ell}))$,
where $(\omega_{i,\ell})\in \oplus_{1\leq i<\ell\leq j}\tF(\ob{K})^k_{\ov{p}_{\nu_{i},\nu_{\ell}}}$
is defined by $\omega_{i,\ell}=0$ if $\ell\neq j$ and $\omega_{i,j}=\omega_{i}$, (cf. \defref{def:balanced} for the construction of $\varphi$).

For proving (b), let $\omega\in \tF(\ob{K})^0_{\ov{p}}$ and $\ov{p}_{\nu_{j}}$ be a predecessor of $\ov{p}$ with $\ov{p}(\tn_{\nu_{j}})\geq 2$. If $\ell\neq \tn_{\nu_{j}}$, then 
$\max(\|\omega\|_{\ell},\|d\omega\|_{\ell}) \leq \ov{p}(\ell)=\ov{p}_{\nu_{j}}(\ell)$ and
$\max(\|\omega\|_{\tn_{\nu_{j}}},\|d\omega\|_{\tn_{\nu_{j}}})\leq 1\leq \ov{p}(\tn_{\nu_{j}})-1=\ov{p}_{\nu_{j}}(\tn_{\nu_{j}})$. We get $\omega\in \tF(\ob{K})^0_{\ov{p}_{\nu_{j}}}$.

 As for (c), let  $\omega\in \tF(\ob{K})_{\ov{q}}^1$ such that $\omega=df$ with $f\in \tF(\ob{K})^0_{\ov{p}}$ and $\ov{q}\leq \ov{p}$. For any $j$, we  have,
$\max(\|f\|_j,\|df\|_j)\leq \max(0,\ov{q}(j))\leq \ov{q}(j)$, which implies $f \in \tF(\ob{K})_{\ov{q}}^0$ and the injectivity of
$H^1(\tF(\ob{K})_{\ov{q}})\to H^1(\tF(\ob{K})_{\ov{p}})$. The same argument works if $\ov{p}=\ov{\infty}$.

(ii) We study now the perverse cohomology  $H(\tF(\ob{K})_\bullet)$. Property (c) comes from the first case and Property (b) is obvious from the previous result on $\tF(\ob{K})^0$. 
Thus, we are reduced to Property (a), which is a consequence of \lemref{lem:formsandcohomology}-(ii)  for $j=1$. Let $j\geq 2$. From \propref{prop:+point2}, we have to prove that the morphism,
$$f^k\colon H^k(\tF(\ob{K})_{\ov{p}_{\nu_{1}}})\pd\cdots\pd H^k(\tF(\ob{K})_{\ov{p}_{\nu_{j}}})\to H^k(\tF(\ob{K})_{\ov{p}}),$$
which sends $\langle [\omega_{1}],\ldots,[\omega_{j}]\rangle$ on 
$f_{1}([\omega_{1}])+\cdots+f_{j}([\omega_{j}]),$
is injective for any $k\leq \ov{p}(\tn_{\nu_{j}})-j+1$.
 Let 
 $$\Upsilon =\langle [\omega_1], \ldots ,[\omega_j]\rangle
 \in H^{k}(\tF(\ob{K})_{\ov{p}_{\nu_{1}}})\pd \cdots \pd H^{k}(\tF(\ob{K})_{\ov{p}_{\nu_{j}}}),$$ 
 such that $f^k(\Upsilon)=0$. The map
 $f_{i,j}\colon H^{k}(\tF(\ob{K})_{\ov{p}_{\nu_{i},\nu_{j}}})\to H^{k}(\tF(\ob{K})_{\ov{p}_{\nu_{i}}})$ being surjective for any $i\in\{1,\ldots,j-1\}$ (see (ii) of \lemref{lem:formsandcohomology}), we have
 $$
 \Upsilon= \langle f_{1,j}[\omega'_1],\ldots , f_{j-1,j}[\omega'_{j-1}],[\omega_j]\rangle=
\langle 0, \ldots , 0,[\omega'_j]\rangle).
$$
This implies $0=f^k(\Upsilon)=f_{j}([\omega'_{j}])$ and $\Upsilon=0$ because $f_{j}$ is injective (see (ii) of \lemref{lem:formsandcohomology}).
\end{proof}
 
 \begin{example}
Let $\ob{K}$ be a \ffs. The statement of
\lemref{lem:formsandcohomology} contains the fact that  $H^1(\cA(\ob{K})_{\ov{t}})=0$ implies $H^1(\cA(\ob{K})_{\ov{p}})=0$, for any GM-perversity $\ov{p}$. This example shows that the reverse way is not true.

We choose $n=3$ and  the \ffs, $\ob{K}$, associated to the cone on the torus, $c(S^1\times S^1)$, stratified by the cone point.
As a cone is a contractible space, from \propref{prop:casonulo},
we have
$H^0(\cA(\ob{K})_{\ov{0}})=H^0(\ob{K};\Q)=\Q$ and $H^1(\cA(\ob{K})_{\ov{0}})=H^1(\ob{K};\Q)=0$. 
The \ffs~$\ob{K}$ has one singular stratum  and a perversity consists of the integer $\ov{p}(3)$. There are only two possibilities, $\ov{p}(3)=0$ or 1, the value 1 corresponding to  $\ov{t}$.
From \exemref{exam:cone}, we have,
 $$H^i_{\ov{t}}(\ob{K};\Q)=\left\{\begin{array}{lcl}
H^i(S^1\times S^1;\Q)&\text{ if }& i\leq 1,\\
0&\text{ if }&i>1.
\end{array}\right.$$
Thus we have got $H^1(\cA(\ob{K})_{\ov{0}})=0$ and $H^1(\cA(\ob{K})_{\ov{t}})=\Q\oplus\Q$.
\end{example}

\section{Minimal models of balanced perverse \cdga's}\label{sec:minimalalgebraic}

\begin{quote}
In this section, we construct a Sullivan minimal model, in the sense of \defref{def:minimalmodel}, of any cohomologically connected, balanced perverse \cdga's. Geometrical applications are given in the next section.
\end{quote}

As in the classical case, we need some connectivity hypothesis for the construction of a model.

\begin{definition}\label{def:cohomologicallyconnectedandregular}
A perverse \cdga, $A_{\bullet}$,
\index{Perverse!CDGA!cohomologically connected}
is \emph{cohomologically connected} if $H^0(A_{\bullet})=\Q_{\bullet}$.
\end{definition}

\begin{theorem}[Construction of a minimal model]\label{thm:constructionminimalmodel}
Let $A_{\bullet}$ be a cohomologically connected, balanced perverse \cdga.
\index{Sullivan!minimal model!of a cohomologically connected balanced CDGA}
Then, there exists a Sullivan minimal model of $A_{\bullet}$,
$$\rho_\bullet\colon B_\bullet=(\land \oplus_{\ov{p}} V_{[\ov{p}]},d)_\bullet\to A_{\bullet},$$
i.e., $\rho_{\ov{p}}(B_{\ov{p}})\subset A_{\ov{p}}$, the restriction
$\rho_{\ov{p}}\colon B_{\ov{p}}\to A_{\ov{p}}$ is a quasi-isomorphism for any $\ov{p}\in\hat{\cP}^n$ and the elements of $V_{[\ov{p}]}$ have a strictly positive degree.
This model is unique up to isomorphism.
\end{theorem}

The construction is done by induction on the degree and on the perverse degree. We first  establish some properties of an inductive step, described in the next statement.

\begin{lemma}\label{lem:etape1model}
Let $\ov{p}$ be a GM-perversity of 
ordered set of
predecessors, $(\ov{p}_{1},\ldots,\ov{p}_{s})$. Let  $A_{\bullet}$ be a cohomologically connected, balanced perverse \cdga~and
$$\rho_{\bullet}\colon B_{\bullet}=(\land \oplus_{\ov{r}<\ov{p}}V_{[\ov{r}]},d)\to  A_{\bullet},$$
be a morphism of perverse \cdga's, such that 
the restriction $\rho_{\ov{r}}\colon B_{\ov{r}}\to  A_{\ov{r}}$ is a quasi-isomorphism 
and $V^0_{[\ov{r}]}=0$,
 for any $\ov{r}<\ov{p}$. Then,  the following properties are satisfied.
\begin{enumerate}[(i)]
\item $B^1_{\ov{p}}=B^1_{\ov{p}_1}+\cdots+B^1_{\ov{p}_s}$.
\item Let $(\ov{r}_{\nu_{1}},\ldots,\ov{r}_{\nu_{j}})$ be any ordered set of $j$ predecessors of a GM-perversity $\ov{r}$,  with $\ov{r}\leq \ov{p}$. We denote by $\tn_{\nu_{i}}$ the peak associated to $\ov{r}_{\nu_{i}}$.
Then, the map $\tilde{\rho}\colon B_{\ov{r}_{\nu_{1}}}+\cdots+B_{\ov{r}_{\nu_{j}}}\to A_{\ov{r}_{\nu_{1}}}\pd\cdots\pd A_{\ov{r}_{\nu_{j}}}$, generated by $\rho$, is such that $H^k(\tilde{\rho})$  is an isomorphism for $k\leq \ov{r}(\tn_{\nu_{j}})-j+1$. As a direct consequence, the map $H^1(\tilde{\rho})\colon H^1(B_{\ov{r}_{\nu_{1}}}+\cdots+B_{\ov{r}_{\nu_{j}}})\to
H^1(A_{\ov{r}_{\nu_{1}}}\pd\cdots\pd A_{\ov{r}_{\nu_{j}}})$
is an isomorphism for any $j$.
\item The map $H^1(\rho_{\ov{p}})\colon H^1(B_{\ov{p}})\to H^1( A_{\ov{p}})$ is injective.
\end{enumerate}
\end{lemma}

\begin{proof}
(i) The inclusion
$B^1_{\ov{p}_1}+\cdots+B^1_{\ov{p}_s}\subset B^1_{\ov{p}}$ being obvious, we consider $\omega\in B^1_{\ov{p}}$. 
This element can be written as a sum $\omega=\sum_i\lambda_i\,\omega_i$, with $\lambda_i$ of degree~0 and 
$\omega_i\in \oplus_{\ov{r}<\ov{p}}V^1_{[\ov{r}]}$. 
As $V^0_{[\ov{r}]}=0$ and $V^1_{[\ov{p}]}=0$, we have $\lambda_i\in \Q$ and $\omega_i\in B_{\ov{p}_i}$ for a certain predecessor, $\ov{p}_i$, of $\ov{p}$. Regrouping these elements, we write $\omega$ as a sum of elements of $(B_{\ov{p}_i})_{1\leq i\leq s}$ and the property (i) is proved.

(ii) Observe  that this property is true if $j=1$, by hypothesis. 
Suppose now that (ii) is satisfied for any ordered sequence of $(j-1)$ predecessors of any GM-perversity less than or equal to $\ov{p}$. We  consider now an ordered family of $j$ predecessors, $(\ov{r}_{\nu_{1}},\ldots,\ov{r}_{\nu_{j}})$, of a GM-perversity $\ov{r}$, with $\ov{r}\leq \ov{p}$. With \corref{cor:+point2}, applied to the bottom line, and \propref{prop:capandgraded}, applied to the upper line, we  have a morphism of short exact sequences, whose vertical maps are induced by $\rho$:
$${\scriptsize \xymatrix@=14pt
{
 0\ar[r]&\ar[d]
 B^*_{\ov{r}_{\nu_{1},\nu_{2}}}+ \cdots+B^*_{\ov{r}_{\nu_{1},\nu_{j}}}\ar[r]&
\ar[d]
 B^*_{\ov{r}_{\nu_{1}}}\oplus (B^*_{\ov{r}_{\nu_{2}}}+\cdots+B^*_{\ov{r}_{\nu_{j}}})\ar[r]&
\ar[d]
B^*_{\ov{r}_{\nu_{1}}}+\cdots+B^*_{\ov{r}_{\nu_{j}}}\ar[r]&0\\
 0\ar[r]&A^*_{\ov{r}_{\nu_{1},\nu_{2}}}\pd \cdots \pd A^*_{\ov{r}_{\nu_{1},\nu_{j}}}\ar[r]&
A^*_{\ov{r}_{\nu_{1}}}\oplus (A^*_{\ov{r}_{\nu_{2}}}\pd \cdots\pd A^*_{\ov{r}_{\nu_{j}}})\ar[r]&
A^*_{\ov{r}_{\nu_{1}}}\pd\cdots\pd A^*_{\ov{r}_{\nu_{j}}}\ar[r]&0
}}$$
for $*\leq \ov{r}(\tn_{\nu_{j}})-j+2$. This morphism induces a morphism between long exact sequences and the induction hypothesis, associated to the five lemma, gives the result.

(iii) For getting the injectivity of
$H^1(\rho_{\ov{p}})\colon H^1(B_{\ov{p}})\to H^1( A_{\ov{p}})$, we decompose it in two maps,
$$\xymatrix@C=1.2cm{
H^1(B_{\ov{p}})=H^1(B_{\ov{p}_1}+\cdots+B_{\ov{p}_s})\ar[r]^-{H^1(\tilde{\rho})}&
H^1(A_{\ov{p}_1}\pd\cdots\pd A_{\ov{p}_s})\ar[r]^-{H^1(\ov{\psi})}&H^1(A_{\ov{p}}),
}$$
where the left-hand equality comes from (i).
The injectivity of $H^1(\tilde{\rho})$ comes from Property (ii) and the injectivity of $H^1(\ov{\psi})$ has been proved in \corref{cor:+point2}. 
\end{proof}

\begin{proof}[Proof of \thmref{thm:constructionminimalmodel}]
The unicity up to isomorphism is a direct consequence of \propref{prop:unicityminimal} and (2) of \propref{prop:propertiesofhomotopy}.
The construction of this model begins with a classical minimal model $(\land V_{[\ov{0}]}, d)$ of the \cdga~$ A_{\ov{0}}$. 

Let $\ov{p}\in\hat{\cP}^n$ be a fixed perversity and
suppose  we have  a morphism of  perverse \cdga's,  
$$\rho_{\bullet}\colon B_{\bullet}=(\land \oplus_{\ov{r}<\ov{p}}V_{[\ov{r}]},d)_{\bullet}\to  A_{\bullet},$$
satisfying the next properties for any GM-perversity $\ov{r}$, with $\ov{r}<\ov{p}$,
\begin{enumerate}[(i)]
\item $ V^{\leq 0}_{[\ov{r}]}=0$,
\item $\rho(B_{\ov{r}})\subset  A_{\ov{r}}$ and the restriction $\rho_{\ov{r}}\colon B_{\ov{r}}\to  A_{\ov{r}}$ is a quasi-isomorphism.
 \newcounter{enumi_saved}
      \setcounter{enumi_saved}{\value{enumi}}
 \end{enumerate}
 
For having a model for the perversity $\ov{p}$, we have to construct
 $$\rho_{\bullet}\colon B_{\bullet}=(\land \oplus_{\ov{r}\leq \ov{p}}V_{[\ov{r}]},d)_{\bullet}\to  A_{\bullet},$$
 verifying (i) and (ii) for $\ov{r}\leq \ov{p}$. Suppose that we have extended $\rho_{\bullet}$ in a morphism of  perverse \cdga's,  still denoted $\rho_{\bullet}$,
$$\rho_{\bullet}\colon B'_{\bullet}=(B\otimes \land V_{[\ov{p}]}^{<m},d)_{\bullet}=(\land( \oplus_{\ov{r}<\ov{p}}V_{[\ov{r}]}\oplus V^{<m}_{[\ov{p}]}),d)_{\bullet}\to  A_{\bullet},$$
such that
\begin{enumerate}[(i)]
  \setcounter{enumi}{\value{enumi_saved}}
\item $V^{\leq 0}_{[\ov{p}]}=0$,
\item $\rho(V^{<m}_{[\ov{p}]})\subset  A_{\ov{p}}$ and the restriction $\rho_{\ov{p}}\colon B'_{\ov{p}}\to  A_{\ov{p}}$ verifies,
\begin{enumerate}[a)]
\item $H^i(\rho_{\ov{p}})$ is an isomorphism for any $i\leq m-1$,
\item $H^i(\rho_{\ov{p}})$ is injective if $i=m$.
\end{enumerate}
\end{enumerate}
If $\ov{r}<\ov{p}$, as $B_{\ov{r}}=B'_{\ov{r}}$,  the map $\rho_{\ov{r}}\colon B'_{\ov{r}}\to A_{\ov{r}}$ is a quasi-isomorphism. Moreover, when this construction is performed for any $m$, by denoting $V_{[\ov{p}]}=\oplus_{m} V^m_{[\ov{p}]}$, we get a morphism,
$\rho_{\bullet}\colon (\land \oplus_{\ov{r}\leq \ov{p}}V_{[\ov{r}]},d)_{\bullet}\to A_{\bullet}$,
satisfying (i) and (ii) for any $\ov{r}\leq \ov{p}$. Thus, we are reduced to properties (iii) and (iv).

In degree~$m=0$, as $V_{[\ov{r}]}^{\leq 0}=0$ for any $\ov{r}\leq \ov{p}$, we have
$(\land \oplus_{\ov{r}\leq \ov{p}}V_{[\ov{r}]})^0_{\bullet}=\Q_{\bullet}$
and the map
$\rho_{\bullet}\colon \Q_{\bullet}\to A_{\bullet}$
is the canonical inclusion. The perverse \cdga, $A_{\bullet}$, being cohomologically connected by hypothesis, $H^0(\rho_{\bullet})_{\ov{r}}$ is an isomorphism, for any $\ov{r}\leq\ov{p}$.

Observe that \lemref{lem:etape1model} implies properties  (iii) and (iv) for $m=1$.
We show now that the morphism $\rho_{\bullet}$, defined on $B'_{\bullet}$, can be extended in a morphism of  perverse \cdga's, $\rho_{\bullet}\colon (B'\otimes \land V^m_{[\ov{p}]},d)_{\bullet}\to A_{\bullet}$, so  that properties (iii) and (iv) are satisfied for $m+1$.
We set
$$\left\{\begin{array}{lcl}
Y^m_{[\ov{p}]}[0]&=&\coker H^m(\rho_{\ov{p}})\colon H^m(B'_{\ov{p}})\to H^m(A_{\ov{p}}),\\
Z^m_{[\ov{p}]}[0]&=&\ker H^{m+1}(\rho_{\ov{p}})\colon H^{m+1}(B'_{\ov{p}})\to H^{m+1}(A_{\ov{p}}),\\
V^m_{[\ov{p}]}[0]&=&Y^m_{\ov{p}}[0]\oplus Z^m_{\ov{p}}[0].
\end{array}\right.$$
We extend the differential $d$ and the morphism $\rho_{\bullet}$ by:
\begin{itemize}
\item if $y\in Y^m_{[\ov{p}]}[0]$, then $dy=0$ and $\rho(y)\in A^m_{\ov{p}}$ is a cocycle representing the class corresponding to  $y\in\coker H^m(\rho_{\ov{p}})$,
\item if $z\in Z^m_{[\ov{p}]}[0]$, then $dz\in B'_{\ov{p}}$ is a cocycle representing the class corresponding to 
$z\in\ker H^{m+1}(\rho_{\ov{p}})$. 
The element $\rho(z)\in A_{\ov{p}}^m$ is defined by the equality $ \rho\circ d=d\circ \rho $.
\end{itemize}
  By construction, we have
  $$d(V^m_{[\ov{p}]}[0])\subset B'_{\ov{p}}=(B\otimes\land V^{< m}_{[\ov{p}]})_{\ov{p}} \text{ and }
  \rho(V^m_{[\ov{p}]}[0])\subset A_{\ov{p}}.$$
  Let $j$ be a fixed integer. Suppose we have defined, for all $i\in\{0,\ldots,j\}$, $V^m_{[\ov{p}]}[i]$, $\rho$ and $d$ on 
  $B'\otimes\land \oplus_{i\leq j}V^{m}_{\ov{p}}[i]$ so that, for all $k\leq j$,
  $$\left\{\begin{array}{l}
  d(V^m_{[\ov{p}]}[k])\in (B'\otimes\land \oplus_{i\leq k-1}V^{m}_{[\ov{p}]}[i])_{\ov{p}},\\
  \rho\circ d=d\circ \rho \text{ and } \rho(V^m_{[\ov{p}]}[k])\subset A_{\ov{p}}.
    \end{array}\right.$$
    We set
\begin{equation}\label{equa:vm+1}
V^m_{[\ov{p}]}[j+1]=\ker H^{m+1}(\rho_{\ov{p}})\colon H^{m+1}(B'\otimes \land (\oplus _{i\leq j}V^m_{[\ov{p}]}[i]))_{\ov{p}}\to H^{m+1}( A_{\ov{p}}).
\end{equation}
If $z\in V^m_{[\ov{p}]}[j+1]$, then 
$dz \in (B'\otimes \land(\oplus_{i\leq j} V^m_{[\ov{p}]}[i]))^{m+1}_{\ov{p}}$ is a cocycle
corresponding to $z$ in
$\ker H^{m+1}(\rho_{\ov{p}})$. The element $\rho(z)\in A^m_{\ov{p}}$ is defined by the equality $\rho\circ d=d\circ\rho$.
 Finally, we set
    $$V^m_{[\ov{p}]}=\oplus_{i\geq 0}V^m_{[\ov{p}]}[i].$$
  We verify now  properties  (iii), (iv)
  for $\rho_{\bullet}\colon (B'\otimes \land V^m_{[\ov{p}]},d)_{\bullet}\to A_{\bullet}$.
  Property (iii) is satisfied by construction and we are reduced to Property (iv). %
  \begin{itemize}
  \item If $i\leq m-1$, then $(B'\otimes\land V^{m}_{[\ov{p}]})^i_{\ov{p}}={B'}^i_{\ov{p}}$ and \emph{$H^i(\rho_{\ov{p}})$ is an isomorphism.}
  \item \emph{The morphism $H^m(\rho_{\ov{p}})$ is surjective} by construction of $Y^m_{\ov{p}}[0]$.
  \item \emph{The morphism $H^m(\rho_{\ov{p}})$ is injective.} To prove that, we consider $\omega\in(B'\otimes\land V^m_{[\ov{p}]})^m_{\ov{p}}$ such that $d\omega=0$ and $\rho(\omega)$ is a boundary in $A^m_{\ov{p}}$. There exists $j\geq -1$ with $\omega \in B'\otimes \land \oplus _{i\leq j+1}V^m_{[\ov{p}]}[i]$. 
  We set $\hat{B}=B'\otimes \land \oplus _{i\leq j}V^m_{[\ov{p}]}[i]$ and decompose $\omega$ in 
  $$\omega=\omega_1+\omega_2\in ({B'}^0\otimes V^m_{[\ov{p}]}[j+1])\oplus \hat{B}^m.$$
  As, by induction, ${B'}^0=\Q$, we have 
  $\omega_1\in V^m_{[\ov{p}]}[j+1]$. The equality  $d\omega=0$ implies
  $d\omega_1\in d(\hat{B}^m)$. By construction of $V^m_{[\ov{p}]}[j+1]$, see (\ref{equa:vm+1}),  this implies $\omega_1=0$ and
  $\omega\in B'\otimes \land \oplus _{i\leq j}V^m_{[\ov{p}]}[i]$. 
  We iterate this process and we obtain 
 $\omega\in B'\otimes \land V^m_{[\ov{p}]}[0]$. As $\rho([\omega])=0$, we have
  $\omega\in B'\otimes \land Z^m_{[\ov{p}]}[0]$ that we decompose, as above, in
  $\omega=\omega_{1}+\omega_{2}\in Z^m_{[\ov{p}]}[0]\oplus \hat{B}^m$. The equality
  $d\omega=0$ implies 
  $d\omega_{1}\in d(\hat{B}^m)$
  and, by construction of 
  $Z^m_{[\ov{p}]}[0]$,
  we get $\omega_{1}=0$ and  $\omega\in B'_{\ov{p}}$. 
  Now  hypothesis (iv) b) implies the nullity in $\hat{B}$ of the cohomology class associated to $\omega$.
  \item \emph{The morphism $H^{m+1}(\rho_{\ov{p}})$ is injective.}
  Let $\omega\in(B'\otimes\land V^m_{[\ov{p}]})^{m+1}_{\ov{p}}$ such that  $d\omega=0$ and $\rho(\omega)$ is a boundary in $A_{\ov{p}}^{m+1}$. There exists an integer
  $j\geq -1$ with $\omega \in (B'\otimes \land \oplus _{i\leq j+1}V^m_{[\ov{p}]}[i])^{m+1}_{\ov{p}}$ and, by construction of
  $V^m_{[\ov{p}]}[j+2]$, there exists also
an element $v\in V^m_{[\ov{p}]}[j+2]$ such that
  $\omega-d(1\otimes v)$ is a boundary in $B'\otimes\land\oplus_{i\leq j+1}V^m_{[\ov{p}]}[i]$. Thus $\omega$ is a boundary in $B'\otimes \land V^m_{[p]}$.
  \end{itemize}
  Set $V^m_{[\ov{p}]}(i)=\oplus_{j\leq i}V^m_{[\ov{p}]}[i]$. By construction, the differential, $d$, verifies, for all $i\geq 0$,
  $$d V^m_{[\ov{p}]}[i]\subset 
  \left((\land\oplus_{\ov{r}<\ov{p}}V_{[\ov{r}]})
  \otimes
  (\land V_{[\ov{p}]}^{<m})
  \otimes
  (\land V^m_{[\ov{p}]}(i-1))\right)_{\ov{p}}.$$
 This is Property (\ref{equa:minimalinduction})   and we have constructed a Sullivan minimal model in the sense of 
  \defref{def:algebredeSullivan}.
     \end{proof}

 \begin{example}\label{exam:modelwithnegativedegrees}
 Consider a lattice of perversities containing
 $$\xymatrix{
 &{\ov{p}_1}&\\
{\ov{p}_2}\ar[ru]&{\ov{p}_3}\ar[u]&{\ov{p}_4}\ar[lu]\\
 {\ov{p}_5}\ar[u]\ar@{-->}[ur]&{\ov{p}_6}\ar[ru]\ar[lu]&{\ov{p}_7}\ar[u]\ar@{-->}[ul]\\
 &{\ov{p}_8}\ar[ru]\ar[u]\ar[lu]&
 }$$
 Suppose that $A_\bullet$ is a perverse \cdga~such that $A_{\ov{p}_8}=\Q\,e$ with $|e|=2$ and $H^+(A_{\ov{p}_i})=0$ if $i<8$. Then the minimal model will contain  elements,
 \begin{itemize}
 \item  of degree 1, $x$ of perversity $\ov{p}_5$, $y$ of perversity $\ov{p}_6$ and $z$ of perversity $\ov{p}_7$ such that
 $dx=dy=dz=e$, 
  \item of degree 0, $f$ of perversity $\ov{p}_2$, $g$ of perversity $\ov{p}_3$, $h$ of perversity $\ov{p}_4$ such that
 $df=x-y$, $dg=x-z$, $dh=y-z$, 
 \item of degree -1, $\alpha$ of perversity $\ov{p}_1$ such that $d\alpha=f-g+h$.
 \end{itemize}
 This example illustrates that we cannot have a model with generators of positive degree for an arbitrary perverse \cdga. 
As $H^1(A_{\ov{p}_6}\pd A_{\ov{p}_7})\to H^1(A_{\ov{p}_4})$
 is not injective, this perverse \cdga~is not balanced, see \corref{cor:+point2}(ii). Moreover, it cannot be the perverse \cdga~of  forms on a \ffs, as shows \propref{prop:formsandaxioms2}.
 \end{example}

 \begin{remark}\label{rem:+ou+point}
 In general, the operations $+$ and $\pd$, applied to the perverse \cdga~of forms on a \ffs, $\ob{K}$, are different.
  More precisely, we have $\cA(\ob{K})_{\ov{p}_1}+\cA(\ob{K})_{\ov{p}_2}=\cA(\ob{K})_{\ov{p}_1}\pd \cA(\ob{K})_{\ov{p}_2}$ but for a sum of three terms or more these operations do not coincide. For instance, in the case of three terms, they correspond to the next short exact sequences, 
   $$\xymatrix@1@=16pt{
 0\ar[r]&
 A_{\ov{p}_1}\cap (A_{\ov{p}_2}+A_{\ov{p}_3})\ar[r]&
 A_{\ov{p}_1}\oplus(A_{\ov{p}_2}+A_{\ov{p}_3})\ar[r]&
 A_{\ov{p}_1}+ A_{\ov{p}_2}+ A_{\ov{p}_3}\ar[r]&
 0}$$
 and
  $$\xymatrix@1@=12pt{
 0\ar[r]&
 (A_{\ov{p}_1}\cap A_{\ov{p}_2})+(A_{\ov{p}_1}\cap A_{\ov{p}_3})\ar[r]&
 A_{\ov{p}_1}\oplus(A_{\ov{p}_2}+A_{\ov{p}_3})\ar[r]&
 A_{\ov{p}_1}\pd A_{\ov{p}_2}\pd A_{\ov{p}_3}\ar[r]&
 0,}$$
 where we have denoted $\cA(\ob{K})$ by $A$.
But the equality
 $A_{\ov{p}_1}\cap (A_{\ov{p}_2}+A_{\ov{p}_3})= (A_{\ov{p}_1}\cap A_{\ov{p}_2})+(A_{\ov{p}_1}\cap A_{\ov{p}_3})$
 is true only in degree less than, or equal to, $\ov{p}_{3}(\tn_{3})-1$. This can be proved by arguments similar to those in the proof of \lemref{lem:formsandcohomology}; we do not need this result in the sequel.
 \end{remark}

\section{Minimal models of filtered spaces}\label{sec:modelCS}

\begin{quote}
The construction of a Sullivan minimal model can be done in the cases of the algebra of perverse forms on a connected \ffs, $\ob{K}$, and of its cohomology. Moreover, in the case $\ob{K}$ is coming from a pseudomanifold, we prove that the model we have built is a topological invariant. 
\end{quote}

Recall, from \defref{def:ffsconected}, that a \ffs, $\ob{K}$, is connected if the face set $\ob{K}^{[0]}$ is connected. For a GM-perversity $\ov{p}$, it is easy to see (cf. \remref{rem:petitemaisbien}) that the only $\ov{p}$-admissible 0-simplices and 1-simplices are elements of $K^{[0]}$, thus, if $\ob{K}$ is connected, the
perverse \cdga, $\cA(\ob{K})_\bullet= \widetilde{A}_{PL,\bullet}(\ob{K})$, is cohomologically connected. The next result is a direct consequence of \propref{prop:formsandaxioms}.

\begin{proposition}\label{prop:formsandaxioms2}
The perverse \cdga~of forms on a \ffs, $\ob{K}$ and their cohomology are balanced. Therefore, if $\ob{K}$ is connected, the perverse \cdga's, $\cA(\ob{K})_\bullet$,  and $H(\cA(\ob{K})_\bullet)$ with the trivial differential, admit Sullivan minimal models.
\end{proposition}

The next definition arises naturally from \propref{prop:formsandaxioms2}.

\begin{definition}\label{def:minimalmodelffs}
If $\ob{K}$ is a connected \ffs, 
\index{Sullivan!minimal model}\index{Sullivan!minimal model!of a connected filtered face set}
the Sullivan minimal model of $\cA(\ob{K})_\bullet$ is called the \emph{minimal model of $\ob{K}$.}
\end{definition}

In perverse degree $\ov{q}=0$,  the Sullivan minimal model of a normal \ffs, $\ob{K}$, is the minimal model of the associated face set. 

\begin{proposition}\label{prop:casonuloproducto}
Let $\ob{K}$ be a connected normal \ffs, of minimal model 
$(\land \oplus_{\ov{p}\in\hat\cP^n}V_{[\ov{p}]},d)_{\bullet}$.
Then $(\land V_{[\ov{0}]},d)$ is a minimal model of the face set associated to $\ob{K}$.
\end{proposition}

\begin{proof}
For that, we need a version of \propref{prop:casonulo} taking in account the structure of algebras. Let
$\Delta=\Delta^{j_{0}}\ast\cdots\ast\Delta^{j_{n}}$ be a filtered simplex and $\mu\colon \tDelta=
c\Delta^{j_{0}}\times\cdots\times c\Delta^{j_{n-1}}\times\Delta^{j_{n}}\to \Delta$ its blow-up. We work with the coordinates introduced in \remref{rem:blowup}, where $\mu$ sends the element
$((a^0_{0},\ldots,a^0_{j_{0}},t_{0}),(a^1_{0},\ldots,a^1_{j_{1}},t_{1}),\ldots, (a^n_{0},\ldots,a^n_{j_{n}}))$
on
$(a^0_{0},\ldots,a^0_{j_{0}}, t_{0}a^1_{0},\ldots,t_{0}a^1_{j_{1}},t_{0}t_{1}a^2_{0},\ldots,t_{0}t_{1}a^2_{j_{2}},\ldots,
t_{0}\ldots t_{n-1}a^n_{0},\ldots, t_{0}\ldots t_{n-1}a^n_{j_{n}})$.

Denote by $(u_{i})_{0\leq i\leq j_{0}+\cdots+j_{n}+n}$ the barycentric coordinates of $\Delta$ and by
$(v^k_{i})_{0\leq i\leq j_{k}}$ the barycentric coordinates of $\Delta^{j_{k}}$. The map $\mu$ induces a map
$\mu^*\colon A_{PL}(\Delta)\to \tA(\Delta)=A_{PL}(c\Delta^{j_{0}})\otimes\cdots\otimes A_{PL}(\Delta^{j_{n}})$, defined by 
$\mu^*(u_{i})=$\\
$$\left\{
\begin{array}{ll}
v_{i}^0\otimes 1\otimes\cdots\otimes 1
&\text{if } 0\leq i\leq j_{0}\\
t_{0}\otimes v_{i-j_{0}-1}^1
\otimes 1\otimes\cdots\otimes 1
&\text{if } j_{0}+1\leq i\leq j_{0}+j_{1}+1\\
t_{0}\otimes t_{1}\otimes v_{i-j_{0}-j_{1}-2}^2
\otimes 1\otimes\cdots\otimes 1
&\text{if } j_{0}+j_{1}+2\leq i\leq j_{0}+j_{1}+j_{2}+2\\
\cdots&\cdots\\
t_{0}\otimes t_{1}\otimes\cdots\otimes t_{n-1}\otimes  v_{i-j_{0}-\ldots - j_{n-1}-n}^n
&\text{if } j_{0}+\cdots+j_{n-1}+n\leq i\leq j_{0}+j_{1}+\cdots +j_{n}+n.
\end{array}\right.
$$
The faces in $\tDelta$ having a factor $\Delta^{j_{n-\ell}}\times\{1\}$, for $\ell\in \{1,\ldots,n\}$, are characterized by $t_{\ell}=0$. Thus, the forms
$\mu^*(u_{i})$
are of perverse degree 0 and we have
$\mu^*(u_{i}) \in \widetilde{A}_{PL,\ov{0}}(\Delta)$. This local construction generates a morphism of CDGA's,
$A_{PL}(\ob{K})\to \widetilde{A}_{PL,\ov{0}}(\ob{K})$, which satisfies the hypothesis of \thmref{thm:extendable} for any connected normal \ffs, $\ob{K}$. Thus, in the next diagram,
$$\xymatrix{
A_{PL}(\ob{K})\ar[r]^{\mu^*}&\widetilde{A}_{PL,\ov{0}}(\ob{K})\\
&(\land V_{[\ov{0}]},d)\ar[u]_{\rho_{\ov{0}}}\ar@{-->}[ul]^{\psi},
}$$
the morphism $\mu^*$ is a quasi-isomorphism (\thmref{thm:extendable}). Let $\rho_{\ov{0}}$ be the minimal model of
$\widetilde{A}_{PL,\ov{0}}(\ob{K})$. As $(\land V_{[\ov{0}]},d)$ is a cofibrant \cdga, there exists (see \propref{prop:propertiesofhomotopy}-(2)) a morphism of \cdga's, $\psi$, making the diagram commutative up to homotopy, and $\psi$ is a quasi-isomorphism.
\end{proof}

We may also compare the structure of products induced, in cohomology, on the blow-ups of Sullivan's forms and on the Thom-Whitney cochains, on one side, together with the product defined by G.~Friedman and J.~McClure (\cite{MR3046315}) in the topological case.

\begin{proposition}\label{prop:producto}
Let $\ob{K}$ be a  \ffs. Then, the integration map, $\int\colon \widetilde{A}_{PL}(\ob{K})\to \widetilde{C}^*(\ob{K})$,
induces an isomorphism of perverse algebras,
$$H^*_{\ov{\bullet}}(\ob{K};\widetilde{A}_{PL})\cong H^*_{\ov{\bullet}}(\ob{K};\widetilde{C}).$$

If $X$ is an $n$-dimensional pseudomanifold and $\ob{K}=\ob{\rm ISing}_*^{\cF}(X)$, the  cup-product of the previous perverse cohomology algebras coincides with the cup-product defined in \cite{MR3046315}.
\end{proposition}

\begin{proof}
We use an argument due to C.~Watkiss (unpublished) in the classical case of Sullivan's theory.
The two inclusions of universal systems of DGA's,
$$\xymatrix@1{
A_{PL}(\Delta^n)\ar[r]&(A_{PL}\otimes C^*)(\Delta^n)&C^*(\Delta^n),\ar[l]
}$$
satisfy the hypotheses of \thmref{thm:extendable} and induce an isomorphism of perverse \cdga's, 
$$H^*_{\ov{\bullet}}(\ob{K};\tA)\cong H^*_{\ov{\bullet}}(\ob{K};\widetilde{A_{PL}\otimes C}) \cong H^*_{\ov{\bullet}}(\ob{K};\tC).$$
This establishes the first assertion. The last part of the statement is proved in \cite[End of Section~4]{2013arXiv1302.2737D}.
\end{proof}

To any filtered space $X$ (\defref{def:espacefiltré}) whose regular part is connected, we  associate a \ffs, 
$\ob{L}=\ob{\rm ISing}_*^{\cF}(X)$, 
such that $\cA(\ob{L})_\bullet$ is a cohomologically connected, balanced perverse \cdga, see \propref{prop:formsandaxioms}.
Thus, by \thmref{thm:constructionminimalmodel}, a perverse minimal model of $\cA(\ob{L})_\bullet$ exists. \emph{We call it the perverse minimal model of $X$.} 
\index{Sullivan!minimal model!of a connected filtered space}
In the case $X$ is normal, the $\ov{0}$-part of this model is the minimal model of $X$, as shows \propref{prop:casonuloproducto}.

\smallskip
The next result is the topological invariance of the minimal model of $\ob{\rm ISing}_*^{\cF}(X)$. 

\begin{theorem}\label{thm:topologicalinvarianceofminimal}
Let $X$ be a PL-pseudomanifold  whose regular part is connected. Then,  the minimal model of $X$
does not depend on the stratification of $X$, in GM-perversity degrees  strictly less than~$\ov{\infty}$.
\end{theorem}

As the $\ov{\infty}$-intersection cohomology of a \ffs, $\ob{K}$, is the ordinary cohomology of its regular part
(\propref{prop:casoinfinito}), we cannot expect any topological invariance for the perversity $\ov{\infty}$. The simple example of a sphere $S^2$ with strata of codimension~0 shows that the regular part can have different homotopy types.
Nevertheless, the knowledge of a model of the regular part of a PL-pseudomanifold reveals crucial for some computation, see \thmref{thm:nodalformal}.

\begin{proof}[Proof of \thmref{thm:topologicalinvarianceofminimal}]
In \cite[Page~150]{MR800845}, King associates to any CS set, $X$, a CS set, $X^*$, which is an intrinsic coarsest stratification of $X$. In particular, the identity map induces a stratified map, $X\to X^*$, in the sense of \defref{def:appstratifiee}.  In \cite[Chapter 2]{IHGreg}, G. Friedman proves that the CS set, $X^*$, associated to a PL-pseudomanifold is also a PL-pseudomanifold.  We begin by proving that the regular part of $X^*$ is connected.

Let $R=X\backslash X_{n-1}$
and
$R^*=X^*\backslash X^*_{n-1}$ be the regular parts of $X$ and $X^*$, respectively. By construction, we have $R\subset R^*$. If $R^*$ is the union of two disjoint open sets, $R^*=U_1 \sqcup U_2$, then we have
$R=(R\cap U_1) \sqcup (R\cap U_2)$. As $R$ is connected, by hypothesis, and open, we must have
$U_1\cap R=\emptyset$ or $U_2\cap R=\emptyset$. The open set $R$ being dense in $X$, we obtain $U_1=\emptyset$ or $U_2=\emptyset$, which implies the connectivity of $R^*$.

We denote by $\ob{L}$ and $\ob{L}^*$ the \ffss~associated to $X$ and $X^*$, respectively, and consider the following diagram whose elements are detailed below,
$$\xymatrix{
\widetilde{A}_{PL,\ov{q}}(\ob{L}^*)
\ar@{-->}[rr]^-{\varphi_1}\ar[d]_{f_{\ob{L}^*}}
&&
\widetilde{A}_{PL,\ov{q}}(\ob{L})
\ar[d]^{f_{\ob{L}}}
\\
\tC^*_{\ov{q}}(\ob{L}^*)
\ar[d]_{g_{\ob{L}^*}}
&&
\tC^*_{\ov{q}}(\ob{L})
\ar[d]^{g_{\ob{L}}}
\\
C^{*}_{\GM,\ov{p}} (\ob{L}^*)=\hom(C_*^{\ov{p}}(X^*),\Q)
\ar[rr]^-{\varphi_2}
&&
C^{*}_{\GM,\ov{p}} (\ob{L})=\hom(C_*^{\ov{p}}(X),\Q)
\\
\hom(K_*^{\ov{p}}({X^*}),\Q)\ar[rr]^-{\varphi_3}
\ar[u]^{h_{\ob{L}^*}}
&&
\hom(K_*^{\ov{p}}({X}),\Q)
\ar[u]_{h_{\ob{L}}}
}$$

The perversities $\ov{p}$ and $\ov{q}$ are elements of $\cP^n$ such that  $\ov{p}(k)+\ov{q}(k)=k-2$, if $k\geq 2$. 
The bottom square commutes by construction and the vertical  maps are quasi-isomorphisms, as this is proved in
\begin{itemize}
\item \corref{cor:quasiisossurQ} for $f_{\ob{L}}$,  $f_{\ob{L}^*}$,  $g_{\ob{L}^*}$ and $g_{\ob{L}}$,
\item and, as $X^*$ is a PL-pseudomanifold, we can use \propref{prop:intersectionetintersection} for $h_{\ob{L}^*}$ and $h_{\ob{L}}$. 
\end{itemize}
The proof that $\varphi_3$ is a quasi-isomorphism being done in \cite[Theorem 9]{MR800845}, the proof of \thmref{thm:topologicalinvarianceofminimal} is reduced to the construction of $\varphi_1$ making the diagram commutative.

From \thmref{thm:stratifiedmapcomplex} (see also \corref{cor:amalgamation2}), we know that the stratified map,
$X\to X^*$, induces an amalgamation of the singular simplices, generated by elementary amalgamations of the type
$\Delta^p\ast\Delta^q\mapsto \emptyset \ast \Delta^{p+q+1}$. 
To prove that these amalgamations induce a morphism of perverse \cdga's, 
$\varphi_1\colon \widetilde{A}_{PL,\bullet}(\ob{L}^*)\to
\widetilde{A}_{PL,\bullet}(\ob{L})$,
making commutative the previous diagram, it is sufficient  to show it locally. This is done in \lemref{lem:amalgamation}.
\end{proof}

\begin{lemma}\label{lem:amalgamation}
Let $i\in\{0,\ldots,n-1\}$. We denote by $\Phi_i$ the identity map on a simplex $\Delta$, 
corresponding to the elementary amalgamation
$$\Delta^{j_0}\ast\cdots\ast\Delta^{j_n}
\mapsto
\Delta^{k_0}\ast\cdots\ast \Delta^{k_n},
$$
with \begin{eqnarray*}
k_a=j_a, &\text{ if }& a<i \text{ or } a>i+1,\\
k_i=-1,&\text{ and }&
k_{i+1}=j_i+j_{i+1}{+1}.
\end{eqnarray*}
Let
$\mu\colon \widetilde{\Delta}_1=c\Delta^{j_0}\times \cdots\times \Delta^{j_n}\to \Delta$
and
$\mu'\colon \widetilde{\Delta}_2=c\Delta^{k_0}\times \cdots\times \Delta^{k_n}\to \Delta$
be the two blow-ups.
Then, for all $i\in\{0,\ldots,n-1\}$, the map $\Phi_i$ lifts in a map $\widetilde{\Phi}_i$ between the two blow-ups, compatible with the face operators of $\Delta$ and such that $\mu'\circ \widetilde{\Phi}_i=\mu$. The induced cochain map,
${\widetilde{\Phi}_i^*}$, is such that,
the following diagram commutes,
$$\xymatrix{
A_{PL}(c\Delta^{j_0})\otimes \cdots\otimes A_{PL}(\Delta^{j_n})\ar[rr]^-{\int}
&&
\Q\\
A_{PL}(c\Delta^{k_0})\otimes \cdots\otimes A_{PL}(\Delta^{k_n})\ar[urr]_-{\int}
\ar[u]^{\widetilde{\Phi}_i^*}
&&
}$$
where $\int\omega=\int_{\mathring{\Delta}}(\mu^{-1})^*\omega$ if
$\omega \in A_{PL}(c\Delta^{j_0})\otimes \cdots\otimes A_{PL}(\Delta^{j_n})$
and a similar formula for the second integral map.
Moreover, the map $\widetilde{\Phi}_i^*$ verifies
$$\displaystyle{
\widetilde{\Phi}_i^*(\widetilde{A}_{PL,\ov{q}}(\Delta^{k_0}\ast \cdots\ast  \Delta^{k_n}))\subset
\widetilde{A}_{PL,\ov{q}}(\Delta^{j_0}\ast \cdots\ast  \Delta^{j_n})
},$$
 for all positive perversity $\ov{q}$.
\end{lemma}

\begin{proof}
Recall, from  \remref{rem:blowup}, that the two blow-ups can be described  by
$$\mu((x_0,t_0),\ldots,(x_{n-1},t_{n-1}),x_n)=
x_0+t_0x_1+t_0t_1x_2+\cdots+t_0\ldots t_{n-1}x_n$$
 with $x_i=(x_{i,0},\ldots,x_{i,j_{i}})\in\R^{j_i+1}$, $t_i\in\R$, $t_{i}+\sum_{k=0}^{j_{i}}x_{i,k}=1$,  
 for all $i$, and a similar formula for $\mu'$.
(With this setting, in the particular case $\Delta^{j_{a}}=\emptyset$, we have $c\Delta^{j_{a}} =\{(0,1)\}$.)
For the study of $\Phi_i$, we have two cases, depending if $i+1=n$ or not.

1)  We begin with  $i+1\neq n$.
We construct a map
$\widetilde{\Phi}_i\colon c\Delta^{j_0}\times \cdots\times \Delta^{j_n}
\to
c\Delta^{k_0}\times \cdots\times \Delta^{k_n}$
by $\widetilde{\Phi}_i=\id\times f_i\times \id$, where
$$f_i\colon c\Delta^{j_i}\times c\Delta^{j_{i+1}}\to c\emptyset \times c(\Delta^{j_i}\ast\Delta^{j_{i+1}})$$
is defined  by
$$f_i((x_i,t_i),(x_{i+1},t_{i+1}))=((0,1),((x_i,t_ix_{i+1}),t_it_{i+1})).$$
We check easily from the definition that the map $\widetilde{\Phi}_i$ verifies $\mu'\circ\widetilde{\Phi}_i=\mu$ and  is compatible with face operators. Thus it induces a \cdga's map
$\widetilde{\Phi}_i^*\colon
A_{PL}(c\Delta^{k_0})\otimes \cdots\otimes A_{PL}(\Delta^{k_n})
\to
A_{PL}(c\Delta^{j_0})\otimes \cdots\otimes A_{PL}(\Delta^{j_n})$.

Let $\omega=\omega_0\otimes\cdots\otimes\omega_n\in A_{PL}(c\Delta^{k_0})\otimes\cdots\otimes A_{PL}(\Delta^{k_n})$. In the next equalities, we use the fact that $\mu$ and $\mu'$ are diffeomorphisms on the interior of the integration domains,
\begin{eqnarray*}
\int_{\widetilde{\Delta}_2}\omega
&=&
\int_{\mathring{\Delta}}({\mu'}^{-1})^*\omega=
\int_{\mathring{\Delta}}\left(({\mu}^{-1})^*\circ \widetilde{\Phi}^*_i\right)(\omega)\\
&=&
\int_{\widetilde{\Delta}_1}\widetilde{\Phi}^*_i(\omega).
\end{eqnarray*}
This proves the equality
$\int\circ \,\widetilde{\Phi}^*_i=\int$.
We show now the compatibility with  perversities, i.e.,
$$\|\omega\|_{\ell}\leq \ov{q}(\ell)\Rightarrow \|\widetilde{\Phi}_i^*\omega\|_{\ell}\leq \ov{q}(\ell),$$
for all $\ell\in\{1,\ldots,n\}$. 
If $\Delta^{j_{n-\ell}}=\emptyset$, the previous implication follows from $\|\widetilde{\Phi}_i^*\omega\|_{\ell}=-\infty$. We suppose now that $\Delta^{j_{n-\ell}}\neq\emptyset$.

--- Consider first  $n-\ell=0$, with $i\neq 0$. (The case $i=0$, $\ell=n$ is done below.)
The commutativity of the next diagram,
$$\xymatrix{
\left(\Delta^{j_0}\times \{1\}\right)\times \cdots\times c\Delta^{j_i}\times c\Delta^{j_{i+1}}\times\cdots\times \Delta^{j_n}
\ar[d]^-{\widetilde{\Phi}_i}\ar[r]^-{\rm pr}&\Delta^{j_0}\times\{1\}\\
\left(\Delta^{j_0}\times \{1\}\right)\times \cdots\times c\emptyset\times c(\Delta^{j_i}\ast\Delta^{j_{i+1}})\times\cdots\times\Delta^{j_n},\ar[ru]_-{\rm pr}&
}$$
 implies
$$\|\widetilde{\Phi}_i^*(\omega)\|_{\ell}\leq \|\omega\|_{\ell},$$
since the projections ${\rm pr}$ are used for the determination of $\|-\|_{\ell}$, see \remref{rem:degreealongfibre}.
A similar argument works for all the perverse degrees, except in the cases $n-\ell=i$ and $n-\ell=i+1$. 

--- We continue  with $n-\ell=i+1$. 
In this case, the restriction of  the map $\widetilde{\Phi}_i=\id\times f_i\times \id$,
$$\xymatrix{
c\Delta^{j_0}\cdots\times c\Delta^{j_i}\times\left( \Delta^{j_{i+1}}\times \{1\}\right)\times \cdots\times \Delta^{j_n}
\ar[d]_{\widetilde{\Phi}_i=\id\times f_i\times \id}\\
c\Delta^{j_0}\cdots\times c\emptyset\times \left((\Delta^{j_i}\ast \Delta^{j_{i+1}})\times \{1\}\right)\times \cdots \times\Delta^{j_n},
}$$
is defined by $f_{i}\colon c\Delta^{j_{i}}\times (\Delta^{j_{i+1}}\times\{1\})\to c\emptyset\times \left((\Delta^{j_i}\ast \Delta^{j_{i+1}})\times \{1\}\right)$, which sends the element
$((x,t),(y,0))$ on $((0,1),((x,ty),0))$. 
(Observe that the face $\Delta^{j_i}\times \{1\}$ of $c\Delta^{j_i}$, used in the definition of perverse degree, 
corresponds to $t=0$, see \remref{rem:blowup}.)
We denote by ${\rm pr_1}$ the projection of the domain of  $\widetilde{\Phi}_i$ on
$c\Delta^{j_0}\times\cdots\times c\Delta^{j_i}\times\left( \Delta^{j_{i+1}}\times \{1\}\right)$, which is used for the determination of
$\|\widetilde{\Phi}_i^*(\omega)\|_{\ell}$,
and by ${\rm pr}_2$ the projection of the codomain of $\widetilde{\Phi}_i$ on
$c\Delta^{j_0}\times \cdots\times c\emptyset\times \left((\Delta^{j_i}\ast \Delta^{j_{i+1}})\times \{1\}\right)$, which is used for the determination of 
$\|\omega\|_{\ell}$. The equality ${\rm pr}_2\circ \widetilde{\Phi}_i=(\id\times f_{i})\circ {\rm pr}_1$ 
implies
$$\|\widetilde{\Phi}_i^*(\omega)\|_{\ell}\leq \|\omega\|_{\ell}.$$

--- The last perversity we have to study corresponds to $n-\ell=i$. The restriction of $\widetilde{\Phi}_i$ to
$\widetilde{S_i}=c\Delta^{j_0}\times\cdots\times\left(\Delta^{j_i}\times\{1\}\right)\times c\Delta^{j_{i+1}}\times\cdots\times \Delta^{j_n}$
is defined by
$f_i((x,0),(y,s))=((0,1),((x,0),0))$. Thus we have
\begin{eqnarray*}
\widetilde{\Phi}_i(\widetilde{S_i})&=&
c\Delta^{j_0}\times \cdots\times %
c\emptyset\times \left(\Delta^{j_i}\times\{1\}\right)\times c\Delta^{j_{i+2}}\times\cdots\times \Delta^{j_n}\\
&\subset&
c\Delta^{j_0}\times \cdots\times %
c\emptyset\times (\Delta^{j_i}\ast\Delta^{j_{i+1}}\times\{1\})\times c\Delta^{j_{i+2}}\times\cdots\times \Delta^{j_n}.
\end{eqnarray*}
We denote by ${\rm pr}_1$ the projection of $\widetilde{S}_i$ on 
$c\Delta^{j_0}\times \cdots\times 
c\Delta^{j_{i-1}}\times \left(\Delta^{j_i}\times\{1\}\right)$ and by ${\rm pr}_2$ the projection of
$c\Delta^{j_0}\times \cdots\times 
c\emptyset\times (\Delta^{j_i}\ast\Delta^{j_{i+1}}\times\{1\})\times c\Delta^{j_{i+2}}\times\cdots\times \Delta^{j_n}$ on 
$c\Delta^{j_0}\times \cdots\times 
c\emptyset\times (\Delta^{j_i}\ast\Delta^{j_{i+1}}\times\{1\})$.
The projection ${\rm pr}_1$ is used for the determination of $\|\widetilde{\Phi}_i^*\omega\|_{\ell}$ and the projection ${\rm pr}_2$ for the determination of $\|\omega\|_{\ell-1}$. The equality
$(\id\times h_{i})\circ {\rm pr}_{1}={\rm pr}_2\circ \widetilde{\Phi}_i$,
where $h_{i}\colon \Delta^{j_i}\times\{1\}
\to
c\emptyset\times\left((\Delta^{j_i}\ast\Delta^{j_{i+1}})\times\{1\}\right)$
is defined by
$h_{i}(x,0)=((0,1),((x,0),0)$,
implies
$$\|\widetilde{\Phi}_i^*(\omega)\|_{\ell}\leq \|\omega\|_{\ell-1}\leq \ov{q}(\ell-1)\leq \ov{q}(\ell).$$

\medskip
2) We study now the case $i+1=n$. The map $\widetilde{\Phi}_{n-1}=\id\times f_{n-1}$, with $f_{n-1}\colon c\Delta^{j_{n-1}}\times  \Delta^{j_n}\to c\emptyset\times\left(\Delta^{j_{n-1}}\ast\Delta^{j_n}\right)$, is defined by
$f_{n-1}((x,t),y)=((0,1),(x,ty))$. The proof goes like in the previous case, except for $\ell=1$.  The restriction of $\widetilde{\Phi}_{n-1}$ to
$\widetilde{S}_{n-1}=c\Delta^{j_0}\times\cdots\times
\left(\Delta^{j_{n-1}}\times\{1\}\right)
\times \Delta^{j_n}$
is defined by $f_{n-1}((x,0),y)=((0,1),(x,0))$ and we have
$\widetilde{\Phi}_{n-1}(\widetilde{S}_{n-1})\subset
c\Delta^{j_0}\times\cdots\times c\emptyset
\times (\Delta^{j_{n-1}}\ast\Delta^{j_{n}})$. In the next  diagram, the map $h_{n-1}$ is defined by
$h_{n-1}(x,0)=((0,1),(x,0))$,
$$\xymatrix{
c\Delta^{j_0}\times\cdots\times
\left(\Delta^{j_{n-1}}\times\{1\}\right)
\times \Delta^{j_n}
\ar[r]^-{\rm pr}
\ar[d]_{\widetilde{\Phi}_{n-1}=\id\times f_{n-1}}
&c\Delta^{j_0}\times\cdots\times
\left(\Delta^{j_{n-1}}\times\{1\}\right)
\ar[dl]^{\id\times h_{n-1}}\\
c\Delta^{j_0}\times\cdots\times
c\emptyset\times\left(\Delta^{j_{n-1}}\ast\Delta^{j_n}\right).&
}$$
The commutativity of this diagram implies, for $\ell=1$,
$$\|\widetilde{\Phi}_{n-1}^*(\omega)\|_{\ell}= \|{\rm pr}^*\circ (\id\times h_{n-1})^*\omega\|_{\ell}\leq 0,$$
since the projection ${\rm pr}$ is used for the determination of $\|-\|_{\ell}$.
\end{proof}

\begin{remark}
(In this remark, we keep the notations of the proof of \thmref{thm:topologicalinvarianceofminimal}.)
In this proof, the hypothesis ``$X$ PL-pseudomanifold'' is used only for proving that the maps $h_{L}$ and $h_{L^*}$ are quasi-isomorphisms. We believe that the map 
$\varphi_{2}$
is a quasi-isomorphism under the weaker hypothesis that $X$ is a recursive CS set in the sense of G. Friedman, \cite{IHGreg}. This should need a direct proof, similar to the proof made by G.~Friedman in \cite{IHGreg}. We do not go further in this direction.
\end{remark}

\chapter{Formality and  examples}\label{chap:formalityexamples}

In this chapter, we consider algebraic structures defined on the field of the rational numbers and any vector space is supposed to be a rational vector space.

\section{Intersection-formality}\label{subsec:formality}
\begin{quote}
A notion of formality is defined and examples are given. This opens a framework for a study of the formality of singular projective algebraic varieties, as asked by M.~Goresky in the introduction of \cite{MR761809}.
We use also
 \exemref{exam:cone} and \exemref{exam:suspensionCP2} to construct
 models, not necessarily cofibrant, for the cone and suspension of a face set.
 They provide an explicit example of a non intersection-formal PL-pseudomanifold which is formal, as space.
 \end{quote}
 
When $\ob{K}$ is a \ffs, a  notion of perverse formality comes naturally from the fact that the perverse \cdga~of forms, $\cA(\ob{K})_\bullet$, and its perverse cohomology algebra are belonging to the same category.

\begin{definition}\label{def:ffsformal}
A connected \ffs, $\ob{K}$, is \emph{intersection-formal}
\index{Intersection-formality!of a filtered face set}
 if there is an isomorphism between the minimal models of $\cA(\ob{K})_\bullet$ and $H(\cA(\ob{K})_\bullet)$, for $\bullet<\ov{\infty}$. 
\end{definition}

An equivalent formulation is the existence of a sequence of quasi-isomorphisms in $\cdgaf$, between $\cA(\ob{K})_\bullet$ and its cohomology.

\begin{definition}\label{def:CSformal}
A PL-pseudomanifold, $X$, whose regular part is connected,  
\index{Intersection-formality!of a PL-pseudomanifold}
is \emph{intersec\-tion-formal} if its associated \ffs, $\ob{\rm ISing}_*^{\cF}(X)$, is intersec\-tion-formal.
\end{definition}

\thmref{thm:topologicalinvarianceofminimal} 
implies
immediately the next result.

\begin{proposition}\label{prop:topologicalinvarianceformality}
The intersection-formality of a  PL-pseudomanifold, whose regular part is connected, is independent of the stratification.
\end{proposition}

In the literature, Massey products are sometimes introduced for the detection of non formal spaces. We emphasize that the existence of non trivial Massey products guaranties the non-formality of a space but the reverse is harder to express. Mention, for instance from \cite[Page 262]{MR0382702}, that  the vanishing of each Massey product is not sufficient for getting the formality; for that, this vanishing has to be done in an uniform way. Before the introduction of examples, we adjust the definition of Massey products and their basic properties (see \cite[Definition 2.89 and Proposition 2.90]{MR2403898}) to the perverse frame.

\begin{definition}\label{def:masseyproduct}
Let $(A,d)_{\bullet}$ be a perverse \cdga, $x\in A_{\ov{p}_{1}}$, $y\in A_{\ov{p}_{2}}$, $z\in A_{\ov{p}_{3}}$ be cocycles, of associated cohomology classes, $[x]$, $[y]$, $[z]$, and such that
there exist $\alpha\in A_{\ov{p}_{1}\oplus\ov{p}_{2}}$, $\beta\in A_{\ov{p}_{2}\oplus\ov{p}_{3}}$ with
$d\alpha=xy$ and $d\beta=yz$. 
\index{Massey product}
The element
$\alpha z - (-1)^{|x|}x\beta$ is a cocycle. The \emph{triple Massey product} is the set
$\langle [x],[y],[z]\rangle\subset H_{\ov{p}_{1}\oplus\ov{p}_{2}\oplus \ov{p}_{3}}(A,d)$
formed by all the cohomology classes $[\alpha z - (-1)^{|x|}x\beta]$, constructed using all the possible choices of the elements $\alpha$ and $\beta$. This Massey product is said \emph{trivial} if $0\in \langle [x],[y],[z]\rangle$.
\end{definition}

If we quotient $H_{\bullet}(A,d)$ by the ideal generated by $[x]$ and $[z]$, the set $\langle [x],[y],[z]\rangle$ projects to a single element. Also, this element is zero if, and only if, $\langle [x],[y],[z]\rangle$ is trivial.

\begin{proposition}\label{prop:masseyproduct}
Let $\ob{K}$ be a  connected \ffs, of minimal model,\linebreak $\rho_{\bullet}\colon (\land \oplus_{\ov{p}\in\hat\cP^n}V_{[\ov{p}]},d)_{\bullet}\to \cA(\ob{K})_\bullet$.
If there exists a non trivial triple  Massey product in $(\land \oplus_{\ov{p}\in\cP^n}V_{[\ov{p}]},d)$, then $\ob{K}$ is not intersection-formal.
\end{proposition}

\begin{proof}
Suppose that $(\land \oplus_{\ov{p}\leq\ov{t}}V_{[\ov{p}]},d)$ is also a perverse minimal model of $(\oplus_{\ov{p}\leq\ov{t}}H_{\ov{p}}(\cA(\ob{K})))$, i.e., there is a quasi-isomorphism
$\varphi_{\bullet}\colon (\land \oplus_{\ov{p}\leq\ov{t}}V_{[\ov{p}]},d)_{\bullet}\to
 (\oplus_{\ov{p}\leq\ov{t}}H_{\ov{p}}(\cA(\ob{K})))_{\bullet}$.
Let
$x\in (\land \oplus_{\ov{p}\leq\ov{t}}V_{[\ov{p}]})_{\ov{p}_{1}}$, 
$y\in (\land \oplus_{\ov{p}\leq\ov{t}}V_{[\ov{p}]})_{{\ov{p}_{2}}}$, 
$z\in (\land \oplus_{\ov{p}\leq\ov{t}}V_{[\ov{p}]})_{{\ov{p}_{3}}}$ 
be cocycles such that
there exist 
$\alpha\in (\land \oplus_{\ov{p}\leq\ov{t}}V_{[\ov{p}]})_{\ov{p}_{1}\oplus\ov{p}_{2}}$, 
$\beta\in (\land \oplus_{\ov{p}\leq\ov{t}}V_{[\ov{p}]})_{\ov{p}_{2}\oplus\ov{p}_{3}}$ with
$d\alpha=xy$ and $d\beta=yz$. The cocycle
$\alpha z - (-1)^{|x|}x\beta$ is sent by $\varphi$ to an element of the ideal generated by $\varphi(x)$ and $\varphi(z)$. The morphism $\varphi$ being a quasi-isomorphism of algebras, the element $[\alpha z - (-1)^{|x|}x\beta]$ is in the ideal generated by $[x]$ and $[z]$. Thus, the Massey product 
$\langle [x],[y],[z]\rangle$
 is trivial.
\end{proof}

We give now some easy constructions of models, beginning with  the cone of a face set, see  \exemref{exam:cone}.

\begin{proposition}\label{prop:modelcone}
Let $\ob{S}$ be a  connected face set of Sullivan  minimal  model $(\land V,d)$. 
\index{Perverse!model!of a cone}\index{Cone on a face set!model of a}
Then a perverse model of the cone, $c\ob{S}\in\fil$, is $(\land V,d)_{\bullet}$, where the perverse degree of  $\omega \in \land V$ is defined by
$\|\omega\|=\max(|\omega|,|d\omega|)$.
\end{proposition}

In the case of the cone, $c\ob{S}\in \fil$, a perversity is determined by the value $\ov{q}(n)$ and, with the described perverse degree, we have $(\land V,d)_{\ov{q}}= \tau_{\leq \ov{q}(n)}(\land V,d)$, see \defref{def:troncature}.

\begin{proof}
From \exemref{exam:cone}, we know that the perverse \cdga, $\widetilde{A}_{PL,\bullet}(c\ob{S})$, is quasi-isomorphic to
$A_{PL}(\ob{S})_{\bullet}$, with
$A_{PL}(\ob{S})_{\ov{q}}=\{\omega\in A_{PL}(\ob{S})\mid \|\omega\|=\max(|\omega|,|d\omega|)\leq \ov{q}(n)\}$.
Let $\varphi\colon (\land V,d)\to A_{PL}(\ob{S})$ be the minimal  model of the \cdga, $A_{PL}(\ob{S})$. As above, we define a perverse \cdga, $(\land V,d)_{\bullet}$, by
$(\land V,d)_{\ov{q}}=\{\omega\in \land V\mid \|\omega\|=\max(|\omega|,|d\omega|)\leq \ov{q}(n)\}$.

As $\varphi$ keeps the degree, i.e., $|\varphi(\omega)|=|\omega|$ if $\varphi(\omega)\neq 0$, we have an induced morphism of perverse \cdga's,
$\varphi_{\bullet}\colon (\land V,d)_{\bullet}\to A_{PL}(c\ob{S})_{\bullet}$,
which induces a quasi-isomorphism,
$\varphi_{\ov{q}}\colon (\land V,d)_{\ov{q}}\to A_{PL}(c\ob{S})_{\ov{q}}$, for any GM-perversity, $\ov{q}$, determined by $\ov{q}(n)$.
\end{proof}

Observe that $(\land V,d)_{\ov{0}}=\Q$ is a model of the contractible face set, $c\ob{S}$, as expected.

\medskip
As the perverse model of a PL-pseudomanifold ``contains'' the model of the underlying space in perverse degree~0, we see that \emph{any intersection-formal PL-pseudomanifold is formal as space.} The next result shows that these two notions, formal and intersection-formal, are distinct.

\begin{proposition}
There exist  PL-pseudomanifolds that are formal as spaces but  not intersection-formal.
\end{proposition}

The next example is an illustration of this statement.

\begin{example}
Denote by $\psi\colon S^2\times S^2\to S^4$ the map obtained by collapsing the two 2-dimensional spheres. We denote by $E$ the pullback of the Hopf fibration $S^3
\to S^7\to S^4$ along $\psi$. The minimal model of the PL-pseudomanifold $E$ is (see \cite[Example 2.91]{MR2403898})
$(\land (x,\alpha,y,\beta,a),d)$, with $|x|=|y|=2$, $|\alpha|=|\beta|=|a|=3$, $d\alpha=x^2$, $d\beta=y^2$, $da=xy$. 
Set $u=\alpha y-xa$ and $v=x\beta-ay$.

The non-zero groups of rational cohomology are
$H^0(E;\Q)=\Q$,
$H^2(E;\Q)=\Q[x]\oplus \Q[y]$, $H^5(E;\Q)=\Q[u]\oplus \Q[v]$, $H^6(E;\Q)=\Q[u y]$.
We observe that $\{[u]\}=\langle [x],[x],[y]\rangle$ and
$\{[v]\}=\langle [x],[y],[y]\rangle$. 
As $xv-uy=d(\alpha\beta)$, we have also $[uy]=[xv]$.
The cone, $cE$, being of dimension~8, a perversity is determined by one integer and we denote by $\ov{\ell}$ the perversity such that $\ov{\ell}(8)=\ell$, the top GM-perversity being  $\ov{6}$. 
From \propref{prop:modelcone}, we compute the perverse minimal model of $cE$ and find
$\rho_{\bullet}\colon (\land \oplus_{\ov{p}\leq \ov{t}} V_{[\ov{p}]},d)_{\bullet}\to (\land (x,\alpha,y,\beta,a),d)_{\bullet}$, with
\begin{itemize}
\item $V_{[\ov{2}]}=\Q \hat{x}\oplus \Q \hat{y}$, $d\hat{x}=d\hat{y}=0$, $\rho(\hat{x})=x$, $\rho(\hat{y})=y$. 
\item $V_{[\ov{4}]}=\Q \hat{\alpha}\oplus \Q\hat{\beta}\oplus \Q\hat{a}$, 
$d\hat{\alpha}= \hat{x}^2$, $d\hat{\beta}=\hat{y}^2$, $d\hat{a}=\hat{x}\hat{y}$,
$\rho(\hat{\alpha})=\alpha$,
$\rho(\hat{\beta})=\beta$, $\rho(\hat{a})=a$.
\item $V_{[\ov{5}]}=\Q\hat{u}\oplus \Q\hat{v}$, $d\hat{u}=d\hat{v}=0$, $\rho(\hat{u})=u$, $\rho(\hat{v})=v$.
\item $V_{[\ov{6}]}=\Q\xi_{1}\oplus \Q\xi_{2}$, $d\xi_{1}=\hat{\alpha} \hat{y} -\hat{x} \hat{a}-\hat{u}$, 
$d\xi_{2}=\hat{x}\hat{\beta} - \hat{a}\hat{y}-\hat{v}$
\end{itemize}
From this determination, we note the existence of two non trivial triple Massey products,
$\hat{u}\in\langle [\hat{x}],[\hat{x}],[\hat{y}]\rangle$
and
$\hat{v}\in \langle [\hat{x}],[\hat{y}], [\hat{y}]\rangle$. \propref{prop:masseyproduct} implies the non intersection-formality of $cE$ despite the formality of $cE$ as space.
\end{example}

\medskip
We continue with perverse models of the suspension of a face set.

\begin{proposition}\label{prop:modelofsuspension}
Let $\ob{S}$ be a  connected face set of Sullivan  minimal  model $(\land V,d)$. 
\index{Perverse!model!of a suspension}\index{Suspension of a face set!model of a}
Then, the perverse minimal model of the suspension, $\Sigma\ob{S}\in \fil$, is the perverse minimal model of the perverse \cdga, $(\land(t,dt)\otimes (\land V,d))_{\bullet}$, with\\
{$(\land(t,dt)\otimes (\land V,d))_{\ov{q}}=
dt\otimes \land t\otimes \land V\oplus \{f(t)\otimes w\in \land t\otimes \land V\mid f(0)=f(1)=0\text{ or } |\omega|\leq \ov{q}(n)\},$}
for any GM-perversity, $\ov{q}$.
\end{proposition}

\begin{proof}
We proceed as in the proof of \propref{prop:modelcone}, by using the results obtained in \exemref{exam:suspensionCP2}, replacing the \cdga, $A_{PL}(\ob{S})$, by the minimal model, $(\land V,d)$, of~$\ob{S}$.
\end{proof}

We note that $(\land(t,dt)\otimes (\land V,d))_{\bullet}$ contains elements of degree~0 and is not a model as in \defref{def:algebredeSullivan}. 
On the other side, an explicit description of the perverse  minimal model of $\Sigma \ob{S}$ is awkward because it is  built by starting with the minimal model of the face set $\Sigma\ob{S}$. 
This last one has the (suspended) dual of a free Lie algebra as vector spaces of indecomposables. 
For instance, if $\ob{S}$ corresponds to the space $\C P(2)$, we have $\Sigma(\C P(2))=S^3\vee S^5$ whose minimal model is
$(\land s\sharp \L(x,y),d)$,
with $|x|=2$, $|y|=4$ and the differential $d$ is the suspension of the transposition of the bracket of the free Lie algebra,
$\L(x,y)$. 
Below, we built the minimal perverse model,
$\varphi_{\bullet}\colon B_{\bullet}\to \widetilde{A}_{PL,\bullet}(\Sigma \C P(2))$, in low degrees, such that $H^i_{\bullet}(\varphi)$ is an isomorphism for $i\leq 6$. (This gives the minimal perverse model in the range of the cohomology.) Taking in account the dimension of the space, the perversities are determined by the value of $\ov{q}(5)$, the top perversity being given by $\ov{t}(5)=3$.  As previously, we denote by $\ov{\ell}$ the perversity such that $\ov{\ell}(5)=\ell$.
Observe also that, as $\C P(2)$ is a formal space, we may replace the model of $\C P(2)$ by  its cohomology algebra, $(H, 0)$. 
In this case, there is no ambiguity in the definition of $\ttau^{\geq \ell} H$ (see \exemref{exam:suspensionCP2}) and 
we are reduced to the determination of the perverse minimal model of 
$E_{\ov{\ell}}= H^{\leq \ell} \oplus s(H^{>\ell})$, with a trivial differential and $H=(\land u)/u^3$, $|u|=2$. The product on $E_{\bullet}$ comes from the product induced by the product of $H$ on the quotient $H^{\leq \ell}$ and the formulae,
$s\eta_{1}\cdot s\eta_{2}=0$, $\eta_{1}\cdot s\eta_{2}=(-1)^{|\eta_{1}|} s(\eta_{1}\cdot\eta_{2})$.
\begin{itemize}
\item If $\ell=0$ or 1, we have $E_{\ov{\ell}}=\Q\oplus sH$ and
$B^{\leq 6}_{\ov{\ell}}=
(\land (\alpha_{3},\alpha_{5}),d)^{\leq 6}$,
with $d\alpha_{3}=d\alpha_{5}=0$, $|\alpha_{i}|=i$ and $\|\alpha_{i}\|=0$. The map $\varphi$ sends $\alpha_{3}$ to $su$ and $\alpha_{5}$ to $s(u^2)$.
\item If $\ell=2$ or 3, we have $E_{\ov{\ell}}=\Q\oplus \Q u \oplus s(\Q u^2)$ and, in the model, we have to kill the class associated to $\alpha_{3}$ and to add a new cocycle which reaches $u$.
Thus, we introduce $\beta_{2}$, $\beta'_{2}$, with $d\beta_{2}=0$, $d\beta'_{2}=\alpha_{3}$. Doing that, we have also introduced a cocycle, $\alpha_{3}\beta'_{2}$, which has to be killed. Also, we observe that the product $\alpha_{3}\beta_{2}$ has to be identified to $\alpha_{5}$. Finally, we set
$$B^{\leq 6}_{\ov{\ell}}=
(\land (\alpha_{3},\alpha_{5}, \beta_{2}, \beta'_{2},\beta_{4},\beta'_{4}),d)^{\leq 6},$$
with $|\beta_{2}|=|\beta'_{2}|=2$, $|\beta_{4}|=|\beta'_{4}|=4$, $d\beta_{2}=0$, $d\beta'_{2}=\alpha_{3}$, $d\beta_{4}=\alpha_{3}\beta_{2}-\alpha_{5}$ and $d\beta'_{4}=\alpha_{3}\beta'_{2}$. The map $\varphi$ sends $\beta_{2}$ to $u$ and $\beta'_{2}$, $\beta_{4}$, $\beta'_{4}$ to $0$. All the $\beta_{i}$ and $\beta'_{i}$ are of perverse degree~2.
(Observe that we do not have to consider $\beta^2_{2}$ and $\beta'_{2}\beta_{2}$  which are not of this perversity.)
\end{itemize}

%
\section{A model for isolated singularities}\label{subsec:isolated}
\begin{quote}
We extend the two situations, of the cone and of the suspension, to the case of PL-pseudomanifolds with isolated singularities.
They are also  considered in \cite[Proposition 5.1]{2013arXiv1302.2737D} where the structure of their Steenrod squares is determined.
\end{quote}

Let $M$ be a PL-pseudomanifold, obtained from a manifold with boundary, $(W,\partial W)$, by attaching cones on the connected components of the boundary, i.e., $M$ is the push out
$$\xymatrix{
\sqcup_{u\in I}\partial_{u} W\ar[r]^-{\iota}\ar[d]&W\ar[d]\\
\sqcup_{u\in I} c(\partial_{u} W)\ar[r]&M,
}$$
where the $\partial_{u} W$'s are the connected components of $\partial W$, and $c(\partial_{u} W)$ is the cone on a  component. 
We filter the pseudomanifold, $M$, by 
the cone points of $c(\partial_{u}W)$.
As the singularities are points, a perversity $\ov{q}$ is determined by one number, $\ov{q}(n)$, with $n=\dim M$.

\begin{proposition}\label{prop:modelisolated}
Let $\ov{q}$ be a GM-perversity and $M$ be an $n$-dimensional PL-pseudo\-manifold, as above. Let $\varphi\colon (A_{1},d_{1})\to (A_{2},d_{2})=\oplus_{u\in I}(A_{2}(u),d_{2})$ be a \emph{surjective} model of the inclusion $\iota\colon\sqcup_{u\in I}\partial_{u} W\to W$. 
\index{Perverse!model!for isolated singularities}
Then a (non cofibrant) perverse model of $M$ is given by:
$$\cM(M)_{\ov{q}}=(A_{1},d_{1})\oplus_{A_{2}}\left( \oplus_{u\in I}\tau_{\leq\ov{q}(n)}(A_{2}(u),d_{2})\right),$$
 where the truncation $\tau_{\leq\ov{q}(n)}$ is defined in \defref{def:troncature}.
\end{proposition}

\begin{proof}[Proof of \propref{prop:modelisolated}]
We start with a pushout of spaces
$$\xymatrix{
\sqcup_{u\in I}\partial_{u} W\ar[r]^-{\iota}\ar[d]&W\ar[d]\\
 \sqcup_{u\in I}c(\partial_{u} W)\ar[r]&M.
}$$
We may suppose that $M$, $W$, $\partial_{u} W$ and $c(\partial_{u} W)$ are triangulated in such a way that any simplex is filtered, for the filtration by the cone point. 

Let $X$ be one of the spaces above, of associated simplicial complex $X^\tau$and of associated \ffs, $\ob{X}^\tau$. 
G. Friedman proves, in particular,  that if the triangulation is full (which is the case in our situation) then the cochains
 $C^*_{\GM,\ov{p}}(X)$ and $C^*_{\GM,\ov{p}}(X^\tau)$ are quasi-isomorphic for any GM-perversity $\ov{p}$ (see \cite[Chapter 3 and Chapter 5]{IHGreg}). 
 Thus, the isomorphism
$$C^*_{\GM,\ov{p}}({M}^\tau)\cong C^*({W}^\tau)\oplus_{(\oplus_{u\in I}C^*({\partial_{u} W}^\tau))}\left(\oplus_{u\in I}C^*_{\GM,\ov{p}}({c(\partial_{u} W)^\tau})\right)$$
gives a quasi-isomorphism
 $$C^*_{\GM,\ov{p}}({M})\simeq C^*({W})\oplus_{(\oplus_{u\in I}C^*({\partial_{u} W}))}\left(\oplus_{u\in I}C^*_{\GM,\ov{p}}{(c(\partial_{u} W)})\right).$$

 Moreover, we know, from \corref{cor:quasiisossurQ}, that $\widetilde{A}_{PL,\ov{q}}(\ob{X})$ and $C^*_{\GM,\ov{p}}(X;\Q)$ are quasi-isomorphic if $\ov{p}+\ov{q}=\ov{t}$.
 This implies that the canonical \cdga~map,
 $$
\widetilde{A}_{PL,\ov{q}}(M)\to  {A}_{PL}( W)\oplus_{(\oplus_{u\in I}{A}_{PL}(\partial_{u} W))}\left(\oplus_{u\in I}\widetilde{A}_{PL,\ov{q}}(c\partial_{u} W)\right),
$$
is a quasi-isomorphism.
The \cdga~of forms on the cone is quasi-isomorphic to a truncation, i.e., there is a quasi-isomorphism
$ \widetilde{A}_{PL,\ov{q}}(c\partial_{u} W)\simeq \tau_{\leq \ov{q}(n)} {A}_{PL}(\partial_{u} W)$ which induces a quasi-isomorphism
$$\widetilde{A}_{PL,\ov{q}}(M)\simeq 
{A}_{PL}( W)\oplus_{(\oplus_{u\in I}{A}_{PL}(\partial_{u} W))}\left(\oplus_{u\in I}\tau_{\leq \ov{q}(n)} {A}_{PL}(\partial_{u} W)\right).
$$
With the surjective model $\varphi\colon (A_{1},d_{1})\to \oplus_{u\in I}(A_{2}(u),d_{2})$ of the statement, we get a morphism of short exact sequences,
$$\footnotesize{\xymatrix@C-=0.4cm{
0\ar[r]&
\ker \varphi\ar[r]\ar[d]&
(A_{1},d_{1})\oplus_{(\oplus_{u\in I}A_{2}(u))} \left(\oplus_{u\in I}\tau_{\leq\ov{q}(n)}(A_{2},d_{2})\right)\ar[r]\ar[d]&
\oplus_{u\in I}\tau_{\leq\ov{q}(n)}(A_{2}(u),d_{2})\ar[r]\ar[d]&
0\\
0\ar[r]&
\ker {A}_{PL}(\iota)\ar[r]&
{A}_{PL}( W)\oplus_{(\oplus_{u\in I}{A}_{PL}(\partial_{u} W))}\left(\oplus_{u\in I}\tau_{\leq \ov{q}(n)} {A}_{PL}(\partial_{u} W)\right)\ar[r]&
\oplus_{u\in I}\tau_{\leq \ov{q}(n)} {A}_{PL}(\partial_{u} W)\ar[r]&
0.
}}$$
The result follows with an application of the five lemma to the associated  long exact sequences.
\end{proof}

In \propref{prop:modelisolated},  the elements of $\cM(M)_{\ov{q}}$ are couples, $(\omega, \varphi(\omega))$ such that  $\omega\in A_{1}$ and
$$\left\{
\begin{array}{lcl}
\varphi(\omega)=0
&\text{if}&
|\omega|>\ov{q}(n),\\
\varphi(\omega) \;\text{is a cocycle}
&\text{if}&
|\omega|=\ov{q}(n),\\
\text{no condition}
&\text{if}&
|\omega|<\ov{q}(n).
\end{array}\right.$$
This implies immediately:
$$H^k_{\ov{q}}(M;\Q)=\left\{\begin{array}{lcl}
H^k(W;\Q)
&\text{if}&
k\leq \ov{q}(n),\\
\ker \left(H^k(W;\Q)\to H^k(\partial W;\Q)\right)%
&\text{if}&
k=\ov{q}(n)+1,\\
H^k(W,\partial W;\Q)
&\text{if}&
k>\ov{q}(n)+1.
\end{array}\right.$$

In the case of simply connected spaces, the model arising in \propref{prop:modelisolated} can be simplified as follows.

\begin{corollary}\label{cor:modelisolated}
Let $\ov{q}$ be a GM-perversity and $M$ be an $n$-dimensional PL-pseudo\-manifold, as above, and such that $W$ and the $\partial_{u}W$'s are simply connected. Let $\varphi\colon (A_{1},d_{1})\to(A_{2},d_{2})= \oplus_{u\in I}(A_{2}(u),d_{2})$ be a model of the inclusion $\iota\colon\sqcup_{u\in I}\partial_{u} W\to W$, such that $A^0_{1}=A_{2}(u)^0=\Q$,  $A^1_{1}=A_{2}(u)^1=0$ and $\varphi$ surjective in strictly positive degrees. Then a (non cofibrant) perverse model of $M$ is given by:
$$\cM(M)_{\ov{q}}=(A_{1},d_{1})\oplus_{A_{2}} \left(\oplus_{u\in I}\tau_{\leq\ov{q}(n)}(A_{2}(u),d_{2})\right),$$
 where the truncation $\tau_{\leq\ov{q}(n)}$ is defined in \defref{def:troncature}.
\end{corollary}

\begin{proof}
Proof of \propref{prop:modelisolated} is still valid with the following modifications.
\begin{itemize}
\item The arguments with long exact sequences are starting in degree~1 instead of degree~0.
\item The result is still true in degree~0 because
$$\left((A_{1},d_{1})\oplus_{(\oplus_{u\in I}A_{2}(u))}\left( \oplus_{u\in I}\tau_{\leq\ov{q}(n)}(A_{2}(u),d_{2})\right)\right)^0=\Q.$$
\end{itemize}
\end{proof}

\section{Thom spaces}\label{subsec:Thom}

\begin{quote}
 The construction done in the previous section is applied to the case of the Thom space associated to a vector bundle and specialized to projective cones. We deduce that all singular quadrics are intersection-formal. 
 \end{quote}
 
Let $\R^m\to E\to B$ be a vector bundle. We denote by $D_{E}\to B$ the associated disk-bundle and by $S_{E}\to B$ the associated sphere-bundle. The \emph{Thom space,} $\Th(E)$, is  the quotient of the disk-bundle by the sphere-bundle. We filter $\Th(E)$ by the point of compactification and a perversity is determined by the number $\ov{q}(n)$ where $n=\dim \Th(E)$. The next result suffices for the study of projective cones.

\begin{proposition}\label{prop:thomformal}
Let $f\colon E\to B$ be a vector bundle of rank $2r$, with $B$ a formal manifold. 
\index{Intersection-formality!of some Thom spaces}
The associated Thom space, 
$\Th(E)$, filtered by the compactification point, is an intersection-formal space.
\end{proposition}

\begin{proof}
We work over the field $\Q$.
Let $\ov{q}$ be a GM-perversity and $\cM({B})$ be any model of the space $B$.
Denote by $c=c(E)\in \cM^{2r}({B})$ a representative  of the Euler class. A surjective model of $S_{E}\to D_{E}$ is given by
$$\varphi\colon (\cM(B)\otimes \land (x,y),D) \to (\cM(B)\otimes \land x,d),$$
with $|x|=2r-1$, $|y|=2r$, $Dx=c-y$, $Dy=0$, $dx=c$, 
$\varphi$ is the identity map on $\cM(B)\otimes \land x$ and $\varphi(y)=0$.  
\propref{prop:modelisolated} gives as perverse model of the Thom space,
$$\cM(\Th(E))_{\ov{q}}=(\cM(B)\otimes \land (x,y),D)\oplus_{\cM(B)\otimes \land x}
\tau_{\leq\ov{q}(n)}(\cM(B)\otimes \land x,d).
$$
The manifold $B$ being formal, we may choose $\cM(B)=(H^*(B),0)$ and obtain,
$$\cM(\Th(E))_{\ov{q}}=(H^*(B)\otimes \land (x,y),D)\oplus_{H^*(B)\otimes \land x}
\tau_{\leq\ov{q}(n)}(H^*(B)\otimes \land x,d).$$
This pullback can be described by:
$$\cM(\Th(E))^k_{\ov{q}}=\left\{
\begin{array}{cl}
(H(B)\otimes \land (x,y))^k&\text{ if } k<\ov{q}(n),\\
\{\omega\in (H(B)\otimes \land (x,y))^k\mid D(\omega)_{|y=0}=0\}&\text{ if }k=\ov{q}(n),\\
\cI\langle y\rangle&\text{ if } k>\ov{q}(n),
\end{array}\right.$$
where $\cI\langle y\rangle=H(B)\otimes \land x\otimes y$ is the differential ideal  generated by $y$.
If $\ov{p}\leq \ov{q}$, the morphisms
$\varphi_{\ov{p}}^{\ov{q}}\colon \cM(\Th(E))_{\ov{p}}\to \cM(\Th(E))_{\ov{q}}$
are the canonical inclusions.
From  this presentation, we recover (see \cite[Page~77]{MR2207421})  the intersection cohomology vector space of the Thom space,
\begin{equation}\label{equa:thom}
H_{\ov{q}}^k(\Th(E))=\left\{
\begin{array}{lcl}
H^k(B)
&\text{if}&
k\leq \ov{q}(n),\\
\im(-\cup c\colon H^{k-2r}(B)\to H^k(B))
&\text{if}&
k=\ov{q}(n)+1,\\
H^{k-2r}(B)
&\text{if}&
k>\ov{q}(n)+1.
\end{array}\right.
\end{equation}
If $\ov{p}\leq \ov{q}$, the morphisms,
$\psi_{\ov{p}}^{\ov{q}}\colon H^k_{\ov{p}}(\Th(E))\to H^k_{\ov{q}}(\Th(E))$
are the canonical inclusions, except for
$\ov{p}(n)+1< k\leq \ov{q}(n)$
where $\psi_{\ov{p}}^{\ov{q}}(\gamma)=\gamma\cup c$. 
Let $\ov{q}_{1}$, $\ov{q}_{2}$ be two GM-perversities and $a_{1}\in H^*_{\ov{q}_{1}}(\Th(E))$,
$a_{2}\in H^*_{\ov{q}_{2}}(\Th(E))$. We specify the product $a_{1}\cdot a_{2}\in H^*_{\ov{q}_{1}\oplus \ov{q}_{2}}(\Th(E))$ as follows.
\begin{itemize}
\item If $|a_{1}|\leq \ov{q}_{1}(n)+1$ and $|a_{2}|< \ov{q}_{2}(n)+1$, we have
$a_{1}\cdot a_{2}=a_{1}\cup a_{2}$.
\item If $|a_{1}|= \ov{q}_{1}(n)+1$ and $|a_{2}|= \ov{q}_{2}(n)+1$, then $a_{1}=a'_{1}\cup c$ and $a_{2}=a'_{2}\cup c$ and we have
$a_{1} \cdot a_{2}=a'_{1}\cup a'_{2}\cup c$.
\item If $|a_{1}|\leq  \ov{q}_{1}(n)+1$ and $|a_{2}|> \ov{q}_{2}(n)+1$, 
with $|a_{1}|+|a_{2}|\leq (\ov{q}_{1}\oplus\ov{q}_{2})(n)+1$,
we have
$a_{1}\cdot a_{2}=a_{1}\cup a_{2}\cup c$.
\item If $|a_{1}|\leq  \ov{q}_{1}(n)+1$ and $|a_{2}|> \ov{q}_{2}(n)+1$, 
with $|a_{1}|+|a_{2}|> (\ov{q}_{1}\oplus\ov{q}_{2})(n)+1$,
we have
$a_{1}\cdot a_{2}=a_{1}\cup a_{2}$.
\item If $|a_{1}|\geq  \ov{q}_{1}(n)+1$ and $|a_{2}|\geq  \ov{q}_{2}(n)+1$, we have
$a_{1}\cdot a_{2}=a_{1}\cup a_{2}\cup c$.
\end{itemize}

We construct now an explicit \cdga's map,  $\Phi_{\ov{q}}\colon \cM(\Th(E))_{\ov{q}}\to (H_{\ov{q}}^k(\Th(E)),0)$.
An element $\omega\in H^*(B)\otimes\land(x,y)$ is a sum of terms,
$a\otimes y^i$ and $b\otimes xy^i$,
with $i\geq 0$ and $a,\,b\in H^*(B)$. 
We define $\Phi_{\ov{q}}(\omega)\in H_{\ov{q}}^k(\Th(E))$ as follows.
\begin{itemize}
\item If $|\omega|\leq \ov{q}(n)+1$, we set $\Phi_{\ov{q}}(a\otimes y^i)=a\cup c^i$ and $\Phi_{\ov{q}}(b\otimes xy^i)=0$.
\item If $|\omega|> \ov{q}(n)+1$, we have $i\geq 1$ and set 
$\Phi_{\ov{q}}(a\otimes y^i)=a\cup c^{i-1}$ and $\Phi_{\ov{q}}(b\otimes xy^i)=0$.
\end{itemize}
By definition, the map $\Phi_{\ov{q}}$ takes value in $H_{\ov{q}}^k(\Th(E))$. 
We prove that $\Phi_{\ov{q}}\circ \varphi_{\ov{p}}^{\ov{q}}=\psi_{\ov{p}}^{\ov{q}}\circ \Phi_{\ov{p}}$, if $\ov{p}\leq \ov{q}$.
This is clear except for $\ov{p}(n)+1< k\leq \ov{q}(n)$, where we have,
$$(\psi_{\ov{p}}^{\ov{q}}\circ \Phi_{\ov{p}})(a\otimes y^i)=\psi_{\ov{p}}^{\ov{q}}(a\cup c^{i-1})=a\cup c^i
=\Phi_{\ov{q}}(a\otimes y^i)=(\Phi_{\ov{q}}\circ \varphi_{\ov{p}}^{\ov{q}})(a\otimes y^i).$$

We check now the compatibility of $\Phi_{\ov{q}}$ with differentials, which reduces, in this case, to $\Phi_{\ov{q}}(D\omega)=0$. The differential of $\omega$ is determined by
$$D(a\otimes y^i)=0\text{ and }
D(b\otimes xy^i)=(-1)^{|b|}(b\cup c\otimes y^i-b\otimes y^{i+1}).$$
Thus, we are only concerned with the terms $\omega=b\otimes xy^i$.

$\bullet$ If $|\omega|< \ov{q}(n)$, then there is no restriction on $\omega$ and we have, 
$$\Phi_{\ov{q}}(D(b\otimes xy^i))=(-1)^{|b|} (b\cup c\cup c^i-
b\cup c^{i+1})=0.$$

$\bullet$ If $|\omega|=\ov{q}(n)$, then  $D(\omega)_{|y=0}=0$ implies $b\cup c=0$ if $i=0$ and we have, 
\begin{eqnarray*}
\Phi_{\ov{q}}(D(b\otimes x))&=&-(-1)^{|b|} \Phi_{\ov{q}}(b\otimes y)=-(-1)^{|b|} b\cup c=0,\\
\Phi_{\ov{q}}(D(b\otimes xy^i))
&=&
(-1)^{|b|}(b\cup c\cup c^i-b\cup c^{i+1})=0, \text{ if } i>0.
\end{eqnarray*}

$\bullet$ If $|\omega|>\ov{q}(n)$, then $b=0$ if $i=0$, and we have,
$$\Phi_{\ov{q}}(D(b\otimes xy^i))=
(-1)^{|b|}(b\cup c\cup c^{i-1}-b\cup c^{i})=0.
$$
Let $\ov{q}_{1}$ and $\ov{q}_{2}$ be two GM-perversities.
For the compatibility with products, it is sufficient to establish
$\Phi_{\ov{q}_{1}\oplus\ov{q}_{2}}((a_{1}\otimes y^i) (a_{2}\otimes y^j))=
\Phi_{\ov{q}_{1}}(a_{1}\otimes y^i)\cdot \Phi_{\ov{q}_{2}}(a_{2}\otimes y^j)$,
i.e.,
$$\Phi_{\ov{q}_{1}\oplus\ov{q}_{2}}(a_{1}\cup a_{2}\otimes y^{i+j})
=
\Phi_{\ov{q}_{1}}(a_{1}\otimes y^i)\cdot \Phi_{\ov{q}_{2}}(a_{2}\otimes y^j).$$
We specify the different cases  and set $\omega_{1}=a_{1}\otimes y^i$, $\omega_{2}=a_{2}\otimes y^j$.
  \begin{itemize}
\item Suppose $|\omega_{1}|\leq \ov{q}_{1}(n)+1$ and $|\omega_{2}|< \ov{q}_{2}(n)+1$. Then, we have,
$$\Phi_{\ov{q}_{1}}(\omega_{1})\cdot\Phi_{\ov{q}_{2}}(\omega_{2})=
(a_{1}\cup c^i)\cdot  (a_{2}\cup c^j)=a_{1}\cup a_{2}\cup c^{i+j}=
\Phi_{\ov{q}_{1}\oplus\ov{q}_{2}}(\omega_{1} \omega_{2}).
$$
\item Suppose $|\omega_{1}|= \ov{q}_{1}(n)+1$ and $|\omega_{2}|= \ov{q}_{2}(n)+1$. Then, we have,
\begin{eqnarray*}
\Phi_{\ov{q}_{1}}(\omega_{1})\cdot \Phi_{\ov{q}_{2}}(\omega_{2})&=&
(a_{1}\cup c^i)\cdot (a_{2}\cup c^{j})=a_{1}\cup  a_{2}\cup c^{i+j-1}
=\Phi_{\ov{q}_{1}\oplus\ov{q}_{2}}(\omega_{1} \omega_{2}).
\end{eqnarray*}
\item Suppose $|\omega_{1}|\leq \ov{q}_{1}(n)+1$ and $|\omega_{2}|> \ov{q}_{2}(n)+1$,
with $|\omega_{1}|+|\omega_{2}|\leq (\ov{q}_{1}\oplus \ov{q}_{2})(n)+1$. Then we have,
\begin{eqnarray*}
\Phi_{\ov{q}_{1}}(\omega_{1})\cdot \Phi_{\ov{q}_{2}}(\omega_{2})&=&
(a_{1}\cup c^i)\cdot (a_{2}\cup c^{j-1})=(a_{1}\cup c^i)\cup (a_{2}\cup c^{j-1})\cup c\\
&=& a_{1}\cup a_{2}\cup c^{i+j}
=
\Phi_{\ov{q}_{1}\oplus\ov{q}_{2}}(\omega_{1} \omega_{2}).
\end{eqnarray*}
\item Suppose $|\omega_{1}|\leq \ov{q}_{1}(n)+1$ and $|\omega_{2}|> \ov{q}_{2}(n)+1$,
with $|\omega_{1}|+|\omega_{2}|> (\ov{q}_{1}\oplus \ov{q}_{2})(n)+1$. 
Then, we have,
\begin{eqnarray*}
\Phi_{\ov{q}_{1}}(\omega_{1})\cdot \Phi_{\ov{q}_{2}}(\omega_{2})&=&
(a_{1}\cup c^i)\cdot (a_{2}\cup c^{j-1})=(a_{1}\cup c^i)\cup (a_{2}\cup c^{j-1})\\
&=& a_{1}\cup a_{2}\cup c^{i+j-1}
=
\Phi_{\ov{q}_{1}\oplus\ov{q}_{2}}(\omega_{1} \omega_{2}).
\end{eqnarray*}
\item Suppose $|\omega_{1}|> \ov{q}_{1}(n)+1$ and $|\omega_{2}|> \ov{q}_{2}(n)+1$. Then, we have,
\begin{eqnarray*}
\Phi_{\ov{q}_{1}}(\omega_{1})\cdot \Phi_{\ov{q}_{2}}(\omega_{2})&=&
(a_{1}\cup c^{i-1})\cdot (a_{2}\cup c^{j-1})=(a_{1}\cup c^{i-1})\cup (a_{2}\cup c^{j-1})\cup c\\
&=& a_{1}\cup a_{2}\cup c^{i+j-1}
=
\Phi_{\ov{q}_{1}\oplus\ov{q}_{2}}(\omega_{1} \omega_{2}).
\end{eqnarray*}
\end{itemize}

\end{proof}

Recall the notion of 
projective cone of a smooth projective variety. Let $S$ be a smooth algebraic subvariety of $\C P^n$, we embedd it in $\iota\colon S\to \C P^{n+1}$ through the canonical map, $ \C P^n\to \C P^{n+1}$, sending $[x_0:\ldots:x_n]$ to
$[x_0:\ldots:x_n:0]$. By definition, the \emph{projective cone of $S$} is the union of all the lines that intersect $S$ and contain the point $[0:\ldots:0:1]$. It can also be expressed as the Thom space of the restriction of the tautological line bundle over $\C P^{n+1}$,
$$\xymatrix{
E\ar[r]\ar[d]&
E(\gamma^1)\ar[d]\\
S\ar@{^(->}[r]^-{\iota}&
\C P^{n+1}.
}$$
The characteristic class of this vector bundle is $c_{1}(E)=\iota^*(c_{1})\in H^2(S)$, with $c_{1}$ the K¨\"ahler class of $\CP^{n+1}$. From \cite{MR0382702} and \propref{prop:thomformal}, we deduce directly the next result.

\begin{proposition}\label{prop:projectiveformal}
The projective cones of a smooth projective variety are intersection-formal. 
\index{Intersection-formality!of singular quadrics}
Thus, any singular quadric is intersection-formal.
\end{proposition}

We use the model of \propref{prop:thomformal} for the complete determination of the intersection cohomology algebra of the Segre embedding.

\begin{example}\label{exem:segre}
\emph{
The Segre embedding} (see \cite{MR901594}) 
\index{Segre embedding}
is defined by $\CP(1)\times \CP(1)\to \CP(3)$, 
$([x_{0}:x_{1}],[y_{0}:y_{1}])\mapsto [x_{0}y_{0}:x_{0}y_{1}:x_{1}y_{0}:x_{1}y_{1}]$. 
Let $u$ and $v$ be the fundamental classes of the $\CP(1)$'s.
The corresponding line bundle, $\C\to E\to \CP(1)\times \CP(1)$,  has for Chern class $c=c_{1}(E)=u+v$. The regular component is $\CP(1)\times \CP(1)$ and the link is the associated sphere-bundle.

We denote by $\ov{\ell}$ the perversity such that $\ov{\ell}(6)=\ell$. For the description of the intersection cohomology of the Thom space, $\Th(E)$, we have only to consider the GM-perversities, $\ov{0}$, $\ov{2}$, $\ov{4}$. 
A perverse model of the associated Thom space is given by
$$(A,\delta)_{\ov{\ell}}=(H\land (x,y),D)\oplus_{H\otimes\land x}\tau_{\leq \ell}(H\otimes\land x,d),$$
where $H$ is the quotient $H=\land(u,v)/(u^2,v^2)$,
$|x|=1$, $|y|=|u|=|v|=2$,
$dx=u+v$, $Dx=u+v-y$, $Dy=0$.
From, this model, we deduce the perverse cohomology algebra,
whose elements are specified below,
\renewcommand{\arraystretch}{1.5}
\setlength{\tabcolsep}{.7cm}
\begin{center}
   \begin{tabular}{| c || c | c|c| c|}
     \hline
    $ H^k_{\ov{\ell}}(\Th(E);\Q)$ & $\ell=0$ or 1& $\ell=2$&$\ell =3$ or 4 \\ \hline
     $k=0$ &$\Q$  & $\Q$&$\Q$ \\ \hline
     $k=1$ &  0 & 0&0 \\\hline
    $ k=2$ & $\Q[y]$ & $\Q[u]\oplus \Q[v]$&$\Q[u]\oplus \Q[v]$\\\hline
     $k=3$ & 0 & 0&0 \\\hline
     $k=4$ & $\Q[\alpha]\oplus\Q[\beta]$ & $\Q[\alpha]\oplus\Q[\beta]$&$\Q[uv]$ \\\hline
     $k=5 $& 0 & 0&0 \\\hline
     $k=6$ & $\Q[\gamma]$ & $\Q[\gamma]$&$\Q[\gamma]$ \\
     \hline
   \end{tabular}
 \end{center}
We specify now the maps $\psi_{\ov{p}}^{\ov{q}}\colon H_{\ov{p}}^k(\Th(E);\Q)\to H_{\ov{q}}^k(\Th(E);\Q)$,
with $\ov{p}\leq \ov{q}$.
\begin{itemize}
\item $\psi_{\ov{0}}^{\ov{2}}$ is determined by $[y]\mapsto [u]+[v]$ in degree 2 and the identity otherwise.
\item $\psi_{\ov{2}}^{\ov{4}}$ is determined by $[\alpha]\mapsto [uv]$, $[\beta]\mapsto [uv]$ in degree 4  and the identity otherwise.
\end{itemize}

\noindent
The structure of perverse algebra is given by the following equalities.
\begin{itemize}
\item The classes $[\alpha]$ and $[\beta]$ correspond to the cocycles $yu$ and $yv$, respectively. The differential
$D(xy)=uy+vy-y^2$ implies  $[y]^2=[\alpha]+[\beta]$. Using similar arguments, we obtain the cohomology algebra of $ H^*_{\ov{0}}(\Th(E);\Q)=H^*(\Th(E);\Q)$, with $2[y]\,[\alpha]=2[y]\,[\beta]=[y]^3=[\gamma]$.
\item As $[\alpha]$ and $[\beta]$ are represented by $yu$ and $yv$, respectively, the product $H^*_{\ov{0}}\otimes H^*_{\ov{2}}\to H^*_{\ov{2}}$ is completed by
$[y]\,[u]=[\alpha]$, $[y]\,[v]=[\beta]$, 
$[u]\,[\alpha]=[v]\,[\beta]=0$, $[u]\,[\beta]=[v]\,[\alpha]=[\gamma]$.
\item Finally, the product $H^*_{\ov{2}}\otimes H^*_{\ov{2}}\to H^*_{\ov{4}}$ reduces to
$[u]\,[v]=[uv]$.
\end{itemize}
The other products are zero for reason of dimension or belong to  the cohomology algebra of the regular component, $H_{\ov{\infty}}^*(\Th(E);\Q)=H$. 
\end{example}

\section{Nodal Hypersurfaces in $\CP(4)$}\label{subsec:nodal}

\begin{quote}
  We continue the study of intersection-formality with nodal hypersurfaces in $\CP(4)$ and prove their intersection-formality.
\end{quote}

A nodal hypersurface in $\CP(4)$
\index{Nodal hypersurfaces}\index{Calabi-Yau quintic}
 is a hypersurface whose  singulatities are ordinary double points, the \emph{nodes.}
 \index{Nodes}
  For instance,  the \emph{Calabi-Yau quintic,} $X_{0}$, is defined by the polynomial
$$P(z)=z_{0}^5+z_{1}^5
+z_{2}^5+z_{3}^5 +z_{4}^5
-5z_{0}z_{1}z_{2}z_{3}z_{4}.$$

\begin{theorem}\label{thm:nodalformal}
Any  nodal hypersurface in $\CP(4)$
\index{Intersection-formality!of nodal hypersurfaces}
 is intersection-formal.
\end{theorem}

Come back to the general case and let $\ov{V}$ be a nodal hypersurface in $\CP(4)$, with $N$ isolated singularities. 
We denote by $\ov{V}_{\reg}$ the regular component and by ${V}$ a small resolution of $\ov{V}$.
Our strategy of proof is the construction of a model of $\ov{V}$ in the spirit of \corref{cor:modelisolated}.  First, we  determine a (classical) Sullivan model of the regular component, $\ov{V}_{\reg}$, from  a geometrical local analysis.
The link around each singular point is a product $\CP(1)_{(i)}\times S^3_{(i)}$ that can be obtained as follows.\\
Let
$D^2\to E_{D}(i)\xrightarrow[]{\pi} \CP(1)_{(i)}\times \CP(1)_{(i)}$
be the normal bundle associated to the Segre embedding
\index{Segre embedding}
$\CP(1)_{(i)}\times \CP(1)_{(i)}\hookrightarrow \CP(3)$.
We set $E_{D}=\sqcup_{i=1}^N E_{D}(i)$
and denote by
$\partial E_{D}=\sqcup_{i=1}^N \CP(1)_{(i)}\times S^3_{(i)}
\xrightarrow[]{\pi_{S}} 
\sqcup_{i=1}^N \CP(1)_{(i)}\times \CP(1)_{(i)}$
 the associated circle-bundles.

We consider also a second resolution, $W$, obtained from the minimal resolution, $V$, with a blow-up, $\Bl$, of the $\CP(1)_{(i)}$'s, i.e.,  we have a commutative diagram, where $f$ and $g$ are embeddings,
\begin{equation}\label{equa:2resolutions}
\xymatrix{
\sqcup_{i=1}^N \CP(1)_{(i)}\times \CP(1)_{(i)}\ar[d]^-{g}\ar[r]
&\sqcup_{i=1}^N \CP(1)_{(i)}\ar[d]^-{f}\ar[r]
&\sqcup_{i=1}^N \ast\ar[d]\\
W\ar[r]^{\Bl}
&V\ar[r]
&
\ov{V}
}\end{equation}
The regular part, $\ov{V}_{\reg}$, is the complement of the divisor $\cD=\sqcup_{i=1}^N \CP(1)_{(i)}\times \CP(1)_{(i)}$ in the non-singular manifold, $W$, and we know from J.~Morgan (see \cite{MR516917}) how to build a rational (non-free) Sullivan model of it. The two resolutions can also be described as pushout's,
\begin{equation}\label{equa:lagrosse}
{W}=\ov{V}_{\reg}\cup_{\sqcup_{i=1}^N \CP(1)_{(i)}\times S^3_{(i)}}  E_{D}
\end{equation}
in contrast with
\begin{equation}\label{equa:lapetite}
{V}=\ov{V}_{\reg}\cup_{\sqcup_{i=1}^N \CP(1)_{(i)}\times S^3_{(i)}} (\sqcup_{i=1}^N \CP(1)_{(i)}\times D^4_{(i)}).
\end{equation}
In terms of normal bundle, the construction of $W$ is done from a replacement of the fibration
$\pi_{S}\colon \partial E_{D}\to \sqcup_{i=1}^N \CP(1)_{(i)}\times \CP(1)_{(i)}$
by its composition with the canonical projection $\pr_{1}\colon \CP(1)_{(i)}\times \CP(1)_{(i)}\to\CP(1)_{(i)}$.
This gives a trivial sphere-bundle, of basis $\CP(1)_{(i)}$ and fiber $F_{(i)}=S^3_{(i)}$, as shows the next diagram,
restricted to one component,
$$\xymatrix@=6pt{
F\ar[rr]\ar[dd]&&\partial E_{D}\ar[rr]\ar[dd]^{\pi_{S}}&&S^{7}\ar[dd]\\
&\fbox{b}&&\fbox{c}&\\
\CP(1)\ar[rr]\ar[dd]&&\CP(1)\times \CP(1)\ar[rr]^-f\ar[dd]^{\pr_{1}}&&\CP(3)\\
&\fbox{a}&&&\\
\ast\ar[rr]&&\CP(1).&&
}$$
The square \fbox{a} and the rectangle  $\fbox{\text{a}} \cup \fbox{\text{b}}$ being pullbacks, the square \fbox{b} is a pullback. As  \fbox{c} is a  pullback also, the rectangle
$\fbox{\text{b}} \cup \fbox{\text{c}}$
is a pullback. This implies that the left-hand vertical map is the Hopf fibration, $S^1\to F=S^3\to\CP(1)$.  
The next cube provides a map, $W\to V$,
compatible with the decomposition in pushouts of (\ref{equa:lagrosse}) and (\ref{equa:lapetite}),
\begin{equation}\label{equa:VandW}
\xymatrix{
&W
\ar[dd]|(.50){\hole} 
&&
E_{D}\ar[ll]\ar[dd]\\
\ov{V}_{\reg}\ar[ru]\ar@{=}[dd]
&&
 \partial E_{D}\ar[ll]\ar[dd]_<<<<<<<<<<<<<<<<{\psi}\ar[ru]
&\\
&V&&\sqcup_{i=1}^N \CP(1)_{(i)}\times D^4_{(i)}
\ar[ll] |(.59){\hole}
\\
\ov{V}_{\reg}\ar[ru]
&&
\sqcup_{i=1}^N \partial(\CP(1)_{(i)}\times D^4_{(i)})\ar[ru]\ar[ll]&
}\end{equation}
We use the front face for a construction of rational models, starting from
$H^*(\CP(1)_{(i)})\otimes H^*(S^3_{(i)})$ as model of $\partial (\CP(1)_{(i)}\times D^4_{(i)})$
and
$\oplus_{i=1}^N(H^*(\CP(1)_{(i)}\times \CP(1)_{(i)})\otimes \land \theta_{i}),d)$
as model of $\partial E_{D}$, with
$d\theta_{i}=a_{i}+b_{i}$, where $a_{i}$ and $b_{i}$ are generators of $H^2(\CP(1)_{(i)}\times \CP(1)_{(i)})$.
The model  of $\partial E_{D}$ arrives as a model of the total spaces of the fibrations,
$S^1\to \partial E_{D}(i)\to \CP(1)_{(i)}\times \CP(1)_{(i)}$,
whose Euler class is $a_{i}+b_{i}$, see the definition of the Segre embedding in \exemref{exem:segre}.
In summary, we have a commutative square of \cdga's,
\begin{equation}\label{equa:modelnormalbundle}
\xymatrix{
\cN'\ar@{->>}[r]^-{\rho'}&
\oplus_{i=1}^N
(H^*(\CP(1)_{(i)})\otimes H^*(\CP(1)_{(i)})\otimes \land \theta_{i},d)\\
\cN\ar[u]_{\sim}\ar[r]^-{\rho}&
(\oplus_{i=1}^N H^*(\CP(1)_{(i)})\otimes H^*(S^3_{i}),0),
\ar[u]_{\cM_{\psi}}^{\sim}
}\end{equation}
where 
\begin{itemize}
\item  the top horizontal map is a surjective model of the inclusion $\partial E_{D}\to \ov{V}_{\reg}$,
\item the right-hand vertical map is  a model of $\psi$,
\item the \cdga, $\cN$, is the fibered product of these two maps.
\end{itemize}
By construction, the two vertical maps are quasi-isomorphisms.
For the determination of the model $\cM_{\psi}$ of $\psi$, we start from the commutative diagram
$$\xymatrix{
\sqcup_{i=1}^N  \partial(\CP(1)_{(i)}\times D^4_{(i)})\ar[d]&
\partial E_{D}\ar[l]_-{\psi}\ar[d]^-{\pi_{S}}\\
\sqcup_{i=1}^N \CP(1)_{(i)}&
\sqcup_{i=1}^N\CP(1)_{(i)}\times \CP(1)_{(i)}\ar[l]_-{\pr_{1}}.
}$$
We denote by $(1_{i}, \alpha_{i})$ a basis of $H^*(\CP(1)_{(i)})$ and $(\ov{1}_{i}, \ov{\alpha}_{i})$ a basis of $H^*(S^3_{(i)})$. 
The left-hand vertical map being a  trivial bundle, we have $\cM_{\psi}(\alpha_{i})=a_{i}$. For degree reasons, the image of $\ov{\alpha}_{i}\in H^3(S^3_{(i)})$ can be written as
$\cM_{\psi}(\ov{\alpha}_{i})=\lambda_{i}a_{i}\theta_{i}+\mu_{i}b_{i}\theta_{i}$, with $\lambda_{i},\mu_{i}\in\Q$. This image  being a cocycle, we have
$$\lambda_{i}a_{i}(a_{i}+b_{i})+\mu_{i}b_{i}(a_{i}+b_{i})=(\lambda_{i}+\mu_{i})a_{i}b_{i}=0,$$
which implies $\lambda_{i}=-\mu_{i}$. As we are working over the field $\Q$, we may suppose $\mu_{i}=1$ and
$\cM_{\psi}(\ov{\alpha}_{i})=(b_{i}-a_{i})\theta_{i}$.

 In diagram (\ref{equa:modelnormalbundle}), we filter $\cN'$ and $\cN$ as follows,
 \begin{eqnarray*}
 \cF^0\cN'=\{x\in \cN'\mid \rho'(x)\in \oplus_{i=1}^N
H^*(\CP(1)_{(i)})\otimes H^*(\CP(1)_{(i)}) \}
 &\subseteq &
 \cF^{-1}\cN'=\cN',\\
  \cF^0\cN=\{x\in \cN\mid \rho(x)\in  \oplus_{i=1}^N
H^*(\CP(1)_{(i)})\}
 &\subseteq &
  \cF^{-1}\cN=\cN.
 \end{eqnarray*}
These filtrations are of length~1 and, thanks to the commutativity of (\ref{equa:modelnormalbundle}), the map $\cN\to \cN'$ is compatible with the filtrations and induces a morphism of spectral sequences, $\Phi_{j}^{r,s}\colon E_{j}^{r,s}\cN \to E_{j}^{r,s}\cN'$, $j=0,1$. 

The \cdga~$(E_{1}^{r,s}\cN',d_{1})$ is determined 
\index{Morgan's model}
by Morgan in \cite[Theorem 2.3]{MR516917}:  $E_{1}^{0,q}\cN'=H^q(W)$ and
$E_{1}^{-1,q}\cN'= \oplus_{i=1}^N
s^{-1} H^q(\CP(1)_{(i)}\times \CP(1)_{(i)})$,
with the convention $(s^r A)^k=A^{r+k}$ for any $r\in\Z$.
The differential $d_{1}$ comes from the Gysin sequence of  the sphere-bundle associated to the normal bundle. Moreover, with \cite[Theorem 10.1]{MR516917},  we know also that the \cdga~$(E_{1}^{*,*}\cN',d_{1})$ is a rational model, $\cM'_{\ov{V}_{\reg}}$, of the complement~$\ov{V}_{\reg}$, i.e.,
$$\cM'_{\ov{V}_{\reg}}=(H^*(W)\oplus \oplus_{i=1}^N s^{-1}(H^*(\CP(1)_{(i)}\times \CP(1)_{(i)}),d),$$
with a law of algebra defined by,
\begin{itemize}
\item $\alpha\cdot\beta=\alpha\cup \beta$, if $\alpha, \beta\in H^*(W)$,
\item $s^{-1}\alpha\cdot s^{-1}\beta=0$, if $s^{-1}\alpha, s^{-1}\beta\in \oplus_{i=1}^N s^{-1}(H^*(\CP(1)\times \CP(1))_{(i)}$,
\item $\alpha\cdot s^{-1}\beta=s^{-1}(g^*(\alpha)\cup \beta)$, if $\alpha\in H^*(W)$ and $s^{-1}\beta\in \oplus_{i=1}^N s^{-1}(H^*(\CP(1)\times \CP(1))_{(i)}$, 
with the previous embedding, $g\colon \sqcup_{i=1}^N \CP(1)_{(i)}\times \CP(1)_{(i)}\to W$, 
(\ref{equa:2resolutions}).
\end{itemize}

Concerning $\cN$, the \cdga~$(E_{0}^{0,*}\cN,d_{0})$ appears as the next (homotopy) pullback,
\begin{equation}\label{equa:ssdegree0}
\xymatrix{
\cN\ar@{->>}[r]
&
(\oplus_{i=1}^N H^*(\CP(1)_{(i)})\otimes H^*(S^3_{(i)}),0)\\
(E_{0}^{0,*}\cN,d_{0})\ar[u]
\ar[r]
&
(\oplus_{i=1}^N H^*(\CP(1)_{(i)}),0).\ar[u]
}
\end{equation}
This square is a model of the (homotopy) pushout
\begin{equation}\label{equa:prethom}
\xymatrix{
\ov{V}_{\reg}\ar[d]&
\sqcup_{i=1}^N \CP(1)_{(i)}\times S^3_{(i)}\ar[l]\ar[d]\\
V&
\sqcup_{i=1}^N \CP(1)_{(i)}\times D^4_{(i)}.\ar[l]_-{f}
}
\end{equation}
Therefore $(E_{0}^{0,*}\cN,d_{0})$ is a model of $V$ and $E_{1}^{0,*}\cN=H(E_{0}^{0,*}\cN,d_{0})=H^*(V)$.
We specify now the term $E_{1}^{-1,*}\cN$. The short exact sequence,
$(E_{0}^{0,q}\cN,d_{0})\hookrightarrow (\cN^q,d_{0})\twoheadrightarrow (E_{0}^{-1,q+1}\cN,d_{0})$,
implies $H^q(E_{0}^{-1,*}\cN,d_{0})\cong H^{q+1}(V,\ov{V}_{\reg})$. Using excision and the Thom isomorphism, we get,
\begin{eqnarray*}
H^{q+1}(V,\ov{V}_{\reg})&\cong &
H^{q+1}(\sqcup_{i=1}^N \CP(1)_{(i)}\times D^4_{(i)},  \sqcup_{i=1}^N \CP(1)_{(i)}\times S^3_{(i)})\\
&\cong&
 \oplus_{i=1}^N H^{q-3}( \CP(1)_{(i)}).
\end{eqnarray*}
In conclusion, as \cdga, we have
$$(E_{1}\cN,d_{1})=(H^*(V)\oplus\oplus_{i=1}^N s^{-3}H^*(\CP(1)_{(i)}),d_{1}),$$
with a  product given by
\begin{itemize}
\item $\alpha\cdot\beta=\alpha\cup \beta$, if $\alpha, \beta\in H^*(V)$,
\item $s^{-3}\alpha\cdot s^{-3}\beta=0$, if $s^{-3}\alpha, s^{-3}\beta\in \oplus_{i=1}^N s^{-3} H^*(\CP(1)_{(i)})$,
\item $\alpha\cdot s^{-3}\beta=s^{-3}(f^*(\alpha)\cup \beta)$, if $\alpha\in H^*(V)$, $s^{-3}\beta\in \oplus_{i=1}^N s^{-3}H^*(\CP(1)_{(i)})$,
\end{itemize}
and a differential $d_{1}$ which is the Gysin differential. 
This differential can be seen as the transfer map of the embedding $f\colon \sqcup_{i=1}^N \CP(1)_{(i)}\to V$ and $d_{1}(s^{-3}\alpha_{i})$ is the fundamental class of $H^*(V)$, through the isomorphisms,
$$H^2(\CP(1)_{(i)})
\xrightarrow[]{\cong}
H_{0}(\CP(1)_{(i)})
\xrightarrow[f_{*}]{\cong}
H_{0}(V)
\xrightarrow[]{\cong}
H^{6}(V).$$

The map of \cdga's, $\Phi\colon (E_{1}\cN,d_{1})\to (E_{1}\cN',d_{1})$, is defined by $H^*(\Bl)\colon H^*(V)\to H^*(W)$ and $\Phi(s^{-3}1_{i})= s^{-1}(b_{i}-a_{i})$, $\Phi(s^{-3}\alpha_{i})=s^{-1}\theta_{i}(b_{i}-a_{i})$. As this map is a quasi-isomorphism, we  choose $(E_{1}\cN,d_{1})$ as \cdga~model of $\ov{V}_{\reg}$. With the description of  $H^*(V)$  in \cite[Theorem 3.2]{MR2240431}, we have proved the next statement. (We notice also that the cohomology of $V$ coincides with the intersection cohomology for the middle perversity of $V$, see \cite{MR696691}.)

\begin{lemma}\label{lem:modelofVreg}
A model, $\cM_{\ov{V}_{\reg}}=\oplus_{k=0}^6 \cM_{\ov{V}_{\reg}}^k$, of the complement,
$\ov{V}_{\reg}=V\backslash \sqcup_{i=1}^N  \CP(1)_{(i)}$,
is described in the next array, the differential being specified below.

\renewcommand{\arraystretch}{1.5}
\setlength{\tabcolsep}{.4cm}
\begin{center}
   \begin{tabular}{| c || c | c|}
     \hline
     & $H^*(V)$&
     $\oplus_{i=1}^N s^{-3}H^*(\CP(1)_{(i)})$\\\hline
      $k=0$&$\Q$&$0$\\\hline
       $k=1$&$0$&$0$\\\hline
        $k=2$&$\Q[\omega,\cE]$&$0$\\\hline
         $k=3$&$\Q[\cA,\cA']$&$\Q[s^{-3}1_{i}\mid 1\leq i \leq N]$\\\hline
          $k=4$&$\Q[\omega^2,\cE']$&$0$\\\hline
     $k=5$&$0$&$\Q[s^{-3}\alpha_{i}\mid 1\leq i \leq N]$\\\hline
      $k=6$ & $\Q[\omega^3]$ & $0$\\
     \hline
   \end{tabular}
 \end{center}
 \end{lemma}

Most of the law of algebra of $H^*(V)$ can be deduced from Poincar\'e duality and the Hard Lefschetz theorem (\cite{MR751966}). (For instance, we choose Poincar\'e dual bases, $\cA=(a_{1},\ldots,a_{v})$ and $\cA'=(a'_{1},\ldots,a'_{v})$.) 
A part stays in the shadow, as the product of two elements of $\cE$, but this information
plays no role for the proof of the intersection-formality. The structure of $H^*(V)$-module on the generators $s^{-3}1_{i}$, $s^{-3}\alpha_{i}$, has been described before the statement.

As for the differential, we already know that  $d(s^{-3}\alpha_{i})=\omega^3$ for any $i=1,\ldots,N$.
 We  specify now the differential $d^3\colon \Q[s^{-3}1_{i}\mid 1\leq i \leq N] \to \Q[\omega^2,\cE']$.
 With $\im d^3=\Q[\cE']$, we  choose a basis of $\Q[\cE']$, extracted from $(d^3(s^{-3}1_{i}))_{1\leq i\leq N}$, 
and suppose that is $(d^3(s^{-3}1_{j}))_{1\leq j\leq p}$. We set $e'_{j}=d^3(s^{-3} 1_{j})$, for $j=1,\ldots,p$. 
 As basis of $\ker d^3$, we may choose a family 
$(s^{-3}1_{k}'=s^{-3}1_{k}+\sum_{j=1}^p \nu_{kj}\,s^{-3}1_{j})_{p+1\leq k\leq N}$, $\nu_{kj}\in\Q$. In summary, we have
\begin{equation}\label{equa:d3}
d^3( s^{-3}1_{i})=\left\{
 \begin{array}{cl}
 e'_{i}&\text{ if } 1\leq i\leq p,\\
 -\sum_{j\leq p}\nu_{ij}\,e'_{j}&\text{ if } p+1\leq i\leq N.
 \end{array}\right.
 \end{equation}
 We denote by $\cE=(e_{1},\ldots,e_{p})$  the orthogonal basis of $\cE'$, for the Poincar\'e duality. The equalities (\ref{equa:d3}) are connected to  $f^*(e_{j})=\sum_{i=1}^N \lambda_{ji}\,\alpha_{i}$, $\lambda_{ji}\in\Q$,  as follows.
 By definition of the algebra structure, we have
 $e_{j}\cdot (s^{-3} 1_{i})=s^{-3}(f^*(e_{j})\cup 1_{i})= \lambda_{ji}\,s^{-3}\alpha_{i}$.
Applying the differential on each side of this equality, we get
$$\lambda_{ji}\,\omega^3=\left\{
\begin{array}{cl}
e_{j}\cdot e'_{i}&\text{ if } 1\leq i\leq p,\\
-e_{j}\cdot\sum_{j\leq p}\nu_{ij}\,e'_{j}&\text{ if } p+1\leq i\leq N.
\end{array}\right.$$
This implies
$$f^*(e_{j})=\alpha_{j}-\sum_{k=p+1}^N\nu_{kj}\,\alpha_{k}.$$
(In the case of the Calabi-Yau quintic, we have $N=125$, $p=24$ and $v=1$.)

\medskip
The injection of the link in the regular space,
$L=\sqcup_{i=1}^N \partial(\CP(1)_{(i)}\times D^4_{(i)})
\to
\ov{V}_{\reg}
$,
  is the bottom map of (\ref{equa:VandW}) and its model, given by (\ref{equa:modelnormalbundle}), is described in the next lemma. 
  
    \begin{lemma}\label{lem:modelinclusion}
A model of the inclusion of the link, $L$, in $\ov{V}_{\reg}$,
  $$\mu\colon (\cM_{\ov{V}_{\reg}},d_{f})\to
  (\cM_{L},0)=\oplus_{i=1}^N
( H^*(\CP(1)_{(i)})\oplus s^{-3}H^*(\CP(1)_{(i)}),0)$$
with
\begin{itemize}
\item $\mu=\id$ on $s^{-3}1_{i}$ and $s^{-3}\alpha_{i}$,
\item $\mu(e_{j})=f^*(e_{j})$, $j=1,\ldots,p$,
\item $\mu(1)=+_{i=1}^N 1_{i}$ and $\mu=0$ on the other elements.
\end{itemize}
Concerning the product on $\cM_{L}$, the elements $1_{i}$'s are  neutral elements and we have
$(s^{-3}1_{i})\cdot\alpha_{i}=s^{-3}\alpha_{i}$. In summary, $\cM_{L}$ is the cohomology algebra of $\sqcup_{i=1}^N \CP(1)_{(i)}\times S_{(i)}^3$.
\end{lemma}

\begin{proof}[Proof of \thmref{thm:nodalformal}]
We transform $\mu$ in a surjective map, in strictly positive degrees, by
$$\mu'=(\cM'_{\ov{V}_{\reg}},d_{f})=(\cM_{\ov{V}_{\reg}}\otimes\land(x_{p+1},\ldots,x_{N},y_{p+1},\ldots, y_{N}),d_{f})
\to (\cM_{L},0),$$
with $|x_{k}|=2$, $|y_{k}|=3$, $d_{f}x_{k}=y_{k}$, $\mu'(x_{k})=\alpha_{k}$, $\mu'(y_{k})=0$.
Now,  \corref{cor:modelisolated} gives an intersection model of $\ov{V}$, 
$$\left(\cM_{\ov{V}}\right)_{\ov{q}}=\cM'_{\ov{V}_{\reg}}\oplus_{\cM_{L}}\tau_{\leq \ov{q}(n)}\cM_{L},$$
together with the canonical inclusions, $\varphi_{\ov{p}}^{\ov{q}}\colon (\cM_{\ov{V}})_{\ov{p}}\to (\cM_{\ov{V}})_{\ov{q}}$, if $\ov{p}\leq \ov{q}$. 
Denote by $\ov{\ell}$ the GM-perversity such that $\ov{\ell}(6)=\ell$. In the next array, we collect the various models and their cohomology.

\renewcommand{\arraystretch}{1.5}
\setlength{\tabcolsep}{.3cm}
 \begin{center}
 \begin{tabular}{| c || c |c |c |c |c |c |c |}
    \hline
k&0&1&2&3&4&5&6\\\hline
$\left(\cM_{\ov{V}}\right)^k_{\ov{0}}$&$\Q$&
0&
$\Q[\omega]$&
$\Q[\cA,\cA',\cB]$&
$(\cM'_{\ov{V}_{\reg}})^4$&
$(\ker \mu')^5$&
$(\cM'_{\ov{V}_{\reg}})^6$\\\hline
$H^k_{\ov{0}}({\ov{V}})$&$\Q$&
0&
$\Q[\omega]$&
$\Q[\cA,\cA',\cB]$&
$\Q[\omega^2,\cE']$&
$0$&
$\Q[\omega^3]$\\\hline
$\left(\cM_{\ov{V}}\right)^k_{\ov{2}}$&$\Q$&
0&
$(\cM'_{\ov{V}_{\reg}})^2$&
$\Q[\cA,\cA',\cB]$&
$(\cM'_{\ov{V}_{\reg}})^4$&
$(\ker \mu')^5$&
$(\cM'_{\ov{V}_{\reg}})^6$\\\hline
$H^k_{\ov{2}}({\ov{V}})$&$\Q$&
0&
$\Q[\omega,\cE]$&
$\Q[\cA,\cA']$&
$\Q[\omega^2,\cE']$&
$0$&
$\Q[\omega^3]$\\\hline
$\left(\cM_{\ov{V}}\right)^k_{\ov{4}}$&$\Q$&
0&
$(\cM'_{\ov{V}_{\reg}})^2$&
$(\cM'_{\ov{V}_{\reg}})^3$&
$(\cM'_{\ov{V}_{\reg}})^4$&
$(\ker \mu')^5$&
$(\cM'_{\ov{V}_{\reg}})^6$\\\hline
$H^k_{\ov{4}}({\ov{V}})$&$\Q$&
0&
$\Q[\omega,\cE]$&
$\Q[\cA,\cA',\cB']$&
$\Q[\omega^2]$&
$0$&
$\Q[\omega^3]$\\\hline
  \end{tabular}
 \end{center}
with
$\cB=(y_{p+1},\ldots,y_{N})$,
$\cB'=(s^{-3}1'_{p+1},\ldots, s^{-3}1'_{N})$.
 Let $\ov{p}\leq \ov{q}$. The morphisms
 $\psi_{\ov{p}}^{\ov{q}}\colon H_{\ov{p}}(\ov{V})\to H_{\ov{q}}(\ov{V})$
are defined by
\begin{itemize}
\item $\psi_{\ov{0}}^{\ov{1}}=\id$ and $\psi_{\ov{1}}^{\ov{2}}$ is the identity except on $\cB$ where it takes the value~0,
\item $\psi_{\ov{2}}^{\ov{3}}$ is the identity except on $\cE'$ where it takes the value~0, and $\psi_{\ov{4}}^{\ov{3}}=\id$.
\end{itemize}

We construct now a perverse \cdga's map, $\chi_{\bullet}$, such that the following diagram commutes.
\begin{equation}\label{equa:lechi}
\xymatrix{
(\cM_{\ov{V}})_{\ov{0}}\ar[r]_{\varphi_{\ov{0}}^{\ov{2}}} 
\ar[d]_{\chi_{\ov{0}}}
&
(\cM_{\ov{V}})_{\ov{2}}\ar[r]_{\varphi_{\ov{2}}^{\ov{4}}}
\ar[d]^{\chi_{\ov{2}}}
&
(\cM_{\ov{V}})_{\ov{4}}
\ar[d]^{\chi_{\ov{4}}}
\\
H_{\ov{0}}(\ov{V})\ar[r]^{\psi_{\ov{0}}^{\ov{2}}}
&H_{\ov{2}}(\ov{V})\ar[r]^{\psi_{\ov{2}}^{\ov{4}}}
&H_{\ov{4}}(\ov{V})
}
\end{equation}
In each degree, we choose a supplementary subspace, $\cZ^c$, of the subspace of cocycles, $\cZ$, and vector space bases $\B_{\cZ^c}$, $\B_{\cZ}$, for each of them, with
$\nu\in\{1,\ldots,v\}$, $j\in\{1,\ldots,p\}$,
$k,k',k''\in\{p+1,\ldots,N\}$.
\begin{itemize}
\item Degree~2: $\B_{\cZ}=(\omega,\cE)$, $\B_{\cZ^c}=(x_{k})$.
\item Degree~3: $\B_{\cZ}=(\cA,\cA',y_{k},s^{-3}1'_{k})$, $\B_{\cZ^c}=(s^{-3}1_{j})$.
\item Degree~4: $\B_{\cZ}=(\omega^2, \cE')$, 
$\B_{\cZ^c}=(\omega x_{k},e_{j}x_{k}, x_{k}x_{k'})$ with $k\leq k'$.
\item Degree~6: $\B_{\cZ}=(\omega^3, \omega y_{k}, a_{\nu}y_{k}, a'_{\nu}y_{k},
s^{-3}1'_{k}y_{k'}, y_{k''}y_{k'''},e'_{j}x_{k}-s^{-3}1_{j}y_{k})$, with $k''<k'''$ and
 $\B_{\cZ^c}=(\omega^2x_{k},e_{j}x_{k}x_{k'},\omega x_{k}x_{k'},x_{k}x_{k'}x_{k''},s^{-3}1_{j}y_{k})$,
 with $k\leq k'\leq k''$.
\end{itemize}
The description of the elements of degree~5 is not necessary because all of them are sent to~0 by $\chi_{\bullet}$.
For any element, $a$, in one of these bases, we set $\chi_{\bullet}(a)=[a]$, if $a\in\B_{\cZ}$, and $\chi_{\bullet}(a)=0$, if $a\in\B_{\cZ^c}$. 
(In this definition, we mean that $\chi_{\ov{q}}(a)$ is defined if the element $a$ is of perverse degree $\ov{q}$.)
The compatibility of $\chi_{\bullet}$ with the differentials and the commutativity of (\ref{equa:lechi}) are direct.
We need only to verify,
\begin{equation}\label{equa:chiproduct}
\chi_{\bullet}(a)\chi_{\bullet}(b)=\chi_{\bullet}(ab).
\end{equation}
This is immediate if $a,\,b\in \B_{\cZ}$. 
Consider now $a\in\B_{\cZ^c}$. If the product $ab\in \B_{\cZ^c}$, the equality (\ref{equa:chiproduct}) is satisfied. 
An inspection of the previous list of bases shows that the only case where $a b\notin\B_{\cZ^c}$ is $e'_{j}x_{k}$. 
We have $\chi_{\bullet}(x_{k})=0$, by definition, and
$$
\chi_{\bullet}(e'_{j}x_{k})=
\chi_{\bullet}(e'_{j}x_{k}-s^{-3}1_{j}y_{k})+\chi_{\bullet}(s^{-3}1_{j}y_{k})=
\chi_{\bullet}(e'_{j}x_{k}-s^{-3}1_{j}y_{k}).
$$
For the determination of the cohomology class associated to
$e'_{j}x_{k}-s^{-3}1_{j}y_{k}$, we first observe that this element is in the kernel of $\mu'$ and thus has perverse degree $\ov{0}$. The triviality of this class comes from
$d_{f}(s^{-3}1_{j}x_{k})=e'_{j}x_{k}-s^{-3}1_{j}y_{k}$
and $s^{-3}1_{j}x_{k}\in \ker \mu'$.
\end{proof}
%
%
\appendix
\chapter{Topological setting}
\section{Filtered spaces}\label{sec:filteredspaces}

\begin{quote}
A filtered space is a topological space together with a filtration by closed subsets. This simple notion is sufficient to define an intersection homology   associated to a loose perversity, with
classical properties, as the existence of a Mayer-Vietoris sequence and some particular case of  K\"unneth formula.
\end{quote}

\begin{definition}\label{def:espacefiltré} 
A \emph{filtered  space} is a topological space, $X$, 
\index{Space!filtered|see{Filtered space}}\index{Filtered!space}
together with a filtration by closed subspaces,
$$X_0\subseteq X_1\subseteq\ldots\subseteq X_n=X,$$
such that $X_n\backslash X_{n-1}$ is not empty.
The \emph{formal dimension} of $X$ is $n$, denoted by $d(X)=n$. 
\index{Dimension!formal}
The subspace $X_i$ is called the \emph{$i$-skeleton} of $X$ but the index $i$ and the formal dimension, $d(X)$, are not necessarily related to a notion of geometrical dimension. 

The connected components, $S$, of $X_{i}\backslash X_{i-1}$ are called \emph{strata} of $X$ and we set  $d(S)=i$. 
The strata in $X_n\backslash X_{n-1}$ are called the \emph{regular strata} of $X$.  
We denote by $\cS_X$ the family of non-empty strata of $X$.
The subspace $X_{n-1}$ is the \emph{singular set} and is  also denoted by~$\Sigma$.
\index{Strata}
\index{Strata!regular}\index{Regular!strata}
\index{Set!singular}
\end{definition} 
  
 By convention, we set $X_j=\emptyset$ if $j<0$ and $X_j=X$ if $j>n$. 
 The filtration of a non-empty topological space $X$ by $X_{n-1}=\emptyset$ and $X_{n}\,=X$  is called \emph{trivial} of formal dimension $n$. \index{Filtration!trivial}
If $(X,(X_i)_{0\leq i\leq n})$ and $(Y,(Y_i)_{0\leq i\leq m})$  are filtered spaces, we define  a canonical filtration on the product $X\times Y$  by
$$(X\times Y)_i = \displaystyle \bigcup_{k + j =i} X_k \times Y_j, \text{ for }i \in \{0, \ldots , n+m\}.$$

\begin{definition}\label{def:conicfiltration}
Let  $(X, (X_i)_{0\leq i\leq n})$ be a filtered space, the open cone on $X$ is the quotient  $\mathring{c}X = X \times [0,1[ \big/ X \times \{ 0 \}$. We denote by
$\vartheta$
the cone point.
The \emph{conic filtration,} $(\mathring{c}X)_{0\leq i\leq n+1}$, on the open cone is defined by $\left(\mathring{c}X\right) _i =\mathring{c}X_{i-1}$, with the convention
$\mathring{c}\,\emptyset=\{ \vartheta\}$.
Observe that $\mathring{c}X \backslash \{ \vartheta \}$ is the product  $X \times ]0,1[$, where $]0,1[$ owns the trivial filtration of formal dimension~1.
\index{Filtration!conic}
\end{definition}

This conic filtration generates a canonical filtration on the suspension $\Sigma X$.
The join $X*Y$ of two filtered spaces is also canonically filtered by observing that $X * Y$ is the union of three open sets, $X \times Y \times ]0,1[$, $X \times \mathring{c}Y$, $\mathring{c}X \times Y$, and the filtrations on the respective intersections coincide.

\begin{definition}\label{def:filteredsimplex}
Let  $(X, (X_i)_{0\leq i\leq n})$ be a filtered space and let $\Delta$ be an euclidean simplex.
\index{Simplex!filtered|see{Filtered simplex}}\index{Filtered!simplex}
A  \emph{filtered simplex} is a continous map, $\sigma\colon\Delta\to X$, such that one of the following equivalent properties is satisfied.
\begin{enumerate}[(a)]
\item For all $i\in \{0,\ldots, n\}$, $\sigma^{-1}X_i$ is a face, $[0,1,\ldots,m_{i}]$, of $\Delta$ or $\sigma^{-1}X_i=\emptyset$.
\item There exists a decomposition $\Delta=\Delta^{j_0}\ast\Delta^{j_1}\ast\cdots\ast\Delta^{j_n}$ 
such that
$$
\sigma^{-1}X_{i} = \Delta^{j_0}\ast\Delta^{j_1}\ast\cdots\ast\Delta^{j_i},
$$
for all~$i \in \{0, \ldots, n\}$. We call it \emph{the $\sigma$-decomposition of $\Delta$.}
\end{enumerate}
\emph{Let $R$ be a commutative ring.} We denote by $C_*(X)$ the usual singular chain complex of $X$, with coefficients in $R$, and by $C_*^{\cF}(X)$ the subcomplex generated by the filtered simplices.
\index{Decomposition@$\sigma$-Decomposition}
\end{definition}
The simplices $\Delta^{j_i}$, which arise in \defref{def:filteredsimplex}, can be empty. (An empty set  appears automatically if the filtration of $X$ is stationary  for an index~$i$.)
 We use the convention $\emptyset * Y=Y$, for any space $Y$.

 \begin{definition}\label{def:soliddecomposition}
 \index{Decomposition@$\sigma$-Decomposition!solid}
 Let $X$ be a filtered space.
 The \emph{solid $\sigma$-decomposition of a filtered simplex, $\sigma\colon \Delta\to X$,} is obtained by keeping only the non-empty elements, $\Delta^{j_{i}}$, of the $\sigma$-decomposition, written in the same order, i.e.,
$\Delta=\Delta_{1}\ast\cdots\ast\Delta_{p}$.
 \end{definition}
 
 \begin{definition}\label{def:perversedegreesimplex}
Let  $(X, (X_i)_{0\leq i\leq n})$ be a filtered space and let $\sigma\colon\Delta \to X$ be a filtered simplex. The \emph{perverse degree} of $\sigma$ is the $(n+1)$-uple,
$$\|\sigma\|=(\|\sigma\|_0,\ldots,\|\sigma\|_n),$$  
where $\|\sigma\|_{\ell}$ is the dimension of the smallest skeleton containing $\sigma^{-1} X_{n-\ell}$, with the convention $\|\sigma\|_\ell=-\infty$ if $\sigma^{-1} X_{n-\ell}=\emptyset$.\\
If $\sigma\colon\Delta=\Delta^{j_0}\ast \cdots\ast\Delta^{j_n}\to X$ is such that $\Delta^{j_0} \ast\cdots\ast\Delta^{j_{n-\ell}}\neq \emptyset$, then
 $\|\sigma\|_{\ell}=\dim \sigma^{-1}X_{n-\ell}=\dim (\Delta^{j_0}\ast\cdots\ast\Delta^{j_{n-\ell}})$.
 \index{Perverse!degree}
 \end{definition}

 \begin{example}\label{exam:filteredsimplex}
 The next picture corresponds to a filtered space of dimension 2 with a curve as subspace $X_{1}$ and a singleton as subspace $X_{0}$. We have represented several situations of filtered simplices together with their decomposition and their perverse degree.

 \centerline{\includegraphics[scale=0.24]{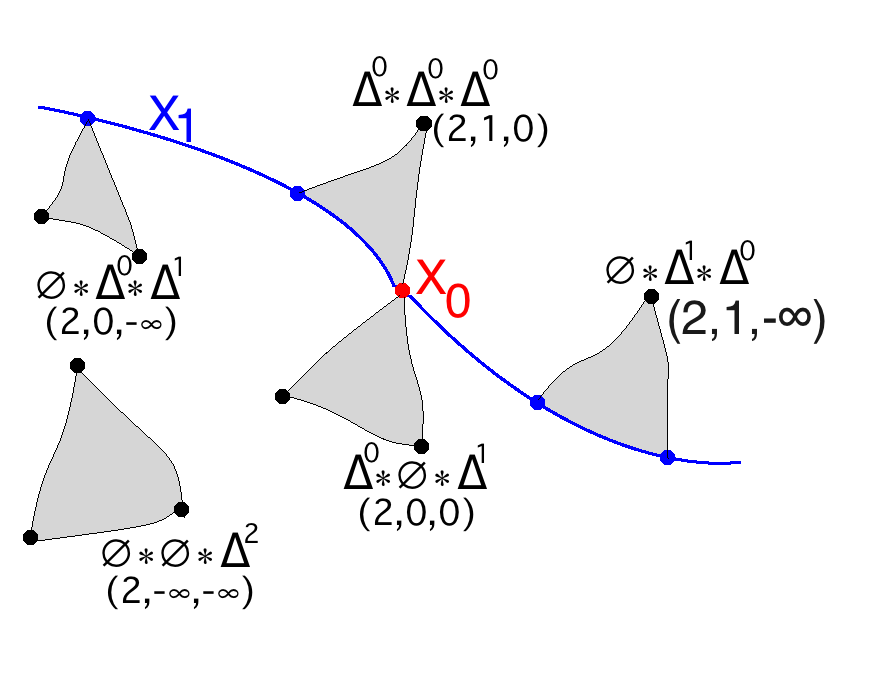}}
 \end{example}

\begin{definition}\label{def:chaineadmis}(\cite{MR572580}, \cite{MR1143404})
Let  $(X, (X_i)_{0\leq i\leq n})$ be a filtered space and let $\ov{p}$ be a loose perversity. 
 A \emph{$\ov{p}$-admissible simplex} is a filtered simplex,
$\sigma\colon\Delta%
\to X$, such that
$$\|\sigma\|_{\ell}\leq \dim \Delta -\ell+\ov{p}(\ell),$$
for all $\ell\in\{0,\ldots,n\}$.
A \emph{chain $c$ is $\ov{p}$-admissible} if there exist $\ov{p}$-admissible simplices, $\sigma_j$, so that
$c=\sum_j\lambda_j\sigma_j$, $\lambda_j\in R$.
\index{Filtered!simplex!admissible for a perversity}
\index{Filtered!chain!admissible for a perversity}
\end{definition}

\begin{remark}\label{rem:perversity0}
A $\ov{p}$-admissible simplex, $\sigma\colon\Delta\to\ob{K}$,  verifies
 $$\dim\Delta=\|\sigma\|_0\leq \dim \Delta -0 + \ov{p}(0).$$
We impose  $\ov{p}(0)=0$ in the definition of loose perversity and this restriction makes vacuous the condition on $\|\sigma\|_0$.
 (On the other side, we may observe also that, with a negative value for $\ov{p}(0)$, we  should not have  any $\ov{p}$-admissible simplex.)
 \end{remark}
 
 Let $\ov{p}$ be a GM-perversity and $(X, (X_i)_{0\leq i\leq n})$ be a filtered space. We check easily (see \remref{rem:petitemaisbien}) that a filtered  0-simplex, $\sigma\colon \Delta^0\to X$,   is $\ov{p}$-admissible if, and only if, $\sigma(\Delta^0)\subset X\backslash X_{n-1}$. This leads naturally to the next definition.
 
 \begin{definition}\label{def:filteredspaceconnected}
 A filtered space, $(X, (X_i)_{0\leq i\leq n})$, is \emph{connected}, 
 \index{Filtered!space!connected}
 if the regular part is it.
 \end{definition}
 
 As usual in this theory, the submodule generated by the admissible $\ov{p}$-chains is not a subcomplex of the singular chains and we need to set the following definition.

\begin{definition}\label{def:chaineintersection} 
Let  $(X, (X_i)_{0\leq i\leq n})$ be a filtered space and let $\ov{p}$ be a loose perversity. 
A  chain $c$ is \emph{an intersection chain} for $\ov{p}$ if $c$ and its boundary $\partial c$ are $\ov{p}$-admissible chains. We denote by $C^{\ov{p}}_*(X)$ the complex of intersection chains, 
with coefficients in a commutative  ring, $R$.
 The associated homology and cohomology are denoted by
$H^{\ov{p}}_*(X)$ and $H_{\ov{p}}^*(X)$  and called \emph{intersection homology and cohomology for $\ov{p}$.} 
\index{Filtered!chain!of intersection for a perversity}
\index{Intersection homology}\index{Intersection cohomology}
\end{definition}

First, we characterize the $\ov{p}$-admissible simplices which are of $\ov{p}$-intersection.

\begin{definition}\label{def:badface}
Let $X$ be a filtered space, $\ov{p}$ be a loose perversity and $\sigma\colon \Delta=\Delta^{j_{0}}\ast\cdots\ast\Delta^{j_{n}}\to X$ be a $\ov{p}$-admissible simplex. We denote by $\ell(\sigma) \in\{0,\ldots,n\}$ the integer such that
\begin{itemize}
\item $\|\sigma\|_{\ell(\sigma) }=\dim\Delta-\ell(\sigma) +\ov{p}(\ell(\sigma) )$, and
\item $\|\sigma\|_{m}<\dim\Delta-m+\ov{p}(m)$, for all $m>\ell(\sigma) $.
\end{itemize}
As $\|\sigma\|_{0}=\dim\Delta-0+\ov{p}(0)$, the integer $\ell(\sigma) $ exists. 
When $\dim(\Delta^{j_{0}}\ast\cdots\ast\Delta^{j_{n-\ell(\sigma) }})\neq \dim\Delta$,
the restriction, $\tau_{\sigma}\colon \Delta^{j_{0}}\ast\cdots\ast\Delta^{j_{n-\ell(\sigma) }}\to X$,
\index{Bad face}
is called \emph{the bad face} of~$\sigma$.
\end{definition}

\begin{example}\label{exam:filteredsimplex2}
\exemref{exam:filteredsimplex} provides an illustration of the previous definition. We choose the perversity $\ov{0}$; a similar argument works for any loose perversity. Here the dimension is equal to ~2. For the simplex $\emptyset\ast\Delta^1\ast\Delta^0$, a computation gives:
$$
\dim\Delta-m+\ov{0}(m)=\left\{
\begin{array}{cl}
2& \text{ if } m=0,\\
1& \text{ if } m=1,\\
0&\text{ if } m=2,
\end{array}\right.
\,\text{ and }
\|\sigma\|_{m}=\left\{
\begin{array}{cl}
2& \text{ if } m=0,\\
1& \text{ if } m=1,\\
-\infty&\text{ if } m=2.
\end{array}\right.
$$
Thus, by definition, its bad face is $\emptyset\ast\Delta^1\ast\emptyset$. 
To be complete, mention that $\emptyset\ast\Delta^0\ast\Delta^1$, $\Delta^0\ast\emptyset\ast\Delta^1$ and $\emptyset\ast\emptyset\ast\Delta^2$ have no bad face. Also, the bad face of $\Delta^0\ast\Delta^0\ast\Delta^0$ is $\Delta^0\ast\emptyset\ast\emptyset$.
\end{example}

\begin{proposition}\label{prop:badface}
Let $X$ be a filtered space, $\ov{p}$ be a loose perversity and $\sigma\colon \Delta\to X$ be a $\ov{p}$-admissible simplex.
\begin{enumerate}[(a)]
\item The simplex $\sigma$ is of $\ov{p}$-intersection if, and only if, it does not contain bad faces.
\item A codimension~1 face   of $\sigma$ is not $\ov{p}$-admissible if, and only if, it contains the bad face of $\sigma$.
\item Let $\sigma'\colon \Delta\to X$ be a $\ov{p}$-admissible simplex of $X$, having a not $\ov{p}$-admissible face $\sigma''$ (of codimension~1) in common with $\sigma$. Then $\sigma''$ contains the bad face of $\sigma$, i.e., $\sigma$ and $\sigma'$ have the same bad face.
\end{enumerate}
\end{proposition}

\begin{proof}
Recall that, by definition, $\sigma\colon \Delta=\Delta^{j_{0}}\ast\cdots\ast\Delta^{j_{n}}\to X$  is $\ov{p}$-admissible if we have
$$\dim\Delta^{j_{0}}\ast\cdots\ast\Delta^{j_{n-m}}=\|\sigma\|_{m}\leq \dim\Delta-m+\ov{p}(m),$$
for all $m\in\{1,\ldots,n\}$. 
Let $F_{i}$ be a face of codimension~1 of $\Delta^{j_{i}}$, with $\Delta^{j_{i}}\neq\emptyset$.
The restriction
$\sigma_{i}\colon \Delta^{j_{0}}\ast\cdots\ast \Delta^{j_{i-1}}\ast F_{i}\ast\Delta^{j_{i+1}}\ast\cdots\ast\Delta^{j_{n}}\to X$
of $\sigma$ is $\ov{p}$-admissible if, and only if,
$$\|\sigma_{i}\|_{m}\leq \dim\Delta-1-m+\ov{p}(m),$$
for all $m\in\{1,\ldots,n\}$. This perverse degree verifies
$$\|\sigma_{i}\|_{m}=\left\{
\begin{array}{ll}
\|\sigma\|_{m}-1 &\text{ if } m\leq n-i,\\
\|\sigma\|_{m} &\text{ if } m>n-i.
\end{array}\right.$$
As $\|\sigma\|_{m}\leq \dim\Delta-1-m+\ov{p}(m)$ if $m>\ell(\sigma)$, the simplex $\sigma_{i}$ is $\ov{p}$-admissible if, and only if,
\begin{equation}\label{equa:mpervers}
\|\sigma\|_{m}<\dim\Delta-m+\ov{p}(m),
\end{equation}
for any $m$ and $i$ such that $n-i<m\leq \ell(\sigma)$ and $\Delta^{j_{i}}\neq\emptyset$.

(a) Inequality (\ref{equa:mpervers}) cannot be verified for $m=\ell(\sigma)$, because of the next contradiction,
\begin{equation}\label{equa:contrapervers}
\dim\Delta-\ell(\sigma)+\ov{p}(\ell(\sigma))=\|\sigma\|_{\ell(\sigma)}<\dim\Delta-\ell(\sigma) +\ov{p}(\ell(\sigma)).
\end{equation}
Recall that $\tau_{\sigma}$ denotes the bad face of $\sigma$, relatively to $\ov{p}$, and that $\sigma_{i}$ is a face  of codimension~1 of $\sigma$. The previous observation implies:
\begin{eqnarray*}
\sigma \text{ is of } \ov{p}\text{-intersection}&\Leftrightarrow&
\sigma_{i} \text{ is } \ov{p}\text{-admissible for any } i\in\{0,\ldots,n\} \\ &&\text{ such that } \Delta^{j_{i}}\neq\emptyset\\
&\Leftrightarrow&  \Delta^{j_{i}}=\emptyset \text{ if } n-\ell(\sigma)<i\\
&\Leftrightarrow& \dim\sigma=\dim\tau_{\sigma}.
\end{eqnarray*}
But, by definition of the bad face, the equality $\dim\sigma=\dim\tau_{\sigma}$ is a contradiction.

\smallskip
(b)  Thanks to the conditions (\ref{equa:mpervers}) and (\ref{equa:contrapervers}), we know that 
 the restriction $\sigma_{i}$ is $\ov{p}$-admissible if, and only if, $i\leq n-\ell(\sigma)$ and $\Delta^{j_{i}}\neq \emptyset$. In other words, $\sigma_{i}$ is not $\ov{p}$-admissible if, and only if, the bad face $\tau_{\sigma}$ of $\sigma$ is a face of $\sigma_{i}$. (Observe that $\dim\tau_{\sigma}\neq\dim\sigma$ as required in the definition of a bad face.)

\smallskip
(c) 
Let $\sigma'\colon \Delta= \Delta^{k_{0}}\ast\cdots\ast\Delta^{k_{n}}\to X$ and
 $\sigma''\colon \Delta^{j_{0}}\ast\cdots\ast\Delta^{j_{i-1}}\ast F_{i}\ast\Delta^{j_{i+1}}\ast\cdots\ast\Delta^{j_{n}}\to X$ be the not $\ov{p}$-admissible  face of $\sigma'$, in common with $\sigma$. As proved in (b), we have $i> n-\ell(\sigma)$.
For $*\leq n-\ell(\sigma)$, if the integers $k_{\ast}$ and $j_{\ast}$ are not equal,  they differ only for one index, $m\in \{0,\ldots,n-\ell(\sigma)\}$, for which we have  $k_{m}=j_{m}+1$. 
By definition, the $\ov{p}$-admissibility of $\sigma'$ implies
$\|\sigma'\|_{\ell(\sigma)}\leq \dim\Delta-\ell(\sigma)+\ov{p}(\ell(\sigma))$ and we have:
\begin{eqnarray*}
\dim\Delta-\ell(\sigma)+\ov{p}(\ell(\sigma))
\geq \|\sigma'\|_{\ell(\sigma)}
&=&\dim (\Delta^{k_{0}}\ast\cdots\ast\Delta^{k_{n-\ell(\sigma)}})\\
&=&\dim (\Delta^{j_{0}}\ast\cdots\ast\Delta^{j_{n-\ell(\sigma)}})+1= \|\sigma\|_{\ell(\sigma)}+1\\
&=&\dim\Delta-\ell(\sigma)+\ov{p}(\ell(\sigma))+1.
\end{eqnarray*}
Thus, the integers $k_{*}$ and $j_{*}$ are equal for $*\leq n-\ell(\sigma)$ and the bad face of $\sigma$ is a face of $\sigma''$.
\end{proof}

In this context of filtered spaces, we prove the existence of a Mayer-Vietoris sequence
\index{Mayer-Vietoris sequence}
 and recall the intersection homology of a cone.

\begin{proposition}\label{prop:propertiesintersectionhomology}
Let $X$
be a filtered space of formal dimension $d(X)=n$ and let $\ov{p}$ be a loose perversity. Then the following properties are satisfied, for intersection homology with coefficients in a commutative ring, $R$.
\begin{enumerate}[(i)]
\item If $\cU=\{U,V\}$ is an open cover of $X$, there exists a long exact sequence
$$\ldots\to H^{\ov{p}}_{*}(U\cap V)\to
 H^{\ov{p}}_{*}(U)\oplus H^{\ov{p}}_{*}(V)\to
 H^{\ov{p}}_*(X)\to H^{\ov{p}}_{*-1}(U\cap V)\to\ldots
$$
where $U\cap V$, $U$, $V$ are endowed with the induced filtration.
\item The canonical projection $\pr\colon \R\times X\to X$ induces an isomorphism
$$H^{\ov{p}}_*(\pr)\colon H^{\ov{p}}_*(\R\times X)\to  H^{\ov{p}}_*(X),$$
whose inverse is induced by the  inclusion
$j_{0}\colon X\to \R\times X$, $x\mapsto (0,x)$.
\item If $U_0\subset U_1\subset \cdots$ are open sets of $X$ so that $X=\cup_iU_i$, then the natural map
$\lim_{\to} H^{\ov{p}}_*(U_i)\to H^{\ov{p}}_*(X)$ is an isomorphism.
\item If $M$ is a topological manifold, then the following sequence is split exact for each i,
$$
\xymatrix@=16pt{
0\ar[r]&(H_*(M)\otimes H_*^{\ov{p}}(X))_i\ar[r]&
H^{\ov{p}}_i(M\times X)\ar[r]&
(H_*(M)\ast H^{\ov{p}}_*(X))_{i-1}\ar[r]&0.
}$$
\item Suppose $X$ is compact and $\ov{p}$ is a perversity. 
If the cone $\mathring{c}X$  is endowed with the conic filtration, 
\index{Intersection homology!of a cone}
we have
$$H^{\ov{p}}_k(\mathring{c}X)=\left\{
\begin{array}{cl}
H^{\ov{p}}_k(X)&\text{ if } k\leq n-\ov{p}(n+1)-1
,\\[.2cm]
0&\text{ if } k\geq n-\ov{p}(n+1)\text{ and } k\neq 0,\\[.2cm]%
R&\text{ if } k\geq n-\ov{p}(n+1) \text{ and } k=0.
\end{array}\right.$$
\end{enumerate}
\end{proposition}

\begin{proof}
\begin{enumerate}[(i)]
\setcounter{enumi}0
\item This is a consequence of \lemref{lem:petitechaine}, as in the classical case.
\end{enumerate}

\begin{enumerate}[(i)]
\setcounter{enumi}1
\item  The map $f\colon \R\times X\times [0,1]\to \R\times X$, $(u,x,t)\mapsto (ut,x)$ induces a chain map
$C_{*}^{\ov{p}}(\R\times X\times [0,1])\to C_{*}^{\ov{p}}(\R\times X)$
and thus a homomorphism
$H_{*}^{\ov{p}}(\R\times X\times [0,1])\to H_{*}^{\ov{p}}(\R\times X)$.
With the notation and the result of \lemref{lem:homotopie1}, as  $f\circ \iota_{0}=j_{0}\circ \pr$ and $f\circ \iota_{1}=\id$, we deduce 
$H^{\ov{p}}_{*}(j_{0})\circ H_{*}^{\ov{p}}(\pr)=\id$. This equality and $\pr\circ j_{0}=\id_{X}$ imply that 
$H^{\ov{p}}_{*}(j_{0})$ and $H_{*}^{\ov{p}}(\pr)$ are inverse of each other.
\end{enumerate}

\begin{enumerate}[(i)]
\setcounter{enumi}2
\item This is obvious.
\end{enumerate}

\begin{enumerate}[(i)]
\setcounter{enumi}3
\item A careful reading of the proof of \cite[Theorem 4]{MR800845} shows that this particular K\"unneth formula exists if properties (i), (ii) and (iii) are satisfied.
\end{enumerate}

\begin{enumerate}[(i)]
\setcounter{enumi}4
\item Let $X$ be a compact filtered space and $\mathring{c}X$ be the associated open cone, endowed with the conic filtration. The next computation is already done in previous works.
\end{enumerate}
Let $\sigma\colon \Delta^k\to \mathring{c}X$ be a $\ov{p}$-admissible filtered simplex. We consider two cases.

$\bullet$ \emph{Let $k=\dim\Delta\leq n-\ov{p}(n+1)$}. As
$$\|\sigma\|_{n+1}\leq k-(n+1)+\ov{p}(n+1)<0,$$
we have $\sigma^{-1}(\mathring{c}X)_0=\emptyset$ and the simplex $\sigma$ does not meet the cone point, $\vartheta$, of $\mathring{c}X$. This implies $C^{\ov{p}}_{\leq n-\ov{p}(n+1)}(\mathring{c}X)=
C^{\ov{p}}_{\leq n-\ov{p}(n+1)}(X\times ]0,1[)$ and property (ii) implies 
$$H^{\ov{p}}_k(\mathring{c}X)=H^{\ov{p}}_k(X\times ]0,1[)= H^{\ov{p}}_k(X),$$
for all $k\leq n-\ov{p}(n+1)-1$.

$\bullet$ \emph{Let $k=\dim\Delta\geq n-\ov{p}(n+1)$.} In this case, we show that the classical homotopy operator, used to prove the contractibility of the singular chain complex of a cone, is compatible with the intersection setting. If $s\in [0,1]$ and $\langle x,t\rangle \in \mathring{c}X$, we set $s\cdot \langle x,t\rangle=\langle x,st\rangle$. If $\sigma\colon \Delta\to \mathring{c}X$ is a simplex, we define $c\sigma\colon \{\vartheta\}\ast\Delta\to \mathring{c}X$ by
$$c\sigma(sa+(1-s)\vartheta)=s\cdot \sigma(a).$$
As $X$ is compact, this map is 
continuous
 and we extend it in a linear map
$c\sigma\colon C_*(\mathring{c}X)\to C_*( \mathring{c}X)$. 
Let $i\in\{1,\ldots,n+1\}$ and $\sigma\colon \Delta\to \mathring{c}X$ be a $\ov{p}$-admissible simplex. We have
\begin{eqnarray*}
(c\sigma)^{-1}((\mathring{c}X)_{n+1-i}))&=&\{sa+(1-s)\vartheta\mid s\cdot \sigma(a)\in (\mathring{c}X)_{n+1-i}
=\mathring{c}X_{n-i}\}\\
&=&\{\vartheta\}\ast\sigma^{-1}(\mathring{c}X_{n-i}),
\end{eqnarray*}
which is a face of $\Delta$. 
If $\sigma^{-1}(\mathring{c}X_{n-i})\neq\emptyset$, we have 
\begin{eqnarray*}
\dim (c\sigma)^{-1}((\mathring{c}X)_{n+1-i})&=&1+\dim \sigma^{-1}(\mathring{c}X_{n-i})\leq 1+\dim \Delta-i+\ov{p}(i)\\
&=&
\dim (\{\vartheta\}\ast\Delta)-i+\ov{p}(i).
\end{eqnarray*}
Therefore, the simplex $c\sigma$ is $\ov{p}$-admissible. 

If $\sigma^{-1}(\mathring{c}X_{n-i})=\emptyset$, then 
$\dim(c\sigma)^{-1}((\mathring{c}X)_{n+1-i})=0$ and, by hypothesis, we have
\begin{eqnarray*}
0&\leq& k-n+\ov{p}(n+1)=k+1-(n+1)+\ov{p}(n+1)\\
&\leq& k+1-i+\ov{p}(i),
\end{eqnarray*}
because $\ov{p}$ is a perversity. The last expression is equal to
$\dim(\{\vartheta\}\ast\Delta)-i+\ov{p}(i)$ and the simplex $c\sigma$ is $\ov{p}$-admissible.

Now we compute $H_k^{\ov{p}}(\mathring{c}X)$ by distinguishing two cases.
\begin{itemize}
\item[---] $k\neq 0$. The result is clear for $k<0$ and we may suppose $k>0$. For any $\eta\in C_k^{\ov{p}}(\mathring{c}X)$, we have $\partial c\eta+c\partial\eta=\eta$, which implies $H_k^{\ov{p}}(\mathring{c}X)=0$.
\item[---] $k=0$. For any $\eta\in C_0^{\ov{p}}(\mathring{c}X)$, we have the formula
$\partial c\eta=\eta-\vartheta$. This implies that all the generators of $C_0^{\ov{p}}(\mathring{c}X)$ are identified in homology and, therefore, $H_0^{\ov{p}}(\mathring{c}X)=R$, because $X_{n}\backslash X_{n-1}\neq\emptyset$.
\end{itemize}
\end{proof}

\begin{lemma}\label{lem:homotopie1}
Let $\ov{p}$ be a loose perversity.
\index{Homotopy!and intersection homology}
For any filtered space, $X$, the canonical injections, $\iota_{0},\,\iota_{1}\colon X\to X\times I$, defined by
$\iota_{k}(x)=(x,k)$ for $k=0,\,1$, induce the same homomorphism in intersection homology,
 $H_{*}^{\ov{p}}(\iota_{0})=H_{*}^{\ov{p}}(\iota_{1})\colon H_{*}^{\ov{p}}(X)\to H_{*}^{\ov{p}}(X\times [0,1])$.
\end{lemma}

\begin{proof}
 In the non-stratified setting, the proof (see \cite[Th\'eor\`eme 5.33]{1198.55001} for instance) uses a chain homotopy between $C_*(\iota_0)$ and $C_*(\iota_1)$, defined as follows.
 If $\Delta^m=\langle e_0,\ldots,e_m\rangle$, we denote by $a_j=(e_j,0)$ and $b_j=(e_j,1)$ the points of $\Delta^m\times I$ corresponding to vertices. A singular $(m+1)$-chain of $\Delta^m\times I$ is defined by
 $$P=\sum_{j=0}^m (-1)^j \langle a_0,\ldots,a_j,b_j,\ldots,b_{m}\rangle.$$
 This chain generates a homotopy $h\colon C_*(X)\to C_{* +1}(X\times I)$ defined by
 $h(\sigma)=(\sigma\times \id)_{*}(P)$ and which satisfies the expected equality
 $(dh+hd)(\sigma)=C_*(\iota_1)(\sigma)-C_*(\iota_0)(\sigma)$.
 
 Let $\sigma$ be a $\ov{p}$-admissible simplex of $X$, we are reduced to prove that $h(\sigma)$ is a $\ov{p}$-admissible simplex of $X\times [0,1]$. Let $j\in\{0,\ldots,m\}$. We denote by $\tau_j\colon \Delta^{m+1}=\langle v_0,\ldots,v_{m+1}\rangle\to \Delta^{m}\times I$
 the $(m+1)$-simplex defined by
 $(v_0,\ldots,v_j,v_{j+1},\ldots,v_{m+1})\mapsto (a_0,\ldots,a_j,b_j,\ldots,b_m)$. 
Let $F$  be any face of the simplicial decomposition
 $$\Delta^m\times I=\cup_{j=0}^{m}\tau_j(\Delta^{m+1}).$$
  As
 $\tau_j^{-1}(F)$ is identified to $F\cap\tau_j(\Delta^{m+1})$, we have  $\dim\tau_j^{-1}(F)\leq \dim F$, for any $0\leq j\leq m$. From this inequality, we deduce
 \begin{eqnarray*}
 \dim \tau_{j}^{-1}(\sigma\times\id)^{-1}(X_{n-i}\times I)
 &\leq &
 \dim \tau_{j}^{-1}(\sigma^{-1}(X_{n-i})\times I)\\
 &\leq &
 \dim (\sigma^{-1}(X_{n-i})\times I)\\
 &\leq &
 m-i+\ov{p}(i)+1.
 \end{eqnarray*}
 By definition of the $\ell$-perverse degree of $h(\sigma)$, we obtain,
  \begin{eqnarray*}
  \|h(\sigma)\|_{\ell}
  &\leq &
  \max_{j}\|(\sigma\times \id)\circ \tau_{j}\|_{\ell}\\
  &\leq &
  m+1-\ell+\ov{p}(\ell)=\dim\Delta^{m+1}-\ell+\ov{p}(\ell),
   \end{eqnarray*}
   and $h(\sigma)$ is $\ov{p}$-admissible.
 \end{proof}

\begin{lemma}\label{lem:petitechaine}
Let $\cU=(U,V)$ be an open cover of the filtered space, $X$, and let $\ov{p}$ be a loose perversity.
We define a homomorphism,
$\varphi\colon C_{*}^{\ov{p}}(U)\oplus  C_{*}^{\ov{p}}(V)\to  C^{\ov{p}}_*(X)$,
by 
$\varphi(\xi_{1},\xi_{2})=\xi_{1}+\xi_{2}$. Then the following properties are satisfied.
\begin{enumerate}[(i)]
\item For any $\xi\in C_{*}^{\ov{p}}(X)$, there exists an integer $k\geq 0$, such that the iterated barycentric subdivision of $\xi$ verifies $(sd)^k\xi\in\im\varphi$.
\item The injection $\im\varphi\hookrightarrow C_{*}^{\ov{p}}(X)$ induces an isomorphism in homology. 
\end{enumerate}
\end{lemma}

\begin{proof}
We adapt the classical proof to the intersection setting (see also \cite{MR2276609} and \cite{MR1245833}). For that, we refer to the presentation  of  \cite[Section 4 of Chapter 4]{MR0210112} or \cite[Section 2.1]{MR1867354}. The key is the construction of  chain maps,
$sd\colon C_*^{\ov{p}}(X)\to C_*^{\ov{p}}(X)$ and  $T\colon C_*^{\ov{p}}(X)\to C_{*+1}^{\ov{p}}(X)$, such that $\partial T+T \partial=\id -sd$.  
We proceed in several steps.

$\bullet$ First, we define $sd$ in the intersection setting. 
Denote by $\Delta^{per}(\Delta)$ the subcomplex generated by the  linear simplices, $\xi$, of $\Delta$ such that 
$$\dim\xi^{-1} (F)\leq \dim F,$$ for any  face $F$ of $\Delta$. 
(Recall from \cite[Page 176]{MR0210112} that a  simplex $\sigma\colon \Delta^q\to \Delta^r$is said \emph{linear} if $\sigma(\sum t_{i}e_{i})=\sum t_{i}\sigma(e_{i})$ for $t_{i}\in [0,1]$ and $\sum t_{i}=1$. A linear simplex is therefore entirely determined by its values on the vertices.)
 Observe that if $\sigma\colon\Delta\to X$ is filtered and $\zeta\in \Delta^{per}(\Delta)$, then
\begin{itemize}
\item[--] $\sigma_*(\zeta)^{-1}(X_i)=\zeta^{-1}\sigma^{-1}(X_i)$ is a face of 
$\Delta$ because $\sigma$ is filtered and $\zeta$ linear,
\item[--] $\|\sigma_*(\zeta)\|_i=\dim \zeta^{-1}\sigma^{-1}(X_{n-i})\leq \dim \sigma^{-1}(X_{n-i})=\|\sigma\|_i$.
\end{itemize}
Therefore, if $\sigma$ is $\ov{p}$-admissible and $\zeta\in \Delta^{per}_{\dim \Delta}(\Delta)$ then 
we have
$$\|\sigma_{*}(\zeta)\|_{i}\leq \|\sigma\|_{i}\leq \dim\Delta-i+\ov{p}(i),$$
and $\sigma_*(\zeta)$ is $\ov{p}$-admissible.
We define
$sd$  on a linear simplex,
$\xi\colon\Delta^{\ell}\to\Delta$,
by
$$sd(\xi)=\left\{\begin{array}{ll}
\xi&\text{ if } \ell =0,\\
b(\xi)\cdot sd(\partial \xi)&\text{ if } \ell >0,
\end{array}\right.$$
where $b(\xi)$ is the barycenter of $\xi(\Delta^{\ell})$ and the definition of
$b(\xi)\cdot sd(\partial \xi)$
is recalled below. 
If $\sigma\colon\Delta\to X$ is a singular simplex, we set  $sd(\sigma)=\sigma_{*}(sd([\Delta]))$, as in \cite{MR0210112} or \cite{MR1867354}.

We want to prove that the image by $sd$ of a chain of $\ov{p}$-intersection is of $\ov{p}$-intersection. As $sd$ commutes with the differential, it is sufficient to prove that the image by $sd$ of a $\ov{p}$-admissible simplex is $\ov{p}$-admissible.
Therefore, we are reduced to the verification of the inclusion, 
$$sd(\Delta^{per}(\Delta))\subset \Delta^{per}(\Delta).$$
(Observe that $[\Delta]\in \Delta^{per}_{\dim\Delta} (\Delta)$.)
 The inclusion
$sd(\Delta^{per}_{\ell}(\Delta))\subset \Delta^{per}_{\ell}(\Delta)$
is clear for $\ell =0$. 
Let $\xi\in \Delta^{per}_{\ell} (\Delta)$.
By induction, we know that $sd(\partial(\xi))=\sum_in_i\tau_i$, with $\tau_i\in \Delta^{per}_{\ell-1}(\Delta)$. Therefore, we have
$b(\xi)\cdot sd(\partial(\xi))=\sum_i\,n_i\,b(\xi)\cdot \tau_i$, where $b(\xi)\cdot\tau_i$ is defined on $\Delta^{\ell}=\Delta^{l-1}\ast\{\vartheta\}$ by
$$(b(\xi)\cdot\tau_i)(t\vartheta+(1-t)a)=tb(\xi)+(1-t)\tau_i(a).$$ 
We have to prove that $b(\xi)\cdot\tau_i\in \Delta^{per}(\Delta)$. If $F$ is a  face of $\Delta$, we claim that
$\dim (b(\xi)\cdot \tau_i)^{-1}(F)\leq \dim F$. We consider two cases:
\begin{itemize}
\item[--] $b(\xi)\notin F$. Then $(b(\xi)\cdot \tau_i)^{-1}(F)=\tau_i^{-1}(F)$ and we already know that
$\dim \tau_i^{-1}(F)\leq \dim F$.
\item[--] $b(\xi)\in F$. In this case, $\Delta^{\ell}=\xi^{-1}(F)=(b(\xi)\cdot\tau_{i})^{-1}(F)$ which gives 
$\dim(b(\xi)\cdot\tau_{i})^{-1}(F)=\dim\xi^{-1}(F)\leq \dim F$.
\end{itemize}

\smallskip
$\bullet$ We extend now the definition of $T$ to the intersection setting.
Denote by $\Delta^{per+1}(\Delta)$ the subspace generated by the  linear simplices, $\xi$, of $\Delta$ such that 
$$\dim\xi^{-1} (F)\leq1+ \dim F,$$ for any  face $F$ of $\Delta$. If $\sigma\colon\Delta\to X$ is filtered and $\zeta\in \Delta^{per+1}(\Delta)$, then
\begin{itemize}
\item[--] $\sigma_*(\zeta)^{-1}(X_i)=\zeta^{-1}\sigma^{-1}(X_i)$ is a face of 
$\Delta$ because $\sigma$ is filtered and $\zeta$ linear,
\item[--] $\|\sigma_*(\zeta)\|_i=\dim \zeta^{-1}\sigma^{-1}(X_{n-i})\leq \dim \sigma^{-1}(X_{n-i})+1=\|\sigma\|_i+1$.
\end{itemize}
Therefore, if $\sigma$ is $\ov{p}$-admissible and $\zeta\in \Delta^{per+1}_{\dim\Delta+1}(\Delta)$,  then 
we have
$$\|\sigma_{*}(\zeta)\|_{i}\leq \|\sigma\|_{i}+1\leq \dim\Delta+1-i+\ov{p}(i),$$
and $\sigma_*(\zeta)$ is $\ov{p}$-admissible.
We define
$T$  on a linear simplex,
$\xi\colon\Delta^{\ell}\to\Delta$,
by
$$T(\xi)=\left\{\begin{array}{ll}
0&\text{ if } \ell =-1,\\
b(\xi)\cdot (\xi-T(\partial(\xi)))&\text{ if } \ell \geq 0,
\end{array}\right.$$
where $b(\xi)$ is the barycenter of $\xi(\Delta^{\ell})$.
If $\sigma\in C_\ell^{\ov{p}}(X)$, we set
$T(\sigma)=\sigma_{*}(T([\Delta^\ell]))$.
As in the case of $sd$, since $T\partial+\partial T=\id-sd$, for proving that the image by $T$ of a $\ov{p}$-intersection simplex is of $\ov{p}$-intersection, it is sufficient to prove that $$T(\Delta^{per}(\Delta))\subset \Delta^{per+1}(\Delta).$$
The inclusion $T(\Delta^{per}_{\ell}(\Delta))\subset \Delta^{per+1}_{\ell+1}(\Delta)$ is obvious if $\ell=-1$. 
Let $\xi\in \Delta^{per}_{\ell} (\Delta)$.
By induction, we know that $\xi-T(\partial(\xi))=\sum_in_i\tau_i$, with $\tau_i\in \Delta^{per+1}_{\ell}(\Delta)$. Therefore, we have
$b(\xi)\cdot (\xi-T(\partial(\xi)))=\sum_i\,n_i\,b(\xi)\cdot \tau_i$ and we
have to prove that $b(\xi)\cdot\tau_i\in \Delta^{per+1}_{\ell+1}(\Delta)$.
If $F$ is a face of $\Delta^{\ell}$, we claim that 
$\dim(b(\xi)\cdot\tau_{i})^{-1}(F)\leq \dim F+1$. We consider two cases:
\begin{itemize}
\item[--] $b(\xi)\notin F$. Then $(b(\xi)\cdot\tau_{i})^{-1}(F)=\tau_{i}^{-1}(F)$ and we  already know that
$\dim\tau_{i}^{-1}(F)\leq\dim F+1$.
\item[--] $b(\xi)\in F$. In this case, $\Delta^{\ell}=\xi^{-1}(F)$ and
$\Delta^{\ell+1}=(b(\xi)\cdot\tau_{i})^{-1}(F)$
which gives
$\dim (b(\xi)\cdot\tau_{i})^{-1}(F)=\dim\xi^{-1}(F)+1\leq \dim F +1$.
\end{itemize}

\smallskip
$\bullet$ Proof of (i). 
Let $\xi=\sum_{j\in J} n_{j}\sigma_{j}$ be a chain of $\ov{p}$-intersection, written as a sum of filtered simplices.
Recall from \defref{def:badface}, the integer $\ell(\sigma)$, associated to a simplex $\sigma$. We extend this definition to the chain $\xi$ by $\ell(\xi)=\max\{\ell(\sigma_{j})\mid j\in J \text{ such that } n_{j}\neq 0\}$ and $\ell(0)=-\infty$.

We have to show the existence of an integer $k$ such that the iterated barycentric subdivision of $\xi$ verifies 
\begin{equation}\label{equa:barycentric}
sd^k\xi\in C_{*}^{\ov{p}}(U)+  C_{*}^{\ov{p}}(V).
\end{equation}
From the classical theory, we can find $r\geq 0$, with $sd^r\xi\in C_{*}(U)+C_{*}(V)$. Thus, we can suppose
\begin{equation}\label{equa:barycentric2}
\xi\in C_{*}^{\ov{p}}(X)\cap(C_{*}(U)+  C_{*}(V)).
\end{equation}
On the set of chains of $\ov{p}$-intersection, we define an equivalence relation by
$$\xi\sim\zeta \text{ if there exists } k \geq 0, \text{ such that } sd^k\xi-sd^k\zeta\in C^*_{\ov{p}}(U)+C^*_{\ov{p}}(V).$$
This equivalence relation is clearly compatible with sums, i.e.,  
$\xi\sim\zeta$ and $\xi'\sim\zeta'$ imply $\xi+\xi'\sim \zeta+\zeta'$. 
The interest of this relation relies in the fact that Property (i) is equivalent to $\xi\sim 0$. We prove it by induction on $\ell(\xi)$. 
This is obviously true if $\ell(\xi)=0$. 
We suppose it is true for any chain of $\ov{p}$-intersection, $\zeta$, such that $\ell(\zeta)<m$ and we consider a chain of $\ov{p}$-intersection, $\xi$, with $\ell(\xi)=m$.
We decompose $\xi$ in $\xi=\xi_{0}+\cdots+\xi_{m}$, with the next characterization of the components $\xi_{i}$.
\begin{itemize}
\item The chain $\xi_{0}$ is formed of simplices which are of $\ov{p}$-intersection, i.e., $\xi_{0}$ is composed of   simplices $\sigma$ such that $\ell(\sigma)=0$.
\item The chain $\xi_{i}$, for $i>0$, is composed of simplices $\sigma$ having a bad face and such that $\ell(\sigma)=i$.
\end{itemize} 
By induction and the compatibility of the equivalence relation with sums, we may suppose $\xi=\xi_{m}$. 
Denote by $(\tau_{i})_{i\in I}$ the bad faces that appear for the simplices of $\xi$. We decompose $\xi$ in $\xi=\sum_{i\in I} \xi(i)$, where the simplices of $\xi(i)$ have $\tau_{i}$ as bad face. 
Let $\sigma$ be a simplex of $\xi(i)$.
As the sum $\xi$ is of $\ov{p}$-intersection, there exists at least a simplex $\sigma'$ which has  $\tau_{i}$ as face. From (c) of \propref{prop:badface}, we know that $\sigma$ and $\sigma'$ have the same bad face. Therefore $\sigma'$ is a simplex of $\xi(i)$ and $\xi(i)$ is of $\ov{p}$-intersection. With the compatibility of the equivalence relation with sums,  we may suppose that each simplex of $\xi$ has the same bad face, 
$\tau\colon \Delta^{j_{0}}\ast\cdots\ast \Delta^{j_{n-m}}\to X$.

Grants to (\ref{equa:barycentric2}), we may suppose  also $\im \tau\subset U$.  There exists a subdivision such that
\begin{equation}\label{equa:subdivi}
sd^k\xi=\sum_{\im\beta_{a}\cap\im\tau=\emptyset}n_{a}\beta_{a}+
\sum_{\im\beta_{b}\cap\im\tau\neq\emptyset}n_{b}\beta_{b}\in C_{*}^{\ov{p}}(X),
\end{equation}
with $\im\beta_{b}\subset U$, for all $b$. Each element $\beta$ of these sums is obtained from a simplex, $\sigma\colon \Delta=\Delta^{j_{0}}\ast\cdots\ast\Delta^{j_{n}}\to X$, of $\xi$, verifying $\ell(\sigma)=m$ and $\dim \Delta^{j_{0}}\ast\cdots\ast \Delta^{j_{n-m}}\neq \dim\Delta$, and from an element $F$ of the iterated barycentric subdivision of $\Delta$, of the same dimension than $\Delta$, i.e.,
$$\beta\colon F\cap (\Delta^{j_{0}}\ast\cdots\ast\Delta^{j_{n}})\to X.$$
Observe:
\begin{eqnarray*}
\|\beta\|_{n}&=&\dim (F\cap \Delta^{j_{0}})\leq \dim \Delta^{j_{0}}< \dim\Delta-n+\ov{p}(n),\\
\ldots&&\\
\|\beta\|_{m+1}&=&
\dim (F\cap (\Delta^{j_{0}}\ast\cdots\ast\Delta^{j_{n-m-1}}))\leq \dim(\Delta^{j_{0}}\ast\cdots\ast\Delta^{j_{n-m-1}})\\
&<& \dim\Delta-(m+1)+\ov{p}(m+1)= \dim F-(m+1)+\ov{p}(m+1),\\
\|\beta\|_{m}&=&
\dim (F\cap (\Delta^{j_{0}}\ast\cdots\ast\Delta^{j_{n-m}}))=
 \dim(\Delta^{j_{0}}\ast\cdots\ast\Delta^{j_{n-m}})\\
&\leq& \dim\Delta-m+\ov{p}(m) = \dim F-m+\ov{p}(m),
\end{eqnarray*}
where the strict inequalities appear for $\|\beta\|_{i}$ with $i>m=\ell(\sigma)$.
From the previous inequalities, we deduce $\ell(\beta)\leq m$. 
Thus, we have proved that
$$sd^k \xi =(sd^k\xi)_{0}+\cdots+(sd^k \xi)_{m-1}+(sd^k \xi)_{m}.$$
From the induction and the compatibility of the equivalence relation with sums, it is sufficient to prove that
$(sd^k\xi)_{m}\sim 0$. For that, we observe that we have $\ell(\beta)=m$ if, and only if, 
$\dim (F\cap (\Delta^{j_{0}}\ast\cdots\ast\Delta^{j_{n-m}}))= \dim(\Delta^{j_{0}}\ast\cdots\ast\Delta^{j_{n-m}})$,
which implies $F\cap (\Delta^{j_{0}}\ast\cdots\ast\Delta^{j_{n-m}})\neq\emptyset$. In particular, if $\ell(\beta)=m$, we have $\im\beta\cap\im\tau\neq\emptyset$ and thus $\im\beta\subset U$.

As $\dim(F\cap (\Delta^{j_{0}}\ast\cdots\ast\Delta^{j_{n-m}}))=\dim(\Delta^{j_{0}}\ast\cdots\ast\Delta^{j_{n-m}})\neq\dim\Delta=\dim F$, we
 have proved that the chain of $\ov{p}$-intersection,
$(sd^k\xi)_{m}$, verifies
$$ (sd^k\xi)_{m}=
\sum_{\im\beta_{b}\cap\im\tau\neq\emptyset, \,\ell(\beta_{b})=m}n_{b}\beta_{b}\in C_{*}^{\ov{p}}(U)
$$
and this gives the inductive step.

\smallskip
$\bullet$ Proof of (ii).
We prove that the inclusion $\iota\colon \im\varphi\hookrightarrow C_{*}^{\ov{p}}(X)$ induces an isomorphism in homology.
Let $[\xi]\in H_{*}^{\ov{p}}(X)$. With the already established properties of $sd$ and $T$, we have $[\xi]=[sd^i\xi]$, for all $i\geq 0$.
\begin{itemize}
\item[--] For the surjectivity, let $[\xi]\in H_{*}^{\ov{p}}(X)$.  Thanks to the property (i), there exists $k\geq 0$ with $sd^k\xi\in\im\varphi$. We get
$$[\xi]=\iota_{*}[sd^k \xi]\in\im\iota_{*}.$$
\item[--] Concerning the injectivity, let $[\alpha]\in H_{*}^{\ov{p}}(\im\varphi)$ and $\xi\in C_{*+1}^{\ov{p}}(X)$ such that
$\alpha=\partial\xi$. With the property (i), there exists $k\geq 0$ with $sd^k\xi\in\im\varphi$. This implies
$$[\alpha]=[sd^k\alpha]=[sd^k(\partial\xi)]=[\partial (sd^k\xi)]$$
and $[\alpha]=0$ in  $H_{*}^{\ov{p}}(\im\varphi)$.
\end{itemize}
\end{proof}

\section{Stratified spaces}\label{sec:stratified}

\begin{quote}
Filtered spaces are sufficient for a definition of intersection homology, but this generality does not bring a good behavior for maps. For having induced morphisms  between intersection homology groups, we introduce stratified spaces and maps, the key point (\thmref{thm:stratifiedmapcomplex}) being the fact that stratified maps decrease the perverse degree of \defref{def:perversedegreesimplex}.
We define also a notion of s-homotopy between stratified maps and prove that two s-homotopic stratified maps induce the same homomorphism in intersection homology.
\end{quote}

\begin{definition}\label{def:espacestratifie}
A \emph{stratified space} is a filtered space
\index{Space!stratified|see{Stratified space}}\index{Stratified!space}
such that 
 any pair of strata, $S$ and $S'$ with $S\cap \overline{S'}\neq \emptyset$,  verifies
$S\subset \ov{S'}$.
\end{definition}

For instance, the CS sets of King \cite{MR800845}
 are stratified spaces, see \thmref{thm:pseudostratifie} and \defref{def:CS}. Also, the previous constructions of cone, suspension, product and join of stratified spaces 
 are stratified.
 
 \begin{definition}\label{def:appstratifiee}
A \emph{stratified map,} $f\colon (X, (X_j)_{0\leq j\leq n})\to (Y,(Y_j)_{0\leq j\leq m})$, is a continuous map between stratified spaces, such that, for any stratum $S\in\cS_X$, there exists a stratum $S^f\in \cS_Y$ with $f(S)\subset S^f$
and $m-d(S^f)\leq n-d(S)$.
A stratified map is \emph{stratum preserving} if $n=m$ and $f^{-1}Y_{n-\ell}=X_{n-\ell}$, for any $\ell\geq 0$.
\index{Maps!stratified}\index{Stratified!map}\index{Stratified!map!stratum preserving}
\end{definition}

The previous conditions on  stratum preserving map correspond to  definitions  of Friedman \cite{MR2009092} and Quinn, \cite{MR928266}, \cite{MR873296}.
Observe  that a continuous map is stratified if, and only if, for any stratum 
$S' \in \cS_{Y}$, $f^{-1}(S')$ is the empty set or there exists a family of strata, $\{S_i \mid i \in I\} \subset \cS_{X}$, such that
$$f^{-1} (S') = \bigcup_{i \in I} S_i, \text{ with } m-d(S')\leq n-d(S_{i}).$$
In the particular case of two stratified spaces on the same  space, $(X,(X_{i})_{1\leq i\leq n})$ and $(X,(X'_{i})_{1\leq i\leq n})$, the identity map,
$\id\colon (X,(X_{i})_{1\leq i\leq n})\to (X,(X'_{i})_{1\leq i\leq n})$,
is stratified if, and only if, the strata coming from the filtration $(X'_{i})_{1\leq i\leq n}$ are union of strata associated to the filtration $(X_{i})_{1\leq i\leq n}$.

 \begin{theorem}\label{thm:stratifiedmapcomplex}
 A stratified map, $f\colon X\to Y$, induces a chain map
 $f_*\colon C_*^{\cF}(X)\to C_*^{\cF}(Y)$, defined by $\sigma\mapsto f\circ \sigma$.
 Moreover, for all $\ell\geq 0$,  
 we have
 $$\|f_*(\sigma)\|_\ell\leq \|\sigma\|_\ell.$$
 \end{theorem}
 
The next result is a direct consequence of Definitions \ref{def:chaineadmis}, \ref{def:chaineintersection} 
and  \thmref{thm:stratifiedmapcomplex}.

\begin{corollary}\label{cor:stratifiedinducedhomology}
A stratified map, $f\colon X\to Y$,
induces a morphism between the complexes of intersection chains,
$f_*\colon C_*^{\ov{p}}(X) \to C_*^{\ov{p}}(Y)$,
for any loose perversity $\ov{p}$.
\end{corollary}

In the next definition, $[0,1]$ is trivially filtered of formal dimension~1 and $X\times [0,1]$ endows the  product filtration.

\begin{definition}\label{def:homotopystratified}
Two stratified maps, $f_0,\,f_1\colon X\to Y$, are \emph{s-homotopic} if there exists a stratified map,
$f\colon X\times [0,1]\to Y$, such that $f(x,0)=f_0(x)$ and $f(x,1)=f_1(x)$.
\index{Maps!homotopic}\index{Homotopy!of stratified maps}
\end{definition}

\begin{proposition}\label{prop:shomotopyandintersectionhomology}
If $f_0,\,f_1\colon X\to Y$ are two s-homotopic stratified maps, they induced the same homomorphism in intersection homology,
$$H^{\ov{p}}_*(f_0)=H^{\ov{p}}_*(f_1)\colon H^{\ov{p}}_*(X)\to H^{\ov{p}}_*(Y),$$
for any loose perversity $\ov{p}$.
\end{proposition}

\begin{proof}
 Let $f\colon X\times I\to Y$ be an s-homotopy from $f_0$ to $f_1$ and
 $\iota_k\colon X\to X\times I$, $x\mapsto (x,k)$, be the canonical inclusions for $k=0,\,1$. 
 The proof follows directly from \corref{cor:stratifiedinducedhomology} and \lemref{lem:homotopie1}.
 \end{proof}

 Before proving \thmref{thm:stratifiedmapcomplex}, we need results on the behavior of strata.

\begin{proposition}\label{prop:propespacestratifie}
If $(X,(X_i)_{0\leq i\leq n})$ is a stratified space, the following properties are satisfied.
\begin{enumerate}[(i)]
\item The relation $S\preceq S'$, defined on the set of strata by  $S\subset \overline{S'}$, is an order relation.
\item The formal dimension map, $d\colon \cS_X\to \N$, is strictly increasing.
\item The closure of a stratum is a union of lower strata, i.e., $\overline{S}=\cup_{S'\preceq S} S'$.
\end{enumerate}
\end{proposition}

\begin{proof}
\begin{enumerate}[(i)] 
\item 
Let $S$ and $S'$ be strata such that
$S\subset \overline{S'}$ and $S'\subset \overline{S}$.
Suppose $d(S)\leq d(S')$. 
Since each stratum, $S$,  is closed in $X_{d(S)}\backslash X_{d(S)-1}$ and $X_{d(S)}$ is closed in $X$, the stratum $S$ is closed in $X\backslash X_{d(S)-1}$, which implies
$$S=\ov{S}\cap (X\backslash X_{d(S)-1}).$$ Then, the inclusion $S'\subset \ov{S}$ implies $S'\backslash X_{d(S)-1}\subset S$.
The inclusion $S'\cap X_{d(S)-1}\subset S'\cap X_{d(S')-1}$ and the equality $S'\cap X_{d(S')-1}=\emptyset$, by definition of $d(S')$, imply $S'\subset S$ and the equality $S'=S$.
This establishes the antisymetry property; the reflexivity and transitivity are obvious.
\item Suppose $S\preceq S'$. 
As $X_{d(S')}$ is closed, we have $S\subset \ov{S'}\subset X_{d(S')}$ which implies $d(S)\leq d(S')$. Since $S\subset X\backslash X_{d(S)-1}$, we have
$$S=S\cap (X\backslash X_{d(S)-1})\subset \ov{S'}\cap (X\backslash X_{d(S)-1}).$$
 If  $d(S)=d(S')$, we know that $S'$ is closed in $X\backslash X_{d(S)-1}$, which implies
$$S'=\ov{S'}\cap(X\backslash X_{d(S)-1}).$$
We get $S\subset S'$ and $S=S'$.
\item The inclusion $\cup_{S'\preceq S} S'\subset \ov{S}$ is obvious. Consider $x\in \ov{S}$. Since the strata form a partition, there exists a unique stratum $S'$ such that $x\in S'$. From $\ov{S}\cap S'\neq \emptyset$ and from \defref{def:espacestratifie}, we deduce $S'\subset \ov{S}$. We get $x\in S'$ and $S'\preceq S$, which proves the remaining part of the equality.
\end{enumerate}
\end{proof}

As strata give a partition, the stratum $S^f$ of \defref{def:appstratifiee} is uniquely determined by $f$ and $S$. 

\begin{proposition}\label{prop:orden}
If $f\colon X\to Y$ is a stratified map, the induced map,
$\cS_f \colon  \!\!({{\cS}}_{X},\preceq)  \to ({{\cS}}_{Y},\preceq)$, defined by
$\cS_f(S)=S^f$,
is increasing.
\end{proposition} 

\begin{proof}
Let $S_1 \preceq S_2$ in ${{\cS}}_{X}$. Since $f$ is continuous and $S_1 \subset \overline{S_2}$, we have
$f(S_1) \subset \ov{f(S_2)}$. This implies $S_1^{f} \cap \ov{S_2^{f}} \neq \emptyset$, which gives $\cS_f(S_1) \preceq \cS_f(S_2)$.
\end{proof}

\begin{lemma}\label{lem:stratasimplex}
Let  $(X, (X_i)_{0\leq i\leq n})$ be a stratified space and let $\sigma\colon\Delta
\to X$ be a filtered simplex.  Then the following properties are satisfied.
\begin{enumerate}[(i)]
\item There is a stratum $S\in\cS_X$ such that $\sigma(\mathring{\Delta})\subset S$.
\item The family of strata that have a non-empty intersection with the image of $\sigma$ is totally ordered. If $S_1\prec S_2\prec \cdots\prec S_p$ is this family, then we have
\begin{equation}\label{equa:petiteequation}
\sigma^{-1}X_i=\left\{
\begin{array}{cl}
\emptyset&\text{ if } i<d(S_1),\\
\sigma^{-1}(S_1\sqcup \cdots \sqcup S_k)&
\text{ if }d(S_k)\leq i < d(S_{k+1}),
\,1\leq k<p,\\
\sigma^{-1}(S_1\sqcup \cdots \sqcup S_p)&
\text{ if } i\geq d(S_p).
\end{array}\right.
\end{equation}
\end{enumerate}
\end{lemma}

\begin{proof} 
(i) Let $k$ be the integer such  that 
$\sigma(\Delta)\subset X_k$ and $\sigma(\Delta)\not\subset X_{k-1}$. Thus  there exists $t\in\Delta$ with $\sigma(t)\notin X_{k-1}$ and we consider the stratum, $S$,  characterized by $\sigma(t)\in S\subset X_k\backslash X_{k-1}$. By hypothesis, the pullback $\Delta'=\sigma^{-1}(X_{k-1})$ is a face (perhaps empty) of $\Delta$. We have $\Delta'\ne \Delta$ because $t\notin \Delta'$. From $\sigma(\Delta\backslash \Delta')\subset X_k\backslash X_{k-1}$ and
$\sigma(t)\in \sigma(\Delta\backslash \Delta')\cap S$, we deduce $\sigma(\Delta\backslash \Delta')\subset S$ and
$$\sigma(\mathring{ \Delta})\subset \sigma(\Delta\backslash \Delta')\subset S.$$

(ii) Let $i\in\{0,\ldots,n\}$. From the definition of filtered simplex, we know that\linebreak
$\sigma^{-1}X_i\backslash \sigma^{-1}X_{i-1}$ is connected. Moreover, this complement can be decomposed in a disjoint union of closed subsets of $\sigma^{-1}X_i\backslash \sigma^{-1}X_{i-1}$,
$$\sigma^{-1}X_i\backslash \sigma^{-1}X_{i-1}=\bigsqcup_{d(S)=i}\sigma^{-1}S.$$
Suppose $\sigma^{-1}X_i\backslash \sigma^{-1}X_{i-1}\neq \emptyset$. As it is connected, it is the pullback of \emph{one} stratum $\tS_i$ of $X_i\backslash X_{i-1}$.
Therefore, there exists a family of integers,
$0\leq d_1<d_2<\cdots <d_p\leq n$, such that
$$\sigma^{-1}X_i\backslash \sigma^{-1}X_{i-1}=
\left\{
\begin{array}{lcl}
\emptyset&\text{if}&i\notin \{d_1,\ldots,d_p\},\\
\sigma^{-1}(\tS_i)&\text{if}&i=d_k \text{ for some } k\in\{1,\ldots,p\}.
\end{array}\right.
$$
Setting $\tS_{d_k}=S_k$, this gives the equalities (\ref{equa:petiteequation}). 

\smallskip
Let $k\in\{2,\ldots,p\}$, we are reduced to prove $S_{k-1}\prec S_k$. Recall the equality
$\sigma^{-1}X_{d_k}=\sigma^{-1}(S_1\sqcup\cdots\sqcup S_k)$. 
 We  have proved the equalities
$
\sigma^{-1}X_{d_k}\backslash \sigma^{-1}X_{d_{k-2}}=\sigma^{-1}(S_{k-1}\sqcup S_k)
$
and
$\displaystyle{
\sigma^{-1}\ov{S_k}=\bigcup_{S \preceq S_k} \sigma^{-1}S }$,
see (iii) of \propref{prop:propespacestratifie}.
In particular, $\sigma^{-1}(\ov{S_k} \cap  (X_{d_{k-1}}\backslash X_{d_{k-2}})) = \sigma^{-1}S_{k-1}$ or $\emptyset$ and 
we have
\begin{eqnarray*}
\sigma^{-1}X_{d_k}\backslash \sigma^{-1}X_{d_{k-2}}&=&
\sigma^{-1}(S_{k-1} \sqcup( \ov{S_k} \cap (X_{d_k}\backslash X_{d_{k-1}}) ))\\
 &=& 
\sigma^{-1}(S_{k-1} \cup (\ov{S_k} \cap (X_{d_k}\backslash X_{d_{k-2}}) ))\\
&=&\sigma^{-1}(S_{k-1}) \cup\sigma^{-1}(\ov{S_k} \cap (X_{d_k}\backslash X_{d_{k-2}}) ).
\end{eqnarray*}

The subset $\sigma^{-1}S_{k-1}$(resp. $\sigma^{-1}(\ov{S_{k}}\cap (X_{d_{k}}\backslash X_{d_{k-2}})$)  is closed in 
$\sigma^{-1}X_{d_{k-1}}\backslash \sigma^{-1}X_{d_{k-2}}$ (resp. $\sigma^{-1}X_{d_{k}}\backslash \sigma^{-1}X_{d_{k-2}}$), which is closed in the connected space $\sigma^{-1}X_{d_k}\backslash \sigma^{-1}X_{d_{k-2}}$. 

  From the formula above and the connectivity of
  $\sigma^{-1}X_{d_k}\backslash \sigma^{-1}X_{d_{k-2}}$, we deduce that
$\sigma^{-1}(S_{k-1})\cap \sigma^{-1}(\ov{S_k}\cap (X_{d_k}\backslash X_{d_{k-2}}))\neq \emptyset$
which implies $S_{k-1}\cap \ov{S_k}\neq \emptyset$ and $S_{k-1}\preceq S_k$, by definition. Finally, as $d(S_{k-1})=d_{k-1}\neq d_k=d(S_k)$, we have $S_{k-1}\prec S_k$.
\end{proof}

\begin{proof}[Proof of \thmref{thm:stratifiedmapcomplex}]
Let $f\colon (X, (X_i)_{0\leq i\leq n})\to (Y, (Y_i)_{0\leq i\leq m})$ be a stratified map
and $\sigma\colon \Delta\to X$ be a filtered simplex.

If $(f\circ\sigma)^{-1}(Y_{m-\ell})=\emptyset$ then $\|f\circ\sigma\|_\ell=-\infty$.

If $\sigma^{-1}X_{n-\ell}=\emptyset$, then from
$(f\circ\sigma)^{-1}(Y_{m-\ell})\subset \sigma^{-1}X_{n-\ell}=\emptyset$, we deduce
$\|f\circ\sigma\|_\ell=\|\sigma\|_\ell=-\infty$. 

We may thus suppose $(f\circ\sigma)^{-1}(Y_{m-\ell})\neq \emptyset$ and $\sigma^{-1}X_{n-\ell}\neq\emptyset$, which imply 
$0\leq \ell\leq m$.
We have to prove that $(f\circ \sigma)^{-1}(Y_{m-\ell})$ is a face of $\Delta$ and that
$$\dim (f\circ \sigma)^{-1}(Y_{m-\ell})\leq \dim \sigma^{-1}(X_{n-\ell}).$$
Recall that the family of strata meeting the image of $\sigma$ can be ordered as
$$S_1\prec S_2\prec \cdots\prec S_p,$$
which implies, with \propref{prop:orden},
$$S_1^f\preceq S_2^f\preceq\cdots\preceq S^f_p.$$
Let $T$ be a stratum of $Y$ such that $T\cap (f\circ \sigma)(\Delta)\neq\emptyset.$, i.e., $f^{-1}(T)\cap \sigma(\Delta)\neq \emptyset$. As the map $f$ is stratified, there exists $k\in\{1,\ldots,p\}$, such that $T=S_k^f$ and
$$m-d(T)\leq n-d(S_k).$$
We denote by $0\leq k_0< \cdots <k_q=p$ the integers such that
$$\left\{\begin{array}{l}
T_1=S_1^f=\cdots=S^f_{k_1}, \text{ with }S^f_{k_1}\neq S^f_{k_1+1},\\
T_2=S^f_{k_1+1}=\cdots=S^f_{k_2}, \text{ with }S^f_{k_2}\neq S^f_{k_2+1},\\
\ldots\\
T_q=S^f_{k_{q-1}+1}=\cdots=S^f_{k_q}=S^f_p.
\end{array}
\right.$$
From that, we deduce:
\begin{enumerate}[(1)]
\item $\{T\in\cS_Y\mid T\cap (f\circ \sigma)(\Delta)\neq\emptyset\}=\{T_1\prec T_2\prec\cdots\prec T_q\}$,
\item $m-d(T_j)\leq n-d(S_{k_j})$,
\item $\displaystyle{f^{-1}(T_j)=\bigcup _{k=k_{j-1}+1}^{k_j}S_k}$, for any $j\in\{1,\ldots,q\}$.
\end{enumerate}
From (ii) of \propref{prop:propespacestratifie}, we deduce:
$$0\leq d(T_1)<d(T_2)<\cdots <d(T_q)\leq m.$$
With the convention $d(T_{q+1})=m$ and $d(T_{0})=0$, we denote by $a\in\{0,\ldots,q\}$ the integer such that
$$d(T_a)\leq m-\ell<d(T_{a+1}).$$
As $(f\circ\sigma)^{-1}Y_{m-\ell}\neq\emptyset$, the integer $a$ is different of 0 and we have obtained
\begin{eqnarray*}
(f\circ\sigma)^{-1}(Y_{m-\ell})&=&\bigcup_{j=0}^{m-\ell} \sigma^{-1}f^{-1}(Y_j\backslash Y_{j-1})=
 \bigcup_{i=0}^a\sigma^{-1}f^{-1}(T_i)\\
&=& \bigcup_{k=1}^{k_a}\sigma^{-1}(S_k)=
 \sigma^{-1}(X_{d(S_{k_a})}),
\end{eqnarray*}
which is a face of the simplex $\Delta$. The simplex $f\circ \sigma$ is thus filtered. From Property (2) above, we deduce
$$d(S_{k_a})\leq n-m+d(T_a)\leq n-m+(m-\ell)\leq n-\ell,$$
which implies
$\sigma^{-1}(X_{d(S_{k_a})})\subset \sigma^{-1}(X_{n-\ell})$
and the expected inequality between the dimensions,
$$\dim (f\circ \sigma)^{-1}(Y_{m-\ell})=\dim \sigma^{-1}(X_{d(S_{k_a})})\leq \dim \sigma^{-1}(X_{n-\ell}).$$
\end{proof}

\begin{corollary}\label{cor:amalgamation2}
Let $f\colon X\to Y$ be a stratified map and $\sigma\colon\Delta\to X$ be a filtered simplex of solid $\sigma$-decomposition
$\Delta=\Delta_{1}\ast\cdots\ast\Delta_{p}$. 
 \index{Decomposition@$\sigma$-Decomposition!solid}
 Then, the solid $(f\circ\sigma)$-decomposition is of the shape
 $\Delta=\Delta'_1\ast\cdots\ast\Delta'_q$,
where the $\Delta'_{k}$'s are obtained by amalgamations of the type
$\Delta^i\ast \Delta^j\mapsto \emptyset\ast\Delta^{i+j+1}$.
In the case of a stratum preserving stratified map, the $\sigma$- and $f\circ\sigma$-decompositions of $\Delta$ are identical, as well as their solid decompositions.
\end{corollary}

\begin{proof}
In the previous proof, with the same notation, we have shown that the solid $(f\circ \sigma)$-decomposition,
 $\Delta=\Delta'_1\ast\Delta'_2\ast\cdots\ast\Delta'_q$,
verifies
 $\Delta'_a=\Delta_{k_{a-1}+1}\ast\cdots\ast \Delta_{k_a}$, for any $a\in \{1,\ldots,q\}$
 and
 $$(f\circ\sigma)^{-1}(T_1\sqcup\cdots\sqcup T_a)=
 \Delta'_1\ast\cdots\ast\Delta'_a,\, \text{ for any } a\in\{1,\ldots,q\}.$$
This establishes the first part of the statement. 

In the case of a stratum preserving stratified map, with property (ii) of \propref{prop:propespacestratifie}, the relation 
  $S_1\prec S_2\prec \cdots\prec S_p$
  and the stratum preserving condition imply
  $S_1^f\prec S_2^f\prec \cdots\prec S_p^f$.
   Therefore the equation (3) above reduces to
  $f^{-1}(T_j)=S_{k_j}$ and we have
  $(f\circ \sigma)^{-1}(T_j)=\sigma^{-1}(S_{k_j})$, as announced. 
  \end{proof}

\section{CS Sets and pseudomanifolds}\label{sec:locallyconelikespaces}

\begin{quote}
Independence of the stratification is one important feature of the theory of Goresky and MacPherson. Proved in \cite{MR696691} for pseudomanifolds by using  sheaves, this result is also established for CS sets, by King, with a singular point of view. Here, we prove the existence of a quasi-isomorphism between King's intersection complex and the complex $C^{\ov{p}}_*(X)$ introduced in \secref{sec:filteredspaces}, when $X$ is a pseudomanifold. From \cite[Theorem 9]{MR800845}, we deduce the independence of the stratification for our intersection homology. 
\end{quote}

\begin{definition}\label{def:CS}
A \emph{CS set} of dimension $n$
 \index{Set!CS}\index{Link!of a point}\index{CS set|see{Set CS}}
  is a filtered space,
$$
\emptyset\subset X_0 \subseteq X_1 \subseteq \cdots \subseteq X_{n-2} \subseteq X_{n-1} \subsetneqq X_n =X,
$$
such that, for all $i$, 
$X_i\backslash X_{i-1}$ is an $i$-dimensional metrizable topological  manifold or the empty set. Moreover, for each point  $x \in X_i \backslash X_{i-1}$, $i\neq n$, there exist
\begin{enumerate}[(a)]
\item an open neighborhood, $V$, of $x$ in $X$, endowed with the induced filtration,
\item an open neighborhood, $U$, of $x$ in  $X_i\backslash X_{i-1}$, 
\item a compact, filtered  space, $L$,  of formal dimension $n-i-1$, whose cone, $\mathring{c}L$ is endowed with the conic filtration, %
\item a   homeomorphism, $\varphi \colon U \times \mathring{c}L\to V$, 
such that
\begin{enumerate}[(i)]
\item $\varphi(u,\vartheta)=u$, for any $u\in U$, with $\vartheta$  the cone point,
\item $\varphi(U\times \mathring{c}L_j)=V\cap X_{i+j+1}$, for any $j\in \{0,\ldots,n-i-1\}$.
\end{enumerate}
\end{enumerate}
The  couple $(V,\varphi)$ is called a \emph{conic chart} of $x$
 and the filtered  space, $L$, the \emph{link} of $x$. 
 \end{definition}

  Observe  that property (d) (i) implies $V\cap X_{i-1}=\emptyset$ 
and both properties (d) (i) and (ii) imply $V\cap X_{i}=\varphi(U\times \{\vartheta\})=U$.
As $X_i\backslash X_{i-1}$ is locally connected, we may choose $U$ connected. 
Observe also, from \defref{def:espacefiltré}, that the link $L$ is not empty.

The CS sets were introduced by Siebenmann (\cite{MR0319207}), without a restriction of finite formal dimension on the filtration. 
They include metrizable topological manifolds, 
differentiable stratified sets X in the sense of Thom,...

Pseudomanifolds are particular cases of CS sets which allow proofs by induction on the dimension.

\begin{definition}\label{def:pseudo}
An \emph{$n$-dimensional topological pseudomanifold} (or pseudomanifold in short) is a CS set of dimension~$n$, in which $X_{n-2}=X_{n-1}$ and all the links, $L$, are $(n-i-1)$-dimensional topological pseudomanifolds.
\index{Pseudomanifold}
\end{definition}

Observe that this definition makes sense with an induction on the dimension, starting from pseudomanifolds of dimension 0 which are discrete topological spaces, by definition. Also, one can prove that, in a pseudomanifold, $X$, the subspace
$X_{n}\backslash X_{n-2}$ is dense.

\begin{theorem}\label{thm:pseudostratifie}
A CS set
\index{Stratified!space}\index{Set!CS}
is a stratified space.
\end{theorem}

\begin{proof}
Let $S$ and $S'$ be strata such that
$S\cap \overline{S'}\neq \emptyset$,  we have to prove
$S\subset \ov{S'}$.
Set $i=d(S)$ and $j=d(S')$. The space $X_i\backslash X_{i-1}$ being locally connected, we may cover $S$ by  connected open sets $U$ of $S$ satisfying the local condition of \defref{def:CS}, for some chart $(V,\varphi)$. Denote by $\cU$ this cover of $S$ and let $U\in\cU$. We first prove the \\
\centerline{\emph{Claim: if $U\cap \ov{S'}\neq \emptyset$, then $U\subset \ov{S'}$.}}

Suppose $U\cap \ov{S'}\neq \emptyset$ and recall $U=V\cap S=\varphi(U\times\{\vartheta\})$ and $V\cap X_{i-1}=\emptyset$. 
The property $U\cap \ov{S'}\neq \emptyset$ implies $V\cap \ov{S'}\neq \emptyset$. As $V$ is open, we deduce $V\cap S'\neq \emptyset$ and $j\geq i$, because $V\cap X_{i-1}=\emptyset$.

$\bullet$ If $i=j$, we have
$$\emptyset\neq S'\cap V\subset X_i\cap V=U=S\cap V,$$
which implies $S\cap S'\neq \emptyset$ and $S=S'$. This gives $U\subset S=S'\subset \ov{S'}$. %

$\bullet$ If $i<j$, we have
\begin{eqnarray*}
\emptyset \neq S'\cap V\subset (X_j\backslash X_{j-1})\cap V&=&
\varphi(U\times (L_{j-i-1}\backslash L_{j-i-2})\times ]0,1[)\\
&=&
\varphi(\bigsqcup_T U\times T\times ]0,1[),
\end{eqnarray*}
where $T$ runs over the connected components of 
$L_{j-i-1}\backslash L_{j-i-2}$. We may write
$$(X_j\backslash X_{j-1})\cap V=\bigsqcup_{S''}\bigsqcup_{C''}C'',$$
where $S''$ runs over the strata of $X_j\backslash X_{j-1}$ and $C''$ over the connected components of $S''\cap V$. By local connectivity of $X_j\backslash X_{j-1}$, the $C''$ are open in $(X_j\backslash X_{j-1})\cap V$.
As $S'\cap V\neq \emptyset$, there  exists  a non empty connected component, $C'$, of $S'\cap V$ and a non empty connected component, $T$, of $L_{j-i-1}\backslash L_{j-i-2}$ with
$$C'=\varphi(U\times T\times ]0,1[).$$
Observe that $U\times\{\vartheta\}$ is included in the closure of $U\times T\times ]0,1[$ in $U\times \mathring{c}L$. This implies
that $U$ is included in the closure of $C'$ in $V$ and gives $U\subset \ov{S'}$.

\smallskip 
As the claim is proved, we know that, for any $U\in\cU$, we have $U\subset \ov{S'}$ or $U\cap \ov{S'}=\emptyset$. The condition
$S\cap \ov{S'}\neq \emptyset$ and the connectivity of $S$ imply $U\subset \ov{S'}$, and therefore $S\subset \ov{S'}$.
\end{proof}

In \cite{MR800845}, King defines a subcomplex of the singular chain complex of a filtered space as follows. 

\begin{definition}\label{def:kingchaineadmis}
Let  $(X, (X_i)_{0\leq i\leq n})$ be a filtered space and let $\ov{p}$ be a loose perversity.  
A \emph{singular simplex, $\sigma\colon \Delta\to X$, is King $\ov{p}$-admissible} if
\index{Simplex!King admissible for a perversity}
$\sigma^{-1}(X_{n-j}\backslash X_{n-j-1})$ is contained in the $(\dim\Delta-j+\ov{p}(j))$-skeleton of $\Delta$, for all $j\in\{0,1,\ldots,n\}$. A \emph{chain $c$ is King $\ov{p}$-admissible} if there exist $\ov{p}$-admissible simplices, $\sigma_j$, so that
\index{Chain!King admissible for a perversity}
$c=\sum_j\lambda_j\sigma_j$,  $\lambda_j\in R$.

Let $K^{\ov{p}}_*(X)$ be the chain complex which consists of the singular chains $c$, such that $c$ and its boundary $\partial c$ are  King $\ov{p}$-admissibles. 
\end{definition}

\begin{proposition}\label{prop:intersectionetintersection}
Let $X$ be a pseudomanifold and let $\ov{p}$ be a  perversity. Then the canonical inclusion $C^{\ov{p}}_*(X)\to K^{\ov{p}}_*(X)$ induces an isomorphism in homology.
\end{proposition}
\begin{proof} Let $\sigma\colon\Delta\to X$ be a $\ov{p}$-admissible simplex. 
As the set $\sigma^{-1}X_{n-i}$ is a face of $\Delta$ of dimension less than, or equal to, $(\dim \Delta-i+\ov{p}(i))$,  its subset 
$\sigma^{-1}(X_{n-i}\backslash X_{n-i-1})$ is also included in this skeleton.
This shows that any $\sigma$ in $C_*^{\ov{p}}(X)$ belongs also to $K_*^{\ov{p}}(X)$ and we have a canonical inclusion 
$C_*^{\ov{p}}(X)\subset K_*^{\ov{p}}(X)$
We have now only to check  the five conditions of \cite[Theorem~10]{MR800845}. They are direct consequences of 
\propref{prop:propertiesintersectionhomology}.
 \end{proof}
 
 Theorem~9 of \cite{MR800845} implies immediately: 
 
 \begin{corollary}\label{cor:independancestratification}
Let X be a pseudomanifold and let $\ov{p}$ be a  perversity. Then the  homology of $C^{\ov{p}}_*(X)$ is independent of the stratification of $X$.
\end{corollary}

 This result is not true if $\ov{p}$ is not a perversity, as shows an example of King, \cite[Page 155]{MR800845}. Also, we mention that, in \cite{IHGreg}, G. Friedman establishes this result in the more general setting of recursive CS sets.
 
 \begin{remark}
  If $\sigma\colon \Delta\to X$ is a simplex of a filtered space $X$ of formal dimension $n$, we consider the two following conditions,
\begin{enumerate}[(a)]
\item $\sigma^{-1} X_{n-j}$ is included in the $(\dim\Delta-j+\ov{p}(j))$-skeleton of $\Delta$,
for all $j$,
\item $\sigma^{-1}(X_{n-j}\backslash X_{n-j-1})$ is included in the $(\dim\Delta-j+\ov{p}(j))$-skeleton of $\Delta$,
for all $j$.
\end{enumerate}
 In the proof of \propref{prop:intersectionetintersection}, we used that condition (a) implies condition (b). As already quoted by King (\cite{MR800845}), in the case of a perversity $\ov{p}$,  condition (b) implies also condition (a). For proving that, we observe that $\sigma^{-1}X_{n-j}$ is a union of $\sigma^{-1}(X_{n-j-k}\backslash X_{n-j-k-1})$ and that
 \begin{eqnarray*}
 \dim\Delta-(j+k)+\ov{p}(j+k)&\leq&
 \dim\Delta-(j+k)+\ov{p}(j)+k\\
 &\leq&
 \dim\Delta-j+\ov{p}(j),
 \end{eqnarray*}
for any $k$, $k\geq 0$.
 \end{remark}

\backmatter
\bibliographystyle{amsplain}
\bibliography{IntersectionPostArxiv}
\printindex
\end{document}

%% file: intro07Arxiv.tex

\chapter*{Introduction}

 Intersection  homology and cohomology  of pseudomanifolds have been introduced by M.~Goresky and 
 R.~MacPherson  in \cite{MR572580} and \cite{MR696691}. 
 The main feature of these theories concerns the transversality of a simplex relatively to a stratum, the notion of general position being replaced by a less restrictive one depending on a parameter called perversity.
 
To simplify the presentation, we consider first, instead of a pseudomanifold, a filtered space defined as a  topological space, $X$, together with a filtration by closed subspaces, $X_0\subseteq X_1\subseteq\cdots\subseteq X_n=X$. We replace also the original definition of Goresky and MacPherson perversity (henceforth GM-perversity) by the less restrictive notion of a map, $\ov{p}\colon\N\to\Z$, such that $\ov{p}(0)=0$, called
 \emph{loose perversity}.

  In \cite{MR572580},
  a singular simplex, $\sigma\colon \Delta\to X$, is called \emph{$\ov{p}$-admissible}  if $\sigma^{-1}(X_{n-j}\backslash X_{n-j-1})$ is included in the $(\dim\Delta-j+\ov{p}(j))$-skeleton of $\Delta$. 
  A $\ov{p}$-admissible chain is a linear combination of $\ov{p}$-admissible simplices. 
  Finally,  the boundary of a $\ov{p}$-admissible chain being not necessary $\ov{p}$-admissible,  one must introduce the notion of chain of \emph{$\ov{p}$-intersection} as a $\ov{p}$-admissible chain, $c$, whose boundary, $\partial c$, is also $\ov{p}$-admissible. 
  Let $K_{*}^{\ov{p}}(X,R)$ be the chain complex which consists of the chains of $\ov{p}$-intersection, with coefficients in a commutative ring, $R$. 
  By definition, its homology is the $\ov{p}$-intersection homology of Goresky and 
 MacPherson, denoted $H_{*}^{\ov{p}}(X;R)$.

Motivated by a simplicial framework, we replace the previous condition on the skeleton by properties on some faces of the simplex $\Delta$. More explicitely, we consider singular simplices, $\sigma\colon\Delta=\Delta^{j_0}\ast\cdots\ast\Delta^{j_n}\to X$, whose domain is decomposed as a join product such that
 $\sigma^{-1}(X_k) =\Delta^{j_0}\ast\cdots\ast\Delta^{j_k}$, for any filtered space $X=(X_i)_{0\leq i\leq n}$.
To any such simplex, called \emph{filtered simplex,} and to any number $i\in \{1,\ldots,n\}$, we associate the perverse degree
 $\|\sigma\|_i=\dim (\Delta^{j_0}\ast\cdots\ast \Delta^{j_{n-i}})$, with the convention $\dim\emptyset=-\infty$.  

On a filtered simplex, the definition of \emph{$\ov{p}$-admissibility} is equivalent to the inequality 
 $\|\sigma\|_i\leq \dim \Delta - i+\ov{p}(i)$, for all~$i\in \{1,\ldots,n\}$ and
 $\ov{p}$-admissible filtered chains are linear combinations of $\ov{p}$-admissible filtered simplices.
We denote by $C^{\ov{p}}_*(X)$ the complex of  $\ov{p}$-admissible filtered chains, $c$, whose boundary $\partial c$ is $\ov{p}$-admissible also. With a more difficult than expected proof, we establish the existence of a Mayer-Vietoris exact sequence for the homology of $C_{*}^{\ov{p}}(X)$, for any filtered space $X$ and loose perversity $\ov{p}$. Then, relying on King's paper (\cite{MR800845}), we prove that 
the canonical inclusion, $C^{\ov{p}}_*(X)\to K^{\ov{p}}_*(X)$, induces an isomorphism in homology,
if $X$ is a pseudomanifold and $\ov{p}$ a GM-perversity. 

 \section*{Simplicial Blow-up and Intersection-Cohomology}
 
At this point, we hold a complex, ready for a simplicial paradigm, and which covers the original situation of \cite{MR572580}.
We develop it in the context of face sets, also called simplicial sets without degeneracies, see  
\cite{MR0300281}.

 Let ${\pmb\Delta}^{[n]}_\cF$ be
 the category whose objects are the filtered euclidean simplices,
 $\Delta=\Delta^{j_0}\ast\cdots\ast\Delta^{j_n}$,
 and maps are joins of face operators, i.e., $f=\ast_{i=0}^n\partial_i$, with $\partial_i$  a face operator.  \emph{\Ffss}~are defined as functors, 
 $\ob{K}$, from $ {\pmb\Delta}^{[n]}_\cF$ to the category of sets. (For instance, the previous filtered simplices of a filtered topological space constitute a \ffs.)
By mimicking the topological situation described above, we associate, to $\ob{K}$ and to a loose perversity $\ov{p}$, a chain complex
 $C^{\GM,\ov{p}}_{*}(\ob{K};R)$, with coefficients in a commutative ring $R$. 
 When $\ob{K}$ is the \ffs~associated to a filtered space, $X$, there is an isomorphism between
 $C_{*}^{\ov{p}}(X)$ and 
 $C_{*}^{\GM,\ov{p}}(\ob{K})$. Thus,
 the topological setting appears as a particular case.
 An obvious candidate for a cohomology theory is the linear dual,
 $$C_{\GM,\ov{p}}^{*}(\ob{K};R)=\hom(C^{\GM,\ov{p}}_{*}(\ob{K};R),R),$$
 whose homology is 
 called the \emph{GM-cohomology of $\ob{K}$,} with coefficients in $R$,
 and denoted
$H_{\GM,\ov{p}}^{*}(\ob{K};R)$.

 We define  another cochain complex  on $\ob{K}$, based on a \emph{simplicial version of a blow-up,} already present in \cite{MR1143404}  for differential forms on singular manifolds  and  for some particular simplicial complexes in  \cite{MR1346255}. 
 To any filtered simplex, $\Delta=\Delta^{j_0}\ast\cdots\ast\Delta^{j_n}$, we associate a prism,
 $\widetilde{\Delta}=c\Delta^{j_0}\times \cdots \times c\Delta^{j_{n-1}}\times \Delta^{j_n}$, where $c\Delta^{j_{i}}$ is the simplicial cone on $\Delta^{j_{i}}$. The prism $\widetilde{\Delta}$ is called \emph{the blow-up of $\Delta$.}
 If $C^*(\Delta^i;R)$ is the simplicial 
 cochain algebra over a commutative ring, $R$,  we set
 $$\widetilde{C}^*(\Delta;R)=C^*(c\Delta^{j_0};R)\otimes \cdots\otimes 
 C^*(c\Delta^{j_{n-1}};R)\otimes C^*(\Delta^{j_n};R).$$
 A TW-cochain on $\ob{K}$ is  a family of cochains $c_{\sigma}\in \tC^*(\Delta;R)$, for each simplex $\sigma\colon\Delta\to\ob{K}$, such that the association $\sigma\mapsto c_{\sigma}$ is compatible with any face operator of ${\pmb\Delta}^{[n]}_\cF$. We denote by  $\widetilde{C}^*(\ob{K};R)$ the cochain complex of TW-cochains on $\ob{K}$.
(We do not want to go into technical points in this introduction but the reader should be aware that this definition makes sense only if $\Delta^{j_n}\neq\emptyset$. Thus, the compatibility condition in the definition of 
 $\widetilde{C}^*(\ob{K};R)$ concerns only the face operators which satisfy this condition.)
  
  For any TW-cochain, $c\in\widetilde{C}^*(\ob{K};R)$, and any $i\in\{1,\ldots,n\}$, 
  we  define a \emph{perverse degree,} $\|c\|_{i}$, which is a sort of degree along the fiber of the projection
  $$c\Delta^{j_{0}}\times\cdots\times (\Delta^{j_{n-i}}\times\{1\})\times\cdots\times c\Delta^{j_{n-1}}\times \Delta^{j_{n}}
  \to
  c\Delta^{j_{0}}\times\cdots\times \Delta^{j_{n-i}}.$$
  (Observe that the blow-up of $\Delta$ transforms the face
  $\Delta^{j_{0}}\ast\cdots\ast\Delta^{j_{n-i}}$ in the domain of this projection.)
  For
  any loose perversity $\ov{q}$, the cochain $c\in \widetilde{C}^*(\ob{K};R)$   is said \emph{$\ov{q}$-admissible} if $\|c\|_i\leq \ov{q}(i)$ for all $i\in\{1,\ldots,n\}$. 
  
  The complex $\widetilde{C}^*_{\ov{q}}(\ob{K};R)$, generated by the $\ov{q}$-admissible TW-cochains, $c$, whose boundary $dc$ is $\ov{q}$-admissible also, is called the \emph{Thom-Whitney complex} of $\ob{K}$, with coefficients in $R$. 
Its homology is  the \emph{Thom-Whitney cohomology of $\ob{K}$,} with coefficients in $R$,
 denoted
$H_{\TW,\ov{q}}^{*}(\ob{K};R)$.
When $R$ is a \emph{field,} these two complexes are connected by a quasi-isomorphism,
 $$\tC^*_{\ov{q}}(\ob{K};R)\xrightarrow[]{\simeq} C^*_{\GM,\ov{p}}(\ob{K};R),$$
 if $\ov{p}$ and $\ov{q}$ are GM-perversities verifying $\ov{q}\geq 0$ and $\ov{p}(j)+\ov{q}(j)=j-2$. 
 
 This opens a  treatment of intersection cohomology with coefficients in $\Z_{2}$ from  TW-cochains and 
 brings a new approach for the study of Steenrod operations in intersection cohomology, as quoted by MacPherson in \cite[Question~4.3]{MR761809}. We  come back to this  aspect in \cite{2013arXiv1302.2737D}.

\section*{Rational Homotopy and Intersection-Formality}

The second part of this work is motivated by a presentation of the  intersection cohomology, similar to Sullivan's presentation of the  rational homotopy type, including the notion of minimal model,
with its topological invariance for PL-pseudomanifolds whose regular part is connected.
This rational intersection theory answers a question raised by Goresky (\cite[Introduction]{MR761809}) in view of a treatment of formality.

The classical rational theory of D.~Sullivan begins with  a functor from simplicial sets to the category of commutative differential graded  algebras (henceforth \cdga's), denoted by $A_{PL}$, 
such that  the cohomology of $A_{PL}(\ob{K})$ is isomorphic to the cohomology algebra of the simplicial set, $\ob{K}$, with rational coefficients.

For having such a functor  in the case of intersection cohomology, we   consider again a  construction on the blow-up of a simplex, similar to the construction of TW-cochains. Here, we replace $C^*(\Delta^i;R)$ by the cochain algebra of rational polynomial forms on $\Delta^i$, $A_{PL}(\Delta^i)$.  We set
 $\widetilde{A}_{PL}(\Delta)=A_{PL}(c\Delta^{j_0})\otimes \cdots\otimes 
 A_{PL}(c\Delta^{j_{n-1}})\otimes A_{PL}(\Delta^{j_n})$.
 A (global) form on a \ffs, $\ob{K}$, is a family of $\omega_{\sigma}\in \widetilde{A}_{PL}(\Delta)$, for each $\sigma\colon\Delta\to \ob{K}$, with compatibility conditions to face operators.
 We denote by $\widetilde{A}_{PL}(\ob{K})$ the \cdga~of global forms on $\ob{K}$.
 
  For any global form $\omega\in\widetilde{A}_{PL}(\ob{K})$, we  define also a \emph{perverse degree,} $\|\omega\|$, in a similar manner than for TW-cochains. For
  any loose perversity $\ov{q}$, a form $\omega$   is said \emph{$\ov{q}$-admissible} if $\|\omega\|_i\leq \ov{q}(i)$ for all 
  $i\in\{1,\ldots,n\}$. 
  The complex $\widetilde{A}_{PL,\ov{q}}(\ob{K})$ is generated by the $\ov{q}$-admissible forms, $\omega$, 
  whose boundary $d\omega$ is $\ov{q}$-admissible also.
We establish first a ``De Rham theorem,'' by proving  that the integration map  induces a quasi-isomorphism,
   $$\int\colon \widetilde{A}_{PL,\ov{q}}(\ob{K})\to C^*_{\GM,\ov{p}}(\ob{K};\Q),$$
    if $\ov{p}$ and $\ov{q}$ are  GM-perversities, verifying $\ov{q}\geq 0$ and $\ov{p}(j)+\ov{q}(j)=j-2$.
    From the previous comparison between GM-cohomology and TW-cohomology, we obtain isomorphisms,
$$H^*_{\GM,\ov{p}}(\ob{K};\Q)\cong H^*_{\TW,\ov{q}}(\ob{K};\Q)\cong H^*(\widetilde{A}_{PL,\ov{q}}(\ob{K})).$$
 If $\ob{K}$ is the \ffs~associated to a pseudomanifold, $X$, 
 these algebras are also isomorphic to the algebra of cohomology defined by G.~Friedman and J.~McClure in \cite{MR3046315}.

In \cite{MR2544388}, M. Hovey defines the algebraic category, $\cdgaf$,  of perverse \cdga's as the monoids in the category of functors from the lattice of GM-perversities to a category of cochain complexes endowed with a perverse degree. Our previous construction of the blow-up of Sullivan's forms gives the expected functor,
$\cA\colon \fil\to \cdgaf$, $\ob{K}\mapsto \cA(\ob{K})_{\bullet}$, with $\cA(\ob{K})_{\ov{q}}=\widetilde{A}_{PL,\ov{q}}(\ob{K})$.

The second main ingredient of Sullivan's theory is the notion of \emph{minimal model.} Here, we define Sullivan minimal perverse models and prove their \emph{existence} (with some connectivity conditions, as in the classical case) and \emph{unicity,}  up to isomorphism.
Some of our results are established in the general case of loose perversities but the construction of the perverse minimal model relies strongly on the  structure of the lattice of  GM-perversities. 
This was first observed by Stienne (\cite{Stienne}) who determined the properties of the predecessors of a GM-perversity and a part of the construction of our perverse minimal model  is inspired from his work. 
If $X$ is a PL-pseudomanifold whose regular part is connected, we establish \emph{the topological invariance of its Sullivan minimal perverse model.}
    
Finally, this rational setting gives a notion of formality and we say that a
    connected \ffs, $\ob{K}$, is  \emph{intersection-formal} if there is an isomorphism 
between the Sullivan minimal perverse models of $\cA(\ob{K})_{\bullet}$ and  $H_{\bullet}(\cA(\ob{K}))$. 
The analogue of triple Massey products is also introduced and their relationship with intersection-formality is detailed. 
We continue with a series of examples as the cone on a manifold which furnishes  a formal pseudomanifold which is not inter\-sect\-ion-formal. We present also (non cofibrant) models of pseudomanifolds with isolated singularities. As an illustration of this construction, we prove that  Thom spaces of vector bundles on a formal space,  projective cones of a smooth projective variety, 
the Calabi Yau quintic, and more generally, \emph{any nodal hypersurface in $\CP(4)$, are intersection-formal.}

\section*{Outline of the paper}

 We supply the  necessary definitions and properties of intersection homology along the text but, for a more complete background, the reader can  consult  the  original papers or one of the books that have appeared in the last years as, for instance,   \cite{MR2401086}, \cite{MR2207421}, \cite{MR932724}, \cite{MR2662593}, \cite{MR2286904}, \cite{IHGreg} or the historical development presented in \cite{MR2330160}.
 
\medskip \noindent
The topological setting is concentrated in the  {\sc Appendix,} with a recall of the intersection homology complex $K^{\ov{p}}_{*}(X)$, and the presentation of a new complex, $C_{*}^{\ov{p}}(X)$, for any filtered space, $X$, and any loose perversity, $\ov{p}$. Several properties of  $C_{*}^{\ov{p}}(X)$ are established:
  
\begin{itemize}
 \item \propref{prop:propertiesintersectionhomology}: The homology theory associated to $C_{*}^{\ov{p}}(X)$ owns a Mayer-Vietoris exact sequence, for any loose perversity $\ov{p}$.
 \item \thmref{thm:stratifiedmapcomplex}, page~\pageref{thm:stratifiedmapcomplex}: For any loose perversity $\ov{p}$, a stratified map (see \defref{def:appstratifiee}) $f\colon X\to Y$, induces a chain map $f_*\colon C_*^{\ov{p}}(X)\to C_*^{\ov{p}}(Y)$.
 \item \thmref{thm:pseudostratifie}, page~\pageref{thm:pseudostratifie}: A CS set (see \defref{def:CS}) is a stratified space.
 \item \propref{prop:intersectionetintersection}: In the case of a pseudomanifold
 (see \defref{def:pseudo}), the complex $C_{*}^{\ov{p}}(X)$ gives the same intersection homology than the Goresky and MacPherson original one.
 \end{itemize}
 \thmref{thm:stratifiedmapcomplex} is a key point in the proof of the topological invariance of the minimal model, stated in \secref{sec:modelCS}. 

 \medskip \noindent
 The description of a simplicial setting for the study of intersection cohomology is done in {\sc \chapref{chap:blowup}.}
Various cohomologies are defined on a \ffs, $\ob{K}$, from the blow-up of universal sytems on euclidean simplices, as, for instance,
the cochains and the forms, $\tC^*(\ob{K};R)$ and $\widetilde{A}_{PL,\ov{q}}(\ob{K})$, introduced above.
\thmref{thm:extendable} (Page \pageref{thm:extendable}) gives a sufficient condition for having an isomorphism between two such theories and \thmref{thm:thetwocochains} (Page \pageref{thm:thetwocochains}) proves the existence of an isomorphism between GM and TW cohomologies, with coefficient in a field, for appropriate GM-perversities.
 
In \emph{\secref{sec:normalisation},} we determine particular cases and examples. 
 In the case $\ov{p}\equiv\ov{\infty}$, the intersection cohomology is the cohomology of the \emph{regular} part, $\ob{K}^{[0]}$, which consists of simplices $\Delta^{j_0}\ast\cdots\ast\Delta^{j_n}$ such that $\Delta^{j_i}=\emptyset$ if $i\leq n-1$. Also, the   $\ov{0}$-intersection cohomology is the cohomology of the face set associated to $\ob{K}$, when  $\ob{K}$ is a normal \ffs~(see \defref{def:ffsregulier}). 
 As it is done in \cite{MR572580} for pseudomanifolds, we prove that any \ffs, $\ob{K}$, admits a unique normalization, $N(\ob{K})$, which has the same intersection cohomology than $\ob{K}$, see \propref{prop:existencenormal}. 
 Examples of cone and suspension of a face set~are also detailed. 
\emph{In \secref{sec:homotopyffs},} we  supply a cylinder object, define the product of a \ffs~with a face set and  a notion of homotopy between filtered face maps.

 \medskip \noindent
In {\sc \chapref{chap:rational},} we develop our rational homotopy project. (The reader can consult \cite{MR0646078}, \cite{MR0258031}, \cite{MR1802847}, \cite{MR2403898}, \cite{MR736299}  or  \cite{MR764769} for a presentation of  algebraic models for the rational homotopy type.)
The ring of coefficients is the rational field $\Q$ in all this chapter.

\emph{\secref{sec:perversedifferentialalgebras}} is mainly a brief recall of Hovey's work (\cite{MR2544388}), augmented with the definition and unicity of the perverse minimal model.
In \emph{\secref{sec:balanced},} we focus on the properties of particular perverse \cdga's, called balanced and based on an explicit description of the predecessors of a GM-perversity. 
The main examples of balanced perverse \cdga's are the blow-ups, $\tF(\ob{K})$, 
and their cohomology, $H_{\bullet}(\ob{K};F)$, for any \ffs, $\ob{K}$, and any universal system of \cdga's, $F$.
In \thmref{thm:constructionminimalmodel} (Page \pageref{thm:constructionminimalmodel}) of \emph{\secref{sec:minimalalgebraic},}  we prove that any balanced, cohomologically connected, perverse \cdga~admits a 
Sullivan minimal perverse model. 
 
In \emph{\secref{sec:modelCS},} we consider a first series of geometric properties of the perverse minimal model of a \ffs, $\ob{K}$, which, by definition, is the Sullivan minimal perverse model of $\cA(\ob{K})_{\bullet}$. 
If $\ob{K}$ is connected and normal, the \cdga~of elements of perverse degree~0 in the minimal model of $\cA(\ob{K})_{\bullet}$ is the minimal model of the face set associated to $\ob{K}$. 
We prove also that the two perverse algebras of cohomology, built from the PL-forms on one side and from the Thom-Whitney cochains on the other side, are isomorphic. 
Moreover, if $X$ is a PL-pseudomanifold whose regular part is connected, we establish the topological invariance of its minimal model, see \thmref{thm:topologicalinvarianceofminimal}, page \pageref{thm:topologicalinvarianceofminimal}. This uses 
the existence of an intrinsic filtration, due to Sullivan, and detailed in \cite{MR800845}, see also \cite{IHGreg}.

 \medskip \noindent
{\sc \chapref{chap:formalityexamples}} is a presentation of intersection-formality with examples,
and the ring of coefficients is the rational field $\Q$ in all this chapter. \emph{\secref{subsec:formality}} contains the definition of intersection-formality, its link with some Massey products and the examples of cone and suspension of a connected face set. In \emph{\secref{subsec:isolated}}, we present an ad hoc perverse model for pseudomanifolds with isolated singularities. This model is applied to the study of Thom spaces in \emph{\secref{subsec:Thom}}, where an explicit perverse model of Thom spaces  is provided. We use it to establish the intersection-formality of the singular quadrics, see \propref{prop:projectiveformal}, page \pageref{prop:projectiveformal}.

\emph{\secref{subsec:nodal}} is concerned with the \emph{intersection-formality of the nodal hypersurfaces in $\C P(4)$,} see \thmref{thm:nodalformal}, page \pageref{thm:nodalformal}. 
Let $\ov{V}$ be a nodal hypersurface, $\ov{V}_{\reg}$ its regular component, $V$ a small resolution of $\ov{V}$ and $W$ a second resolution obtained from $V$ by a succession of blow-ups of
$\C P(1)$'s. As the regular part, $\ov{V}_{\reg}$, is a complement of a divisor in $W$, we may use a model of J. Morgan (\cite{MR516917}) for having a model of $\ov{V}_{\reg}$. The next step in the proof of \thmref{thm:nodalformal} is the construction of a model of $\ov{V}$ using the results of  \secref{subsec:isolated}. Finally, an explicit computation shows the intersection-formality of $\ov{V}$.

 \medskip \noindent
 
{\hrule width 2cm}
 
 \medskip\noindent
 In all this text, for a graded perverse object, $a$, we denote by $|a|$ its degree and by $\|a\|$ its perverse degree. A chain map is called a quasi-isomorphism (or a weak equivalence) if it induces an isomorphism in homology.

 \medskip \noindent
 
{\hrule width 2cm}
 
 \medskip\noindent
 
We would like to thank Greg Friedman for reading a previous version of this work and making valuable suggestions which have contributed to improve the writing. 